%% file: arxiv_v2.tex
\documentclass[dvipsnames,svgnames,table]{amsart}
\usepackage[T1]{fontenc}
\usepackage{a4wide, amsfonts, amsmath, amssymb, amsthm, calc, colortbl, enumitem, graphics, graphicx, hvfloat, hyperref, mathtools, multicol, multido, multirow, pdflscape, pifont, shuffle, thmtools, url, wasysym, xargs, xcolor, bm, booktabs, wasysym}
\hypersetup{colorlinks=true, citecolor=RoyalBlue, linkcolor=RoyalBlue, hypertexnames=false, bookmarksdepth=3, bookmarksopen}
\usepackage{tikz}
\usetikzlibrary{angles, arrows, arrows.meta, backgrounds, calc, calligraphy, cd, decorations, decorations.markings, decorations.pathmorphing, decorations.pathreplacing, fit, intersections, matrix, patterns, positioning, shapes, shapes.misc, trees}
\usepackage{tikz-qtree}
\usepackage{comment}
\usepackage{ulem}
\usepackage[noabbrev,capitalise,nameinlink]{cleveref}
\usepackage[export]{adjustbox}
\newcommand{\set}[2]{\left\{ #1 \;\middle|\; #2 \right\}}

\newcommand{\bigset}[2]{\bigg\{ #1 \;\bigg|\; #2 \bigg\}}

\usepackage[utf8]{inputenc}
\usepackage{lmodern}
\usepackage[
    protrusion=true,
    activate={true,nocompatibility},
    final,
    tracking=true,
    kerning=true,
    spacing=true,
    factor=1100]{microtype}
\microtypecontext{spacing=nonfrench}
\SetTracking{encoding={*}, shape=sc}{40}

\numberwithin{equation}{section}

\theoremstyle{definition}
\newtheorem{theorem}{Theorem}[section]
\newtheorem{definition}[theorem]{Definition}
\newtheorem{convention}[theorem]{Convention}
\newtheorem{lemma}[theorem]{Lemma}
\newtheorem{proposition}[theorem]{Proposition}
\newtheorem{corollary}[theorem]{Corollary}
\newtheorem*{theorem*}{Theorem}
\newtheorem{problem}[theorem]{Problem}
\theoremstyle{remark}
\newtheorem{remark}[theorem]{Remark}
\newtheorem{example}[theorem]{Example}

\newcommand{\blossom}{\text{\ding{96}}}
\newcommand{\vdim}{\underline{\dim}}
\newcommand{\pref}{\mathsf{prefix}}
\newcommand{\suf}{\mathsf{suffix}}
\DeclareMathOperator{\inc}{in}
\DeclareMathOperator{\out}{out}
\DeclareMathOperator{\In}{Pref}
\DeclareMathOperator{\Out}{Suff}
\DeclareMathOperator{\indeg}{indeg}
\DeclareMathOperator{\outdeg}{outdeg}
\DeclareMathOperator{\Routes}{\mathcal{R}}
\DeclareMathOperator{\head}{head}
\DeclareMathOperator{\tail}{tail}
\DeclareMathOperator{\Hom}{Hom}
\DeclareMathOperator{\Adj}{Adj}
\DeclareMathOperator{\conv}{conv}
\DeclareMathOperator{\spann}{span}
\newcommand{\cyclicflows}{\F^{\circlearrowleft}}
\newcommand{\funny}{\F_{\bf 0}^{\circlearrowleft}}
\newcommand{\bdd}{\widehat{\F}}
\DeclareMathOperator{\IndShift}{\mathcal{T}}
\DeclareMathOperator{\KentuckyBij}{\Gamma}
\DeclareMathOperator{\RoutesToStrings}{\Phi}
\DeclareMathOperator{\coloring}{\omega}
\newcommand{\tcb}{\textcolor{blue}}
\newcommand{\tcr}{\textcolor{red}}
\newcommand{\doppel}{\mathcal{M}}
\newcommand{\F}{\mathcal{F}}
\newcommand{\Fred}{\F_{\rm red}}
\newcommand{\cE}{\mathcal{E}}

\newcommand{\exc}{\spann_\RR(\cE)}
\newcommand{\ZZ}{\mathbb{Z}}
\newcommand{\NN}{\mathbb{N}}
\newcommand{\KK}{\mathbb{K}}
\newcommand{\RR}{\mathbb{R}}
\newcommand{\cyclegraph}{H}
\newcommand{\route}{\rho}
\newcommand{\Cliques}{\mathcal{K}}
\newcommand{\MaxCliques}{\Cliques_{\rm max}}
\newcommand{\Cycles}{{\sf Cycles}}
\newcommand{\ve}{\mathbf{e}}
\newcommand{\vx}{\mathbf{x}}
\newcommand{\DKK}{{\sf DKK}}
\newcommand{\DKKred}{\DKK_{\rm red}}
\newcommand{\NK}{{\sf NK}}
\newcommand{\NKred}{\NK_{\rm red}}
\newcommand{\SP}{{\sf SP}}
\newcommand{\HN}{{\sf HN}}
\newcommand{\Normal}{\mathcal{N}}
\newcommand{\gfan}{\pmb{g}{\sf fan}}
\newcommand{\Top}{\operatorname{\sf Top}}
\newcommand{\Bot}{\operatorname{\sf Bot}}
\newcommand{\peak}{\operatorname{\sf peaks}}
\newcommand{\deep}{\operatorname{\sf valleys}}
\newcommand{\h}{\operatorname{\sf c}}
\newcommand{\rep}{\operatorname{rep}}
\DeclareMathOperator{\vol}{vol}
\newcommand{\cyclohedron}{{\mathcal{C}}}
\newcommand{\sources}{{\mathsf{sources}}}
\newcommand{\sinks}{{\mathsf{sinks}}}
\newcommand{\internal}{{\mathsf{internal}}}

\newcommand{\ool}{\overline{\overline{\Lambda}}}

\newcommand{\blackXgrah}{$\bm{\diagup}\kern-3.3mm\bm{\diagdown}$}
\newcommand{\blackXXgraph}{{\raisebox{-.8mm}{\blackXgrah}\kern-1.7mm\raisebox{.8mm}{\blackXgrah}}}
\newcommand{\colorfulX}{$\color{blue}{\bm{\diagup}}\kern-3.3mm\color{red}{\bm{\diagdown}}$}
\newcommand{\colorfulXXgraph}{{\raisebox{-.8mm}{\colorfulX}\kern-1.7mm\raisebox{.8mm}{\colorfulX}}}
\newcommand{\colorfulHashGraph}{{\kern2mm\raisebox{-.6mm}{\colorfulXXgraph}\kern-6.5mm\raisebox{1.3mm}{\colorfulXXgraph}\kern2mm}}
\newcommand{\TheBlackCycletikz}{\begin{tikzpicture}[scale=.15]
            \draw[thick] (0,0) circle (25pt);
            \draw[thick] ( 45:10pt) -- ( 45:40pt);
            \draw[thick] (135:10pt) -- (135:40pt);
            \draw[thick] (225:10pt) -- (225:40pt);
            \draw[thick] (315:10pt) -- (315:40pt);
        \end{tikzpicture}}
\newcommand{\TheCycletikz}{\begin{tikzpicture}[scale=.15]
            \draw[thick,red] (0,0) circle (25pt);
            \draw[thick,blue] ( 45:10pt) -- ( 45:40pt);
            \draw[thick,blue] (135:10pt) -- (135:40pt);
            \draw[thick,blue] (225:10pt) -- (225:40pt);
            \draw[thick,blue] (315:10pt) -- (315:40pt);
        \end{tikzpicture}}

\newcommand{\TheBlackCycle}[1]{\raisebox{-.6mm}{$\TheBlackCycletikz_{#1}$}}
\newcommand{\TheCycle}[1]{\raisebox{-.6mm}{$\TheCycletikz_{#1}$}}

\DeclareMathOperator{\nf}{netflow}

\definecolor{defcolor}{HTML}{e2062c}
\newcommand{\textnew}[1]{\textit{\textbf{\color{defcolor}{#1}}}}
\newcommand{\mathnew}[1]{{\color{defcolor}#1}}

\title[Flows on graphs with cycles, locally gentle algebras, and the Mutoperhedron]{Flow cones of graphs with cycles,\\ locally gentle algebras, and the Mutoperhedron}

\author[Abram]{Antoine Abram}
\address[A.~Abram]{Université de Liège, Belgique}
\email{antoine.abram@uliege.be}
\urladdr{\url{https://antoineabram.codeberg.page}}

\author[Bastidas]{Jose Bastidas}
\address[J.~Bastidas]{LACIM, Université du Québec à Montréal, Canada}
\email{bastidas.math@proton.me}
\urladdr{\url{https://bastidas-jose.codeberg.page}}

\author[Dequêne]{Benjamin Dequêne}
\address[B.~Dequêne]{School of Mathematics, University of Leeds, UK}
\email{b.d.dequene@leeds.ac.uk}
\urladdr{\url{https://sites.google.com/view/benjamin-dequene/}}

\author[Morales]{Alejandro H. Morales}
\address[A. H. Morales]{LACIM, Université du Québec à Montréal, Canada}
\email{morales\_borrero.alejandro@uqam.ca}
\urladdr{\url{https://sites.google.com/view/ahmorales}}

\author[Park]{GaYee Park}
\address[G. Park]{Department of Mathematics, Dartmouth College, USA}
\email{gayee.park@dartmouth.edu}
\urladdr{\url{https://sites.google.com/view/gayeepark}}

\author[Thomas]{Hugh Thomas}
\address[H. Thomas]{LACIM, Université du Québec à Montréal, Canada}
\email{thomas.hugh\_r@uqam.ca}
\urladdr{\url{https://hughrthomas.github.io/}}

\begin{document}

\begin{abstract}
Flow cones of a directed acyclic graph admit a family of unimodular triangulations given by Danilov, Karzanov, and Koshevoy (DKK) whose normal fans are related to (generalizations) of the associahedron and permutohedron. A correspondence between these triangulations for certain graphs and maximal cones of a $g$-vector fan of a gentle quiver associated to the graph was discovered by von Bell, Braun, Bruegge, Hanely, Peterson, Serhiyenko, and Yip in 2022. This correspondence has been fruitful in uncovering lattice structures in the triangulations. We start by showing that this correspondence is actually a linear isomorphism. We then consider flow cones of certain graphs with cycles. For this case, we give a DKK-like triangulation of the cone, and extend the correspondence to the finite $g$-vector fan of a corresponding locally gentle quiver. In addition, we extend to cyclic graphs a mysterious result of Postnikov--Stanley and Baldoni--Vergne, giving the volume of flow polytopes of acyclic graphs as the number of certain integer flows on the same graph.

We illustrate our results with a two-parameter family of cyclic graphs that includes a cycle graph and nested 2-cycles as special cases. We show that the fans of its DKK-like triangulations are respectively isomorphic to the normal fan of the cyclohedron and of a new polytope with the same $f$-vector but different combinatorial type than the permutohedron, which we call the mutoperhedron.
\end{abstract}

\vspace*{-1cm}
\maketitle
\setcounter{tocdepth}{1}
\tableofcontents
\vspace*{-1cm}

\section{Introduction}

Representation theory of quivers studies a specific category associated to a directed graph. One might therefore imagine that there would be plentiful connections between the representation theory of quivers and other topics in the study of directed graphs. Up until recently, that has not been true --- indeed, there has been scarcely any communication between the two subjects. Happily, this is no longer the case. This paper is a contribution to a quite new connection between the representation theory of gentle quivers and the study of flows in directed graphs. This connection was first made in \cite{Kentuckygentle} by a research group at the University of Kentucky, of von Bell, Braun, Bruegge, Hanely, Peterson, Serhiyenko, and Yip, and has since been further developed in \cite{BerSer_Wilting,BERGGREN} by Berggren--Serhiyenko and Berggren, respectively.

We will now provide a brief introduction to both sides of this connection. Full details of definitions can be found later in the paper.

On the one hand, we are interested in flows in a directed graph $G=(V,E)$. Typically, one assumes that $G$ has no cycles; however, as we will show, it can be productive to relax that assumption. The set $V$ of vertices is partitioned into three disjoint sets: sources, sinks, and internal vertices. A \textnew{balanced flow} is an assignment of a non-negative real number weight to each edge of $G$, such that at any internal vertex, the sum of the weights of the incoming edges equals the sum of the weights of the outgoing edges. The collection of all balanced flows forms a cone $\F^+(G)$, and a particular slice of this cone gives rise to a polytope called the \textnew{flow polytope} $\F_1(G)$. This cone and polytope have been the subject of considerable study \cite{BV,berggren2025framingtriangulations,braun2024volumeinequalitiesflowpolytopes,s-permutahedra,Hille,MeszarosMorales,Meszaros-StDizier,RietschWilliams}.  In particular, one is interested in triangulating $\F^+(G)$, i.e., in dividing it into simplicial, unimodular cones \cite{Kentuckygentle,braun2025equatorialflowtriangulationsgorenstein,GonzalezHanusaYip,MMSt,vonBellCeballos,vonBellGonzalezCetinaYip}.
Danilov, Karzanov, and Koshevoy constructed a family of such triangulations \cite{DKK}. These triangulations depend on some auxiliary data called a \textnew{framing}; for now, we shall ignore this. We write $\DKK(G)$ for the triangulation of $\F^+(G)$, which is itself a polyhedral fan. This fan has certain rays that are contained in every maximal cone. It is natural to quotient $\DKK(G)$ by the subspace generated by these rays, obtaining a reduced fan $\DKKred(G)$.

Suppose that the graph $G$ has the property that each internal vertex has two incoming arrows and two outgoing arrows.
The key insight of \cite{Kentuckygentle} is that, given a choice of \textnew{ample framing} for $G$, it is possible to associate to $G$ a gentle quiver $(Q_G,R_G)$, defining a finite-dimensional algebra. Then, the maximal cones of $\DKK(G)$ (or of $\DKKred(G)$) are in natural correspondence with the maximal cones of the $g$-vector fan of $(Q_G,R_G)$.

The $g$-vector fan is an invariant of a finite-dimensional algebra, introduced by Adachi, Iyama, and Reiten in \cite{adachi-iyama-reiten}. Since that time, it has received considerable attention \cite{HPS18, KPY23, PPPP23,PPP}.
Making the connection between triangulations of the flow cone and the $g$-vector fan of a corresponding algebra is not just a curiosity: it also has implications on both sides of the connection. This has already been exploited in \cite{Kentuckygentle} and \cite{BerSer_Wilting}.

On the flow cone side, the $g$-vector fan is known to be simplicial and unimodular, so, as we shall discuss further below, this provides a new proof of unimodularity of the DKK triangulation for these graphs.  Also, and more significantly, it suggests the possibility of widening the range of graphs $G$ that can be considered. While flows on graphs with cycles are just beginning to be considered \cite{BERGGREN}, for certain graphs with cycles, there is still a corresponding quiver with relations $(Q_G,R_G)$. This quiver may not be gentle, but it is still locally gentle. Thanks to the work in \cite{aokiyurikusa,PPPlocg}, we can associate to it a finite $g$-vector fan, which we show contains a triangulation of the flow cone.

On the other hand, for algebraists interested in $g$-vector fans, our results provide new approaches to their study, in particular an algorithmic method for enumerating maximal cones.

We now summarize the contributions of our paper in more detail.

\subsection{Linear isomorphism of $g$-vector fan and $\DKKred(G)$}
\label{ss:intro_linear_iso}

Let $G=(V,E)$ be a connected acyclic directed graph.
A framing $F$ on $G$ is a choice, for each internal vertex $v$ of $G$, of an order for its incoming edges and an order for its outgoing edges.
The framing $F$ is said to be \textnew{ample} if the corresponding reduced fan $\DKKred(G,F)$ is {\em complete}.
In \cite{Kentuckygentle}, the authors show that ample framings exist essentially only for \textnew{full} graphs; that is, graphs in which each internal vertex has exactly two incoming and two outgoing edges, and that such a framing is equivalent to a certain bi-colouring of the edges of $G$. See \cref{ss:DAG_to_gentle} for details.
Moreover, using a different terminology, they prove that the reduced fan $\DKKred(G,F)$ is combinatorially isomorphic to the $g$-vector fan of an associated gentle quiver $(Q_G,R_G)$.

Our first result shows that the combinatorial isomorphism between the reduced fan $\DKKred(G,F)$ and $\gfan(Q_G,R_G)$ is, in fact, a {\em linear} isomorphism.  We achieve this by constructing an explicit linear map $\phi$ (\cref{def:phi}) that realizes this isomorphism. This map was independently considered by Berggren in~\cite{BERGGREN}.

\begin{theorem}[\cref{prop:linear-iso}]
    Let $(G,F)$ be an amply framed DAG and $(Q_G,R_G)$ be the associated reduced gentle quiver.
    The map $\phi$ above induces a linear isomorphism between $\DKKred(G,F)$ and $\gfan(Q_G,R_G)$.
\end{theorem}

\subsection{$g$-vector fans of locally gentle quivers}

Palu, Pilaud, and Plamondon introduce in \cite{PPP} the reduced non-kissing complex $\NKred(Q,R)$ of a gentle quiver $(Q,R)$ and prove that the $g$-vector fan of $(Q,R)$ is a realization of this complex.
In a follow-up paper, \cite{PPPlocg}, they extend the definition of $\NKred(Q,R)$ to the case where $(Q,R)$ is only locally gentle. In the locally gentle case, they do not show that $\NKred(Q,R)$ can be interpreted as a $g$-vector fan; in fact, they do not show that $\NKred(Q,R)$ admits a realization as a fan.

Subsequently, Aoki and Yurikasa showed in \cite{aokiyurikusa} how to define a $g$-vector fan for the complete gentle algebra defined from a locally gentle quiver $(Q,R)$. We show that indeed $\NKred(Q,R)$ can be interpreted as a $g$-vector fan.

\begin{theorem}[\cref{thm:connection_g-vector_NKred}]
    If $(Q,R)$ is a locally gentle quiver, the $g$-vector fan of the complete gentle algebra defined by $(Q,R)$ is given by a suitable embedding of $\NKred(Q,R)$ into a real vector space.
\end{theorem}

If an algebra has only finitely many brick representations, then it is known that its $g$-vector fan is complete, and it is the outer normal fan of a Minkowski sum of {\em Harder--Narasimhan (HN) polytopes} \cite{aoki2024fanspolytopestiltingtheory, JF2}.
We further show that, under good circumstances, this result extends to the locally gentle case.

\begin{proposition}[\cref{prop:locallygentlegvectorfan}]
    If $(Q,R)$ is a locally gentle quiver such that every non-oriented cycle contains a relation, then the $g$-vector fan of $(Q,R)$ is the outer normal fan of a suitable sum of Harder--Narasimhan polytopes.
\end{proposition}

As a consequence of this result and the linear isomorphism discussed in \cref{ss:intro_linear_iso},
we obtain a realization of the reduced fan $\DKKred(G,F)$ of an amply framed DAG $(G,F)$ as the normal fan of a Minkowski sum of \emph{order polytopes} of {\em fence posets} (equivalently, a Minkowski sum of Harder--Narasimhan polytopes).

\begin{theorem}[\cref{cor:DKKred-is-polytopal}]
    The reduced fan $\DKKred(G,F)$ is polytopal. Explicitly, it is the normal fan of a Minkowski sum of order polytopes.
\end{theorem}

\subsection{Triangulations of flow cones of cyclic graphs}
\label{ss:intro-tiangulations_cyclic}

After studying the case of flow cones of directed acyclic graphs, we relax the acyclic assumption and allow full graphs that have directed cycles.
These full cyclic graphs $\cyclegraph$ must have a variant of the ample framing called \textnew{cyclic ample framing} requiring every minimal directed cycle to be monochromatic.
Note that, in his study of flows for graphs with cycles, Berggren \cite{BERGGREN} has the opposite assumption; that is, every directed cycle is required to be non-monochromatic.
We show that the flow cone $\F^+(\cyclegraph)$ of a cyclic graph $\cyclegraph$ with a cyclic ample framing has a DKK-like unimodular triangulation denoted by $\DKK(\cyclegraph)$.

\begin{theorem}[\cref{thm:cyclic-DKK}]
    Let $\cyclegraph$ be a graph with a cyclic ample framing.
    Then, $\DKK(\cyclegraph)$ is a unimodular triangulation of the flow cone $\F^+(\cyclegraph)$ of $\cyclegraph$.
\end{theorem}

As in the acyclic case, there is a reduced fan $\DKKred(\cyclegraph)$ that is complete whenever $\cyclegraph$ is endowed with a cyclic ample framing.
We show that $\DKKred(\cyclegraph)$ is linearly isomorphic to a subfan of the $g$-vector fan of a locally gentle quiver constructed from $\cyclegraph$ (\cref{thm:quotient-subfan}).
Consequently, the reduced fan $\DKKred(\cyclegraph)$ is the normal fan of a Minkowski sum of projections of Harder--Narasimhan polytopes (\cref{thm:polytope-dual-to-cyclic-gfan}).

\subsection{Cyclic volume integer flows}

In the case of acyclic graphs $G$,
since any unimodular triangulation of the flow cone $\F^+(G)$ restricts to a unimodular triangulation of the flow polytope~$\F_1(G)$,
we can compute the volume $\vol(\F_1(G))$ of the flow polytope by counting the number of maximal simplices in $\DKK(G,F)$.
This contrasts with an earlier intriguing formula by Postnikov--Stanley \cite{PS} and Baldoni--Vergne \cite{BV}, which gives the volume as the number of certain integer flows on $G$ satisfying a variant of the balanced flow condition.
That is, for a directed acyclic graph $G$, we have that
\begin{equation} \label{eq: vol as number int flows}
\vol(\F_1(G)) = \# \F_{\bf 0}^\NN(G),
\end{equation}
where $\# \F_{\bf 0}^\NN(G)$ is the number of integer flows on $G$ with zero flow on the sources of $G$ and net flow at each internal vertex $v$ given by its indegree minus 1 (see \cref{thm:innerflowsindeg}).  Because of this result, we call the integer flows counted on the RHS above \textnew{volume integer flows}. In practice, counting these integer flows is an effective tool to calculate the volume of the flow polytope \cite{CKM,Z}.
Moreover, thanks to the link between DKK triangulations and $g$-vector fans, counting these flows also becomes a way to enumerate maximal cones in these latter fans.
Lastly, Morales--Mészáros--Striker gave in \cite{MMSt} a correspondence between the maximal cones of the DKK triangulation of a framed DAG $(G,F)$ and the volume integer flows on $G$.

We extend \eqref{eq: vol as number int flows} to the case of graphs with cyclic ample framings.
To do this, first we define the \textnew{cyclic volume integer flows} of $\cyclegraph$ and extend the correspondence from \cite{MMSt} to give a bijection between these flows and the maximal cliques in the DKK-like triangulation of $\cyclegraph$ (\cref{prop:cliques-to-cyclic-funny}).
When $\cyclegraph$ has at least one directed cycle, the polyhedron $\F_1(\cyclegraph)$ is unbounded and has no finite volume.
We thus introduce a polyhedral complex called the \textnew{flow complex} of $\cyclegraph$ denoted $\bdd(\cyclegraph)$ by restricting the amount of flow that can loop completely around a cycle.
This complex coincides with the flow polytope whenever $\cyclegraph$ is acyclic, and its volume is consistently given by the number of cyclic volume integer flows of $\cyclegraph$.

\begin{theorem}[\cref{thm:normVolCyclicIntegerFlow}]
    Let $\cyclegraph$ be a directed graph with a cyclic ample framing, then
    \[
        \vol\,\bdd(\cyclegraph) = \#\funny(\cyclegraph),
    \]
    where $\#\funny(\cyclegraph)$ is the number of cyclic volume integer flows of $H$.
\end{theorem}

\subsection{Cyclohedron, mutoperhedron, and beyond}

One trend in the study of volumes and DKK triangulations of flow polytopes of acyclic graphs is to find families of graphs whose flow polytopes have interesting volume formulas \cite{CKM,JangKim,Yip2019,Z}. Another more recent trend is to realize geometrically known posets, complexes, and lattices as the dual graphs of DKK triangulations for a suitable choice of a graph and framing \cite{Kentuckygentle,s-permutahedra,MeszarosSubword,vonBellCeballos,vonBellGonzalezCetinaYip}. For instance, for the path graph of $n-1$ internal vertices with an ample framing, the flow polytope has volume given by a Catalan number $\frac{1}{n+1}\binom{2n}{n}$ and the normal fan of the DKK triangulation is isomorphic to the normal fan of the {\em associahedron} \cite{vonBellGonzalezCetinaYip} (\cref{fig:full_graph_and_routes}).

We continue these trends in the case of flow cones of cyclic graphs using the results described above. For the cycle graph $\TheCycle{n}$ with $n$ internal vertices and a cyclic ample framing, we show that the fan of its reduced DKK-like triangulation is combinatorially isomorphic to the normal fan of the {\em cyclohedron}: the polytope that encodes the centrally symmetric triangulations of a regular $2n$-gon. (\cref{prop:DKKfan-cyclo} and \cref{fig:ex:DKKred_H3}). As a consequence of the results mentioned in \cref{ss:intro-tiangulations_cyclic}, we obtain a realization of the cyclohedron as a slice of the {\em type D associahedron}.

The cycle graph is part of a two-parameter family of graphs $H_{c,p}$ of $c$ nested monochromatic $p$-cycles (\cref{fig:Hcp}). By counting cyclic volume integer flows, we obtain a simple formula for the volume of the flow complex of $H_{c,p}$.

\begin{theorem}[\cref{thm:hkr-vol}]
    For the graph $\cyclegraph_{c,p}$ we have that
    \[
        \vol\bdd(\cyclegraph_{c,p}) =  \binom{(c+1)(p-1)}{p-1,\ldots,p-1}.
    \]
In particular, $\vol\bdd(\TheBlackCycle{p})=\#\funny(\TheBlackCycle{p})=\binom{2p-2}{p-1}$ and $\vol\bdd(\cyclegraph_{c,2})=\#\funny(\cyclegraph_{c,2})=(c+1)!$.
\end{theorem}

We call the new simple polytope dual to the DKK triangulation of $\F(H_{c,2})$ the \textnew{mutoperhedron} \textnew{$\doppel_c$}
(\cref{fig:doppel3}). The choice of this name stems from the following surprising result.

\begin{theorem}[\cref{thm:f-vect_doppel}]
 The mutoperhedron $\doppel_c$ and the permutohedron $\Pi_{c+1}$ have the same $f$-vector and $h$-vector. However, for $c \geq 3$, these polytopes are {\em not} combinatorially isomorphic.
\end{theorem}

\subsection*{Outline}  We begin, in \cref{sec:bg graphs flows}, by recalling crucial background on flows of a framed connected acyclic directed graph (DAG) $(G,F)$, including the flow cone, the flow polytope, and the Danilov--Karzanov--Koshevoy (DKK) triangulations of those objects. Then, in \cref{s:gentle}, we recall essential tools and results from the study of gentle algebras and complete gentle algebras, in particular focusing on $g$-vectors of $\tau$-rigid modules over them, and the properties of their $g$-vector fans.
This section already contains some new results,
relating $g$-vector fans of complete gentle algebras to the reduced non-kissing complexes of \cite{PPPlocg}.

In \cref{s:conections-acyclic},
we show that the DKK fan associated with the reduced DKK triangulation of the flow cone of $(G,F)$ is linearly isomorphic to the $g$-vector fan of the gentle algebra associated to $(Q_G, R_G)$. In \cref{s:HN}, we prove that the type $A_n$ shard polytopes of \cite{ppr20shard} are exactly the Harder--Narasimhan polytopes of the indecomposable modules of some gentle quiver.

Then, in \cref{s:flow-cyclics}, we relax our assumptions and work with directed graphs that admit directed cycles. In particular, we construct a DKK-like triangulation of the flow cone of these graphs.
In \cref{s:completegentlealg}, we generalize previous results by describing a linear isomorphism from the reduced triangulation $\DKKred(\cyclegraph)$ to a certain subfan of the $g$-vector fan of an associated locally gentle quiver $(Q_\cyclegraph,R_\cyclegraph)$.

Finally, we focus on the two-parameter family of graphs $H_{c,p}$ with cyclic ample framing in \cref{s:HCP} and study the mutoperhedron in \cref{ss:TheDoppel}.

\subsection*{Acknowledgements} This paper grew out of a working group at LACIM in Montr\'eal, which included Sarah Brauner, Jake Levinson, and Franco Saliola. We thank them for their contributions. We thank Jonah Berggren and Khrystyna Serhiyenko for discussions surrounding their work \cite{BerSer_Wilting,BERGGREN}, and also Rafael Gonz\'alez De L\'eon, Chris Hanusa, Joel Kamnitzer, Pierre-Guy Plamondon, and Martha Yip for helpful discussions. This work was facilitated by computer experiments using Sage~\cite{sagemath} and its geometric combinatorics features developed by the Sage-Combinat community~\cite{Sage-Combinat}.

The first-named author was partially supported by NSERC doctoral research scholarship. The third-named author thanks the \emph{CHARMS program grant (ANR-19-CE40-0017-02)}, and the \emph{Engineering and Physical Sciences Research Council (EP/W007509/1)} for their partial funding support.
The fourth and sixth-named authors acknowledge support from the FRQNT team grant 10.69777/341288.
The fourth-named author acknowledges the support of the Natural Sciences and Engineering Research Council of Canada (NSERC) funding reference number RGPIN-2024-06246, and was partially supported by the NSF grant DMS-2154019.
The sixth-named author acknowledges the support of the Natural Sciences and Engineering Research Council of Canada (NSERC) funding reference number RGPIN-2022-03960 and the Canada Research Chairs Program CRC-2021-00120.

\section{Background on directed graphs and flow polytopes}
\label{sec:bg graphs flows}

In this section, we recall some definitions and set our notation.
We mainly follow the conventions and notation of \cite{Kentuckygentle,DKK}.

\subsection{Notation,
walks,
and routes}
\label{ss:paths-and-routes}

Let $G = (V,E) = \big(V(G),E(G)\big)$ be a connected acyclic directed graph, commonly referred to as a \emph{DAG}.
For a vertex $v \in V$ of $G$,
we denote by \textnew{$\inc(v)$} and \textnew{$\out(v)$} the set of incoming and outgoing edges of $v$,
with sizes denoted by \textnew{$\indeg(v)$} and \textnew{$\outdeg(v)$}, respectively.
A vertex $v \in V$ is a \textnew{source} if $\indeg(v) = 0$, a \textnew{sink} if $\outdeg(v) = 0$, or \textnew{internal} if it is neither a source nor a sink.
We let \textnew{$\sources(G)$}, \textnew{$\sinks(G)$}, and \textnew{$\internal(G)$} denote the set of sources, sinks, and internal vertices of $G$, respectively.
If $G$ contains at least two vertices, these sets partition $V$.

Given an edge $e = (u,v) \in E$,
we say that $u$ is the \textnew{tail} and $v$ is the \textnew{head} of $e$, and we write
\[
    \mathnew{\tail(e)} := u
    \qquad\text{ and }\qquad
    \mathnew{\head(e)} := v.
\]
Finally, an edge $e \in E$ between internal vertices is called \textnew{idle} if it is either the only edge coming out of its tail or the only edge going to its head.

A(n oriented) \textnew{walk} in $G$
is either a single vertex $v \in V$\footnote{This will be called a {\em lazy path} in the language of quivers in \cref{ss:vocabulary}.} or a sequence of edges $e_1,\dots,e_k$ such that $\head(e_i) = \tail(e_{i+1})$ for all $i$.
A \textnew{route} is a maximal walk; that is, a walk that starts at a source and ends at a sink.
We let \textnew{$\Routes(G)$} denote the set of all routes of $G$.

Given a route $\route = e_1,\dots,e_k \in \Routes(G)$ passing through an internal vertex $v$,
the \textnew{prefix} and \textnew{suffix} of $\route$ at $v$ are the
subwalks
\[
    \mathnew{\pref(\route,v)} := e_1,\dots,e_i
    \qquad\text{ and }\qquad
    \mathnew{\suf(\route,v)} := e_{i+1},\dots,e_k,
\]
where $i \in [k-1]$ is the index such that $\head(e_i) = v = \tail(e_{i+1})$.
Lastly, let \textnew{$\In(v)$} and \textnew{$\Out(v)$} be the set of
walks from a source to $v$ and from $v$ to a sink, respectively.
That is,
\begin{align*}
    \mathnew{\In(v)} & :=
    \set{ \pref(\route,v) }{ \route \in \Routes(G) \text{ goes through } v }, &
    \text{ and }\\
    \mathnew{\Out(v)} & :=
    \set{ \suf(\route,v) }{ \route \in \Routes(G) \text{ goes through } v }.
\end{align*}

\subsection{Flows}
\label{ss:flows-defs}

A \textnew{flow} $f$ of a graph $G$ is an assignment of a real number $f(e)$ to each edge $e \in E$ of~$G$. That is, a flow is an element of $\RR^E$.
Given a flow $f \in \RR^E$ and a vertex $v \in V$, the \textnew{netflow} of $f$ at $v$ is
\[
    \mathnew{\nf_f(v)} :=
    \sum_{e\in \out(v)} f(e) - \sum_{e \in \inc(v)} f(e).
\]
A flow $f \in \RR^E$ is said to be \textnew{balanced} if $\nf_f(v) = 0$ for each internal vertex $v \in \internal(G)$.

The \textnew{flow space} of $G$ is the subspace
\[
    \mathnew{\F(G)} := \set{ f \in \RR^E }{ f \text{ is balanced} }
\]
of all balanced flows, and the \textnew{flow cone} $\F^+(G)$ is the polyhedral cone in $\F(G)$ consisting of nonnegative balanced flows. That is,
\[
    \mathnew{\F^+(G)} := \F(G) \cap (\RR_{\geq 0})^E = \set{ f \in \F(G) }{ f(e) \geq 0 \text{ for all } e \in E }.
\]
The facet-defining inequalities of $\F^+(G)$ inside $\F(G)$ are precisely $f(e) \geq 0$ for non-idle edges $e$,
and the rays of $\F^+(G)$ are determined by routes, as we now explain.
We identify routes $\route \in \Routes(G)$ with the unit flow ${\bf 1}_\route \in \RR^E$ along it:
\[
    \mathnew{{\bf 1}_\route(e)} := \begin{cases}
        1 & \text{if } e \in \route,\\
        0 & \text{otherwise.}
    \end{cases}
\]
The rays of $\F^+(G)$ are precisely the rays of the form $\RR_{\geq 0}\{{\bf 1}_\route\}$ for $\route \in \Routes(G)$.

Observe that the balance conditions for flows in the flow cone only involve internal vertices.
Thus, the flow cone (and the flow space) does not change if we require that sources and sinks have degree one.
Moreover, by the balance condition, the flow on any idle edge is completely determined by the flow on its adjacent edges.
Thus, without loss of generality, we restrict to the following case.

\begin{convention} \label{convention:graphs no idle edges}
The graphs we consider have no idle edges and are such that all sources and sinks have degree one.
\end{convention}

The \textnew{flow polytope} $\F_1(G)$ of $G$ is the slice of the cone of nonnegative balanced flows $\F^+(G)$ consisting of {\em unit flows}:
\[
    \mathnew{\F_1(G)} := \bigset{ f \in \F^+(G) }{\smashoperator[r]{\sum_{\substack{v \in \sources(G) \\ e \in \out(v)}}} \quad f(e) = 1 }.
\]
Equivalently, it is the polytope with vertices $\set{ {\bf 1}_\route }{ \route \in \Routes(G) }$.
Note that
\[
    \dim(\F(G)) = \# E(G) - \# \internal(G)
    \hspace{.05\linewidth}\text{and}\hspace{.05\linewidth}
    \dim(\F_1(G)) = \# E(G) - \# \internal(G) - 1.
\]

More generally, given ${\bf a} = (a_v)_{v} \in \RR^{\internal(G)}$
and ${\bf c} = (c_s)_{s} \in \RR^{\sources(G)}$, we let
\begin{align*}
    \mathnew{\F(G,{\bf a})} & := \set{ f \in \RR^E }{ \nf_f(v) = a_v \text{ for all } v \in \internal(G) }, & \text{and} \\
    \mathnew{\F_{\bf c}(G,{\bf a})} & := \set{ f \in \F(G,{\bf a}) }{ \nf_f(s) = c_s \text{ for all } s \in \sources(G) }.
\end{align*}
So $\F(G) = \F(G,{\bf 0})$.

Lastly, we denote by $\F^\NN(G,{\bf a})$ and $\F_{\bf c}^\NN(G,{\bf a})$ the set of nonnegative integer points (simply called \textnew{integer flows}) of $\F(G,{\bf a})$ and $\F_{\bf c}(G,{\bf a})$, respectively.

\begin{remark}
Note that $\F_{\bf c}^\NN(G,{\bf a})$ is a finite set and its size can be computed via the following multivariate generating function
\begin{equation} \label{gf: kpf}
    \sum_{\bf a, \bf c} \#\F_{\bf c}^\NN(G,{\bf a}) \prod_{v \in \sources(G)} z_s^{c_s}\prod_{v \in \internal(G)} z_v^{a_v} \,=\, \prod_{\substack{(u,v) \in E(G),\\ v\in \sinks(G)}} (1-z_u)^{-1}\prod_{\substack{(u,v) \in E(G)\\ v\not\in \sinks(G)}} (1-z_uz_v^{-1})^{-1}.
\end{equation}
\end{remark}

\subsection{DKK triangulations}\label{subsectDKK}

Given a connected DAG $G$, a \textnew{framing} at an internal vertex $v$ of $G$ is a
choice of total orders $\preceq_{\inc(v)}$ and $\preceq_{\out(v)}$ on the sets $\inc(v)$ and $\out(v)$.
A \textnew{framed DAG} $(G,F)$ is a connected DAG $G$ equipped with a framing $(\preceq^F_{\inc(v)} , \preceq^F_{\out(v)})$ at each internal vertex.

Fix a framed DAG $(G,F)$.
The framing $F$ induces a total order on each of the sets $\In(v)$ of walks from a source to an internal vertex $v$, as follows.
Given two distinct walks $\route, \route' \in \In(v)$ from a source to the same internal vertex $v \in \internal(G)$, let $e \in \route$ and $e' \in \route'$ be the last edges at which $\route$ and $\route'$ differ and $u = \head(e) = \head(e')$.
Then,
\[
    \route \prec^F_{\In(v)} \route' \quad\text{if and only if}\quad e \prec^F_{\inc(u)} e'.
\]
Dually, $F$ also induces a total order on $\Out(v)$:
given two distinct walks $\route,\route' \in \Out(v)$,
\[
    \route \prec^F_{\Out(v)} \route' \quad\text{if and only if}\quad e \prec^F_{\out(u)} e',
\]
where $e \in \route$ and $e' \in \route'$ are the first edges at which $\route$ and $\route'$ differ and $u = \tail(e) = \tail(e')$.

Since we only work with one framing $F$ at a time, we drop the superscript $F$ from the notation above.

\begin{definition}[\cite{DKK}]
\label{def:compatible-routes}
Let $G$ be a connected DAG with a fixed framing $F$.
Two routes $\route,\route' \in \Routes(G)$ that have a common internal vertex $v$ are \textnew{incompatible at $v$}
if $\pref(\route,v) \prec_{\In(v)} \pref(\route',v)$ and $ \suf(\route',v) \prec_{\Out(v)} \suf(\route,v)$ or vice versa.
Otherwise, the routes are \textnew{compatible at $v$}.

Two routes $\route,\route' \in \Routes(G)$ are \textnew{compatible} if they are compatible at every common internal vertex.
A route $\route \in \Routes(G)$ is \textnew{exceptional} if it is compatible with all other routes.
We denote the set of exceptional routes of $(G,F)$ by $\cE = \cE(G,F)$.
\end{definition}

A \textnew{clique} $K \subseteq \Routes(G)$ is a set of pairwise compatible routes.
We denote by $\mathnew{\Cliques(G)} = \Cliques(G,F)$ the set of cliques of the framed DAG $(G,F)$,
and by $\mathnew{\MaxCliques(G)} \subseteq \Cliques(G)$ the collection of maximal cliques.
Given a clique $K \in \Cliques(G)$, let $\sigma_K$ be the cone spanned by the (indicator vector of) routes in $K$, and
$\Delta_K$ be its intersection with the flow polytope $\F_1(G)$.
Explicitly,
\begin{align*}
    \mathnew{\sigma_K} & := \RR_{\geq 0}\set{ {\bf 1}_\route }{ \route \in K } \subseteq \RR^E, \text{and}\\
    \mathnew{\Delta_K} & := \conv\set{ {\bf 1}_\route }{ \route \in K } \subseteq \RR^E.
\end{align*}
Danilov--Karzanov--Koshevoy \cite{DKK} proved that these cones (resp. simplices)
form a regular unimodular triangulation of the flow cone (resp. flow polytope).

\begin{theorem}[{\cite[Thm. 1, 2]{DKK}}]
	\label{thm:DKK}
	Let $(G,F)$ be a framed DAG.
	The collection of cones
	\[ \mathnew{\DKK(G,F)} := \set{ \sigma_K }{ K \in \Cliques(G) }\]
	(resp. of simplices $\set{ \Delta_K }{ K \in \Cliques(G) }$)
	yields a regular unimodular triangulation of $\F^+(G)$ (resp. of $\F_1(G)$).
\end{theorem}

Recall that a polyhedral cone $\sigma \subseteq \RR^d$ is \textnew{rational} if it can be written as $\sigma = \RR_{\geq 0} \{ v_1 , \dots , v_k \}$ with each $v_i \in \ZZ^d$. In this case, we always choose $k$ minimal and $v_i$ to be the smallest nonzero integer vector in the ray $\RR_{\geq 0} \{v_i\}$ it spans. With this notation, the cone $\sigma$ is \textnew{unimodular} if $\{ v_1 , \dots , v_k \}$ is a $\ZZ$-basis for $\spann_\RR \{ v_1,\ldots,v_k \} \cap \ZZ^d$. In particular, a unimodular cone is necessarily simplicial.
A triangulation $\Sigma = \{\sigma_i\}_i$ of a cone is unimodular if each cone $\sigma_i$ is unimodular.

A simplex $\Delta = \conv\{v_1,\ldots,v_k\} \subseteq \RR^d$ is \textnew{unimodular} if the cone $\RR_{\geq 0} \Delta = \RR_{\geq 0} \{ v_1 , \dots , v_k \}$ is unimodular, and a triangulation $\{\Delta_i\}_i$ of a polytope is {unimodular} if all its simplices are unimodular.
Crucially, a unimodular simplex has normalized volume $1$.
For the definition of a \textnew{regular} triangulation see \cite[\S 2.2]{triangulations_book}.

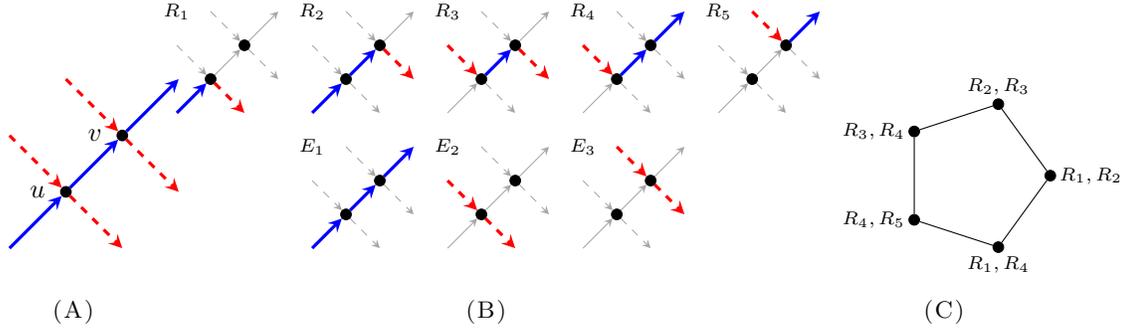
\begin{figure}[ht]
	\centering
	\begin{subfigure}{.12\linewidth}
		\centering
        \input{figures/master_example_acyclic_graph_quiver_smaller.tex}
        \caption{}
		\label{subfig:full_graph_and_routes1}
	\end{subfigure}
	\begin{subfigure}{.6\linewidth}
		\centering
        \input{figures/master_example_routes.tex}
        \caption{}
		\label{subfig:full_graph_and_routes3}
	\end{subfigure} 
    \begin{subfigure}{.2\linewidth}
        \centering
		\begin{tikzpicture}
            \node[circle, draw, inner sep=0pt, minimum size=4pt,fill] (a)   at (  0:1) {};
            \node[right] at (  0:1) {\scriptsize $R_1,R_2$};
            \node[circle, draw, inner sep=0pt, minimum size=4pt,fill] (b)   at ( 72:1) {};
            \node[above] at ( 72:1) {\scriptsize $R_2,R_3$};
            \node[circle, draw, inner sep=0pt, minimum size=4pt,fill] (c)   at (144:1) {};
            \node[left] at (144:1) {\scriptsize $R_3,R_4$};
            \node[circle, draw, inner sep=0pt, minimum size=4pt,fill] (d)   at (216:1) {};
            \node[left] at (216:1) {\scriptsize $R_4,R_5$};
            \node[circle, draw, inner sep=0pt, minimum size=4pt,fill] (e)   at (288:1) {};
            \node[below] at (288:1) {\scriptsize $R_1,R_4$};
            \draw (a) -- (b) -- (c) -- (d) -- (e) -- (a);
        \end{tikzpicture}
       \caption{}
   \label{subfig:full_graph_and_routes2}
	\end{subfigure}
	\caption{
        (A) A full DAG $G = \blackXXgraph$ with an ample framing $F$ (at each vertex, $\text{\tcr{red}} \prec \text{\tcb{blue}}$),
        (B) the routes of $(G,F)$ of which the last three are exceptional, and
        (C) illustration of the dual of the triangulation $\DKKred(G,F)$.}
    \label{fig:full_graph_and_routes}
\end{figure}

Since the exceptional routes $\cE(G,F)$ are in every maximal cone of $\DKK(G,F)$, it is natural to quotient out by the linear span their indicator vectors $\mathnew{\spann_\RR(\cE)} := \RR \set{ {\bf 1}_\route }{ \route \in \cE(G,F) }$.
Let \textnew{$\F(G)_{\rm red} :=\F(G)/\spann_\RR(\cE)$} be the \textnew{reduced flow space} and denote by $\pi_\cE : \F(G) \to \F(G)_{\rm red}$ be the corresponding projection map.
The \textnew{reduced fan $\DKKred(G,F)$} is the projection of the triangulation $\DKK(G,F)$ to $\F(G)_{\rm red}$:
\[
    \DKKred(G,F) := \set{ \pi_\cE(\sigma_K) }{ \sigma_K \in \DKK(G,F) }.
\]

\begin{definition}[\cite{DKK}]
\label{def:ample_framing_acyclic}
A framing $F$ of $G$ is \textnew{ample} if the cone $\spann_\RR(\cE) \cap \F^+(G)$
is not contained in any facet of the flow cone $\F^+(G)$.
Equivalently, $F$ is ample if the reduced fan $\DKKred(G,F)$ is \textnew{complete}, i.e. the union of the cones in $\DKKred(G,F)$ equals $\F(G)_{\rm red}$.
\end{definition}

In the same work, Danilov--Karzanov--Koshevoy \cite[Prop.~5]{DKK}
showed that
a framing $F$ on $G$ is ample if and only if each non-idle edge of $G$ belongs to an exceptional route in $\cE(G,F)$.

\begin{example}
    \label{ex:master_ex_acyclic_1}
    Let $\colorfulXXgraph := (\blackXXgraph,F)$ be the amply framed DAG in \cref{subfig:full_graph_and_routes1}.
    It has 8 routes, 3 of which are exceptional.
    The non-exceptional routes $R_1,R_2,R_3,R_4,R_5$ and the exceptional routes $E_1,E_2,E_3$ are illustrated in \cref{subfig:full_graph_and_routes3}.
    The latter appear in all the maximal cliques of $\colorfulXXgraph$.
    These maximal cliques, with exceptional routes omitted, are:
    \[
        \{R_1,R_2\}, \{R_2,R_3\},\{R_3,R_4\},\{R_4,R_5\},\{R_1,R_5\}.
    \]
    The dual graph of the triangulation, where two cliques are adjacent if they differ by one route, is illustrated in \cref{subfig:full_graph_and_routes2}.
\end{example}

Von Bell--Braun--Bruegge--Hanely--Peterson--Serhiyenko--Yip \cite{Kentuckygentle}
completely characterized the graphs with no idle edges that admit an ample framing.
A DAG $G$ is called \textnew{full} if every internal vertex $v \in \internal(G)$ satisfies $\indeg(v) = \outdeg(v) = 2$.

\begin{theorem}[{\cite[Lemma 3.1, Theorem 3.9]{Kentuckygentle}}]
    \label{lem:ample-deg2}
    A graph $G$ with no idle edges admits an ample framing if and only if $G$ is full.
\end{theorem}

In an ample framing of a full graph $G$, if an edge $(u,v)$ is first (resp. second) in $\out(u)$ then it also must be first (resp. second) in $\inc(v)$ \cite[Corollary 3.11]{Kentuckygentle}.
Thus, an ample framing on a full graph $G$ is equivalent to a bi-colouring $\coloring : E(G) \to \{\text{\tcr{red}},\text{\tcb{blue}}\}$ of the edges of $G$ such that any two edges having the same head or the same tail receive different colours.
Henceforth, we adopt the convention that $\text{\tcr{red}} < \text{\tcb{blue}}$ and will sometimes write $\tcr{1}$ and $\tcb{2}$ instead of $\text{\tcr{red}}$ and $\text{\tcb{blue}}$.
With this interpretation, a route is exceptional if and only if it is monochromatic. Moreover, we can characterize incompatibility with the following condition. Because of this, we denote such (full) graphs with an ample framing by $(G,\coloring)$ and say that $(G,\coloring)$ is an \textnew{amply framed DAG}.

\begin{corollary}[{\cite[\S 3.1]{Kentuckygentle}}]
	\label{cor:coherence_full_graphs_ample_framings}
	Given an amply framed DAG $(G,\coloring)$, two routes $\route$ and $\route'$ are incompatible if and only if $\route$ and $\route'$ have a maximal common subwalk going from a vertex $u$ to a vertex $v$ (potentially $u=v$) such that
    the colour of the incoming edge to $u$ in $\route$ and the colour of the outgoing edge from $v$ in $\route$ are different.
    (Equivalently, by the maximality of the common subwalk, the colour of the incoming edge to $u$ in $\route'$ and the colour of the outgoing edge from $v$ in $\route'$ are different.)
\end{corollary}

\begin{remark}
In \cref{s:conections-acyclic} we will interpret this incompatibility condition in the context of non-kissing complexes from \cite{PPP,PPPlocg}.
\end{remark}

The following results will be helpful in what follows.

\begin{proposition}
    \label{prop:max-cliques-use-all-edges}
    Let $(G,F)$ be a framed DAG.
    Then, any maximal clique $K$ of $(G,F)$ uses every edge of $G$.
\end{proposition}

\begin{proof}
By \cref{thm:DKK}, maximal cliques $K$ correspond to the rays of the full-dimensional cones in the $\DKK$ triangulation of $\F(G)$.
Thus, any maximal clique contains $\dim(\F(G)) = \# E(G) - \# \internal(G)$ routes.
If $K$ does not use an edge $e$ of $G$, then $K$ is also a maximal clique of the graph $G \setminus e$, which is a contradiction since $\dim(\F(G \setminus e)) = \dim(\F(G)) - 1$.
\end{proof}

\begin{remark}
In \cite{BerSer_Wilting} Berggren and Serhiyenko explained that $\DKK(G,F)$ of a framed DAG $(G,F)$ with a certain condition called {\em rooted}, is a subcomplex of $\DKK(\tilde{G},\tilde{F})$ where $(\tilde{G},\tilde{F})$ is an amply framed DAG obtained from $(G,F)$ by elementary graph operations.
That is to say, a good understanding of the amply framed case is sufficient to permit an understanding of a much broader class of framed DAGs.
\end{remark}

\subsection{Volume integer flows}

Let $(G,F)$ be a framed acyclic directed graph.
Since the DKK triangulation of $\F_1(G)$ is unimodular,
the \textnew{normalized volume} of $\F_1(G)$, denoted by $\vol(\F_1(G))$,
is precisely the number of maximal cliques of $(G,F)$.
The following result of Baldoni--Vergne \cite{BV} and Postnikov--Stanley \cite{PS}
computes $\vol(\F_1(G))$ in terms of certain integer flows on $G$.

\begin{theorem}[\cite{BV,PS}]
	\label{thm:innerflowsindeg}
	For a connected DAG $G$, we have that
	\[
		\vol(\F_1(G)) = \# \F_{\bf 0}^\NN(G,{\bf d}),
	\]
where ${\bf d} := (\indeg(v)-1)_{v \in \internal(G)}$.
\end{theorem}

The integer flows in $\F_{\bf 0}^\NN(G,{\bf d})$ have netflow $0$ at each source and netflow $\indeg(v) - 1$ at each internal vertex $v$.
Due to the result above, we call these the \textnew{volume integer flows} of $G$.

\begin{remark}
From the generating function in \eqref{gf: kpf}, the number of integer flows can be computed as
\begin{equation} \label{gf: volume integer flows}
   \# \F_{\bf 0}^\NN(G,{\bf d}) = [{\bf z}^{\bf d}] \prod_{\substack{(u,v) \in E(G)\\ u \in \internal(G), v\in \sinks(G)}} (1-z_u)^{-1} \prod_{\substack{(u,v)\in E(G)\\ u,v \in \internal(G)}} (1-z_uz_v^{-1})^{-1},
\end{equation}
where $[\cdot]$ denotes coefficient extraction and ${\bf z}^{\bf d}=\prod_{v \in \internal(G)} z_v^{\indeg(v)-1}$.
\end{remark}

The following is an immediate consequence of \cref{thm:DKK,thm:innerflowsindeg}.

\begin{corollary}
    \label{cor:cliques_and_integer_flows}
    Let $(G,F)$ be a framed DAG. Then,
    \[
        \#\MaxCliques(G,F) = \# \F_{\bf 0}^\NN(G,{\bf d}).
    \]
\end{corollary}

\begin{example}
    \label{ex:master_ex_acyclic_2}
    Let $\colorfulXXgraph$ be the amply framed DAG in \cref{subfig:full_graph_and_routes1}.
    We can see that the flow polytope $\F_1(\blackXXgraph)$ has normalized volume~5 in two different ways:
    \begin{enumerate}
        \item The framed DAG $\colorfulXXgraph$ has five maximal cliques, which are shown in \cref{subfig:full_graph_and_routes2}.
        \item The DAG $\colorfulXXgraph$ has five integer flows with netflow $1 = \indeg(v) - 1$ at every internal vertex~$v$ and netflow $0$ at each source.
        These integer flows are illustrated in \cref{fig:example_volume_integer_flows}.
    \end{enumerate}

\end{example}

\begin{figure}[ht]
	\centering
    \input{figures/volume_integer_flows.tex}
	\caption{
        The five volume integer flows from \cref{subfig:full_graph_and_routes3} (only edges with non-zero flows are labelled).
        Below each flow, we list the clique that corresponds to it via the bijection $\Phi^{-1}$. We do not list the exceptional routes $E_1,E_2,E_3$, which are present in each clique.
        }
    \label{fig:example_volume_integer_flows}
\end{figure}
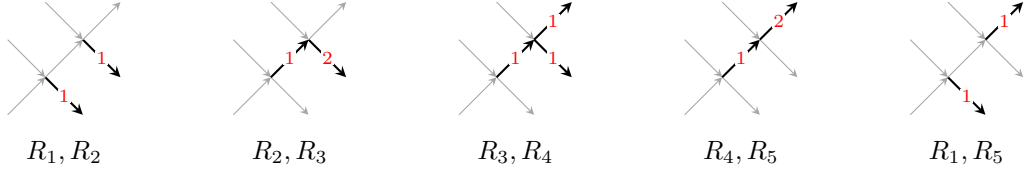

\begin{example}
    \label{ex:big_ex_acyclic}
    Let $\colorfulHashGraph$ be the amply framed DAG in \cref{subfig:full_edge_labeling1}.
    We can see that the flow polytope $\F_1(\colorfulHashGraph)$ has normalized volume~42 in two different ways:
    \begin{enumerate}
    \item The framed DAG $\colorfulHashGraph$ has 42 maximal cliques, which are cumbersome to compute.
    \item The DAG $\colorfulHashGraph$ has 42 integer flows with netflow $1 = \indeg(v) - 1$ at every internal vertex~$v$ and netflow $0$ at each source. This number of flows can be computed from the generating function in \eqref{gf: volume integer flows}. Indeed,
    labelling the internal vertices 1,2,3,4 with 1 on the left and 4 on the right, we have

    \begin{multline*}
    \# \F_{\bf 0}^\NN(\colorfulHashGraph,{\bf 1}) =  [z_1z_2z_3z_4] \\\left(1-z_{2}\right)^{-1} \left(1-z_{3}\right)^{-1} \left(1-z_{4}\right)^{-2}\left(1-\frac{z_{1}}{z_{2}}\right)^{-1} \left(1-\frac{z_{1}}{z_{3}}\right)^{-1} \left(1-\frac{z_{2}}{z_{4}}\right)^{-1} \left(1-\frac{z_{3}}{z_{4}}\right)^{-1}
    =  42.
    \end{multline*}
    \end{enumerate}
\end{example}

For any framed DAG $(G,F)$, Mészáros--Morales--Striker construct in \cite{MMSt} a bijective correspondence
$\Phi = \Phi_{(G,F)} : \MaxCliques(G,F) \to \F_{\bf 0}^\NN(G,{\bf d})$,
between the maximal cliques of $(G,F)$ and the volume integer flows of $G$, thus giving a combinatorial proof of \cref{cor:cliques_and_integer_flows}.
We now recall their construction.

Given a maximal clique $K \in \MaxCliques(G,F)$, let $\mathnew{\Phi(K)} \in \ZZ^E$ be the flow
\[
    \mathnew{\Phi(K)(e)} := n_K(e) - 1,
\]
where
\[
    \mathnew{n_K(e)} := \# \set{\pref(\route,\head(e))}{\route \in K,\; e \in \route}
\]
is the number of different prefixes up to edge $e$ among the routes $\route \in K$ containing $e$.

\begin{theorem}[{\cite{MMSt}}]
	\label{cor:bijection_clique_to_integer_flows}
	Let $(G,F)$ be a framed DAG.
	Then, for any maximal clique ${K \in \MaxCliques(G,F)}$,
    the flow $\Phi(K)$ is a volume integer flow.
	Moreover, the map $\Phi : \MaxCliques(G,F) \to \F_{\bf 0}^\NN(G,{\bf d})$ is a bijection.
\end{theorem}

\begin{example}
	\label{ex:master_ex_acyclic_3}
    \cref{fig:example_volume_integer_flows} shows
    each volume integer flow of $\blackXXgraph$ labelled by the non-exceptional routes of $\colorfulXXgraph$ in the corresponding maximal clique.
\end{example}

\section{Background on gentle algebras and $g$-vector fans}
\label{s:gentle}

\subsection{Some vocabulary}
\label{ss:vocabulary}
A \textnew{quiver} is a quadruplet $Q=(Q_0,Q_1,s,t)$ where:
\begin{itemize}[itemsep=1mm]
    \item $Q_0$ and $Q_1$ are sets, respectively called the \textnew{vertex} and \textnew{arrow sets} of $Q$;
    \item $s,t: Q_1 \longrightarrow Q_0$ are functions, respectively called \textnew{source} and \textnew{target functions}.
\end{itemize}
For $\alpha\in Q_1$, the vertex $s(\alpha)$ is the vertex at which $\alpha$ begins and $t(\alpha)$ is the vertex at which $\alpha$ ends.
A quiver is said to be \textnew{finite} whenever $Q_0$ and $Q_1$ are finite sets. From now on, we assume that all the quivers we will study are finite.

A \textnew{path} in $Q$ is either a formal element $e_q$ for some $q \in Q_0$, called the \textnew{lazy path at $q$}, or a finite sequence of arrows $\alpha_1,\ldots, \alpha_n \in Q_1$, which will be written $\alpha_1 \alpha_2 \cdots \alpha_n$, such that $s(\alpha_{i+1}) = t(\alpha_i)$ for all $i \in \{1, \ldots, n-1\}$. Define the \textnew{length} of a path as the number of arrows that constitute it. For a given path $p = \alpha_1 \alpha_2 \cdots \alpha_n$, we write its source $\mathnew{s(p)} := s(\alpha_1)$ and its target $\mathnew{t(p)} := t(\alpha_n)$. For a lazy path $s(e_q) = t(e_q) = q$. A path $p$ is an \textnew{oriented cycle} whenever $p$ is non-lazy and $s(p) = t(p)$. (Note that contrary to some conventions, an oriented cycle is allowed to repeat vertices and edges.) A quiver is said to be \textnew{acyclic} if it does not contain any oriented cycle.

For any arrow $\alpha \in Q_1$, we introduce its \textnew{formal inverse $\alpha^{-1}$} such that $s(\alpha^{-1}) = t(\alpha)$ and~$t(\alpha^{-1}) = s(\alpha)$. A \textnew{walk} in $Q$ is either a lazy path or a finite sequence of arrows and inverse arrows $\beta_1 \beta_2 \ldots \beta_m$ such that $s(\beta_{i+1}) = t(\beta_i)$ for all $i \in \{1, \ldots, m-1\}$. For a given walk $w$, we write its source $s(w)$ and its target $t(w)$.

\subsection{Gentle quivers}
\label{ss:gentle-quiver}
A \textnew{locally gentle quiver} is a pair $(Q,R)$ with $Q$ a quiver and $R$ a subset of paths in $Q$ of length $2$, such that for any $\alpha \in Q_1$:
\begin{itemize}[itemsep=1mm]
\item there exists at most one $\beta \in Q_1$ such that $t(\alpha)=s(\beta)$ and $\alpha \beta \notin R$;
\item there exists at most one $\beta' \in Q_1$ such that $t(\alpha)=s(\beta')$ and $\alpha \beta' \in R$;
\item there exists at most one $\gamma \in Q_1$ such that $t(\gamma)=s(\alpha)$ and $\gamma \alpha \notin R$;
\item there exists at most one $\gamma' \in Q_1$ such that $t(\gamma')=s(\alpha)$ and $\gamma' \alpha \in R$;
\end{itemize}
We draw a locally gentle quiver $(Q,R)$ as the directed graph corresponding to $Q$ with an additional dashed curve joining $\alpha$ and $\beta$ whenever $\alpha \beta \in R$.

One way to express the four conditions that make $(Q,R)$ a locally gentle quiver is by saying that, near each vertex $v$, the quiver looks like this, where we are allowed to remove arrows, and also to make identifications among the vertices that appear.

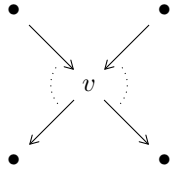
\[\begin{tikzpicture}[>= angle 60,->]
        \node (v) at (0,0) {$v$};
        \node (a) at (1,1) {$\bullet$};
        \node (b) at (1,-1) {$\bullet$};
        \node (c) at (-1,-1) {$\bullet$};
        \node (d) at (-1,1) {$\bullet$};
        \draw (a) to (v);
        \draw (v) to (b);
        \draw (d) to (v);
        \draw (v) to (c);

        \draw[dotted,-] ([yshift=-.25cm, xshift=-.2cm]v.west) arc[start angle = -140, end angle = -220, x radius=.4cm, y radius =.4cm];
		\draw[dotted,-] ([yshift=-.25cm,xshift=.2cm]v.east) arc[start angle = 320, end angle = 400, x radius=.4cm, y radius =.4cm];
		\end {tikzpicture} \]

A locally gentle quiver $(Q,R)$ is said to be \textnew{gentle} whenever every oriented cycle $\alpha_1 \ldots \alpha_k$ has $\alpha_i \alpha_{i+1} \in R$ for some $i \in \{1,\ldots,k-1\}$ or else $\alpha_k \alpha_1 \in R$.

Let $\KK$ be an algebraically closed field, and let $Q$ be a quiver. The \textnew{path algebra} $\KK Q$ is the $\KK$-vector space with basis the set of the paths in $Q$, equipped with a multiplication defined by concatenation:
\[p_1 \cdot p_2 = \begin{cases}
    p_1 p_2 & \text{if } t(p_1)=s(p_2), \\
    0 & \text{otherwise.}
\end{cases}\]
A $\KK$-algebra $\Lambda$ is said to be \textnew{gentle} whenever there exists a gentle quiver $(Q,R)$ such that $\Lambda\simeq\KK Q/\langle R\rangle$, where $\langle R \rangle$ is the two-sided ideal generated by the elements of $R$. A gentle algebra is necessarily finite-dimensional.

\subsection{Representations of a locally gentle quiver}

Let $(Q,R)$ be a locally gentle quiver. A \textnew{representation} of $(Q,R)$ (over $\KK$) is a pair $M = ((M_q)_{q \in Q_0}, (M_\alpha)_{\alpha \in Q_1})$ such that:
\begin{itemize}[itemsep=1mm]
\item for all $q \in Q_0$, $M_q$ is a $\KK$-vector space;
\item for all $\alpha \in Q_1$, $M_\alpha : M_{s(\alpha)} \longrightarrow M_{t(\alpha)}$ is a $\KK$-linear transformation;
\item for all $\alpha \beta \in R$, $M_\beta M_\alpha = 0$.
\end{itemize}
One can see a representation of $(Q,R)$ as an assignment of a $\KK$-vector space to each vertex of $Q$, and an assignment of a $\KK$-linear transformation to each arrow, satisfying the vanishing conditions given by $R$. We say that $M$ is \textnew{finite-dimensional} whenever $M_q$ is finite-dimensional for all $q \in Q_0$. Write $\mathnew{\rep_\KK(Q,R)}$ for the set of finite-dimensional representations of $(Q,R)$ (over $\KK$).

Given $M,N \in \rep_\KK(Q,R)$, we define a \textnew{morphism} of representations $h : M\longrightarrow N$ to be a collection of linear maps $(h_q : M_q \longrightarrow N_q)_{q \in Q_0}$ such that for all $\alpha \in Q_1$, $h_{t(\alpha)} M_\alpha = N_\alpha h_{s(\alpha)}$. Write $\Hom(M,N)$ for the $\KK$-vector space of these morphisms.

A morphism $h\in\Hom(M,N)$ is an \textnew{isomorphism} provided there exists a morphism denoted by $h^{-1}\in\Hom(N,M)$ with $h\circ h^{-1}=\textrm{id}_N$ and $h^{-1}\circ h=\textrm{id}_M$. Equivalently, $h$ is an isomorphism if $h_q$ is bijective for every $q \in Q_0$. If such an isomorphism exists, we say that $M$ and $N$ are \textnew{isomorphic}, and we write $M \simeq N$.

Endowing $\rep_\KK(Q,R)$ with the morphisms described above puts a category structure on $\rep_{\KK}(Q,R)$. This category is equivalent to the category of finite-dimensional $(\KK Q/\langle R \rangle)$-modules \cite[Theorem III.1.6]{ASS06}.

A \textnew{subrepresentation} of a representation $M$ consists of a subspace $N_q$ of $M_q$ for each $q\in Q_0$, such that for any $\alpha\in Q_1$, we have that $M_\alpha(N_{s(\alpha)}) \subseteq N_{t(\alpha)}$. Equivalently, a subrepresentation of $M$ is a representation $N$ with a morphism from $N$ to $M$ which is an inclusion at each vertex.
See \cref{ex:HN-polytope} for a list of the subrepresentations of the representation $\KK \longleftarrow \KK \longrightarrow \KK$.

\subsection{Indecomposable representations of locally gentle quivers}

Fix a locally gentle quiver $(Q,R)$. Let $M,N \in \rep_\KK(Q,R)$. We define the \textnew{direct sum} of $M$ and $N$, denoted by $M \oplus N$, to be the representation such that:
\begin{itemize}[itemsep=1mm]
    \item for all $q \in Q_0$, $(M \oplus N)_q = M_q \oplus N_q$;
    \item for all $\alpha \in Q_1$, we define $(M\oplus N)_\alpha$ by $(M \oplus N)_\alpha(x,y) = (M_\alpha(x),N_\alpha(y))$, for any $ x \in M_{s(\alpha)}$ and $y \in N_{s(\alpha)}$.
\end{itemize}
Note that $M \oplus N \in \rep_\KK(Q,R)$. A nonzero representation $E \in \rep_\KK(Q,R)$ is said to be \textnew{indecomposable} if whenever $E \cong M \oplus N$, for some $M,N \in \rep_\KK(Q,R)$, then either $M \cong 0$ or $N \cong 0$.
Butler and Ringel first classified the indecomposable representations of gentle quivers \cite{BR87}. This description was extended to the locally gentle quivers by Crawley-Boevey \cite{CB89}. We do not provide a full description here. Readers interested in more details from a combinatorial perspective can consult \cite{BDMTY20,PPP}.

A walk $w$ in $Q$ is said to be \textnew{reduced} if either it is a lazy path or $w = \beta_1 \ldots \beta_m$ such that $\beta_{i+1} \neq \beta_i^{-1}$ for all $i \in \{1, \ldots, m-1\}$. A \textnew{string} $\rho=\beta_1\dots\beta_m$ in $(Q,R)$ is a reduced walk such that for any $i \in \{1,\ldots,m-1\}$, both $\beta_i \beta_{i+1} \notin R$ and $\beta_{i+1}^{-1} \beta_{i}^{-1} \notin R$.
\begin{example}
    \label{ex:An-example}
    A quiver is said to be of type $A_n$ if it has $n$ vertices $\{1,2,\dots,n\}$ and $n-1$ arrows, connecting vertices $i$ and $i+1$ for $1\leq i\leq n-1$. Note that the orientations of the arrows are not fixed, so there are $2^{n-1}$ different quivers of type $A_n$.

    Consider $Q$ any $A_n$ quiver, with no relations. There are $n^2$ strings of $(Q,\varnothing)$, uniquely determined by specifying their starting and ending vertices.
\end{example}

\begin{example}
\label{ex:runningGQR}
Let us consider the following gentle quiver.
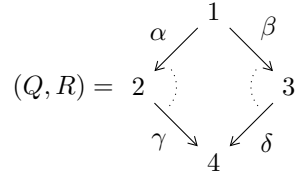
\[\begin{tikzpicture}[>= angle 60,->]
	    \node (QR) at (-1,0){$(Q,R) =$};
		\node (a) at (1,1) {$1$};
		\node (b) at (0,0) {$2$};
        \node (c) at (2,0) {$3$};
		\node (d) at (1,-1) {$4$};
		\draw (a) to node[above left]{$\alpha$} (b);
        \draw (a) to node[above right]{$\beta$} (c);
        \draw (b) to node[below left]{$\gamma$} (d);
        \draw (c) to node[below right]{$\delta$} (d);
		\draw[dotted,-] ([yshift=-.25cm, xshift=-.2cm]c.west) arc[start angle = -140, end angle = -220, x radius=.4cm, y radius =.4cm];
		\draw[dotted,-] ([yshift=-.25cm,xshift=.2cm]b.east) arc[start angle = 320, end angle = 400, x radius=.4cm, y radius =.4cm];
		\end {tikzpicture} \]
There are sixteen strings in $(Q, R)$: in addition of the lazy paths and the paths given by exactly one arrow or one inverse arrow, we have $\alpha^{-1} \beta$, $\beta^{-1} \alpha$, $\gamma \delta^{-1}$, and $\delta \gamma^{-1}$.
\end{example}

Consider a string $\rho = \beta_1 \ldots \beta_m$ of $(Q,R)$. We define the \textnew{standard string representation} $\mathnew{M(\rho)} \in \rep_\KK(Q,R)$ as follows:
\begin{itemize}[itemsep=1mm]
    \item set $v_0 = s(\beta_1)$ and, for all $i \in \{1,\ldots,m\}$, $v_i = t(\beta_i)$;
    \item consider formal variables $x_0,\ldots,x_m$;
    \item for all $q \in Q_0$, we define $M(\rho)_q$ as the $\KK$-vector space generated with basis $\set{x_i}{v_i = q}$;
    \item for all $\alpha \in Q_1$, we define $M(\rho)_\alpha: M(\rho)_{s(\alpha)} \longrightarrow M(\rho)_{t(\alpha)}$ as the linear transformation such that:
    \[M(\rho)_\alpha (x_i) = \begin{cases}
        x_{i-1} & \text{if } \alpha = \beta_i^{-1}; \\
        x_{i+1} & \text{if } \alpha = \beta_{i+1};\\
        0 & \text{otherwise.}
    \end{cases}\]
\end{itemize}
Note that the first two cases are mutually exclusive since $\rho$ is reduced,
thus $M(\rho)$ is a well-defined representation of $(Q,R)$.
A representation $M$ is a \textnew{string representation} if it is isomorphic to some standard string representation.

\begin{example}
\label{ex:An-example-cont}
Continuing \cref{ex:An-example}, let $Q$ be an $A_n$ quiver with no relations. Let $\gamma_{ij}$ denote the unique string starting at vertex $i$ and ending at vertex $j$. The standard string representation of $\gamma_{ij}$ is one-dimensional at each vertex $k$ with $i\leq k\leq j$ (if $i\leq j$) or $j\leq k\leq i$ (if $i\geq j$). The linear maps between successive 1-dimensional vector spaces are identity maps, while the linear maps involving a 0-dimensional vector space are (necessarily) the zero map.
\end{example}

\begin{example} Let us take the gentle quiver in \cref{ex:runningGQR}. The figure below shows the string representation of $\alpha^{-1} \beta$.
    \[\begin{tikzpicture}[>= angle 60,->]
	    \node (QR) at (-1.5,0){$M(\alpha^{-1}\beta) = $};
		\node (a) at (1,1) {$\KK$};
		\node (b) at (0,0) {$\KK$};
        \node (c) at (2,0) {$\KK$};
		\node (d) at (1,-1) {$0$};
		\draw (a) to node[above left]{$1$} (b);
        \draw (a) to node[above right]{$1$} (c);
        \draw (b) to (d);
        \draw (c) to (d);
		\draw[dotted,-] ([yshift=-.25cm, xshift=-.2cm]c.west) arc[start angle = -140, end angle = -220, x radius=.4cm, y radius =.4cm];
		\draw[dotted,-] ([yshift=-.25cm,xshift=.2cm]b.east) arc[start angle = 320, end angle = 400, x radius=.4cm, y radius =.4cm];
		\end {tikzpicture} \]
\end{example}

\cref{ex:An-example-cont} is an essential motivating example. For a quiver of type $A_n$ with no relations, it has long been understood that the indecomposable representations are exactly the string representations, and two string representations are isomorphic if and only if the strings are either identical or one is the reverse of the other. This extends to the setting of gentle quivers.

\begin{theorem} [\cite{BR87,CB89,CB18}]
    Let $(Q,R)$ be a locally gentle quiver and $\KK$ a field. Any string representation is indecomposable, and two string representations are isomorphic if and only if the two strings either coincide or one is the reversal of the other.
\end{theorem}

In fact, it is possible, and not much more difficult, to describe all the indecomposable representations of a locally gentle quiver $(Q,R)$: they come in two disjoint classes, the string representations already discussed, and band representations. See \cite{BR87, CB89}.
\textnew{Band representations} are associated to cyclic strings, where a non-lazy string $\rho = \beta_1 \ldots \beta_m$ is said to be \textnew{cyclic} whenever $ \beta_2 \ldots \beta_m \beta_1$ is also a string in $(Q,R)$.

If $(Q,R)$ is sufficiently nice, then all indecomposable representations are string representations.

\begin{proposition}[\cite{BR87,CB89}]
    Let $(Q,R)$ be a locally gentle quiver. If $(Q,R)$ does not admit a cyclic string, then every indecomposable finite-dimensional representation of $(Q,R)$ is a string representation.
\end{proposition}

Note that if the hypotheses of the proposition hold, then $(Q,R)$ is actually gentle, since every oriented cycle includes a relation.

\subsection{$g$-vector fan of an algebra}
\label{ss:general-g-vector}

Let $\Lambda$ be a finite-dimensional algebra with $n$ isomorphism classes of simple modules. (In particular, one can think about a gentle algebra corresponding to a quiver with $n$ vertices, but in this section, we prefer to be more general, to underline that these definitions apply to arbitrary finite-dimensional algebras. In subsequent sections, we will return to the gentle setting and present a combinatorial interpretation of this material.) A reference for the material in this section is \cite{adachi-iyama-reiten}. We number the simple modules $S_1$ to $S_n$ in some order, and we let $P_1, \dots, P_n$ be their projective covers.

\begin{definition}
    \label{def:compatible-modules}
    An indecomposable module $M$ of $\Lambda$ is called \textnew{$\tau$-rigid} if $\Hom(M,\tau M)=0$, where $\tau$ denotes the Auslander--Reiten translation. Two $\tau$-rigid indecomposable modules $M,N$ are said to be \textnew{compatible} if we have $\Hom(M,\tau N)=\Hom(N,\tau M)=0$.
\end{definition}
We will not define the Auslander--Reiten translation in full generality.
Rather, in \cref{thm:g-vect-blossom}, we will present an explicit combinatorial description of $\tau$-rigidity and compatibility.

\begin{example}
\label{ex:g-vector_running_example}
Consider the quiver
\[\begin{tikzpicture}[>= angle 60,->]
	    \node (QR) at (-1,0.5){$Q =$};
		\node (a) at (1,1) {$2$};
		\node (b) at (0,0) {$1$};
		\draw (a) to node[above left]{$\alpha$} (b);
		\end {tikzpicture}. \]
This quiver has the following three isomorphism classes of indecomposable representations:
\[
    \begin{tikzpicture}[>= angle 60,->]
		\node (QR) at (-1,0.5){$M_1=$};
        \node (a) at (1,1) {$0$};
		\node (b) at (0,0) {$\KK$};
		\draw (a) to node[above left] {}(b);

        \begin{scope}[xshift=4cm]
            \node (QR) at (-1,0.5){$M_2=$};
		\node (a) at (1,1) {$\KK$};
		\node (b) at (0,0) {$\KK$};
		\draw (a) to node[above left] {$1$} (b);
        \end{scope}

        \begin{scope}[xshift=8cm]
             \node (QR) at (-1,0.5){$M_3=$};
		\node (a) at (1,1) {$\KK$};
		\node (b) at (0,0) {$0$};
		\draw (a) to node[above left] {}(b);
        \end{scope}
	\end{tikzpicture}
\]

Their Auslander--Reiten translations are as follows.
\begin{itemize}[itemsep=1mm]
    \item $\tau M_1 = 0 = \tau M_2$; and,
    \item $\tau M_3 \cong M_1$,
\end{itemize}
Note that, aside from the endomorphism spaces, the only non-zero homomorphism spaces among indecomposable modules are $\Hom(M_1, M_2)$ and $\Hom(M_2, M_3)$. So the compatible pairs are $(M_1, M_2)$ and $(M_2,M_3)$.
\end{example}

For a $\tau$-rigid module $M$, consider its minimal projective presentation:

\[\begin{tikzcd}
	P' & P & M & 0
	\arrow[from=1-1, to=1-2]
    \arrow[from=1-2, to=1-3]
    \arrow[from=1-3, to=1-4]
\end{tikzcd},\]

Here, $P$ and $P'$ are projective modules, which are therefore direct sums of the indecomposable modules $P_1,\ldots,P_n$.
By the minimality assumption, and the hypothesis that $M$ is $\tau$-rigid, $P$ and $P'$ will not have any common direct summands \cite[Proposition 2.5]{adachi-iyama-reiten}, so we can record the information of the isomorphism classes of $P$ and $P'$ in a single vector $\pmb{g}_M = (g_1,\dots,g_n)$ where $g_i$ is the number of copies of $P_i$ in a direct sum decomposition of $P$, if this is positive, or otherwise negative the number of copies of $P_i$ in a direct sum decomposition of $P'$. The vector $\pmb{g}_M$ is the called the \textnew{$g$-vector} of $M$.

\begin{example}
    \label{ex:g-vector_running_example2}
Continuing \cref{ex:g-vector_running_example}, the $g$-vectors of the three indecomposable modules are:
\[
    \pmb{g}_{M_1}=\ve_1,\quad
    \pmb{g}_{M_2}= \ve_2,\quad
    \pmb{g}_{M_3}= \ve_2-\ve_1.
\]
\end{example}

There is a fan whose rays are given by the $g$-vectors of the $\tau$-rigid indecomposable modules, and whose faces are spanned by rays corresponding to pairwise compatible collections.
This fan is never complete. It turns out to be natural to add further $n$ rays, generated by the negatives of the standard basis vectors. Naively, the $g$-vector $-\ve_i$ would correspond to the ``presentation''
\[\begin{tikzcd}
	P_i & 0 & M & 0
	\arrow[from=1-1, to=1-2]
    \arrow[from=1-2, to=1-3]
    \arrow[from=1-3, to=1-4],
\end{tikzcd}\]
but there is no $\Lambda$-module $M$ of which this is the minimal presentation. It is nonetheless convenient to consider a formal object denoted $P_i[1]$, which corresponds to this ``presentation.'' We assign to $P_i[1]$ the $g$-vector $-\ve_i$ and call these additional objects \textnew{shifted projective modules}.
We further say that $P_i[1]$ is compatible with $P_j[1]$ for all $j$, and that $P_i[1]$ is compatible with any module which has no support at vertex $i$. (It is possible to give a less formal description of the objects $P_i[1]$, and a more conceptual explanation of why they should be included. For this, see \cite{adachi-iyama-reiten,derksen-fei}.) Adding these further $n$ rays and defining cones by the extension of compatibility to include these new rays, we get a fan \cite[Theorem 1.9]{DIJ}, which is called the \textnew{$g$-vector fan of $\Lambda$}, and denoted $\gfan(\Lambda)$.
If $(Q,R)$ is the quiver of $\Lambda$, we also write $\gfan(Q,R)$ for $\gfan(\Lambda)$.

For an algebra with a finite number of indecomposable $\tau$-rigid modules, the $g$-vector fan is complete \cite[Theorem 4.7]{A21}.

\begin{example}
Continuing \cref{ex:g-vector_running_example,ex:g-vector_running_example2}, the $g$-vector fan $\gfan(Q)$ is illustrated in \cref{fig:gvect_running_example}. There is a two-dimensional cone of the fan generated by each pair of cyclically adjacent rays.
\end{example}

\begin{figure}[!ht]
    \centering
    \begin{tikzpicture}[>= angle 60,->, scale=1.5]
	   \draw[-,step=1.0,dash pattern={on 10pt off 2pt on 5pt off 2pt}, black, opacity=0.2,line width=.4mm] (-1.5,-1.5) grid (1.5,1.5);
        \draw[-, black,line width=.4mm] (-1.5,0) edge (1.5,0);
        \draw[-, black,line width=.4mm] (0,-1.5) edge (0,1.5);
        \draw[-,black,line width=.4mm] (0,0) edge (-1.5,1.5);
        \draw[blue,line width=.6mm] (0,0) -- (0,1);
        \draw[blue,line width=.6mm] (0,0) -- (1,0);
        \draw[blue,line width=.6mm] (0,0) -- (-1,1);
        \draw[blue,line width=.6mm] (0,0) -- (0,-1);
        \draw[blue,line width=.6mm] (0,0) -- (-1,0);
        \node[blue] at (-0.7,-0.3){$P_1[1]$};
        \node[blue] at (0.3,-0.7){$P_2[1]$};
        \node[blue] at (0.7,-0.3){$M_1$};
        \node[blue] at (0.3,0.7){$M_2$};
        \node[blue] at (-0.5,0.8){$M_3$};
    \end{tikzpicture}
    \caption{The $g$-vector fan $\gfan(Q)$ of the quiver $Q$ in \cref{ex:g-vector_running_example}.
        The vector $\pmb{g}_{M}$ is simply labelled by $M$.}
    \label{fig:gvect_running_example}
\end{figure}
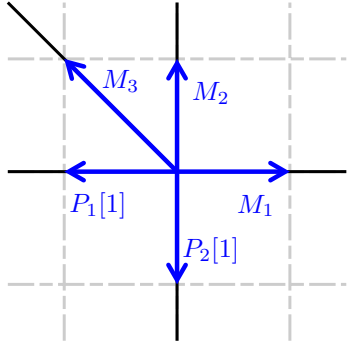

\subsection{Blossoming quivers}
\label{ss:blossoming_quivers}
Let $(Q,R)$ be a gentle quiver. We now describe a new quiver $(Q^\blossom,R^\blossom)$, containing $(Q,R)$ as a full subquiver. The idea is straightforward: we want all the vertices of $Q$ to have indegree 2 and outdegree 2. They cannot have a larger indegree or outdegree, so the only problem is that the indegree or outdegree might be too small. In this case, add new arrows pointing into or out of the vertices to make up the necessary degree. The other end of each such added arrow is a node that has no other arrows incident to it.
Following \cite{PPP}, we call these new arrows and vertices \textnew{blossoming arrows} and \textnew{blossoming vertices}, and the resulting quiver $Q^\blossom$ the \textnew{blossoming quiver} of $Q$.

It is also straightforward to define a new set of \textnew{blossoming relations} $R^\blossom$ whose restriction to $Q$ is $R$, and such that $(Q^\blossom,R^\blossom)$ is a gentle quiver; the blossoming quiver is uniquely defined up to isomorphism.

\begin{example} We continue \cref{ex:g-vector_running_example}. The quiver with no relations $(Q,\varnothing)$ is gentle. Its blossoming quiver is shown below:
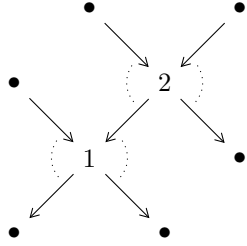
\[\begin{tikzpicture}[>= angle 60,->]
        \node (1a) at (0,2) {$\bullet$};
        \node (1b) at (2,2) {$\bullet$};
        \node (2a) at (-1,1) {$\bullet$};
        \node (2b) at (-1,-1) {$\bullet$};
         \node (c) at (2,0) {$\bullet$};
		\node (d) at (1,-1) {$\bullet$};
		\node (a) at (1,1) {$2$};
		\node (b) at (0,0) {$1$};
		\draw (a) to (b);
       \draw (a) to (c);
        \draw (b) to (d);
        \draw (1a) to (a);
        \draw (1b) to (a);
        \draw (2a) to (b);
        \draw (b) to (2b);
        \draw[dotted,-] ([yshift=-.25cm, xshift=-.2cm]a.west) arc[start angle = -140, end angle = -220, x radius=.4cm, y radius =.4cm];
         \draw[dotted,-] ([yshift=-.25cm, xshift=-.2cm]b.west) arc[start angle = -140, end angle = -220, x radius=.4cm, y radius =.4cm];
         \draw[dotted,-] ([yshift=-.25cm,xshift=.2cm]a.east) arc[start angle = 320, end angle = 400, x radius=.4cm, y radius =.4cm];
		 \draw[dotted,-] ([yshift=-.25cm,xshift=.2cm]b.east) arc[start angle = 320, end angle = 400, x radius=.4cm, y radius =.4cm];
		\end {tikzpicture} \]
\end{example}

A \textnew{maximal string} is a non-lazy string in $(Q^\blossom,R^\blossom)$ which starts and ends at a blossoming vertex. We tacitly identify a maximal string with its reversal. A maximal string is called \textnew{straight} if it (or its inverse) is a path in $(Q^\blossom, R^\blossom)$, i.e., if all its arrows are forward arrows or all its arrows are inverse arrows.
A string that is not straight is called \textnew{bent}.

It is a remarkable fact, established in \cite{BDMTY20,PPP}, that, for $(Q,R)$ a gentle algebra, the isomorphism classes of indecomposable modules plus the shifted projective modules of $(Q,R)$ are in natural correspondence with the bent maximal strings in $(Q^\blossom, R^\blossom)$ (where a maximal string and its reversal are identified). We will recall a combinatorial way to see this bijective correspondence.

Given $\alpha \in Q_1^\blossom$, let $\nu_\alpha$ be the longest path $\nu$ in $(Q^\blossom,R^\blossom)$ such that $\nu\alpha^{-1}$ is a string in $(Q^\blossom,R^\blossom)$.
We define the \textnew{cohook} at $\alpha$ in $(Q^\blossom,R^\blossom)$ as the string $ \h_{(\alpha)} = \nu_\alpha\alpha^{-1} $. If $(Q,R)$ is gentle, then the cohook at $\alpha$ is a well-defined string beginning at a blossoming vertex.

\begin{definition}
\label{def:add-cohooks}
Let $\rho$ be a string in $(Q,R)$. We define a string $\tau(\rho)$ in $(Q^\blossom,R^\blossom)$ as follows. We view $\rho$ as a string in $(Q^\blossom, R^\blossom)$ whose support is included in $(Q,R)$. In $(Q^\blossom, R^\blossom)$, there exists a unique pair $(\alpha, \beta)$ of arrows in $(Q^\blossom)_1$ such that $\alpha^{-1} \rho \beta$ is a string in $(Q^\blossom, R^\blossom)$.
We then set
\[
    \mathnew{\tau(\rho)} := \h_{(\alpha)} \rho \h_{(\beta)}^{-1}.
\]
\end{definition}

For $\rho$ a string in $(Q,R)$, we define $\tau M_\rho$ to be the representation $M_{\tau(\rho)}$ of $(Q^\blossom,R^\blossom)$.
(For those familiar with Auslander--Reiten translation, $\tau M_\rho$ is the Auslander--Reiten translation of $M_\rho$ in the category of representations of $(Q^\blossom,R^\blossom)$. See \cite{BDMTY20,PPP} for more details.)

\begin{example}
\label{ex:running-bigger}
Consider the blossoming quiver of the gentle quiver considered in \cref{ex:runningGQR}. The result of adding cohooks to the lazy string at 1 is shown with thick arrows in \cref{fig:max-string}.

\begin{figure}[ht]
    \centering
    \begin{tikzpicture}[>= angle 60,->]
        \node[gray!70] (1a) at (0,2) {$\bullet$};
        \node[gray!70] (1b) at (2,2) {$\bullet$};
        \node (2a) at (-1,1) {$\bullet$};
        \node[gray!70] (2b) at (-1,-1) {$\bullet$};
        \node (3a) at (3,1) {$\bullet$};
        \node[gray!70] (3b) at (3,-1) {$\bullet$};
        \node[gray!70] (4a) at (0,-2) {$\bullet$};
        \node[gray!70] (4b) at (2,-2) {$\bullet$};
		\node (a) at (1,1) {$1$};
		\node (b) at (0,0) {$2$};
        \node (c) at (2,0) {$3$};
		\node[gray!70] (d) at (1,-1) {$4$};
		\draw[very thick] (a) to (b);
        \draw[very thick] (a) to (c);
        \draw[gray!70] (c) to (d);
        \draw[gray!70] (b) to (d);
        \draw[gray!70] (1a) to (a);
        \draw[gray!70] (1b) to (a);
        \draw[very thick] (2a) to (b);
        \draw[gray!70] (b) to (2b);
        \draw[very thick] (3a) to (c);
        \draw[gray!70] (c) to (3b);
        \draw[gray!70] (d) to (4a);
        \draw[gray!70] (d) to (4b);
        \draw[dotted,-] ([yshift=-.25cm, xshift=-.2cm]a.west) arc[start angle = -140, end angle = -220, x radius=.4cm, y radius =.4cm];
        \draw[dotted,-] ([yshift=-.25cm, xshift=-.2cm]b.west) arc[start angle = -140, end angle = -220, x radius=.4cm, y radius =.4cm];
		\draw[dotted,-] ([yshift=-.25cm, xshift=-.2cm]c.west) arc[start angle = -140, end angle = -220, x radius=.4cm, y radius =.4cm];
        \draw[dotted,-] ([yshift=-.25cm, xshift=-.2cm]d.west) arc[start angle = -140, end angle = -220, x radius=.4cm, y radius =.4cm];
        \draw[dotted,-] ([yshift=-.25cm,xshift=.2cm]a.east) arc[start angle = 320, end angle = 400, x radius=.4cm, y radius =.4cm];
		\draw[dotted,-] ([yshift=-.25cm,xshift=.2cm]b.east) arc[start angle = 320, end angle = 400, x radius=.4cm, y radius =.4cm];
        \draw[dotted,-] ([yshift=-.25cm,xshift=.2cm]c.east) arc[start angle = 320, end angle = 400, x radius=.4cm, y radius =.4cm];
        \draw[dotted,-] ([yshift=-.25cm,xshift=.2cm]d.east) arc[start angle = 320, end angle = 400, x radius=.4cm, y radius =.4cm];
    \end {tikzpicture}
    \caption{The blossoming quiver $(Q^\blossom,R^\blossom)$ for the gentle quiver $(Q,R)$ from \cref{ex:runningGQR}, with the maximal string (drawn with thick arrows) corresponding to the lazy path at $1$.}
    \label{fig:max-string}
\end{figure}
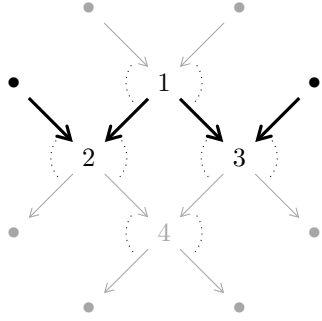
\end{example}

We also define $\tau$ on the shifted projective objects of $(Q,R)$. Let $i \in Q_0$. There are two strings in $(Q^\blossom,R^\blossom)$ which can be obtained by starting at vertex $i$ and continuing as far as possible following reverse arrows. Define $\rho_{I_i}$ to be the maximal string obtained by gluing together these two strings, and let $I_i$ be the corresponding representation of $(Q^\blossom,R^\blossom)$. We then define $\tau P_i[1]=I_i$.

Given $(Q,R)$, we denote by $\mathnew{\IndShift(Q,R)}$ the set of isomorphism classes of indecomposable string representations of $(Q,R)$, and the shifted projective modules $P_i[1]$ for $i \in Q_0$.

\begin{theorem}[{\cite[Lemma 2.42]{PPP}\cite[Theorem 5.1]{BDMTY20}}]
\label{thm:maxstringandindec}
Let $(Q,R)$ be a gentle quiver. Then there is a bijective correspondence $\Psi$:
\begin{equation}
    \label{eq:def-Psi}
    \Psi : \IndShift(Q,R) \rightarrow \big\{ \text{bent maximal strings of } (Q^\blossom,R^\blossom) \big\}.
\end{equation}

The bijection $\Psi$ sends $M_\rho$ to $\tau(\rho)$, and sends $P_i[1]$ to $\rho_{I_i}$.
\end{theorem}

\begin{example}
\label{ex:running example quiver and indecomposable modules}
We continue \cref{ex:g-vector_running_example}. The five elements of $\IndShift(Q,R)$ correspond to the following five bent maximal strings.
$$
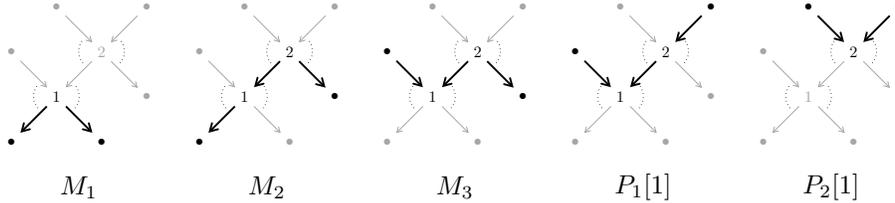
\begin{tikzpicture}[xscale=2.5]
\node(d1) at (3,0){\scalebox{0.6}{\begin{tikzpicture}[>= angle 60,->]
        \node[gray!70] (1a) at (0,2) {$\bullet$};
        \node (1b) at (2,2) {$\bullet$};
        \node (2a) at (-1,1) {$\bullet$};
        \node[gray!70] (2b) at (-1,-1) {$\bullet$};
         \node[gray!70] (c) at (2,0) {$\bullet$};
		\node[gray!70] (d) at (1,-1) {$\bullet$};
		\node (a) at (1,1) {$2$};
		\node (b) at (0,0) {$1$};
		\draw[very thick] (a) to (b);
       \draw[gray!70] (a) to (c);
        \draw[gray!70] (b) to (d);
        \draw[gray!70] (1a) to (a);
        \draw[very thick] (1b) to (a);
        \draw[very thick] (2a) to (b);
        \draw[gray!70] (b) to (2b);
        \draw[dotted,-] ([yshift=-.25cm, xshift=-.2cm]a.west) arc[start angle = -140, end angle = -220, x radius=.4cm, y radius =.4cm];
         \draw[dotted,-] ([yshift=-.25cm, xshift=-.2cm]b.west) arc[start angle = -140, end angle = -220, x radius=.4cm, y radius =.4cm];
         \draw[dotted,-] ([yshift=-.25cm,xshift=.2cm]a.east) arc[start angle = 320, end angle = 400, x radius=.4cm, y radius =.4cm];
		 \draw[dotted,-] ([yshift=-.25cm,xshift=.2cm]b.east) arc[start angle = 320, end angle = 400, x radius=.4cm, y radius =.4cm];
		\end {tikzpicture}}};
    \node(e1) at (4,0){\scalebox{0.6}{\begin{tikzpicture}[>= angle 60,->]
        \node (1a) at (0,2) {$\bullet$};
        \node (1b) at (2,2) {$\bullet$};
        \node[gray!70] (2a) at (-1,1) {$\bullet$};
        \node[gray!70] (2b) at (-1,-1) {$\bullet$};
         \node[gray!70] (c) at (2,0) {$\bullet$};
		\node[gray!70] (d) at (1,-1) {$\bullet$};
		\node (a) at (1,1) {$2$};
		\node[gray!70] (b) at (0,0) {$1$};
		\draw[gray!70] (a) to (b);
       \draw[gray!70] (a) to (c);
        \draw[gray!70] (b) to (d);
        \draw[very thick] (1a) to (a);
        \draw[very thick] (1b) to (a);
        \draw[gray!70] (2a) to (b);
        \draw[gray!70] (b) to (2b);
        \draw[dotted,-] ([yshift=-.25cm, xshift=-.2cm]a.west) arc[start angle = -140, end angle = -220, x radius=.4cm, y radius =.4cm];
         \draw[dotted,-] ([yshift=-.25cm, xshift=-.2cm]b.west) arc[start angle = -140, end angle = -220, x radius=.4cm, y radius =.4cm];
         \draw[dotted,-] ([yshift=-.25cm,xshift=.2cm]a.east) arc[start angle = 320, end angle = 400, x radius=.4cm, y radius =.4cm];
		 \draw[dotted,-] ([yshift=-.25cm,xshift=.2cm]b.east) arc[start angle = 320, end angle = 400, x radius=.4cm, y radius =.4cm];
		\end {tikzpicture}}};
    \node(c1) at (2,0){\scalebox{0.6}{\begin{tikzpicture}[>= angle 60,->]
        \node[gray!70] (1a) at (0,2) {$\bullet$};
        \node[gray!70] (1b) at (2,2) {$\bullet$};
        \node (2a) at (-1,1) {$\bullet$};
        \node[gray!70] (2b) at (-1,-1) {$\bullet$};
         \node (c) at (2,0) {$\bullet$};
		\node[gray!70] (d) at (1,-1) {$\bullet$};
		\node (a) at (1,1) {$2$};
		\node (b) at (0,0) {$1$};
		\draw[very thick] (a) to (b);
       \draw[very thick] (a) to (c);
        \draw[gray!70] (b) to (d);
        \draw[gray!70] (1a) to (a);
        \draw[gray!70] (1b) to (a);
        \draw[very thick] (2a) to (b);
        \draw[gray!70] (b) to (2b);
        \draw[dotted,-] ([yshift=-.25cm, xshift=-.2cm]a.west) arc[start angle = -140, end angle = -220, x radius=.4cm, y radius =.4cm];
         \draw[dotted,-] ([yshift=-.25cm, xshift=-.2cm]b.west) arc[start angle = -140, end angle = -220, x radius=.4cm, y radius =.4cm];
         \draw[dotted,-] ([yshift=-.25cm,xshift=.2cm]a.east) arc[start angle = 320, end angle = 400, x radius=.4cm, y radius =.4cm];
		 \draw[dotted,-] ([yshift=-.25cm,xshift=.2cm]b.east) arc[start angle = 320, end angle = 400, x radius=.4cm, y radius =.4cm];
		\end {tikzpicture}}};
    \node(b1) at (1,0){\scalebox{0.6}{\begin{tikzpicture}[>= angle 60,->]
        \node[gray!70] (1a) at (0,2) {$\bullet$};
        \node[gray!70] (1b) at (2,2) {$\bullet$};
        \node[gray!70] (2a) at (-1,1) {$\bullet$};
        \node (2b) at (-1,-1) {$\bullet$};
         \node (c) at (2,0) {$\bullet$};
		\node[gray!70] (d) at (1,-1) {$\bullet$};
		\node (a) at (1,1) {$2$};
		\node (b) at (0,0) {$1$};
		\draw[very thick] (a) to (b);
       \draw[very thick] (a) to (c);
        \draw[gray!70] (b) to (d);
        \draw[gray!70] (1a) to (a);
        \draw[gray!70] (1b) to (a);
        \draw[gray!70] (2a) to (b);
        \draw[very thick] (b) to (2b);
        \draw[dotted,-] ([yshift=-.25cm, xshift=-.2cm]a.west) arc[start angle = -140, end angle = -220, x radius=.4cm, y radius =.4cm];
         \draw[dotted,-] ([yshift=-.25cm, xshift=-.2cm]b.west) arc[start angle = -140, end angle = -220, x radius=.4cm, y radius =.4cm];
         \draw[dotted,-] ([yshift=-.25cm,xshift=.2cm]a.east) arc[start angle = 320, end angle = 400, x radius=.4cm, y radius =.4cm];
		 \draw[dotted,-] ([yshift=-.25cm,xshift=.2cm]b.east) arc[start angle = 320, end angle = 400, x radius=.4cm, y radius =.4cm];
		\end {tikzpicture}}};

    \node(a1) at (0,0){\scalebox{0.6}{\begin{tikzpicture}[>= angle 60,->]
        \node[gray!70] (1a) at (0,2) {$\bullet$};
        \node[gray!70] (1b) at (2,2) {$\bullet$};
        \node[gray!70] (2a) at (-1,1) {$\bullet$};
        \node (2b) at (-1,-1) {$\bullet$};
         \node[gray!70] (c) at (2,0) {$\bullet$};
		\node (d) at (1,-1) {$\bullet$};
		\node[gray!70] (a) at (1,1) {$2$};
		\node (b) at (0,0) {$1$};
		\draw[gray!70] (a) to (b);
       \draw[gray!70] (a) to (c);
        \draw[very thick] (b) to (d);
        \draw[gray!70] (1a) to (a);
        \draw[gray!70] (1b) to (a);
        \draw[gray!70] (2a) to (b);
        \draw[very thick] (b) to (2b);
        \draw[dotted,-] ([yshift=-.25cm, xshift=-.2cm]a.west) arc[start angle = -140, end angle = -220, x radius=.4cm, y radius =.4cm];
         \draw[dotted,-] ([yshift=-.25cm, xshift=-.2cm]b.west) arc[start angle = -140, end angle = -220, x radius=.4cm, y radius =.4cm];
         \draw[dotted,-] ([yshift=-.25cm,xshift=.2cm]a.east) arc[start angle = 320, end angle = 400, x radius=.4cm, y radius =.4cm];
		 \draw[dotted,-] ([yshift=-.25cm,xshift=.2cm]b.east) arc[start angle = 320, end angle = 400, x radius=.4cm, y radius =.4cm];
		\end {tikzpicture}}};
          \node (a2) at (0,-1.5) {$M_1$};
                 \node (b2) at (1,-1.5) {$M_2$};
                 \node (c2) at (2,-1.5) {$M_3$};
                 \node (d2) at (3,-1.5) {$P_1[1]$};
                 \node (e2) at (4,-1.5){$P_2[1]$};
\end{tikzpicture}$$
\end{example}

\subsection{Non-kissing complex for gentle quivers}
\label{ss:nonkissing_gentle}
This subsection follows \cite[Section 2]{PPPlocg}. Let $(Q,R)$ be a gentle quiver, with blossoming quiver $(Q^\blossom,R^\blossom)$. Let $\rho = \beta_1 \ldots \beta_k$ be a string of $(Q^\blossom,R^\blossom)$. A \textnew{substring} of $\rho$ is either
 \begin{itemize}[itemsep=1mm]
     \item a lazy path $e_{s(\beta_i)}$ for some $i \in \{1,\ldots,k\}$ or $e_{t(\beta_k)}$; or,
     \item a string $\beta_i \ldots \beta_j$ for some $1 \leqslant i \leqslant j \leqslant k$.
 \end{itemize}
 A substring $\sigma = \beta_i \ldots \beta_j$ is said to be \textnew{on the top} of $\rho$ whenever both of the following conditions are satisfied:
 \begin{itemize}[itemsep=1mm]
     \item either $i=1$ or $\beta_{i-1}^{-1} \in Q_1$, and
     \item either $j=n$ or $\beta_{j+1} \in Q_1$.
 \end{itemize}
 Dually, we say that $\sigma$ is \textnew{at the bottom} of $\rho$ whenever both of the following conditions are satisfied:
 \begin{itemize}[itemsep=1mm]
     \item either $i=1$ or $\beta_{i-1} \in Q_1$, and,
     \item either $j=n$ or $\beta_{j+1}^{-1} \in Q_1$.
 \end{itemize}
 (Note that these definitions can in particular be applied to cases where $\sigma$ is a lazy path.) See \cref{fig:TopBot} to visualize them. We write $\Top^0(\rho)$ for the multiset of substrings of $\rho$ that are on the top of $\rho$ and contained in $Q$ (rather than $Q^\blossom$). We consider these strings up to reversal. Similarly, we write $\Bot^0(\rho)$ for the multiset of substrings of $\rho$ that are at the bottom of $\rho$ and contained in $Q$.
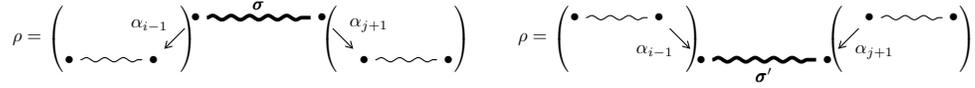
\begin{figure}[h!]
\centering
		\scalebox{0.7}{\begin{tikzpicture}[yshift = 1cm, xshift = 5cm, ->,line width=0.2mm,>= angle 60,color=black, scale=0.8]
			\node (rho) at (-1,.5){$\rho =$};
			\node (1) at (0,0){$\bullet$};
			\node (2) at (2,0){$\bullet$};
			\node (3) at (3,1){$\bullet$};
			\node (4) at (6,1){$\bullet$};
			\node (5) at (7,0){$\bullet$};
			\node (6) at (9,0){$\bullet$};
			\draw[-,decorate, decoration={snake,amplitude=.4mm}] (1) -- (2);
			\draw (3) -- node[above left]{$\alpha_{i-1}$} (2);
			\draw[-,line width=0.7mm,decorate, decoration={snake,amplitude=.4mm}] (3) -- node[above]{$\pmb \sigma$} (4);
			\draw (4) -- node[above right]{$\alpha_{j+1}$} (5);
			\draw[-,decorate, decoration={snake,amplitude=.4mm}] (5) -- (6);
			\node at (-.3,0.5){$\left(\vphantom{\begin{matrix}
						\\
						\\
						\\
				\end{matrix}}\right.$};
			\node at (2.8,0.5){$\left.\vphantom{\begin{matrix}
						\\
						\\
						\\
				\end{matrix}}\right)$};
			\node at (6.2,0.5){$\left(\vphantom{\begin{matrix}
						\\
						\\
						\\
				\end{matrix}}\right.$};
			\node at (9.3,0.5){$\left.\vphantom{\begin{matrix}
						\\
						\\
						\\
				\end{matrix}}\right)$};

			\begin{scope}[xshift = 12cm,yshift=1cm]
				\node (rho) at (-1,-.5){$\rho =$};
				\node (1) at (0,0){$\bullet$};
				\node (2) at (2,0){$\bullet$};
				\node (3) at (3,-1){$\bullet$};
				\node (4) at (6,-1){$\bullet$};
				\node (5) at (7,0){$\bullet$};
				\node (6) at (9,0){$\bullet$};
				\draw[-,decorate, decoration={snake,amplitude=.4mm}] (1) -- (2);
				\draw (2) -- node[below left]{$\alpha_{i-1}$} (3);
				\draw[-,line width=0.7mm,decorate, decoration={snake,amplitude=.4mm}] (3) -- node[below]{$\pmb \sigma'$} (4);
				\draw (5) -- node[below right]{$\alpha_{j+1}$} (4);
				\draw[-,decorate, decoration={snake,amplitude=.4mm}] (5) -- (6);
				\node at (-.3,-0.5){$\left(\vphantom{\begin{matrix}
							\\
							\\
							\\
					\end{matrix}}\right.$};
				\node at (2.8,-0.5){$\left.\vphantom{\begin{matrix}
							\\
							\\
							\\
					\end{matrix}}\right)$};
				\node at (6.2,-0.5){$\left(\vphantom{\begin{matrix}
							\\
							\\
							\\
					\end{matrix}}\right.$};
				\node at (9.3,-0.5){$\left.\vphantom{\begin{matrix}
							\\
							\\
							\\
					\end{matrix}}\right)$};
			\end{scope}
		\end{tikzpicture}}
  \caption{Illustration of a substring $\sigma$ at the top of $\rho$ (left), and a substring $\sigma'$ at the bottom of $\rho$ (right).}
  \label{fig:TopBot}
  \end{figure}

\begin{definition}
    Let $(\rho,\mu)$ be a pair of maximal strings of $(Q,R)$. We say that $\rho$ \textnew{kisses} $\mu$ if $\Top^0(\rho) \cap \Bot^0(\mu) \neq \varnothing$. We say that $\rho$ and $\mu$ are \textnew{kissing} if $\rho$ kisses $\mu$ or $\mu$ kisses $\rho$ (or both). Otherwise, $\rho$ and $\mu$ are \textnew{non-kissing}. We say that a string $\mu$ is \textnew{self-kissing} if $\mu$ kisses itself.
\end{definition}

\begin{definition}
The \textnew{non-kissing complex} of $(Q,R)$ is the simplicial complex $\NK(Q,R)$
whose vertices are given by non-self-kissing maximal strings of $(Q^\blossom,R^\blossom)$, and whose faces are given by the collections of pairwise non-kissing maximal strings.

Recall that a maximal string is called straight if it (or its reversal) uses only arrows from $Q_1$, not inverse arrows.
A maximal string that is straight cannot kiss any other string, so that they appear in all facets of $\NK(Q,R)$. The \textnew{reduced non-kissing complex} $\NKred(Q,R)$ is the simplicial complex whose faces are the collections of pairwise non-kissing bent strings of $(Q,R)$.
\end{definition}

\begin{example} We continue \cref{ex:g-vector_running_example}. The reduced non-kissing complex is shown in \cref{fig:reducednonkisscx}.

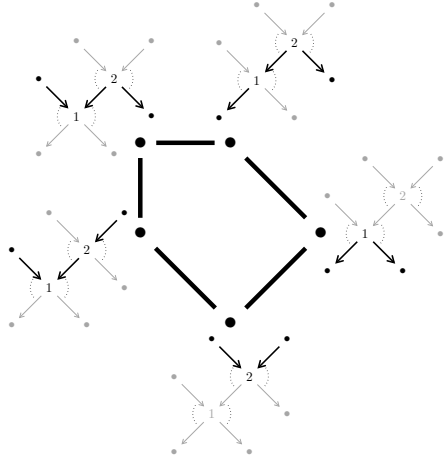
\begin{figure}[!ht]
\centering
\begin{tikzpicture}[scale=1.2] \node (a) at (1,0) {$\bullet$} ;
\node (b) at (0,1) {$\bullet$};
\node (c) at (-1,1) {$\bullet$};
\node (d) at (-1,0) {$\bullet$};
\node (e) at (0,-1){$\bullet$};
\draw[line width=.6mm] (a) -- (b);
\draw[line width=.6mm] (b) -- (c);
\draw[line width=.6mm] (c) -- (d);
\draw[line width=.6mm] (d) -- (e);
\draw[line width=.6mm] (e) -- (a);
\node(d1) at (-1.8,-.4){\scalebox{0.5}{\begin{tikzpicture}[>= angle 60,->]
        \node[gray!70] (1a) at (0,2) {$\bullet$};
        \node (1b) at (2,2) {$\bullet$};
        \node (2a) at (-1,1) {$\bullet$};
        \node[gray!70] (2b) at (-1,-1) {$\bullet$};
         \node[gray!70] (c) at (2,0) {$\bullet$};
		\node[gray!70] (d) at (1,-1) {$\bullet$};
		\node (a) at (1,1) {$2$};
		\node (b) at (0,0) {$1$};
		\draw[very thick] (a) to (b);
       \draw[gray!70] (a) to (c);
        \draw[gray!70] (b) to (d);
        \draw[gray!70] (1a) to (a);
        \draw[very thick] (1b) to (a);
        \draw[very thick] (2a) to (b);
        \draw[gray!70] (b) to (2b);
        \draw[dotted,-] ([yshift=-.25cm, xshift=-.2cm]a.west) arc[start angle = -140, end angle = -220, x radius=.4cm, y radius =.4cm];
         \draw[dotted,-] ([yshift=-.25cm, xshift=-.2cm]b.west) arc[start angle = -140, end angle = -220, x radius=.4cm, y radius =.4cm];
         \draw[dotted,-] ([yshift=-.25cm,xshift=.2cm]a.east) arc[start angle = 320, end angle = 400, x radius=.4cm, y radius =.4cm];
		 \draw[dotted,-] ([yshift=-.25cm,xshift=.2cm]b.east) arc[start angle = 320, end angle = 400, x radius=.4cm, y radius =.4cm];
		\end {tikzpicture}}};
    \node(e1) at (0,-1.8){\scalebox{0.5}{\begin{tikzpicture}[>= angle 60,->]
        \node (1a) at (0,2) {$\bullet$};
        \node (1b) at (2,2) {$\bullet$};
        \node[gray!70] (2a) at (-1,1) {$\bullet$};
        \node[gray!70] (2b) at (-1,-1) {$\bullet$};
         \node[gray!70] (c) at (2,0) {$\bullet$};
		\node[gray!70] (d) at (1,-1) {$\bullet$};
		\node (a) at (1,1) {$2$};
		\node[gray!70] (b) at (0,0) {$1$};
		\draw[gray!70] (a) to (b);
       \draw[gray!70] (a) to (c);
        \draw[gray!70] (b) to (d);
        \draw[very thick] (1a) to (a);
        \draw[very thick] (1b) to (a);
        \draw[gray!70] (2a) to (b);
        \draw[gray!70] (b) to (2b);
        \draw[dotted,-] ([yshift=-.25cm, xshift=-.2cm]a.west) arc[start angle = -140, end angle = -220, x radius=.4cm, y radius =.4cm];
         \draw[dotted,-] ([yshift=-.25cm, xshift=-.2cm]b.west) arc[start angle = -140, end angle = -220, x radius=.4cm, y radius =.4cm];
         \draw[dotted,-] ([yshift=-.25cm,xshift=.2cm]a.east) arc[start angle = 320, end angle = 400, x radius=.4cm, y radius =.4cm];
		 \draw[dotted,-] ([yshift=-.25cm,xshift=.2cm]b.east) arc[start angle = 320, end angle = 400, x radius=.4cm, y radius =.4cm];
		\end {tikzpicture}}};
    \node(c1) at (-1.5,1.5){\scalebox{0.5}{\begin{tikzpicture}[>= angle 60,->]
        \node[gray!70] (1a) at (0,2) {$\bullet$};
        \node[gray!70] (1b) at (2,2) {$\bullet$};
        \node (2a) at (-1,1) {$\bullet$};
        \node[gray!70] (2b) at (-1,-1) {$\bullet$};
         \node (c) at (2,0) {$\bullet$};
		\node[gray!70] (d) at (1,-1) {$\bullet$};
		\node (a) at (1,1) {$2$};
		\node (b) at (0,0) {$1$};
		\draw[very thick] (a) to (b);
       \draw[very thick] (a) to (c);
        \draw[gray!70] (b) to (d);
        \draw[gray!70] (1a) to (a);
        \draw[gray!70] (1b) to (a);
        \draw[very thick] (2a) to (b);
        \draw[gray!70] (b) to (2b);
        \draw[dotted,-] ([yshift=-.25cm, xshift=-.2cm]a.west) arc[start angle = -140, end angle = -220, x radius=.4cm, y radius =.4cm];
         \draw[dotted,-] ([yshift=-.25cm, xshift=-.2cm]b.west) arc[start angle = -140, end angle = -220, x radius=.4cm, y radius =.4cm];
         \draw[dotted,-] ([yshift=-.25cm,xshift=.2cm]a.east) arc[start angle = 320, end angle = 400, x radius=.4cm, y radius =.4cm];
		 \draw[dotted,-] ([yshift=-.25cm,xshift=.2cm]b.east) arc[start angle = 320, end angle = 400, x radius=.4cm, y radius =.4cm];
		\end {tikzpicture}}};
    \node(b1) at (.5,1.9){\scalebox{0.5}{\begin{tikzpicture}[>= angle 60,->]
        \node[gray!70] (1a) at (0,2) {$\bullet$};
        \node[gray!70] (1b) at (2,2) {$\bullet$};
        \node[gray!70] (2a) at (-1,1) {$\bullet$};
        \node (2b) at (-1,-1) {$\bullet$};
         \node (c) at (2,0) {$\bullet$};
		\node[gray!70] (d) at (1,-1) {$\bullet$};
		\node (a) at (1,1) {$2$};
		\node (b) at (0,0) {$1$};
		\draw[very thick] (a) to (b);
       \draw[very thick] (a) to (c);
        \draw[gray!70] (b) to (d);
        \draw[gray!70] (1a) to (a);
        \draw[gray!70] (1b) to (a);
        \draw[gray!70] (2a) to (b);
        \draw[very thick] (b) to (2b);
        \draw[dotted,-] ([yshift=-.25cm, xshift=-.2cm]a.west) arc[start angle = -140, end angle = -220, x radius=.4cm, y radius =.4cm];
         \draw[dotted,-] ([yshift=-.25cm, xshift=-.2cm]b.west) arc[start angle = -140, end angle = -220, x radius=.4cm, y radius =.4cm];
         \draw[dotted,-] ([yshift=-.25cm,xshift=.2cm]a.east) arc[start angle = 320, end angle = 400, x radius=.4cm, y radius =.4cm];
		 \draw[dotted,-] ([yshift=-.25cm,xshift=.2cm]b.east) arc[start angle = 320, end angle = 400, x radius=.4cm, y radius =.4cm];
		\end {tikzpicture}}};
    \node(a1) at (1.7,.2){\scalebox{0.5}{\begin{tikzpicture}[>= angle 60,->]
        \node[gray!70] (1a) at (0,2) {$\bullet$};
        \node[gray!70] (1b) at (2,2) {$\bullet$};
        \node[gray!70] (2a) at (-1,1) {$\bullet$};
        \node (2b) at (-1,-1) {$\bullet$};
         \node[gray!70] (c) at (2,0) {$\bullet$};
		\node (d) at (1,-1) {$\bullet$};
		\node[gray!70] (a) at (1,1) {$2$};
		\node (b) at (0,0) {$1$};
		\draw[gray!70] (a) to (b);
       \draw[gray!70] (a) to (c);
        \draw[very thick] (b) to (d);
        \draw[gray!70] (1a) to (a);
        \draw[gray!70] (1b) to (a);
        \draw[gray!70] (2a) to (b);
        \draw[very thick] (b) to (2b);
        \draw[dotted,-] ([yshift=-.25cm, xshift=-.2cm]a.west) arc[start angle = -140, end angle = -220, x radius=.4cm, y radius =.4cm];
         \draw[dotted,-] ([yshift=-.25cm, xshift=-.2cm]b.west) arc[start angle = -140, end angle = -220, x radius=.4cm, y radius =.4cm];
         \draw[dotted,-] ([yshift=-.25cm,xshift=.2cm]a.east) arc[start angle = 320, end angle = 400, x radius=.4cm, y radius =.4cm];
		 \draw[dotted,-] ([yshift=-.25cm,xshift=.2cm]b.east) arc[start angle = 320, end angle = 400, x radius=.4cm, y radius =.4cm];
		\end {tikzpicture}}};
\end{tikzpicture}
\caption{The reduced non-kissing complex $\NKred(Q,R)$ of the gentle quiver given in \cref{ex:running example quiver and indecomposable modules}.}
\label{fig:reducednonkisscx}
\end{figure}
\end{example}

\subsection{$g$-vectors for gentle algebras}
\label{ss:gentle-g}

Let $(Q,R)$ be a gentle quiver. We will now explain how to determine the $g$-vector of an indecomposable module or shifted projective for $(Q,R)$, in terms of the maximal string in $(Q^\blossom, R^\blossom)$ corresponding to it under $\Psi$.

Let $\rho = \beta_1 \ldots \beta_k$ be a maximal string in $(Q^\blossom,R^\blossom)$. A \textnew{peak} of $\rho$ is a vertex $v$ such that the lazy path $e_v$ is on top of $\rho$. Similarly, a \textnew{valley} of $\rho$ is a vertex $v$ such that $e_v$ is at the bottom of $\rho$. We denote by $\peak(\rho)$ and $\deep(\rho)$ the respective multisets of peaks and valleys of $\rho$. (These are multisets because a string could pass through the same vertex multiple times; each of those occurrences should be considered separately.)

The $g$-vector of $\rho$ is given by
\begin{equation}
    \label{eq:g-rho}
        \pmb{g}(\rho) = \sum_{v \in \peak(\rho)} \ve_v -\sum_{v \in \deep(\rho)} \ve_v\end{equation}
where $(\ve_q)_{q \in Q_0}$ is the standard basis of $\ZZ^{Q_0}$.

\begin{example} Continuing \cref{ex:running-bigger}, let $\rho$ be the bent maximal path corresponding to the lazy path at $1$. Looking at \cref{fig:max-string}, we see that $\rho$ has a peak at 1 and valleys at 2 and 3, so
$\pmb{g}(\rho) = \ve_1 - \ve_2 - \ve_3 = (1,-1,-1,0)$.
\end{example}

The following theorem combines several results of \cite{BDMTY20,PPP},
and explains how the combinatorics of the non-kissing complex $\NK(Q,R)$ (\cref{ss:nonkissing_gentle}) corresponds to the combinatorics of compatible $g$-vectors (\cref{ss:general-g-vector}).

\begin{theorem}
\label{thm:g-vect-blossom}
Let $(Q,R)$ be a gentle quiver.
\begin{enumerate}
    \item \cite[Proposition 1.43]{PPP}
For $M$ an indecomposable module or shifted projective for $(Q,R)$, then $\pmb{g}_M=\pmb{g}(\Psi(M))$.
\item \cite[Proposition 1.52, Lemma 2.45]{PPP}, \cite[Theorem 4.3]{BDMTY20} An indecomposable module $M$ is $\tau$-rigid if and only if $\Psi(M)$ is non self-kissing.
\item \cite[Proposition 1.52, Lemma 2.45]{PPP} \cite[Theorem 4.3]{BDMTY20} Two $\tau$-rigid modules $M,N$ are compatible if and only if $\Psi(M)$ and $\Psi(N)$ are non-kissing.
\item \cite[Theorem 4.17]{PPP} The $g$-vector fan of $(Q,R)$ realizes the reduced non-kissing complex of $\NKred(Q,R)$ as a fan.
\end{enumerate}
\end{theorem}

The previous theorem amounts to saying that, for $(Q,R)$ a gentle algebra, $\gfan(Q,R)$ can be entirely understood in terms of the combinatorics of kissings of maximal strings in $(Q^\blossom,R^\blossom)$.

\begin{example} We continue \cref{ex:g-vector_running_example}.
Consider the bent maximal strings from \cref{ex:running example quiver and indecomposable modules}. Applying \eqref{eq:g-rho} to each of them, we obtain a $g$-vector, which, we can see, agrees with the $g$-vectors assigned to the objects which correspond to them under the bijection of \cref{thm:maxstringandindec}.
\end{example}

\subsection{Brick representations}
\label{ss:brick}

Let $\Lambda$ be an algebra over the field $\KK$. A $\Lambda$-module $M$ is called a \textnew{brick} if $\Hom(M,M)$ is a skew field, i.e., all its non-zero elements are invertible. If $M$ is finite-dimensional over $\KK$ and $\KK$ is algebraically closed, then $M$ is a brick if and only if $\Hom(M,M)\simeq \KK$.

For $\Lambda$ a locally gentle algebra, the representation $M(\rho)$ is a brick provided there is no substring of $\rho$ which occurs in both the top and the bottom of $\rho$, since such substrings give rise to non-invertible endomorphisms.

\begin{theorem}
\label{thm:finite-bricks}
Let $(Q,R)$ be a gentle quiver. The following conditions are equivalent:
\begin{enumerate}

\item Every cycle in $Q$ has at least one relation.
\item $(Q,R)$ has a finite number of indecomposable representations.
\item $(Q,R)$ has a finite number of brick representations.
\end{enumerate}\end{theorem}

\begin{proof}
(1) implies that $(Q,R)$ has no cyclic strings, so it has no band representations. Further, there is a bound on the maximum possible length of a string, so there are only finitely many strings. This implies (2). Conversely, given (2), there cannot be any band representations, which implies (1).

The equivalence of (2) and (3) was shown in \cite{Pla,Mou}.
\end{proof}

\subsection{Locally gentle quivers}
\label{ss:locally}

We now discuss how to extend the results of \cref{ss:general-g-vector,ss:blossoming_quivers,ss:nonkissing_gentle,ss:gentle-g,ss:brick} to include the case of locally gentle quivers.

Let $(Q,R)$ be a locally gentle quiver. We define $(Q^\blossom,R^\blossom)$ exactly as before. The definition of a maximal string is slightly enhanced, following \cite{PPPlocg}. At each end, rather than necessarily terminating at a blossoming vertex, it is permissible for a string to end by cycling infinitely around an oriented cycle without relations in $(Q,R)$, either with or against the arrows. Two maximal strings $\rho$ and $\mu$ kiss if there is a finite substring in $Q$ which is on the top of $\rho$ and the bottom of $\mu$, or vice versa. Palu--Pilaud--Plamondon \cite{PPPlocg} then define the non-kissing complex $\NK(Q,R)$ and reduced non-kissing complex $\NKred(Q,R)$ just as before.

However, a key point to understand is that, in the locally gentle case, the complexes $\NK(Q,R)$ and $\NKred(Q,R)$ are just defined as simplicial complexes, without any direct connection to representation theory, and are not realized as fans. We will now explain how to make the connection to representation theory, following the work of Aoki and Yurikusa \cite{aokiyurikusa}. This will also realize these complexes as fans.

To establish this connection, we first need to specify what algebra we want to relate to $(Q,R)$ when $(Q,R)$ is a locally gentle quiver. The approach that we will take is to consider the corresponding \textnew{complete gentle algebra}, which we now define. Let $\widehat{\KK Q}$ be the completion of $\KK Q$ with respect to the topology defined by the arrow ideal. The corresponding complete gentle algebra is $\Lambda = \widehat{\KK Q}/ \overline{\langle R\rangle}$, where $\overline {\langle R \rangle}$ is the closure of the ideal generated by the relations in $R$. If $(Q,R)$ is gentle, then $\Lambda$ is just the finite-dimensional gentle algebra $\KK Q/\langle R \rangle$, but otherwise $\Lambda$ is an infinite-dimensional algebra and so, by definition, is not gentle.

Next, we need to consider the $g$-vector fan of a complete gentle algebra.
In \cref{ss:general-g-vector}, we gave the definition of the $g$-vector fan only for finite-dimensional algebras. For general infinite-dimensional algebras, the Auslander--Reiten translation $\tau$ is not defined, so the definition of the $g$-vector fan in \cref{ss:general-g-vector} does not apply.

Nevertheless, Aoki and Yurikusa \cite{aokiyurikusa}
provide an approach to $g$-vector fans in terms of two-term complexes of finitely generated projective modules, which allows them to define the $g$-vector fan of a complete gentle algebra (or somewhat more general algebras).
We will not delve into all the details here. The key point for our purposes is that this $g$-vector fan can be identified with the $g$-vector fan of another algebra, which is finite-dimensional and to which the definition in \cref{ss:general-g-vector} therefore applies.

Define $\mathfrak m$ to be the ideal in $\Lambda$ generated by the sum over all vertices $v$ of the oriented cycles without relations starting and ending at $v$.
Let $\overline\Lambda = \Lambda/\mathfrak m$. The algebra $\overline \Lambda$ is finite-dimensional, so it has a $g$-vector fan in the sense we have already defined. Aoki and Yurikusa then prove the following theorem:

\begin{proposition}[{\cite[Proposition 7.9]{aokiyurikusa}}] \label{prop:ay}
Let $\Lambda$ be the complete gentle algebra associated to a locally gentle quiver $(Q,R)$, and let $\overline \Lambda=\Lambda/\mathfrak m$ be the quotient defined as above. Then $\gfan(\Lambda)=\gfan(\overline \Lambda)$. \end{proposition}

While this is a theorem for Aoki and Yurikusa, because they have a definition of $g$-vector fan which is applicable to (some) infinite-dimensional algebras, including all complete gentle algebras, for our purposes, we will treat it as the definition of $\gfan(\Lambda)$.

For our purposes, it turns out to be convenient to work with another algebra $\ool$, which we now define.

\begin{definition} Let $(Q,R)$ be a locally gentle algebra. Define $\ool$ to be $\KK Q/\langle R, C\rangle$, where $C$ is the ideal generated by the set of oriented cycles without relations, starting and ending at vertex $v$ for all $v$.
\end{definition}

$\ool$ is a finite-dimensional algebra, and it is quotient of $\Lambda$ and of $\overline \Lambda$.

We can now state the main result of this subsection:

\begin{theorem}
    \label{thm:connection_g-vector_NKred} Let $(Q,R)$ be a locally gentle quiver, and let $\Lambda=\widehat{\KK Q}/\overline {\langle R\rangle}$. Then the $g$-vector fans of $\Lambda$, $\overline \Lambda$, and $\ool$ coincide, and realize $\NKred(Q,R)$ as a fan.
\end{theorem}

We will break the proof of the theorem down into several distinct statements. The fact that the $g$-vector fans of $\Lambda$ and $\overline \Lambda$ is \cref{prop:ay}. We now turn to showing that the $g$-vector fans of $\overline \Lambda$ and $\ool$ coincide.

\begin{proposition} \label{prop:add-a-bar} $\gfan(\ool)=\gfan(\overline \Lambda)$.\end{proposition}

\begin{proof}
The algebras $\overline \Lambda$ and $\ool$ are very closely related. Write $t$ for the sum of all the oriented cycles without relations, so that $\mathfrak m$ is generated by $t$. If a vertex $v$ is on only one oriented cycle, then $e_vte_v \in \mathfrak m$: thus $\mathfrak m$ already contains the oriented cycle that starts and ends at $v$. If each vertex lay on at most one oriented cycle, then the elements of the ideal $C$ would already all be in $\mathfrak m$, and so $\ool$ and $\overline \Lambda$ would be isomorphic.

$\ool$ and $\overline \Lambda$ differ if there is a vertex $v$ which lies on two oriented cycles without relations. In this case, $e_vte_v\in\mathfrak m$ is the sum of the two cycles that start and end at $v$. This means that the projective module $P_v$ for $\overline \Lambda$ is projective-injective and non-uniserial, with socle at $v$. This is different from the projective at $v$ in $\ool$, which is isomorphic to $P_v/S_v$. To summarize, $\ool$ is the quotient of $\overline \Lambda$ by the socle of some projective-injective modules.

Now, by \cite[Theorem 3.3(2)]{Adachi}, there is a combinatorial isomorphism between the $g$-vector fans of $\overline \Lambda$ and $\ool$. In fact, as we shall explain, this combinatorial isomorphism is realized by the identity map. We first describe the combinatorial isomorphism. It sends the ray corresponding to the projective at vertex $i$ of $\overline \Lambda$ to the ray corresponding to the projective at vertex $i$ of $\ool$, and similarly for the shifted projectives. Any other ray of $\overline \Lambda$ corresponds to a module $M$ for $\Lambda$ which is also a module over $\ool$, and the ray corresponding to $M$ as a $\overline \Lambda$-module is sent to the ray corresponding to $M$ as a $\ool$-module.

Now we just have to check that the corresponding rays in the two $g$-vector fans have identical $g$-vectors. This is obvious for the projective and shifted projective modules. For another ray of $\gfan(\overline \Lambda)$, corresponding to a module $M$, consider its minimal projective presentation as a $\overline \Lambda$ module, say $P'\xrightarrow{f} P\rightarrow M\rightarrow 0$. Tensoring by $\ool$, we obtain a projective presentation of $M$ as a module over $\ool$. To check that the presentation is still minimal, observe that if it weren't, that would be because there is a summand of $\ool\otimes_{\overline\Lambda} P'$ which is now sent to zero by $\ool\otimes_{\overline\Lambda}f$. In order for that to happen, it would have to be the case that that summand was formerly sent to the socle of some projective-injective $P_i$. In that case, since the socle is $S_i$, the summand would also necessarily be $P_i$. So the minimal projective presentation of $M$ over $\overline \Lambda$ would have $P_i$ present in both factors, which is impossible for a rigid $\overline\Lambda$ module \cite[Proposition 2.5]{adachi-iyama-reiten}. Thus, the $g$-vector of $M$ viewed as a $\overline \Lambda$ module and as a $\ool$ module agree. The two $g$-vector fans therefore coincide.
\end{proof}

$\ool$ is typically not a gentle algebra, because of the generators corresponding to the oriented cycles. Nonetheless, $\ool$ is what is called a \textnew{string algebra}. This implies that, as for gentle algebras, all its indecomposable representations are strings or bands \cite{BR87}.
Since bands are not $\tau$-rigid, they do not appear in the $g$-vector fan. The string representations of $\ool$ can be identified with the strings (in our usual sense) for $\Lambda$ which do not contain a substring going completely around an oriented cycle.

We now analyze the $g$-vectors of $\tau$-rigid strings for $\ool$.

\begin{lemma}
    \label{lem:gvect_in_locgen}
Let $\overline X$ be a $\tau$-rigid string representation of $\ool$, corresponding to the string $\gamma_1\dots\gamma_r$. Let $\alpha$ and $\beta$ be such that $\alpha^{-1}\gamma_1$ and $\gamma_r\beta$ are not in $R$. Then
$$\pmb{g}_{\overline X}\,=\, \sum_{i \in \peak(\overline X)} \ve_i - \sum_{i\in \deep(\overline X)}\ve_i - \delta_\alpha \ve_{t(\alpha)} - \delta_\beta\ve_{t(\beta)} $$
where $\delta_\alpha$ is 1 unless adding $\alpha^{-1}$ to $\gamma_1\dots\gamma_r$ creates a full oriented cycle, in which case it is zero, and similarly for $\delta_\beta$.
\end{lemma}

\begin{proof} The proof is essentially the same as that of \cref{thm:g-vect-blossom}~(1). In the setting of that theorem, we are counting peaks and valleys of the maximal walk associated to the string; this includes contributions for what in the present setting would be $t(\alpha)$ and $t(\beta)$. There is no such contribution if adding $\alpha$ (respectively, $\beta$) creates a full cycle, because of the relation in $\ool$ killing the full cycle.\end{proof}

\begin{corollary}
    \label{cor:walks-blossom-cycle}
    The maximal walk in $(Q^\blossom,R^\blossom)$ corresponding to $\overline X$ is given by adding a cohook at either end of the string for $\overline X$, unless the first arrow of the cohook would complete a full oriented cycle, in which case, instead of adding a cohook, the maximal walk ends by spiralling infinitely around the cycle in the direction of its arrows.
\end{corollary}

\begin{proof} [Proof of \cref{thm:connection_g-vector_NKred}] Thanks to \cref{prop:ay,prop:add-a-bar}, we have established that the three $g$-vector fans coincide. \cref{cor:walks-blossom-cycle} shows that the rays of the $g$-vector fan for $\ool$ are identified with
walks which are allowed either to end at a blossoming vertex, or to spiral infinitely around an oriented cycle in either direction. Further, two rays are compatible if the corresponding walks do not kiss.
This exactly describes the non-kissing complex of \cite{PPPlocg}, so this completes the proof of the theorem.\end{proof}

For future use, we also note that the finite-dimensional bricks of $\ool$ can be identified with bricks of $\Lambda$ whose corresponding string never goes the whole way around an oriented cycle.

\begin{lemma}
    \label{lem:overline-bricks}
    $\ool$ has finitely many bricks if and only if there is at least one relation in any non-oriented cycle.
\end{lemma}

\section{Relationship between DKK triangulations and \except{toc}{\linebreak} $g$-vector fans for acyclic graphs}
\label{s:conections-acyclic}

In this section, we strengthen the bijection between the DKK triangulation of an amply framed DAG $G$ and the $g$-vector fan of a gentle quiver $Q_G$ associated to $G$ building on work in \cite{Kentuckygentle}, by viewing it in the context of \cite{PPP,PPPlocg} and $g$-vector fans.

\subsection{Gentle quiver from an amply framed DAG}
\label{ss:DAG_to_gentle}

The following construction, connecting amply framed DAGs and gentle algebras, is one of the key insights in \cite{Kentuckygentle}. Recall that all our DAGs are full.

\begin{definition}[\cite{Kentuckygentle}]
    \label{def:acyclic-blossoming}
    Let $(G,\coloring)$ be an amply framed DAG,
    define its associated \textnew{blossoming quiver} $\mathnew{(Q^\blossom_G,R^\blossom_G)}$ by \emph{reversing} the \tcb{blue} edges of $G$.
    That is,
    $(Q^\blossom_G)_0 = V(G)$, $(Q^\blossom_G)_1 = E(G)$,
    \begin{equation}
    \label{eq:arrows-G-blossom}
        s(e) = \begin{cases}
            \tail(e) & \text{if } \coloring(e) = \text{\tcr{red}},\\
            \head(e) & \text{if } \coloring(e) = \text{\tcb{blue}},
        \end{cases}
        \quad\text{and}\quad
        t(e) = \begin{cases}
            \head(e) & \text{if } \coloring(e) = \text{\tcr{red}},\\
            \tail(e) & \text{if } \coloring(e) = \text{\tcb{blue}}.
        \end{cases}
    \end{equation}
    The relations in $R^\blossom_G$ are all the length-two paths $e_1 e_2$ in $Q^{\blossom}_G$ such that
    $\coloring(e_1) \neq \coloring(e_2)$.
    In addition, let $\mathnew{({Q}_G,{R}_G)}$ be the restriction of $(Q^\blossom_G,R^\blossom_G)$ to the internal vertices $\internal(G)$.
    See \cref{fig:full_edge_labeling} for an example.
\end{definition}

\begin{figure}[ht]
	\centering
	\begin{subfigure}{.3\linewidth}
    	\centering
        \input{figures/quiver_grid_new_convention_a.tex}
        \caption{}
    	 \label{subfig:full_edge_labeling1}
	\end{subfigure}
	\hfill
	\begin{subfigure}{.3\linewidth}
		\centering
        \input{figures/quiver_grid_new_convention_b.tex}
        \caption{}
	\end{subfigure}
	\hfill
	\begin{subfigure}{.2\linewidth}
		\centering
        \input{figures/quiver_grid_new_convention_c.tex}
        \caption{}
	\end{subfigure}
    \caption{
        (A) A full graph $G$ with an ample framing encoded by a bi-colouring of its edges,
        (B) the associated blossoming quiver $(Q^{\blossom}_G,R_G^{\blossom})$ obtained from $G$ by reversing the edges with colour $\textcolor{blue}{2}$
        and with relations all the non-monochromatic length-two paths (see \cref{def:acyclic-blossoming}),
        (C) the quiver $(Q_G,R_G)$ obtained by restricting $(Q^{\blossom}_G,R_G^{\blossom})$ to the internal vertices of $G$.
        }
    \label{fig:full_edge_labeling}
\end{figure}
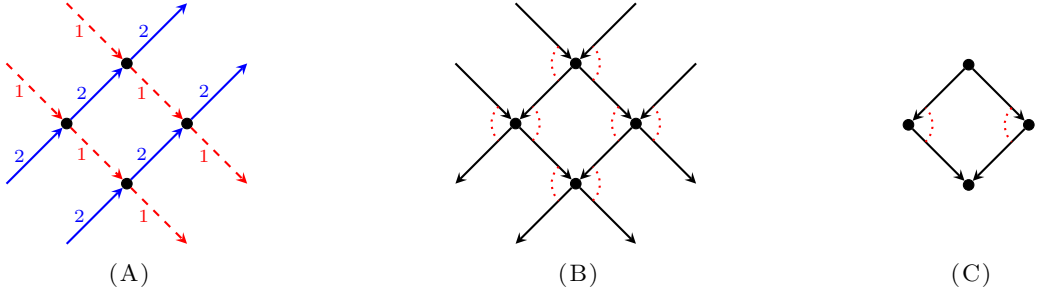

Since our graphs satisfy \cref{convention:graphs no idle edges}, one can verify from the definition of gentle quivers in \cref{ss:gentle-quiver} that both $(Q^\blossom_G,R^\blossom_G)$ and $(Q_G,R_G)$ are gentle (see \cite[Prop. 5.6]{Kentuckygentle}).
The naming above reflects that $(Q^\blossom_G,R^\blossom_G)$ is the \emph{blossoming quiver} of $(Q_G,R_G)$ as defined in \cite{PPP}, see \cref{ss:blossoming_quivers}.

\begin{remark}
We can naturally identify a walk in $G$ with the string of $(Q^\blossom_G,R^\blossom_G)$ that uses the same set of (reversed) arrows.
Explicitly, given a walk $e_1,\dots,e_k$ of $G$, the corresponding string of $(Q^\blossom_G,R^\blossom_G)$ is $e_1^{\epsilon_1} \ldots e_k^{\epsilon_k}$, where
\begin{equation}
    \label{eq:path_to_string}
    \epsilon_i =
    \begin{cases}
        {+}1 & \text{if } \coloring(\alpha_i)=\text{\tcr{red}},\\
        {-}1 & \text{if } \coloring(\alpha_i)=\text{\tcb{blue}}.
    \end{cases}
\end{equation}
In this manner, non-exceptional (resp. exceptional) routes of $G$ correspond to bent (resp. straight) maximal strings of $(Q^\blossom_G,R^\blossom_G)$.
We denote this last bijection by
\begin{equation}
    \label{eq:paths_are_strings}
    \RoutesToStrings :
    \Routes(G) \setminus \cE(G) \to
    \big\{ \text{bent maximal strings of } (Q^\blossom,R^\blossom) \big\}
\end{equation}
\end{remark}

\begin{definition}[\cite{Kentuckygentle}]
    \label{def:KentuckyBij}
    Given an amply framed DAG $(G,\coloring)$, let $\KentuckyBij$ be the composition
    \[
        \mathnew{\KentuckyBij} :=  \Psi^{-1} \circ \RoutesToStrings : \Routes(G) \setminus \cE(G) \to \IndShift(Q_G,R_G),
    \]
    where $\Psi$ is the bijection in \eqref{eq:def-Psi}.
\end{definition}

The following explicit description of $\KentuckyBij$ appears in \cite[Proof of Theorem~5.8]{Kentuckygentle}.
    Given an amply framed DAG $(G,\coloring)$,
    and a non-exceptional route $\route = \alpha_1,\ldots,\alpha_k$ of $G$,
    the module $\KentuckyBij(\route)$ is determined as follows.
    \begin{enumerate}[label=(\roman*)]
        \item If
        $\coloring(\route) = (\tcr{1}^a,\tcb{2}^{k-a})$,
        with $1 \leq a \leq k-1$ (so $\coloring(\alpha_a)=\tcr{1}$ and $\coloring(\alpha_{a+1})=\tcb{2}$),
        then $\KentuckyBij(\route)$ is
        the shifted projective $P_v[1]$, where $v=\head(\alpha_a)=\tail(\alpha_{a+1})$.

        \item If
        $\coloring(\route)=(\tcr{1}^a,\tcb{2},\coloring(\alpha_{a+2}),\ldots,\coloring(\alpha_{j-1}),\tcr{1},\tcb{2}^b)$,
        with $a,b \geq 0$,
        then $\KentuckyBij(\route)$ is
        the module corresponding to the string $\alpha_{a+2}^{\epsilon_{a+2}} \ldots \alpha_{j-1}^{\epsilon_{j-1}}$, with $\epsilon_i$ as in \eqref{eq:path_to_string}.
        When $j=a+2$,
        this is the lazy string at vertex $v = \head(\alpha_{a+1})$,
        which corresponds to the projective module~$P_v$.
    \end{enumerate}

\begin{example}
    \label{ex:bijection_routes_to_modules}
    Let $G = \colorfulXXgraph$ be the amply framed DAG from \cref{ex:master_ex_acyclic_1}.
    The corresponding quiver $(Q_{G},R_{G})$ on its internal vertices is precisely
    the type $A_2$ quiver in \cref{subfig:bijection_routes_to_modules1}, and also in \cref{ex:g-vector_running_example}.
    The result of applying the bijection $\KentuckyBij$ to the five non-exceptional routes of $\colorfulXXgraph$ is:
    \begin{gather*}
        \KentuckyBij(R_1) = P_u = 0 \to \KK, \hspace{.1\linewidth}
        \KentuckyBij(R_2) = P_v = \KK \to \KK, \hspace{.1\linewidth}
        \KentuckyBij(R_3) = \KK \to 0, \\[10pt]
        \KentuckyBij(R_4) = P_u[1], \hspace{.1\linewidth}
        \KentuckyBij(R_5) = P_v[1].
    \end{gather*}
    This is further illustrated in \cref{subfig:bijection_routes_to_modules2}.
\end{example}
\begin{figure}[ht]
	\centering
	\begin{subfigure}{\textwidth}
		\centering
        \[\input{figures/master_example_acyclic_graph_quiver.tex}\]
	    \caption{}
	    \label{subfig:bijection_routes_to_modules1}
	\end{subfigure}
	\begin{subfigure}{\textwidth}
        \centering
        \input{figures/master_example_bijection_routes.tex}
        \caption{}
        \label{subfig:bijection_routes_to_modules2}
	\end{subfigure}
	\caption{(A) The amply framed DAG $G = \colorfulXXgraph$ and its associated reduced quiver $(Q_G,R_G)$, which is precisely that in \cref{ex:g-vector_running_example}.
	(B) The bijective map $\Gamma : \Routes(G) \setminus \cE(G) \to \IndShift(Q_G,R_G)$ between non-exceptional routes in $G$ and indecomposable modules of the quiver $(Q_G,R_G)$.}
\end{figure}

The following result shows that the map $\Gamma$ is a bijection that sends compatible routes to compatible indecomposable modules. Thus, $\Gamma$ induces a bijection between maximal cliques of routes and maximal collections of compatible modules.

\begin{theorem}[{\cite[Theorem 5.8, Lemma 5.14]{Kentuckygentle}}]
    \label{thm:comb-iso-fans}
    The map $\KentuckyBij$ is a bijection.
    Moreover, two routes $\route,\route' \in \Routes(G) \setminus \cE(G)$ are compatible (in the sense of \cref{def:compatible-routes})
    if and only if the indecomposable modules
    $\KentuckyBij(\route)$ and $\KentuckyBij(\route')$ are compatible (in the sense of \cref{def:compatible-modules}).
\end{theorem}

\begin{example}
Continuing with \cref{ex:bijection_routes_to_modules}, the image of the maximal cliques gives the maximal collection of compatible indecomposable modules in $\IndShift(Q_G,R_G)$:
\[
\{\KentuckyBij(R_1),\KentuckyBij(R_2)\}, \{\KentuckyBij(R_2),\KentuckyBij(R_3)\},\{\KentuckyBij(R_3),\KentuckyBij(R_4)\},\{\KentuckyBij(R_4),\KentuckyBij(R_5)\},\{\KentuckyBij(R_1),\KentuckyBij(R_5)\}.
\]
Compare with \cref{ex:g-vector_running_example}.
\end{example}

\subsection{Linear isomorphism between the reduced DKK fan and the $g$-vector fan}
\label{s:fan-iso}

Another interpretation of \cref{thm:comb-iso-fans} is that $\KentuckyBij$ is a bijection from the rays of $\DKKred(G,F)$ to the rays of $\gfan(Q_G,R_G)$ which induces a combinatorial isomorphism of fans.
In this section, we further strengthen this connection by demonstrating that, in fact, there exists a linear isomorphism between $\DKKred(G,F)$ and $\gfan(Q_G,R_G)$, which agrees with $\KentuckyBij$ at the combinatorial level.

\begin{definition}
    \label{def:phi}
    Let $G = (V,E)$ be an amply framed DAG and $V_0 = \internal(G)$ its collection of internal vertices.
    Define $\phi : \RR^E \to \RR^{V_0}$ to be the composition of the following two linear maps:
    \[
    \pi_1: \RR^E \to \RR^{V} : \ve_e \mapsto \begin{cases}
            \tfrac{1}{2} ( \ve_{\tail(e)} - \ve_{\head(e)} ) & \text{if } \coloring(e) = \text{\tcr{red}},\\
            \tfrac{1}{2} ( \ve_{\head(e)} - \ve_{\tail(e)} ) & \text{if } \coloring(e) = \text{\tcb{blue}};
        \end{cases}
    \]
    and the projection $\pi_2:\RR^{V} \to \RR^{V_0}$ sending $\ve_v$ to zero whenever $v$ is not an internal vertex.
\end{definition}

\begin{remark}
This map was independently found by Berggren in \cite[\S 7.3]{BERGGREN}.
\end{remark}

\begin{theorem}
    \label{prop:linear-iso}
    Let $(G,F)$ be an amply framed DAG and $(Q_G,R_G)$ be the associated reduced gentle quiver.
    The map $\phi$ above induces a linear isomorphism between $\DKKred(G,F)$ and $\gfan(Q_G,R_G)$.
\end{theorem}

\begin{proof}
    Let $\KentuckyBij : \Routes(G)\setminus\cE(G) \to \IndShift(Q_G,R_G)$ be the bijection in \cref{def:KentuckyBij}.
    We will show that
    \begin{enumerate}
    \item $\phi({\bf 1}_\route) = 0$ for all exceptional routes $\route \in \cE(G)$,
    so the map $\phi\big|_{\F(G)} : \F(G) \to \RR^{(Q_G)_0}$
    descends to a linear map from the reduced flow space $\overline{\phi} : \F(G)_{\rm red} \to \RR^{(Q_G)_0}$.

    \item $\phi({\bf 1}_\route) = \pmb{g}_{\KentuckyBij(\route)}$ for all non-exceptional routes $\route \in \Routes(G)\setminus \cE(G)$,
    so the map $\overline{\phi}$ above induces a linear isomorphism between $\DKKred$ and $\gfan(Q)$.
    \end{enumerate}

    Observe that $\phi({\bf 1}_\route)_v = 0$ whenever the route $\route$ does not go through the vertex $v \in \internal(G)$.
    Moreover, the acyclicity of $G$ implies every vertex can be visited at most once by $\route$.

    Let $\route \in \cE(G)$ be an exceptional route.
    Thus, every internal vertex visited by $\route$ is the head and the tail of two edges of $\route$ of the same colour.
    Hence, $\phi({\bf 1}_\route)_v = \pm \tfrac{1}{2} \mp \tfrac{1}{2} = 0$.
    Therefore, $\phi(\spann_\RR(\cE)) = 0$ and $\overline{\phi} : \F(G)_{\rm red} \to \RR^{(Q_G)_0} : [{\bf 1}_\route] \mapsto \phi({\bf 1}_\route)$ is well-defined.

    Now, let $\route \in \Routes(G) \setminus \cE(G)$ be a non-exceptional route,
    which we identify with the corresponding maximal bent string of the blossoming quiver $(Q_G^\blossom, R_G^\blossom)$ via the bijection $\RoutesToStrings$ in \eqref{eq:paths_are_strings}.
    Let $v \in \internal(G)$ be an internal vertex visited by $\route$, and
    $e_1 , e_2 \in E$ be the edges on $\route$ with $\head(e_1) = v = \tail(e_2)$.
    By the definition of $(Q_G^\blossom)_1$ in \eqref{eq:arrows-G-blossom},
    $v$ is a peak (resp. a valley) of $\route$ if and only if
    $\coloring(e_1) = \text{\tcb{blue}}$ (resp. \text{\tcr{red}}) and $\coloring(e_2) = \text{\tcr{red}}$ (resp. \text{\tcb{blue}}).
    Therefore,
    \[
        \phi({\bf 1}_\route)_v = \begin{cases}
            \tfrac{1}{2} + \tfrac{1}{2} = 1 & \text{if $v$ is a peak of $\route$,}\\
            {-}\tfrac{1}{2} {-} \tfrac{1}{2} = {-}1 & \text{if $v$ is a valley of $\route$,}\\
            \pm \tfrac{1}{2} \mp \tfrac{1}{2} = 0 & \text{otherwise.}
        \end{cases}
    \]
    Thus, $\phi({\bf 1}_\route)$ is precisely the $g$-vector of the string $\route$ as defined in \eqref{eq:g-rho},
    and, by part (1) of \cref{thm:g-vect-blossom}, it equals $\pmb{g}_{\KentuckyBij(\route)}$.

    We have concluded that $\overline{\phi}$ sends the rays of $\DKKred$ to the rays of $\gfan(Q)$.
    \cref{thm:comb-iso-fans} implies that this map is bijective and, moreover, it respects the fan structure;
    meaning that a set of rays $\DKKred$ spans a cone of the fan if and only if their images in $\RR^{V_0}$ span a cone of $\gfan(Q)$.
\end{proof}

\begin{example}
    \label{ex:image_DKK_XX}
    Consider the amply framed DAG $G = \colorfulXXgraph$ whose routes appear in \cref{subfig:bijection_routes_to_modules2}.
    For example, along the route $R_1$, the vertex $u$ has an incoming arrow coloured \tcb{blue} and an outgoing arrow coloured $\text{\tcr{red}}$; thus $\phi({\bf 1}_{R_1})_u = \tfrac{1}{2} + \tfrac{1}{2} = 1$. Since this route does not visit $v$, we have $\phi({\bf 1}_{R_1}) = (1,0)$.
    The image under $\phi$ of the indicator vectors of the remaining non-exceptional routes is illustrated below.
    
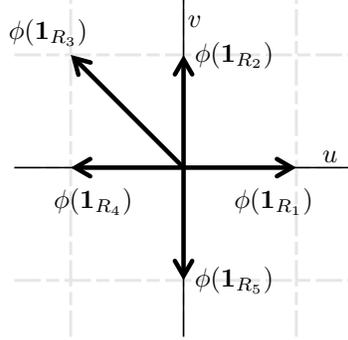
\[\begin{tikzpicture}[>= angle 60,->, scale=1.5]
	   \draw[-,step=1.0,dash pattern={on 10pt off 2pt on 5pt off 2pt}, black, opacity=0.1,line width=.4mm] (-1.5,-1.5) grid (1.5,1.5);
        \draw[-, black,line width=.2mm] (-1.5,0) edge (1.5,0);
        \draw[-, black,line width=.2mm] (0,-1.5) edge (0,1.5);
        \draw[line width=.6mm] (0,0) -- (0,1);
        \draw[line width=.6mm] (0,0) -- (1,0);
        \draw[line width=.6mm] (0,0) -- (0,-1);
        \draw[line width=.6mm] (0,0) -- (-1,0);
        \draw[line width=.6mm] (0,0) -- (-1,1);
        \node at (-0.8,-0.3){$\phi({\bf 1}_{R_4})$};
        \node at (0.45,-1){$\phi({\bf 1}_{R_5})$};
        \node at ( 0.8,-0.3){$\phi({\bf 1}_{R_1})$};
        \node at (0.45,1){$\phi({\bf 1}_{R_2})$};
        \node at (-1.2,1.2){$\phi({\bf 1}_{R_3})$};
        \node at (1.3,0.1){$u$};
        \node at (0.1,1.3){$v$};
    \end{tikzpicture}
    \]
\end{example}

\section{Harder--Narasimhan polytopes for \except{toc}{\linebreak} locally gentle quivers and $g$-vector fans}
\label{s:HN}

In this section, we review the definition of the $\HN$ polytope of a representation $M$ of a gentle quiver $(Q,R)$, and its relationship with $\gfan(Q,R)$.
Notably, in the first part of this section, we extend these results to the case where $(Q,R)$ is merely a locally gentle.

In the second part, we exhibit an intriguing connection with polytopes arising in the study of lattice quotients of the weak order.
Specifically, we show that \emph{shard polytopes} are precisely the $\HN$ polytopes of the quiver in \eqref{eq:quiver-An-shards} below.

\subsection{Definitions and extension to locally gentle algebras}

Fix a quiver $Q$ with vertex set $Q_0 = \{1,\ldots,n\}$.
The Harder--Narasimhan polytope (HN polytope) of a $Q$-representation $M$, first defined by Baumann--Kamnitzer--Tingley in \cite{BKT}, is a polytope that captures information on the dimension of $M$ and of its subrepresentations.
Explicitly, the \textnew{dimension vector} of a $Q$-representation $M$ is
\[ \mathnew{\vdim M} := (\dim M_1,\ldots, \dim M_n) \in \ZZ^n. \]

\begin{definition}[\cite{BKT}]
    The \textnew{Harder--Narasimhan polytope} $\HN(M)$ of the $Q$-representation $M$, is
    \[ \mathnew{\HN(M)} := \conv \set{ \vdim N }{ N \text{ a subrepresentation of } M }. \]
\end{definition}

\begin{example}
    \label{ex:HN-polytope}
    Let $Q$ be the quiver
    \[\begin{tikzpicture}[-stealth]
    \node (A) at (0,0) {$\bullet$};
    \node (B) at (-1.0,0) {$\bullet$};
    \node (C) at ( 1.0,0) {$\bullet$};
    \draw (A) -- (B);
    \draw (A) -- (C);
    \end{tikzpicture},\]
    and let $M$ be the following representation of $Q$:
    \[ \begin{tikzpicture}[-stealth]\node (A) at (0,0) {$\KK$};\node (B) at (-1.0,0) {$\KK$};\node (C) at ( 1.0,0) {$\KK$};\draw (A) -- (B);\draw (A) -- (C);\end{tikzpicture}. \]
    The following table shows all the subrepresentations of $M$ and their dimension vectors.
    \begin{center}
    \begin{tabular}{ c c c } \toprule
    Subrepresentation $N$ & & $\vdim N$ \smallskip \\ \midrule
    \begin{tikzpicture}[-stealth]\node (A) at (0,0) {$0$};\node (B) at (-1.0,0) {$0$};\node (C) at ( 1.0,0) {$0$};\draw (A) -- (B);\draw (A) -- (C);\end{tikzpicture} & & $(0,0,0)$ \\
    \begin{tikzpicture}[-stealth]\node (A) at (0,0) {$0$};\node (B) at (-1.0,0) {$\KK$};\node (C) at ( 1.0,0) {$0$};\draw (A) -- (B);\draw (A) -- (C);\end{tikzpicture} & & $(1,0,0)$ \\
    \begin{tikzpicture}[-stealth]\node (A) at (0,0) {$0$};\node (B) at (-1.0,0) {$0$};\node (C) at ( 1.0,0) {$\KK$};\draw (A) -- (B);\draw (A) -- (C);\end{tikzpicture} & & $(0,0,1)$ \\
    \begin{tikzpicture}[-stealth]\node (A) at (0,0) {$0$};\node (B) at (-1.0,0) {$\KK$};\node (C) at ( 1.0,0) {$\KK$};\draw (A) -- (B);\draw (A) -- (C);\end{tikzpicture} & & $(1,0,1)$ \\
    \begin{tikzpicture}[-stealth]\node (A) at (0,0) {$\KK$};\node (B) at (-1.0,0) {$\KK$};\node (C) at ( 1.0,0) {$\KK$};\draw (A) -- (B);\draw (A) -- (C);\end{tikzpicture} & & $(1,1,1)$ \\ \bottomrule
    \end{tabular}
    \end{center}
    The polytope $\HN(M)$, which is the convex hull of these five vectors, is depicted in \cref{fig:ex_HN_polytope}.
    \begin{figure}[!ht]
    \centering
        \input{figures/sq_pyramid.tex}
        \caption[]{The Harder--Narasimhan (HN) polytope $\HN(M)$ of the quiver representation $M = {\begin{tikzpicture}[-stealth, baseline=-0.7ex] \node (A) at (0,0) {$\KK$}; \node (B) at (-1.0,0) {$\KK$}; \node (C) at ( 1.0,0) {$\KK$};\draw (A) -- (B);\draw (A) -- (C);
        \end{tikzpicture}}$}
        \label{fig:ex_HN_polytope}
    \end{figure}
\end{example}

Importantly, by work of Fei \cite{JF2} and Aoki--Higashitani--Iyama--Kase--Mizuno \cite{aoki2024fanspolytopestiltingtheory},
the HN polytopes of certain representations of $(Q,R)$ carry enough information
to recover the $g$-vector fan of the corresponding algebra.

\begin{theorem}[\cite{JF2,aoki2024fanspolytopestiltingtheory}]
\label{thm:JFei}
Let $\Lambda$ be a finite-dimensional algebra having finitely many brick representations.
Then, the $g$-vector fan of $\Lambda$ equals the normal fan of the Minkowski sum of all HN polytopes of its brick representations.
That is,
\[
    \gfan(\Lambda) = \Normal \Biggl( \sum_{\substack{\text{brick modules}\\ M \text{ of } \Lambda}} \HN(M) \Biggr).
\]
\end{theorem}

Recall from \cref{ss:gentle-quiver} that a gentle quiver $(Q,R)$ determines a finite-dimensional algebra, and that the gentle algebras with finitely many bricks are characterized in \cref{thm:finite-bricks}.

We will be interested in relaxing the assumption that $(Q,R)$ is finite-dimensional and instead consider locally gentle quivers.

\begin{proposition}
\label{prop:locallygentlegvectorfan}
Let $(Q,R)$ be a locally gentle quiver such that every minimal non-oriented cycle contains a relation, and let $\ool$ be the quotient of the corresponding complete gentle algebra by the ideal generated by the oriented cycles, as defined in \cref{ss:locally}.
Then,
\[
    \gfan(Q,R) = \Normal \Biggl( \sum_{\substack{\text{brick modules} \\ M \text{ of } \ool}} \HN(M) \Biggr).
\]
\end{proposition}

\begin{proof}
By \cref{lem:overline-bricks}, the bricks for $\ool$ correspond to non-self-kissing strings in $(Q,R)$ which never go completely around an oriented cycle. By the hypothesis of the theorem, there are only finitely many of these. Thus, \cref{thm:JFei} applies, and the $g$-vector fan of $\ool$ is the outer normal fan of the sum of the bricks of $\ool$. As explained in \cref{ss:locally}, by results of \cite{aokiyurikusa}, the $g$-vector fan of $\ool$ is identified with the $g$-vector fan of $\Lambda$, which is $\gfan(Q,R)$.
\end{proof}

Let $G = (G,\coloring)$ be an amply framed DAG,
and let $(Q^\blossom_G,R^\blossom_G)$ be the associated blossoming quiver, as in \cref{def:acyclic-blossoming}.
Then, $(Q^\blossom_G,R^\blossom_G)$ satisfies the equivalent conditions of \cref{thm:finite-bricks}.
Indeed, consider a cycle, oriented or not, in $Q^\blossom_G$, which never goes through a relation.
Going around the cycle in $Q^\blossom_G$, the direction of the arrows changes exactly when the colour of the arrows changes.
Thus, since $Q_G^\blossom$ is obtained from $G$ by reversing the blue arrows, such a cycle in $Q^\blossom_G$ would correspond to a directed cycle in $G$, but $G$ is acyclic by hypothesis. Thus, every cycle of $Q_G^\blossom$ has a relation, so \cref{thm:finite-bricks} applies, which shows that \cref{thm:JFei} applies to $(Q^\blossom_G,R^\blossom_G)$.

\subsubsection{DKK as a normal fan}
    Given a (oriented) walk $\route$ in $G$, let $P_\route$ be the partial order on the vertices visited by $\route$ defined by the following cover relations:
    \[
        \set{ \head(e) \lessdot \tail(e) }{ e \in \route, \coloring(e) = \text{\tcr{red}}} \,\cup\,
        \set{ \tail(e) \lessdot \head(e) }{ e \in \route, \coloring(e) = \text{\tcb{blue}}}.
    \]
    For a walk $\route$ that only visits internal vertices of $G$, let $\Delta_\route \subseteq \mathbb{R}^{\internal(G)}$ be the convex hull of the indicator vector of all order ideals of $P_\route$.
    This is the \textnew{order polytope} of the poset $P_\route$ \cite{stanley-order-polytopes}.
    
    We identify $\Fred(G,\coloring)$, the ambient space of $\DKKred(G,\coloring)$, with $(\mathbb{R}^{\internal(G)})^*$ using the isomorphism $\overline{\phi} : \Fred(G,\coloring) \to \mathbb{R}^{\internal(G)}$ in the proof of \cref{prop:linear-iso}.
    Explicitly, for $f \in \Fred(G,\coloring)$ and $\vx \in \mathbb{R}^{\internal(G)}$, let $f(\vx) = \langle \overline{\phi}(f) , \vx \rangle$, where $\langle \cdot , \cdot \rangle$ denotes the usual inner product of $\mathbb{R}^{\internal(G)}$.

\begin{corollary}
    \label{cor:DKKred-is-polytopal}
    Let $G = (G,\coloring)$ be an amply framed DAG.
    Then, the reduced fan $\DKKred(G,\coloring)$ is polytopal.
    Explicitly,
    \[
        \DKKred(G,\coloring) = \Normal\Biggl(\sum_\route \Delta_\route\Biggr),
    \]
    where the sum runs over all walks $\route$ of $G$ that only visit internal vertices.
\end{corollary}

\begin{proof}
    The result follows by combining \cref{thm:JFei} with \cref{prop:linear-iso}.
    The order polytope $\Delta_\route$ is precisely the HN polytope of the module $M(\route')$, where $\route'$ is the string of $(Q_G,R_G)$ using the same nodes and (reversed) arrows as $\route$.
\end{proof}

\begin{example}
    The graph $\colorfulXXgraph$ has exactly three walks only visiting internal vertices:
    $u$, $v$, and $\tcb{e}$, where $\tcb{e}$ is the (blue) edge $(u,v)$.
    The corresponding partial orders $P_\route$, order polytopes $\Delta_\route$, and their Minkowski sum $\sum \Delta_\route$ are shown below.
    \[
        \begin{tikzpicture}
            \node[circle,fill,inner sep=1.5pt] (u1) at (0,1.5) {};
            \node[left] at (0,1.5) {$u$};
            \draw[very thick] (-.5,0) --++ (1,0);
            \node[left] at (-.5,0) {\scriptsize ${\bf 1}_{\varnothing}$};
            \node[right] at ( .5,0) {\scriptsize ${\bf 1}_{\{u\}}$};

            \node[circle,fill,inner sep=1.5pt] (v1) at (2,1.9) {};
            \node[left] at (2,1.9) {$v$};
            \draw[very thick] (2,0) --++ (0,1);
            \node[right] at (2,0) {\scriptsize ${\bf 1}_{\varnothing}$};
            \node[right] at (2,1) {\scriptsize ${\bf 1}_{\{v\}}$};

            \node[circle,fill,inner sep=1.5pt] (u2) at (4.2,1.9) {};
            \node[left] at (4.2,1.9) {$v$};
            \node[circle,fill,inner sep=1.5pt] (v2) at (3.8,1.5) {};
            \node[left] at (3.8,1.5) {$u$};
            \draw (u2) -- (v2);
            \draw[very thick,fill=gray!30!white] (3.5,0) --++ (1,0) --++ (0,1) -- cycle;
            \node[left] at (3.5,0) {\scriptsize ${\bf 1}_{\varnothing}$};
            \node[right] at (4.5,0) {\scriptsize ${\bf 1}_{\{u\}}$};
            \node[right] at (4.5,1) {\scriptsize ${\bf 1}_{\{u,v\}}$};

            \draw[very thick,fill=gray!30!white] (8.5,0) --++ (2,0) --++ (0,2) --++ (-1,0) --++ (-1,-1) -- cycle;
        \end{tikzpicture}
    \]
    Compare with the image of $\DKKred(\colorfulXXgraph)$ in \cref{ex:image_DKK_XX}.
\end{example}

\subsection{Connections with shard polytopes}

In this section, we show that the \emph{shard polytopes} of type $A$, introduced by Padrol--Pilaud--Ritter \cite{ppr20shard} in their study of polytopal realizations of certain lattice quotients of the weak order, are in fact Harder--Narasimhan polytopes.
Specifically, we prove that they are the NH polytopes of the indecomposable representations of the bound quiver
\begin{equation}
    \label{eq:quiver-An-shards}
    \begin{gathered}
    \begin{tikzpicture}
        [->,>={Latex[length=6pt]},vertex/.style={circle,draw=blue!50,fill=blue!20,inner sep=1pt, minimum size=3mm}]
        \node (a) at (0,0) [vertex] {1};
        \node (b) at (2,0) [vertex] {2};
        \node at (3.2,0) {$\cdots$};
        \node (c) at (4,0) {};
        \node (d) at (6,0) [vertex] {$n$};
        \path [dotted,-] (b)--(c);
        \path (a) edge [bend left] node [above] {$\alpha_2$} (b);
        \path (b) edge [bend left] node [below]{$\beta_2$} (a);
        \draw[densely dotted,-] (.5,0) ++(225:.2) arc (225:135:.2);
        \draw[densely dotted,-] (1.5,0) ++(45:.2) arc (45:-45:.2);
        \path (c) edge [bend left] node [above] {$\alpha_{n}$} (d);
        \path (d) edge [bend left] node [below]{$\beta_{n}$} (c);
        \draw[densely dotted,-] (4.5,0) ++(225:.2) arc (225:135:.2);
        \draw[densely dotted,-] (5.5,0) ++(45:.2) arc (45:-45:.2);
    \end{tikzpicture}
    \end{gathered}
\end{equation}

\subsubsection{Arcs and shard polytopes}

The choice of a \emph{base region} of a real hyperplane arrangement~$\mathcal{A}$
 induces a partial order on its set of regions, called the \emph{poset of regions};
see \cite{edelman84regions,mandel82thesis}.
When $\mathcal{A}$ is the braid arrangement in $\RR^{n+1}$ (of type $A_n$), this order is, modulo isomorphism, independent of the chosen base region, and it agrees with the usual weak order on the symmetric group.
In his study of lattice properties of (quotients of) these posets, Reading \cite{ReadingShards} describes a decomposition of the hyperplanes of $\mathcal{A}$ into pieces called \emph{shards}, which are in bijection with the join-irreducible elements of the poset of regions.
In the case of the braid arrangement, there is a concrete combinatorial model for shards given by \emph{arcs}, see \cite{ReadingArcs}.

\begin{definition}[\cite{ppr20shard,ReadingArcs}]
    A (type $A_n$) \textnew{arc} is a tuple $\alpha = (a,b,A,B)$ such that $a, b \in \ZZ$ with $1\leq a < b \leq n+1$, and $A,B$ partition the set $[a+1, b-1] = \{a+1 , a+2 , \dots, b-1 \}$.
\end{definition}

The arc $\alpha = (a,b,A,B)$ can be represented by a curve connecting $a$ and $b$, going over the vertices in $A$ and under the vertices in $B$. This is shown in \cref{fig:arcs} below.

\begin{figure}[ht]
    \centering
    \begin{tikzpicture}[scale=.8]
        \node[circle, draw, inner sep=0pt, minimum size=6pt,fill] at (0,0) {};
        \node[circle, draw, inner sep=0pt, minimum size=6pt,fill] at (1,0) {};
        \node[circle, draw, inner sep=0pt, minimum size=6pt,fill] at (2,0) {};
        \node[circle, draw, inner sep=0pt, minimum size=6pt,fill] at (3,0) {};
        \node[circle, draw, inner sep=0pt, minimum size=6pt,fill] at (4,0) {};
        \draw[very thick] (0,0) arc(180:360:.75);
        \draw[very thick] (1.5,0) arc(180:0:.75);
    \end{tikzpicture}
    \hspace{.2\linewidth}
    \begin{tikzpicture}[scale=.8]
        \node[circle, draw, inner sep=0pt, minimum size=6pt,fill] at (0,0) {};
        \node[circle, draw, inner sep=0pt, minimum size=6pt,fill] at (1,0) {};
        \node[circle, draw, inner sep=0pt, minimum size=6pt,fill] at (2,0) {};
        \node[circle, draw, inner sep=0pt, minimum size=6pt,fill] at (3,0) {};
        \node[circle, draw, inner sep=0pt, minimum size=6pt,fill] at (4,0) {};
        \draw[very thick] (0,0) arc(180:360:.75);
        \draw[very thick] (1.5,0) arc(180:0:1.25);
    \end{tikzpicture}
    \caption{Two arcs for $n = 4$.
    The one on the left is $(1,4,\{3\},\{2\})$ and the one on the right is $(1,5,\{3,4\},\{2\})$.}
    \label{fig:arcs}
\end{figure}
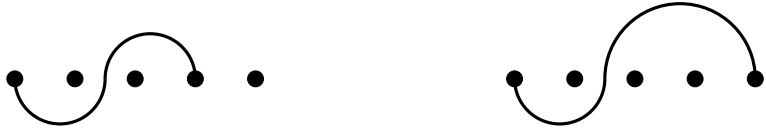

Padrol--Pilaud--Ritter associate to each arc $\alpha$ a polytope $\SP(\alpha)$ that we describe below.

\begin{definition}[\cite{ppr20shard}]
    Fix an arc $\alpha=(a,b,A,B)$. An \textnew{$\alpha$-alternating matching} is a possibly empty sequence
    \[ m =(a_1,b_1,a_2,b_2,\ldots,a_k,b_k) \] satisfying $a\leq a_1 < b_1 < a_2 < \cdots < a_k < b_k \leq b$, with $a_i \in A \cup \{a\}$ and $b_i \in B \cup \{b\}$.
    The \textnew{characteristic vector} of an $\alpha$-matching $m$ is
    \[
        \chi_m := \sum_{i=1}^k (\ve_{a_i}-\ve_{b_i}) \in \RR^{n+1}.
    \]
\end{definition}

\begin{definition}[\cite{ppr20shard}]
    The \textnew{shard polytope} $\SP(\alpha)$ associated to an arc $\alpha$ is the convex hull of the characteristic vectors of all $\alpha$-alternating matchings. That is,
    \[
        \SP(\alpha) := \conv \set{\chi_m }{ m \text{ an } \alpha\text{-alternating matching} }.
    \]
\end{definition}

\begin{example}
    \label{ex:shard-polytope}
    Consider the arc $\alpha = (1 , 4 , \{3\} , \{2\})$ (shown on the left in \cref{fig:arcs} for $n = 4$).
    It has five alternating matchings, including the empty matching, which we list below together with their characteristic vectors.
    \begin{center}\begin{tabular}{ c c c } \toprule
        $\alpha$-alternating matching $m$ & & $\chi_m$ \smallskip \\ \midrule
        $()       $ & & $ 0 $ \\
        $(1,2)    $ & & $ \ve_1 - \ve_2 $ \\
        $(1,4)    $ & & $ \ve_1 - \ve_4 $ \\
        $(3,4)    $ & & $ \ve_3 - \ve_4 $ \\
        $(1,2,3,4)$ & & $ \ve_1 - \ve_2 + \ve_3 - \ve_4 $ \\ \bottomrule
    \end{tabular}\end{center}
    Observe that the sum of characteristic vectors of $()$ and $(1,2,3,4)$ is equal to the sum of characteristic vectors of $(1,2)$ and $(3,4)$, hence the shard polytope $\SP(\alpha)$ is a square pyramid with cone point $ \ve_1 - \ve_4 $.
\end{example}

\subsubsection{The bijection}

Let $(Q,R)$ be the gentle quiver in \eqref{eq:quiver-An-shards}.
Explicitly, $Q$ has vertex set $[n]$, arrows $\alpha_i = (i-1,i)$, $\beta_i = (i,i-1)$ for $i \in [2,n]$,
and the relations in $R$ are $\alpha_i \beta_i$ and $\beta_i \alpha_i$ for all $i$.
The next result shows that, up to a unimodular change of coordinates, HN polytopes coming from the gentle quiver $(Q,R)$ are exactly type A shard polytopes.
An equivalent statement in terms of {\em order polytopes} and type A shard polytopes appeared in \cite[Prop. 5.9]{FMP} (see also \cref{rem:order vs HN vs shards}).

\begin{theorem} \label{thm: HN are shards}
    Up to a global unimodular change of coordinates, the HN polytopes of the indecomposable modules of the gentle quiver $(Q,R)$  in \eqref{eq:quiver-An-shards} are exactly the shard polytopes of type $A_n$.
\end{theorem}

\begin{proof}
    A string $\rho$ of $(Q,R)$ is determined by its endpoints and the collection of arrows $\alpha$ and $\beta$ used to go from one endpoint to the other.
    Thus, a string $\rho$ is uniquely encoded by the quadruplet $(c,d,C,D)$ where $c \leq d \in [n]$ are the endpoints of $\rho$, so $\rho = \gamma_{c+1} \dots \gamma_{d}$ for some choice of $\gamma_i \in \{ \alpha_i , \beta^{-1}_i \}$, and the sets $C := \set{i \in [c+1,d]}{\gamma_i = \alpha_i}$ and $D := \set{i \in [c+1,d]}{\gamma_i = \beta^{-1}_i}$ partition $[c+1,d]$.
    The dimension vector of $M_\rho$ consists of a string of $1$s on positions between $c$ and $d$, and $0$s everywhere else.
    Thus, the map $\Theta$ sending a string $\rho$ encoded by the quadruplet $(c,d,C,D)$
    to the arc $\Theta(\rho) := (c,d+1,C,D)$ is a bijection.

    Now, let $\iota : \RR^n \to \RR^{n+1}$ be the embedding given by $\iota(e_i) = e_i - e_{i+1}$. Note that this embedding is unimodular.
    We claim that, for all strings $\rho$,
    \[
        \iota (\HN(M_\rho)) = \SP(\Theta(\rho)).
    \]
    To prove this, we show that there is a correspondence between submodules of $M_\rho$ and $\Theta(\rho)$-matchings, and that moreover $\iota$ sends the dimension vector of a submodule to the characteristic vector of the corresponding matching.

    Let $N$ be a submodule of $M_\rho$ and let $[a_1,b_1],\dots,[a_k,b_k] \subseteq [n]$ be the maximal intervals of indices $j \in [n]$ for which $\dim N_j = 1$, in particular $a_{i+1} \geq b_i + 2$ for all $i$. For example, if
    \[
        \vdim N = (0,1,1,1,0,0,1,0,1,1),
    \]
    then
    \[
        a_1 = 2,\, b_1 = 4,
        \quad
        a_2 = b_2 = 7,
        \quad
        a_3 = 9,\, b_3 = 10.
    \]
    Then, for all $i = 1,\dots,k$, we have
    \begin{itemize}
        \item either $a_i = c$ (and $i = 1$) or the arrow in $\rho$ to the left of $a_i$ points towards $a_i$. In the latter case, the arrow is $\alpha_{a_i}$ and $a_i \in C$;
        \item either $b_i = d$ (and $i = k$) or the arrow in $\rho$ to the right of $b_i$ points towards $b_i$. In the latter case, this arrow is $\beta_{b_i+1}$ and $b_i + 1 \in D$.
    \end{itemize}
    Thus, the sequence $m_N = (a_1,b_1+1,a_2,b_2+1,\dots,a_k,b_k+1)$ is a $\Theta(\rho)$-alternating matching.
    Since the two conditions above characterize all possible dimension vectors of submodules of $M_\rho$, this is a bijective correspondence.
    Now, observe that
    \[
        (e_{a_i} - e_{a_i + 1}) + (e_{a_i + 1} - e_{a_i + 2}) + \dots + (e_{b_i} - e_{b_i + 1}) = e_{a_i} - e_{b_i + 1}.
    \]
    Therefore, $\iota(\vdim N) = \chi(m_N)$ for all submodules $N$ and $\iota (\HN(M_\rho)) = \SP(\Theta(\rho))$.
\end{proof}

\begin{example}
    Consider the string $\rho = \beta^{-1}_2 \alpha_3$ of $(Q,R)$, which is encoded by the quadruplet $(1,3,\{3\},\{2\})$.
    The string module $M_\rho$ is precisely the quiver representation in \cref{ex:HN-polytope} and the arc $\Theta(\rho) = (1,4,\{3\},\{2\})$ is that of \cref{ex:shard-polytope}.
    The correspondence between subrepresentations of $M_\rho$ and alternating matchings of $\Theta(\rho)$ is described in the following table.
 \begin{center}\begin{tabular}{ c c c } \toprule
    Subrepresentations $N$ of $M_\rho$ & & $\Theta(\rho)$-alternating matchings \smallskip \\ \midrule
    \begin{tikzpicture}[-stealth]\node (A) at (0,0) {$0$};\node (B) at (-1.0,0) {$0$};\node (C) at ( 1.0,0) {$0$};\draw (A) -- (B);\draw (A) -- (C);\end{tikzpicture} & & $() $ \\
    \begin{tikzpicture}[-stealth]\node (A) at (0,0) {$0$};\node (B) at (-1.0,0) {$\KK$};\node (C) at ( 1.0,0) {$0$};\draw (A) -- (B);\draw (A) -- (C);\end{tikzpicture} & & $(1,2) $ \\
    \begin{tikzpicture}[-stealth]\node (A) at (0,0) {$0$};\node (B) at (-1.0,0) {$0$};\node (C) at ( 1.0,0) {$\KK$};\draw (A) -- (B);\draw (A) -- (C);\end{tikzpicture} & & $(3,4) $ \\
    \begin{tikzpicture}[-stealth]\node (A) at (0,0) {$0$};\node (B) at (-1.0,0) {$\KK$};\node (C) at ( 1.0,0) {$\KK$};\draw (A) -- (B);\draw (A) -- (C);\end{tikzpicture} & & $(1,2,3,4)$ \\
    \begin{tikzpicture}[-stealth]\node (A) at (0,0) {$\KK$};\node (B) at (-1.0,0) {$\KK$};\node (C) at ( 1.0,0) {$\KK$};\draw (A) -- (B);\draw (A) -- (C);\end{tikzpicture} & & $(1,4) $ \\ \bottomrule
 \end{tabular}\end{center}
\end{example}

\begin{remark}
    Given a locally gentle quiver $(Q,R)$, each $\HN(M_\rho)$ for a string $\rho$ that does not pass the same vertex more than once is linearly isomorphic to a shard polytope.
    However, for an arbitrary locally gentle quiver, the linear isomorphism depends on $\rho$, and there may not be a global linear map that sends every $\HN$ polytope of $(Q,R)$ to a shard polytope.
    In particular, even if all bricks of $(Q,R)$ correspond to strings of this form,
    its $g$-vector fan is not necessarily (isomorphic to) a coarsening of the braid fan.
    For an interesting example of this phenomenon, see \cref{ss:TheDoppel}.
\end{remark}

\begin{remark}
\cref{thm: HN are shards} implies that the shard polytopes and the HN polytopes of indecomposable modules for the quiver from \eqref{eq:quiver-An-shards} are {\em integrally equivalent} (see \cite[\S 4]{PS}), and therefore have the same combinatorial type, same volume, and Ehrhart polynomials.
\end{remark}

\begin{remark}
    \label{rem:order vs HN vs shards}

    By definition, the vertices of the  HN polytopes of the gentle quiver in \eqref{eq:quiver-An-shards} correspond to indicator vectors for order ideals in the poset given by the orientations of the arrows of the subquiver. This poset is called a {\em fence poset} (see \cite{FMP}) and the polytope with these vertices is called the  {\em order polytope} \cite{stanley-order-polytopes} of the associated fence poset. Consequently, by \cref{thm: HN are shards} up to a global unimodular change of coordinates, shard polytopes of type A also correspond to order polytopes of fence posets. The last statement also appeared in \cite[Prop. 5.9]{FMP}.
\end{remark}

\section{Flows on (certain) cyclic graphs}
\label{s:flow-cyclics}

In this section, we relax our assumptions regarding directed graphs and allow them to have directed cycles.
We typically use $\cyclegraph$ to denote (possibly) cyclic graphs, and reserve $G$ to denote acyclic graphs.

\subsection{Minimal cycles, good routes, and bad routes}

We identify an oriented cycle $C$ in $\cyclegraph$ with the set of edges that constitute it.
An oriented cycle $C$ is \textnew{minimal} if no proper subset of edges $C' \subsetneq C$ is a cycle.
We denote the set of all minimal oriented cycles in $\cyclegraph$ as $\mathnew{\Cycles(\cyclegraph)}$.
Unless stated otherwise, a \emph{cycle} of $\cyclegraph$ will refer to a {minimal} oriented cycle.

We now extend the definition of a route in a cyclic graph $\cyclegraph$.

\begin{definition}
    \label{def:route-in-cyclic}
    A \textnew{route} of $\cyclegraph$ is either
    \begin{itemize}
        \item a walk from a source to a sink, or
        \item a ({minimal} oriented) cycle $C \in \Cycles(\cyclegraph)$.
    \end{itemize}
    Moreover, for every cycle $C \in \Cycles(\cyclegraph)$ and edge $e \in C$, we set $\mathnew{\pref(C,e)} := C$.
\end{definition}

In particular, if $\cyclegraph$ is acyclic, routes are just as defined in \cref{ss:paths-and-routes}.

\begin{example}
    \label{ex:routes-Hn}
    The \textnew{blossomed $n$-cycle $\TheBlackCycle{n}$} is the oriented $n$-cycle with a new sink and source attached to each vertex.
    It has exactly one minimal cycle consisting of $n$ edges (coloured \tcr{red} in \cref{fig:n-cycle1}).
\end{example}

\begin{figure}[ht]
    \centering
    \input{figures/n-cycle.tex}
    \caption{
        The blossomed $n$-cycle from \cref{ex:routes-Hn}
        with its cyclic ample framing described in \cref{ex:cyclic-framing-Hn}.
    }
    \label{fig:n-cycle1}
\end{figure}
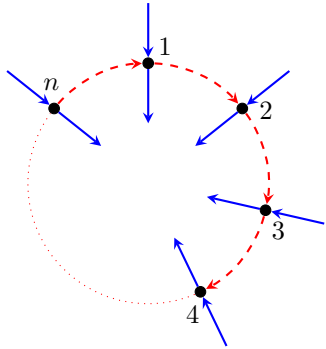

As a walk from a source to a sink might loop an arbitrary number of times around a cycle of $\cyclegraph$,
a cyclic graph has infinitely many routes. We thus distinguish two types of walks from a source to a sink.

\begin{definition}
    Let $\cyclegraph$ be a cyclic graph and $\route$ be a walk
    from a source to a sink.
    The route $\route$ is said to be
    \begin{itemize}
        \item \textnew{good} if it does not contain any cycle $C \in \Cycles(\cyclegraph)$;
        \item \textnew{bad} otherwise.
    \end{itemize}
    Let $\mathnew{\Routes(\cyclegraph)}$ denote the set of \textnew{cycles and good routes} of $\cyclegraph$.
\end{definition}

\

\begin{example}
    \label{ex:good-routes-Hn}
    We continue with the cyclic graph $\cyclegraph = \TheBlackCycle{n}$ from \cref{ex:routes-Hn}.
    Its good routes are in bijection with the set $[n]^2$ as follows.
    Given a pair $(a,b) \in [n]^2$, let $\route_{(a,b)}$ denote the good route that enters the cycle at the vertex $a$,
    traverses the cycle until reaching vertex $b$ for the first time, and immediately exits the cycle.
    The good route of $\TheBlackCycle{6}$ in \cref{fig:ex-bad_route} is $\route_{(6,2)}$.
\end{example}

\begin{figure}[ht]
    \centering
    \input{figures/ex_good_route.tex}
    \hspace{.1\linewidth}
    \input{figures/ex_bad_route.tex}
    \caption{
        Left: An example of a good route on the blossomed $6$-cycle, corresponding to the pair $(6,2) \in [6]^2$.
        Right: An example of a bad route on the same graph.
        (The red arrows that seem to be parallel are in fact the same arrow, used twice by the route.)
    }
    \label{fig:ex-bad_route}
\end{figure}
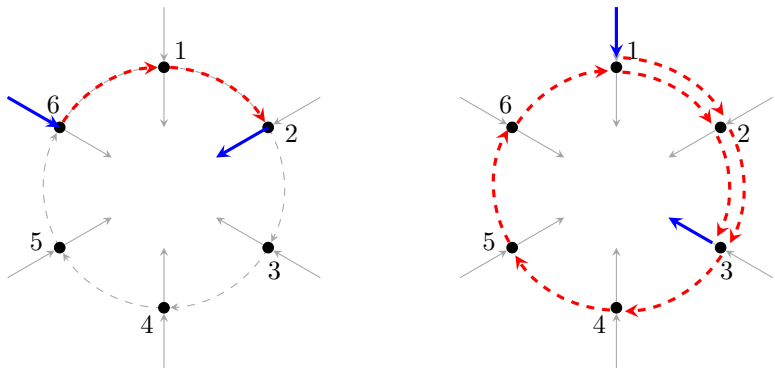

We make the following simple but important observation.

\begin{lemma}
    \label{lem:good-route-vertex}
    A good route cannot visit the same vertex more than once.
\end{lemma}

\begin{proof}
    Otherwise, if $\route = e_1,e_2,\dots,e_k$ visits a vertex more than once,
    the subwalk $C = e_i,\dots,e_j$ where $i \leq j$ are minimal such that $\tail(e_i) = \head(e_j)$
    is a minimal cycle contained in $\route$.
\end{proof}

\begin{corollary}
    A cyclic graph $\cyclegraph$ has only finitely many good routes.
\end{corollary}

\begin{proof}
    In order for there to be infinitely many good routes, some would have to visit some vertex more than once.
    This is impossible by the previous lemma.
\end{proof}

\subsection{Cyclic ample framings and compatibility}

We will be interested in defining an analogue of an ample framing for cyclic graphs $\cyclegraph$ that generalizes the definition for acyclic graphs.
Thus, following \cref{lem:ample-deg2}, we assume henceforth that $\cyclegraph$ is \textnew{full}, meaning $\indeg(v) = \outdeg(v) = 2$ for all internal vertices $v \in \internal(\cyclegraph)$, and consider certain framings that arise from a bi-colouring $\coloring : E \to \{\text{\tcr{red}},\text{\tcb{blue}}\}$ of the edges of $\cyclegraph$.

\begin{definition}
    \label{def:cycle-ample-framing}
    A bi-colouring (framing) $\coloring$ of the edges of $\cyclegraph$ is a \textnew{cyclic ample framing} if the following conditions hold:
    \begin{enumerate}[label=(\roman*)]
        \item
            \label{def:cycle-ample-framing-cond1}
            every internal vertex $v \in \internal(\cyclegraph)$ has exactly one \tcr{red} and one \tcb{blue} incoming edge,
	        and one \tcr{red} and one \tcb{blue} outgoing edge:
            \[
                \set{ \coloring(e) }{ e \in \inc(v) } =
                \set{ \coloring(e) }{ e \in \out(v) } =
                \{ \text{\tcr{red}} , \text{\tcb{blue}} \};
            \]

        \item
            \label{def:cycle-ample-framing-cond2}
            every {minimal} directed cycle $C \in \Cycles(\cyclegraph)$ is monochromatic (all its edges have the same colour).
    \end{enumerate}
\end{definition}

\begin{remark}
    In \cite{BERGGREN} Berggren makes the complementary assumption that  every minimal directed cycle is non-monochromatic. It would be interesting to see if there is a common generalization of both assumptions.
\end{remark}

We make the following key observations.

\begin{lemma}
    \label{lem:cyclic_three_obs}
    Let $\cyclegraph$ be a cyclic graph that admits a cyclic ample framing.
    Then,
    \begin{enumerate}
        \item
            \label{lem:cyclic_three_obs1}
            any two distinct cycles $C,C' \in \Cycles(\cyclegraph)$ are edge-disjoint,

        \item
            \label{lem:cyclic_three_obs2}
            moreover, any two distinct cycles $C,C' \in \Cycles(\cyclegraph)$ can have at most one common vertex,

        \item
            \label{lem:cyclic_three_obs3}
            any good route $\route$ can intersect (at a common subwalk) a cycle $C \in \Cycles(\cyclegraph)$ at most once.
    \end{enumerate}
\end{lemma}

\begin{proof}
    Claim~(1) follows directly from condition~\ref{def:cycle-ample-framing-cond1}.
    For claim~(2), suppose $C,C' \in \Cycles(\cyclegraph)$ share vertices $v_1 \neq v_2$. Then, the cycle obtained by going from $v_1$ to $v_2$ along $C$ and then back to $v_1$ along $C'$ is minimal and not monochromatic.

    Claim~(3) follows by a similar argument: suppose $\route$ intersects a cycle $C$ at two different times and let $v_0$ be the vertex at which $\route$ leaves $C$ for the first time and $v_1$ be the vertex at which $\route$ enters $C$ for the second time.
    Since $\route$ is a good route, $v_0 \neq v_1$, otherwise $\route$ would have completely traversed a cycle.
    Then, concatenating the walk from $v_0$ to $v_1$ along $\route$ and the walk from $v_1$ to $v_0$ along $C$ yields a non-monochromatic cycle $C'$ that is not minimal by condition~\ref{def:cycle-ample-framing-cond2}. But then, a minimal cycle contained in $C'$ is edge-disjoint with $C$ by claim (1) and thus it is a cycle completely traversed by $\route$, contradicting that $\route$ is a good route.
\end{proof}

\begin{example}
    \label{ex:cyclic-framing-Hn}
    Let $\cyclegraph = \TheBlackCycle{n}$ be the blossomed $n$-cycle in \cref{fig:n-cycle1}.
    Up to exchanging colours, it admits only one cyclic ample framing:
    the edges on the cycle are coloured \tcr{red} and the remaining edges,
    those incident to either a source or a sink, are coloured \text{\tcb{blue}}.
    We denote this framed graph by $\TheCycle{n}$.
\end{example}

\begin{lemma}
    \label{lem:monochromatic-are-good}
    Let $\cyclegraph$ be a graph with a fixed cyclic ample framing and $\route$ be a monochromatic route.
    Then, $\route \in \Routes(\cyclegraph)$; that is, $\route$ is either a minimal cycle or a good route.
\end{lemma}

\begin{proof}
    Observe that \cref{def:cycle-ample-framing}~\ref{def:cycle-ample-framing-cond1} implies that any edge $e$ of $\cyclegraph$ is contained in a unique monochromatic route.
    Namely, the only monochromatic route containing $e \in E(\cyclegraph)$ is the one obtained by extending the walk $\route = e$ at its head (resp. tail) by choosing the only other edge coloured $\coloring(e)$ incident to its head (resp. tail).
    The result follows since, if this walk contains a cycle, then it necessarily is a cycle.
\end{proof}

We extend the notion of compatibility of routes to graphs with a cyclic ample framing, following the characterization in \cref{cor:coherence_full_graphs_ample_framings}.

\begin{definition}
    \label{def:incompatible}
    Two routes $\route$ and $\route'$ of $\cyclegraph$ are \textnew{incompatible} if they satisfy the following two conditions:
    \begin{enumerate}
        \item both $\route$ and $\route'$ are walks from a source to a sink; and
        \item there is a common subwalk (possibly a single vertex) $S$ such that when we write $\route = P S Q$ and $\route' = P' S Q'$,
        the last edge of $P$ and the first edge of $Q'$ are coloured \text{\tcb{blue}}, and
        the first edge of $Q$ and the last edge of $P'$ are coloured \text{\tcr{red}}.
    \end{enumerate}
    Otherwise, $\route$ and $\route'$ are \textnew{compatible}.
    A route that is compatible with any other route is called \textnew{exceptional}.
\end{definition}

Observe that monochromatic routes, which include all the minimal cycles, are exceptional.
The following result shows that these are the only exceptional routes.

\begin{lemma}
    \label{lem:cyclic-monochrom-equals-excep}
    A route is exceptional if and only if it is monochromatic.
\end{lemma}

\begin{proof}
    It follows from the definition that any monochromatic route is exceptional.

    If $\route$ is not monochromatic, in particular not a cycle, write $\route = e_1,e_2,\dots,e_k$
    and let $e_i,e_{i+1}$ be the first edges along $\route$ with different colours and $v_0 = \head(e_i) = \tail(e_{i+1})$.
    Let $e' \in \inc(v_0) \setminus \{e_i\}$ and $e'' \in \out(v_0) \setminus \{e_{i+1}\}$
    be the two other edges incident to $v_0$.
    Let $\route'$ (resp. $\route''$) be any walk from a source to $\tail(e')$ (resp. from $\head(e'')$ to a sink).
    Then, the route $\route'e'e''\route''$ is incompatible with~$\route$ and thus $\route$ is not exceptional.
\end{proof}

\begin{example}
    Recall the labelling of good routes of the blossomed $n$-cycle $\TheCycle{n}$ in \cref{ex:good-routes-Hn}.
    Two good routes $\route_{(a,b)}$ and $\route_{(c,d)}$ compatible if either:
    \begin{enumerate}[label=(\roman*)]
        \item the cyclic intervals $[a,b]$ and $[c,d]$ are disjoint; or
        \item $a \preceq c \preceq d \preceq b$ in the order $a \prec a+1 \prec \dots \prec n \prec 1 \prec \dots \prec a-1$; or
        \item $c \preceq a \preceq b \preceq d$ in the order $c \prec c+1 \prec \dots \prec n \prec 1 \prec \dots \prec c-1$.
    \end{enumerate}
    The $n+1$ exceptional routes of $\TheCycle{n}$ are the \tcr{red} cycle and the good routes $\route_{(a,a)}$ for $a \in [n]$.
\end{example}

\begin{lemma}
    A bad route is not compatible with itself.
\end{lemma}

\begin{proof}
    Let $\route$ be a bad route. That is, $\route$ goes completely around some directed cycle $C \in \Cycles(\cyclegraph)$, say of colour \text{\tcr{red}}. Let $v_{\rm in}$ and $v_{\rm out}$ be the vertices of $C$ where $\route$ enters and exits $C$, respectively, and let $S$ be the walk along $C$ going from $v_{\rm in}$ to $v_{\rm out}$ without using any edge more than once. Thus, $S$ is a proper subwalk of $C$, possibly a single vertex if $v_{\rm in} = v_{\rm out}$. Since $\route$ goes completely around $C$, it visits both $v_{\rm in}$ and $v_{\rm out}$ at least twice. We can then write
    $\route = P S Q$ and $\route = P' S Q'$ where $P$ is the subwalk of $\route$ that reaches $v_{\rm in}$ for the first time and $Q'$ is the subwalk of $\route$ leaving from $v_{\rm out}$ for the last time.
    In particular,
    the last edge of $P$ and the first edge of $Q'$ are coloured \text{\tcb{blue}}.
    Since $P S$ is the prefix of $\route$ that first reaches $v_{\rm out}$, the first edge of $Q$ is in $C$ and is therefore coloured \text{\tcr{red}}.
    Similarly, the last edge of $P'$ is also in $C$ and \text{\tcr{red}}. This shows that $\route$ is incompatible with itself.
\end{proof}

\begin{lemma}
    A good route is compatible with itself.
\end{lemma}

\begin{proof}
    Because a good route cannot revisit the same vertex twice, there is nothing to check in the definition of compatibility.
\end{proof}

Just as in the acyclic case, we call \textnew{cliques} the sets of pairwise compatible routes,
and denote by $\mathnew{\Cliques(\cyclegraph)} = \Cliques(\cyclegraph,\coloring)$ (resp. $\mathnew{\MaxCliques(\cyclegraph)}$) the set of (resp. maximal) cliques of the cyclic amply framed graph $(\cyclegraph,\coloring)$.

\subsection{Graphs admitting a cyclic ample framing}

In this section, we characterize when a cyclic graph with a chosen set of routes
arises from a cyclic ample framing.
This provides a generalization of \cite[Theorem 3.9]{Kentuckygentle} to the context of cyclic ample framings.

\begin{definition}[{\cite{Kentuckygentle}}]
    Let $\cyclegraph$ be a directed graph and $X \subseteq \Routes(\cyclegraph)$ be a collection of cycles and good routes of $\cyclegraph$.
    The \textnew{adjacency graph of $X$}, denoted $\mathnew{\Adj(\cyclegraph, X)}$, is the graph with vertex set $X$ and an edge $\{ \route,\route' \} \in \binom{X}{2}$ for every pair of distinct routes passing though a common internal vertex.
\end{definition}

\begin{theorem}
    Let $\cyclegraph$ be a full directed graph
    and $X \subseteq \Routes(\cyclegraph)$ be a collection of cycles and good routes of $\cyclegraph$.
    Then, there exists a cyclic ample framing $\coloring$ of $\cyclegraph$ such that $X = \cE(\cyclegraph,\coloring)$ if and only if
    \begin{enumerate}[label=(\roman*)]
	\item
	    $\Cycles(\cyclegraph) \subseteq X \subseteq \Routes(\cyclegraph)$,
        \label{cond:one}

	\item
	    every edge of $\cyclegraph$ is contained
	    in exactly one route of $X$, and
        \label{cond:two}

	\item
	    $\Adj(\cyclegraph,X)$ is bipartite.
        \label{cond:three}
    \end{enumerate}
\end{theorem}

Observe that whenever $\cyclegraph$ is acyclic, condition~\ref{cond:one} holds trivially and this characterization agrees with that in \cite[Theorem 3.9]{Kentuckygentle}.

\begin{proof}
    First, suppose $\coloring$ is a cyclic ample framing on $\cyclegraph$ and $X = \cE(\cyclegraph,\coloring)$ is its set of exceptional routes.
    Then, \cref{lem:monochromatic-are-good,lem:cyclic-monochrom-equals-excep} imply that $\Cycles(\cyclegraph) \subseteq X \subseteq \Routes(\cyclegraph)$.
    Moreover, the proof of \cref{lem:monochromatic-are-good} also shows that any edge is in exactly one monochromatic route,
    and the first condition on \cref{def:cycle-ample-framing} implies that two distinct monochromatic routes of the same colour cannot intersect, so $\Adj(\cyclegraph,X)$ is bipartite.

Conversely, suppose the collection of routes $X$ satisfies the three conditions above.
Let $X = X_1 \sqcup X_2$ be the partition such that $\Adj(\cyclegraph,X)$ contains no edges within $X_i$ for $i=1,2$.
This partition is unique since $\cyclegraph$, and therefore $\Adj(\cyclegraph,X)$, is connected.
Now, let $\coloring$ be the bi-colouring of the edges of $\cyclegraph$ defined by
\[
    \coloring(e) := \begin{cases}
    \text{\tcr{red}} & \text{if $e$ is in a route in $X_1$},\\
    \text{\tcb{blue}} & \text{if $e$ is in a route in $X_2$}.
    \end{cases}
\]
Observe that $\coloring$ is well-defined by condition~\ref{cond:two}.
Moreover, condition~\ref{cond:one} implies that any cycle is monochromatic and, together with condition~\ref{cond:three},
that the two incoming (resp. outgoing) edges at any internal vertex have distinct colours (since $X$ consists of only good routes).
Thus, $\coloring$ is a cyclic ample framing,
the routes in $X$ are monochromatic (equivalently, by \cref{lem:cyclic-monochrom-equals-excep}, exceptional) with respect to $\coloring$,
and no other route is exceptional by the proof of \cref{lem:monochromatic-are-good}.
\end{proof}

However, giving a complete characterization of those cyclic graphs admitting a cyclic ample framing seems to be more difficult than in the acyclic case, as the following examples illustrate.

\begin{example}
    \label{ex:no-cyclic-ample-framing}
    Consider the two cyclic graphs in \cref{fig:cex-cyclicampfram}. Neither of them admits a cyclic ample framing.

    \begin{itemize}
        \item For the graph on the left, the bottom-most arrow implies that both cycles have the same colour, while the top-most arrows imply that both cycles have different colours.

        \item For the graph on the right, the two edges leaving the cycle should have the same colour (since they leave a monochromatic cycle), but cannot have the same colour since they share the same head.
    \end{itemize}
\end{example}

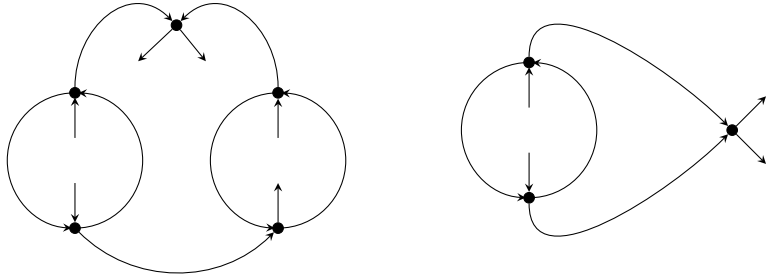
\begin{figure}[ht]
    \centering
    \input{figures/cex_amp_fram_cyc.tex}
    \caption{
        Two examples of directed graphs that do not admit a cyclic ample framing.
        See \cref{ex:no-cyclic-ample-framing}.
    }
    \label{fig:cex-cyclicampfram}
\end{figure}

\begin{problem}
    Give a complete characterization of the graphs admitting a cyclic ample framing.
\end{problem}

\subsection{A DKK-like triangulation}

Flows $f : E(\cyclegraph) \to \RR$, the flow space $\F(\cyclegraph) \subseteq \RR^E$, and the flow cone $\F^+ = \F^+(\cyclegraph)$ of a cyclic graph $\cyclegraph = (V,E)$ are defined in the same manner as for acyclic graphs; see \cref{ss:flows-defs}.

In this section, we show that the flow cone $\F^+(\cyclegraph)$ of a cyclic graph $\cyclegraph$ with a cyclic ample framing admits a DKK-like unimodular triangulation.

\begin{theorem}
    \label{thm:olny-ex}
    Let $\cyclegraph$ be a graph with a cyclic ample framing, and let $f \in \F^+(\cyclegraph)$ be a nonnegative flow on $\cyclegraph$.
    Then $f$ can be uniquely written as a sum of flows along a set of compatible routes.
    Moreover, if $f$ is an integer flow, the coefficients of this sum are integers.
\end{theorem}

\begin{proof}
For each cycle $C \in \Cycles(\cyclegraph)$, let $m_C = \min\set{f(e) }{ e \in C }$ be the minimum flow along its edges,
and let
\[
    f_0 := f \; - \quad \smashoperator{\sum_{C \in \Cycles(\cyclegraph)}} \quad m_C {\bf 1}_C.
\]
Since distinct cycles in $\Cycles(\cyclegraph)$ are {edge-disjoint}, $f_0$ is still a nonnegative flow on $\cyclegraph$.
Moreover, the flow $f_0$ has zero flow on at least one edge in each cycle of $\cyclegraph$.
Let $G$ be the graph obtained from $\cyclegraph$ by removing all the edges on a cycle in $\Cycles(\cyclegraph)$ on which $f_0$ has zero flow.
Observe that $G$ is acyclic and the routes of $G$ are good routes of $\cyclegraph$.
Since the cyclic ample framing on $\cyclegraph$ restricts to a (not necessarily ample) framing of $G$ and $f_0 \in \F^+(G)$,
\cite[Theorem 1]{DKK} shows that there is a unique clique $K \in \Cliques(G)$ and coefficients $a_\route \in \RR_{> 0}$ for $\route \in K$ such that $f_0 = \sum_{\route \in K} a_\route {\bf 1}_\route$.
Thus,
\begin{equation}
    \label{eq:cyclicDKK-decomposition}
    f = \sum_{C \in \Cycles(\cyclegraph)} m_C {\bf 1}_C + \sum_{\route \in K} a_\route {\bf 1}_\route
\end{equation}
is a sum of flows along the set of compatible routes $K \cup \Cycles(\cyclegraph)$.
Moreover, if $f$ is an integer flow, then $m_C \in \ZZ$, so $f_0$ is again an integer flow, and $a_\route \in \ZZ$ by \cite[Theorem 1]{DKK}.

We now prove the uniqueness of the decomposition in \eqref{eq:cyclicDKK-decomposition}.
Suppose we have some other expression for $f$ as in the statement of the theorem.
It follows that there is some cycle $C \in \Cycles(\cyclegraph)$
and a set of good routes $K$ such that $\sum_{\route \in K} {\bf 1}_\route$ has positive flow in all the edges of $C$.
We now show that $K$ cannot possibly be a clique.
By choosing $K$ minimal, we may assume that for any pair $\route,\route' \in K$, we do not have $\route' \cap C \subseteq \route \cap C$, since then we could remove $\route'$ from $K$.
There must be some successive pair of edges of $C$, say $e_1$ and $e_2$, such that some $\route \in K$ contains $e_1$ but not $e_2$, while another $\route' \in K$ contains $e_2$.
Then, $\route$ and $\route'$ are incompatible along their intersection (which may be the single vertex between $e_1$ and $e_2$, but may also include $e_1$ and some previous edges).
\end{proof}

\begin{example}
    Below we show the decomposition described in \cref{thm:olny-ex} for a flow on $\TheCycle{4}$.
    \[
        \input{figures/4-cycle-flow.tex}
    \]
\end{example}

As a consequence, we obtain the following explicit description of a unimodular triangulation of the flow cone of $\cyclegraph$.
Let $\DKK(\cyclegraph)$ denote the following collection of cones in the flow space of $\cyclegraph$,
\begin{equation}
    \label{eq:DKK-cyclic}
    \mathnew{\DKK(\cyclegraph)} := \bigcup_{S}
    \set{ \sigma + \RR_{\geq 0}\Cycles(\cyclegraph) }{ \sigma \in \DKK(\cyclegraph \setminus S) },
\end{equation}
where the union is over all possible choices $S \in \prod_{C \in \Cycles(\cyclegraph)} \set{e}{e \in C}$ of one edge in each {minimal} cycle of $\cyclegraph$.
The following key result is an immediate consequence of \cref{thm:olny-ex}.

\begin{theorem}
    \label{thm:cyclic-DKK}
    Let $\cyclegraph$ be a graph with a cyclic ample framing.
    Then, $\DKK(\cyclegraph)$ is a unimodular triangulation of the flow cone of $\cyclegraph$.
\end{theorem}

We use this result in the next section to express the number of maximal cones in $\DKK(\cyclegraph)$ as a sum of the number of maximal cones in (many) graphs without cycles.

\begin{remark}
    In the acyclic case, DKK triangulations are defined for any framed graph, not only for graphs with ample framings.
    In the cyclic case, we can relax the cyclic ample framing conditions by demanding that all minimal cycles of $\cyclegraph$ are exceptional (by having the edges on any cycle either all be first at both their head and tail, or all be last at both their head and tail).
    A route in a graph $\cyclegraph$ with only exceptional cycles is \textnew{good} if it is self-compatible. This is equivalent to requiring that the route does not completely
    traverse any cycle before leaving.
    In this case, the flow cone admits a triangulation $\DKK$ as in \cref{thm:cyclic-DKK},
    where the indicator vector of the cycles and of the good routes form the rays of the triangulation.
\end{remark}

\subsection{Cyclic volume integer flows}

In this section, we define and give a correspondence between certain integer flows on $H$ and the maximal cliques in $\DKK(H)$ using and extending the correspondence from \cite{MMSt} in the acyclic case. Let $\cyclegraph$ be a graph with a fixed cyclic ample framing.

\begin{definition}
    \label{def:cycintfl}
    An integer flow $f : E(\cyclegraph) \to \NN$ is a \textnew{cyclic integer flow} if
    \begin{enumerate}[label=(\roman*)]
        \item
            \label{def:cycintfl1}
            every cycle $C \in \Cycles(\cyclegraph)$ has exactly one edge $e_C \in C$ with $f(e_C) = 0$,

        \item
            \label{def:cycintfl2}
            the netflow at every internal vertex is $1$, except for the tails of the edges $e_C$ where the netflow is $0$.
    \end{enumerate}
    We let $\mathnew{\cyclicflows(\cyclegraph)}$ denote the collection of cyclic integer flows on $\cyclegraph$.

    More generally, given ${\bf c} = (c_s)_{s} \in \NN^{\sources(\cyclegraph)}$, we let
\[
    \mathnew{\cyclicflows_{\bf c}(\cyclegraph)} := \set{ f \in \cyclicflows(\cyclegraph) }{ \nf_f(s) = c_s \text{ for all } s \in \sources(\cyclegraph) }.
\]

    Elements of $\cyclicflows_{\bf 0}(\cyclegraph)$ are called \textnew{cyclic volume integer flows} for reasons that will be clear soon (see \cref{thm:normVolCyclicIntegerFlow}).
\end{definition}

Define a map $\Phi_\cyclegraph : \MaxCliques(\cyclegraph) \to \NN^E$ by
\[
    \Phi_\cyclegraph(K)(e) = n_K(e) - 1,
\]
where $n_K(e)$ is the number of different prefixes until the edge $e$ for routes in $K$.
Recall from \cref{def:route-in-cyclic} that $\pref(C,e) = C$ for all $e \in C$.
We illustrate this map in \cref{ex:cyclicClique} below.
When $\cyclegraph$ is acyclic, this is the same map as in \cref{cor:bijection_clique_to_integer_flows}.

\begin{theorem}
    \label{prop:cliques-to-cyclic-funny}
    The map $\Phi_\cyclegraph$ is a bijection between the maximal cliques and the cyclic volume integer flows of $\cyclegraph$.
\end{theorem}

\begin{proof}
    Let $K$ be a maximal clique of $\cyclegraph$.
    We first show that for each cycle $C \in \Cycles(\cyclegraph)$, there is exactly one edge $e$ of $C$ which is not used by any route in $K \setminus \{C\}$,
    and thus $\Phi_\cyclegraph(K)$ satisfies condition~\ref{def:cycintfl1} in \cref{def:cycintfl}.
    Observe that such an edge can be used only by $C$ or by a good route, since minimal cycles are edge-disjoint.

    Fix a cycle $C \in \Cycles(\cyclegraph)$.
    By the proof of \cref{prop:max-cliques-use-all-edges},
    every edge of $\cyclegraph$ is used in at least one route in $K$.
    Suppose two edges $e,e'$ in a cycle $C$ are only used by $C$ itself.
    Then, $K' = K \setminus \{C\}$ is a clique of $\cyclegraph' = \cyclegraph \setminus \{e,e'\}$ with size lager than $\dim \F^+(\cyclegraph') = \dim \F^+(\cyclegraph) - 2 = |K| -2$, which is impossible.
    Thus, at most one edge of $C$ is not used by any good route of $K$.
    Among the good routes in $K$, let $\route$ be one that uses a maximal collection of edges in $C$.
    Since $\route$ is good, it intersects $C$ at most once and cannot use all of its edges.
    Let $e \in C$ be the edge along the cycle whose tail is the vertex at which $\route$ leaves $C$.
    If $e$ is used in a different good route $\route' \in K$, then the maximality of $\route$ inside $C$ implies that $\route$ and $\route'$ are incompatible.
    Thus, $e$ is the only edge in $C$ not used by any good route in $K$.

    We now verify that $\Phi_\cyclegraph(K)$ satisfies condition~\ref{def:cycintfl2} in \cref{def:cycintfl}.
    For each cycle $C$, let $e_C$ be the edge in $C$ that is not used by any other route in $K$ and let $G = \cyclegraph \setminus \set{e_C}{C \in \Cycles(\cyclegraph)}$.
    Observe that
    \[
        \indeg_\cyclegraph(v) = \begin{cases}
        \indeg_G(v) + 1 & \text{if $v = \head(e_C)$ for some $C \in \Cycles(\cyclegraph)$},\\
        \indeg_G(v) & \text{otherwise}.
        \end{cases}
    \]
    Consider the maximal clique $K' := K \setminus \Cycles(\cyclegraph)$ of $G$
    and the corresponding (acyclic) volume integer flow $\Phi_G(K')$.
    We compare the netflow of $\Phi_G(K')$ and $\Phi_\cyclegraph(K)$ at each internal vertex of $\cyclegraph$.
    Since, by definition, a cycle $C$ is a prefix of each of its edges, we have
    \begin{equation}
        \label{eq:flow-restriction}
        \Phi_\cyclegraph(K)(e) = \begin{cases}
            \Phi_G(K')(e) & \text{if $e$ is not in any cycle of }\cyclegraph, \\
            \Phi_G(K')(e) + 1 & \text{if $e \in C \setminus \{e_C\}$ for some cycle $C$ of }\cyclegraph, \\
            0 & \text{if $e = e_C$ for some cycle $C$ of }\cyclegraph.
        \end{cases}
    \end{equation}
    Therefore, if $v$ is not in any cycle, the flows at its incident edges do not change. Similarly, if $v$ is in a cycle $C$ but is not incident to $e_C$, then there is an increase of $1$ in the flow of one of the incoming and one of the outgoing edges; those belonging to $C$.
    In both cases,
    \[
        \nf_{\Phi_\cyclegraph(K)}(v) = \nf_{\Phi_G(K')}(v) = \indeg_G(v) - 1 = \indeg_H(v) - 1 = 1.
    \]
    On the other hand, if $v = \head(e_C)$ for some $C$, then
    \[
        \nf_{\Phi_\cyclegraph(K)}(v) = \nf_{\Phi_G(K')}(v) + 1 = \indeg_G(v) - 1 + 1= \indeg_H(v) - 1 = 1.
    \]
    Lastly, if $v = \tail(e_C)$ for some $C$, then
    \[
        \nf_{\Phi_\cyclegraph(K)}(v) = \nf_{\Phi_G(K')}(v) - 1 = \indeg_G(v) - 1 - 1= \indeg_H(v) - 2 = 0.
    \]
    This completes the proof that $\Phi_\cyclegraph$ is a well-defined map.
    Moreover, this also shows that $\Phi_\cyclegraph$ is injective,
    as any maximal clique uniquely determines the set of edges $e_C$, one for each cycle $C$, not used by a good route of $K$ and $\Phi_G$ is injective.

    Lastly, given a cyclic volume integer flow $f$ of $\cyclegraph$, let $\{e_C\}_C$ be the collection of edges in cycles with $f(e_C) = 0$ and let $G = \cyclegraph \setminus \{e_C\}_C$. Let $f'$ be the integer flow on $G$ defined from $f$ by reversing \eqref{eq:flow-restriction}.
    The same computations above show that $f'$ is a volume integer flow of $G$, and therefore there is a clique $K'$ of $G$ with $\Phi_G(K') = f'$. Let $K := K' \cup \Cycles(\cyclegraph)$.
    Then, since $K$ is a clique of $\cyclegraph$ and $\Phi_\cyclegraph(K) = f$, we conclude that $\Phi_\cyclegraph$ is bijective.
\end{proof}

\begin{example}
    \label{ex:cyclicClique}
    Consider the following clique (exceptional routes omitted) on the blossomed $3$-cycle~$\TheCycle{3}$ and the corresponding cyclic volume integer flow:
    \[
        \input{figures/ex_clique_flows_H3_original}
    \]
    Up to rotational symmetry, there is only one more cyclic volume integer flow on $\TheCycle{3}$, namely:
    \[
        \input{figures/ex_other_clique_H3.tex}
    \]
    In both representations of cyclic volume integer flows, the square vertices are those with netflow $0$.
\end{example}

\subsection{A polyhedral complex}

Recall from \cref{thm:innerflowsindeg} that, for an acyclic graph $G$, the volume of the flow polytope $\F_1(G)$ is given by the number of volume integer flows on $G$. However, when the graph $\cyclegraph$ has at least one directed cycle $C$, the defining inequalities of $\F_1(\cyclegraph)$ yield an unbounded polyhedron. Indeed, if ${\bf 1}_C$ is the unit flow on the cycle $C \in \Cycles(\cyclegraph)$, then $\lambda {\bf 1}_C \in \F_1(\cyclegraph)$ for all~$\lambda \geq 0$.

This motivates the definition of the following polyhedral complex.

\begin{definition}
    \label{def:unbounded_object}
    For a graph $\cyclegraph$ admitting a cyclic ample framing, its \textnew{flow complex} is
    \[
        \mathnew{\bdd(\cyclegraph)} := \bigset{ f \in \F^+(\cyclegraph) }{ \smashoperator[r]{\sum_{ \substack{v \in \sources(\cyclegraph) \\ e \in \out(v)} }} \quad f(e) + \smashoperator[r]{\sum_{C \in \Cycles(\cyclegraph)}} \quad \min\set{ f(e) }{ e \in C } = 1 }.
    \]
\end{definition}
Observe that if $\cyclegraph$ is acyclic, then $\bdd(\cyclegraph) = \F_1(\cyclegraph)$.

\begin{example}
    In general, the object $\bdd(\cyclegraph)$ is not convex.
    For instance, consider the blossomed $6$-cycle $\TheBlackCycle{6}$ from \cref{ex:cyclic-framing-Hn} and let $\route_1,\route_2$ be the following two routes:
    \[
        \input{figures/ex_nonconvex_flows.tex}
    \]
    Then, ${\bf 1}_{\route_1},{\bf 1}_{\route_2} \in \bdd(\TheBlackCycle{6})$, but $f = \frac{{\bf 1}_{\route_1} + {\bf 1}_{\route_2}}{2} \notin \bdd(\TheBlackCycle{6})$ (depicted on the right) since
    \[
        \smashoperator[r]{\sum_{ \substack{v \in \sources(\TheBlackCycle{6}) \\ e \in \out(v)} }} \quad f(e) +
        \smashoperator[r]{\sum_{C \in \Cycles(\TheBlackCycle{6})}} \quad \min
        \set{ f(e) }{ e \in C }
        = .5 + .5 + .5 \neq 1.
    \]
\end{example}

The following result is a cyclic analogue of the formula by Baldoni--Vergne and Postnikov--Stanley
(see \cref{thm:innerflowsindeg}) to compute the volume of flow polytopes.

\begin{theorem}
    \label{thm:normVolCyclicIntegerFlow}
    Let $\cyclegraph$ be a directed graph with a cyclic ample framing, then
    \[
        \vol\,\bdd(\cyclegraph) = \#\funny(\cyclegraph).
    \]
\end{theorem}

\begin{proof}
    First notice that ${\bf 1}_\route \in \bdd(\cyclegraph)$ for every $\route \in \Routes(\cyclegraph)$.
    Indeed, for a good route $\route$, ${\bf 1}_\route$ has flow $1$ at exactly one edge incident to a source and has flow $0$ at least one edge of every cycle;
    and for a cycle $C$, ${\bf 1}_C$ has $0$ flow at every edge incident to a source (since those edges cannot be part of a cycle) and at every edge of a cycle $C' \neq C$.
    Since ${\bf 1}_\route$ is the smallest nonzero integer point in $\RR_{\geq 0} \{ {\bf 1}_\route \}$, the intersection of $\DKK(\cyclegraph)$ with $\bdd(\cyclegraph)$,
    \[
        \bigset{ \conv\set{{\bf 1}_\route}{\route \in K} }{ K \in \Cliques(\cyclegraph) },
    \]
    is a unimodular triangulation of $\bdd(\cyclegraph)$.
    It follows that $\vol\,\bdd(\cyclegraph)$ equals the number of maximal cliques of $\cyclegraph$
    and \cref{prop:cliques-to-cyclic-funny} implies the result.
\end{proof}

\begin{example}
    The smallest example we can consider is for the graph $\TheBlackCycle{2}$, in which case $\bdd(\TheBlackCycle{2}) \subseteq \RR^6$ has dimension $6 - 2 - 1 = 3$.
    The graph $\TheBlackCycle{2}$ has two cyclic volume integer flows:
    \[
        \input{figures/ex_bdd_object_flows.tex}
    \]
    and therefore $\vol \bdd(\TheBlackCycle{2}) = 2$.
    Concretely, $\bdd(\TheBlackCycle{2})$ is the union of two unimodular tetrahedra sharing the common facet $\Delta_\cE = \conv\set{{\bf 1}_\route}{\route \in \cE(\TheCycle{2})}$.
    However, the two tetrahedra do not lie in the same $3$-dimensional affine subspace of $\RR^6$.
    Under a suitable projection, in which we collapse $\Delta_\cE$ to a single point, this object looks like the union of two segments (shown in cyan below) sharing a common vertex (shown in orange):
    \[
        \input{figures/ex_bdd_object.tex}
    \]
\end{example}

\section{Connection to complete gentle algebras}
\label{s:completegentlealg}

In the previous section, we defined a triangulation of the flow cone for a graph $\cyclegraph$ with a cyclic ample framing. In this section, we use the results in  \cref{ss:locally} to relate this triangulation to the complete gentle algebra arising from the locally gentle quiver $(Q_\cyclegraph,R_\cyclegraph)$ extending the work done by \cite{Kentuckygentle} for the acyclic case.
We obtain a linear isomorphism from the reduced fan $\DKKred(\cyclegraph,\coloring)$ to a slice of a $g$-vector fan.
This, in turn, gives a realization of $\DKKred(\cyclegraph,\coloring)$ as the normal fan of a polytope.

\subsection{A slice of the $g$-vector fan}

Let $\Fred$ denote the quotient of $\F$ by the subspace $\exc$ spanned by the indicator vectors of exceptional routes, and let $\DKKred$ be the fan in $\Fred$ obtained by projecting the cones in $\DKK$.
The arguments below follow those for the acyclic case in \cite{DKK}.

\begin{proposition}
    \label{prop:dkk_cyclic_complete}
    Let $\cyclegraph$ be a graph with a cyclic ample framing.
    Then, the fan $\DKKred$ is complete in the space $\Fred$.
\end{proposition}

\begin{proof}
    The facets of the cone $\F^+ \subseteq \F$ are given by inequalities of the form $f(e) \geq 0$ as $e$ runs over the edges of $\cyclegraph$.
    Since every edge of $\cyclegraph$ appears on an exceptional route, the intersection $\F^+ \cap \exc$ is not contained in any facet of $\F^+$ and therefore $\F^+ / \exc= \F / \exc = \Fred$.
    It follows that the fan $\DKKred$ is complete.
\end{proof}

Given a cyclic amply framed graph $\cyclegraph = (\cyclegraph,\coloring)$, let $(Q^\blossom_\cyclegraph,R^\blossom_\cyclegraph)$ be the blossoming quiver obtained by reversing the arrows of $\cyclegraph$ coloured \tcb{blue} and with relations $e_1 e_2$ whenever $\coloring(e_1) \neq \coloring(e_2)$, just as in \cref{def:acyclic-blossoming} for acyclic graphs.
Also let $(Q_\cyclegraph,R_\cyclegraph)$ be the restriction of $(Q^\blossom_\cyclegraph,R^\blossom_\cyclegraph)$ to the internal vertices of $\cyclegraph$.
Observe that vertices in any cycle of $\cyclegraph$ are in $Q_\cyclegraph$, as they are neither sources nor sinks of $\cyclegraph$.

The key difference with \cref{ss:DAG_to_gentle} is that now the quivers $(Q^\blossom_\cyclegraph,R^\blossom_\cyclegraph)$ and $(Q_\cyclegraph,R_\cyclegraph)$ are not necessarily gentle, but only locally gentle.

\begin{example}
    When $(\cyclegraph,\coloring)$ is the blossomed $n$-cycle $\TheCycle{n}$ in \cref{fig:n-cycle1},
    the corresponding locally gentle quivers are:
        \[
           (Q^\blossom_{\TheCycle{n}},R^\blossom_{\TheCycle{n}}) =
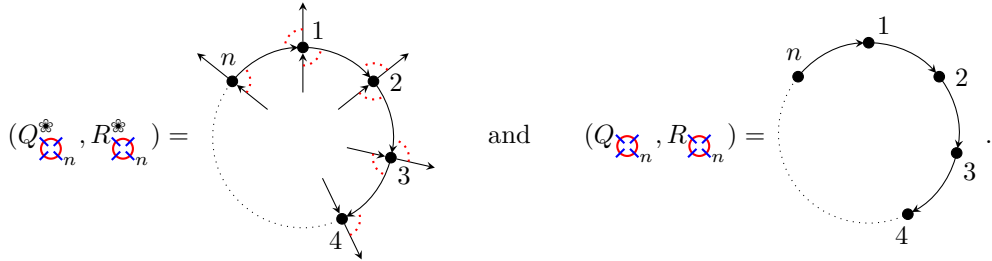
           \begin{gathered}\input{figures/n-cycle-locally-gentle.tex}\end{gathered}
           \qquad\text{and}\qquad
           (Q_{\TheCycle{n}},R_{\TheCycle{n}}) =
           \begin{gathered}\input{figures/n-cycle-locally-gentle-reduced.tex}\end{gathered}.
        \]
\end{example}

In the acyclic case, one can use the fact that the $g$-vector fan $\gfan(Q_G,R_G)$ is complete to deduce that $\DKKred(G)$ is also complete.
Surprisingly, the following result in the cyclic case allows us to deduce the completeness of a certain subfan of $\gfan(Q_\cyclegraph,R_\cyclegraph)$ using the fact that $\DKKred$ is complete. Given a cyclic graph $H$, define the subspace
\[
        \mathnew{W_\cyclegraph} := \bigset{ \vx \in \RR^{(Q_\cyclegraph)_0} }{ \sum_{v \in V(C)} \vx_v = 0 \text{ for all } C \in \Cycles(\cyclegraph) }.
    \]

\begin{theorem}
    \label{thm:quotient-subfan}
    Let $\cyclegraph$ be a cyclic graph with a fixed cyclic ample framing.
    The collection of cones of $\gfan(Q_\cyclegraph,R_\cyclegraph)$ lying in $W_H$
    forms a complete fan inside~$W_\cyclegraph$.
    Moreover, this fan is linearly isomorphic to $\DKKred(\cyclegraph)$.
\end{theorem}

\begin{example}
    For the blossomed $n$-cycle $\TheCycle{n}$, the subspace $W_{\TheBlackCycle{n}} \subseteq \RR^n$ is the hyperplane where the sum of all the coordinates is zero.
    In this case, the first statement of \cref{thm:quotient-subfan} is equivalent to the fact that a projective module and a shifted projective module are never in the same cone of the $g$-vector fan.
    Indeed, any projective $P_v$ of $(Q_\cyclegraph,R_\cyclegraph)$ is supported in all the vertices of $Q_\cyclegraph$.
\end{example}

\begin{proof}[{Proof of \cref{thm:quotient-subfan}}]
    The linear isomorphism is again realized by the map
    $\overline{\phi} : \Fred \to \RR^{(Q_\cyclegraph)_0}$ induced by the composition
    $\phi$ of
    \[
    \RR^E \to \RR^{V} : \ve_e \mapsto \begin{cases}
        \tfrac{1}{2} ( \ve_{\tail(e)} - \ve_{\head(e)} ) & \text{if } \coloring(e) = \text{\tcr{red}}, \\
        \tfrac{1}{2} ( \ve_{\head(e)} - \ve_{\tail(e)} ) & \text{if } \coloring(e) = \text{\tcb{blue}},
        \end{cases}
    \]
    and the projection $\RR^{V} \to \RR^{(Q_\cyclegraph)_0}$, as in \cref{s:fan-iso}.

    By \cref{lem:cyclic-monochrom-equals-excep}, exceptional routes are monochromatic,
    and therefore $\phi({\bf 1}_\route) = 0$ for all exceptional routes $\route \in \cE(\cyclegraph,\coloring)$. Thus, the map $\overline{\phi} : \Fred \to \RR^{(Q_\cyclegraph)_0}$ is well-defined.

    We now show that for any good non-exceptional route $\route \in \Routes(\cyclegraph) \setminus \cE(\cyclegraph)$, the image of its indicator vector $\phi({\bf 1}_\route)$ is a $g$-vector lying in the subspace $W_\cyclegraph$ and therefore the whole image of $\phi$ lies inside~$W_\cyclegraph$.

    Let $\route$ be a good route of $\cyclegraph$ and $C \in \Cycles(\cyclegraph)$.
    (Recall that $\route$ can enter $C$ at most once by \cref{lem:cyclic_three_obs}\ref{lem:cyclic_three_obs3}.)
    If $\route$ and $C$ do not share any vertices, then $\phi({\bf 1}_\route)_v = 0$ for all $v \in V(C)$.
    Otherwise, let $v_0$ and $v_1$ be the vertices of $C$ at which $\route$ enters and exits $C$, respectively.
    If $v_0 = v_1$, then the arrows of $\route$ incident to $v_0$ have the same colour; and we have $\phi({\bf 1}_\route)_v = 0$ for each vertex $v$ of~$C$.
    On the other hand, if $v_0 \neq v_1$, then $\route$ has a peak and a valley at $v_0$ and $v_1$ (the order, peak-valley or valley-peak, depends on the colour of the cycle), so
    $
        \phi({\bf 1}_\route)_{v_1} = - \phi({\bf 1}_\route)_{v_0}
    $
    and $\phi({\bf 1}_\route)_v = 0$ for all other vertices $v$ of $C$.
    In either case,
    \[
        \sum_{v \in C} \phi({\bf 1}_\route)_v = 0,
    \]
    and $\phi({\bf 1}_\route) \in W_\cyclegraph$.
    Now, let $\route^*$ be the string of $(Q_\cyclegraph,R_\cyclegraph)$ obtained by removing hooks from $\route$. Then, by \cref{lem:gvect_in_locgen}, $\phi({\bf 1}_\route) = \pmb{g}(\route^*)$.
    Thus, $\phi({\bf 1}_\route)$ is a $g$-vector lying in the space $W_\cyclegraph$ as claimed.

    Moreover, by \cref{cor:walks-blossom-cycle}, any $g$-vector of $(Q_\cyclegraph,R_\cyclegraph)$ is of this form.
    Thus, $\phi$ is a bijection to the union of all the cones in $(Q_\cyclegraph,R_\cyclegraph)$ whose rays lie inside $W_\cyclegraph$.
    Since $\DKKred(\cyclegraph)$ is a complete fan by \cref{prop:dkk_cyclic_complete} and $\phi$ is a linear map, its image is the subspace of $\RR^{(Q_\cyclegraph)_0}$ spanned by the rays of $(Q_\cyclegraph,R_\cyclegraph)$ inside $W_\cyclegraph$.
    We show that this space is all of $W_\cyclegraph$  by a dimension argument.
    Note that $\dim(W_\cyclegraph) = \# V_0 - \# \Cycles(\cyclegraph)$.
    On the other hand, since an exceptional route of $\cyclegraph$ is either a cycle or the unique monochromatic route leaving from a source, we have
    \[
        \dim(\Fred) = \#E - \# V_0 - (\# \sources(\cyclegraph) + \# \Cycles(\cyclegraph)) =  \dim(W_\cyclegraph).
    \]
    The last equality follows since, by a standard double-counting argument,
    \[
        2 \#E = \#\sources(\cyclegraph) + 4 \# V_0 + \#\sinks(\cyclegraph) \text{ and } \#\sources(\cyclegraph) = \#\sinks(\cyclegraph).
        \qedhere
    \]
\end{proof}

\begin{example}
    \label{ex:DKKred_H3}
    Recall from \cref{ex:routes-Hn} that good routes of
    the blossomed $n$-cycle $\TheCycle{n}$
    are in bijection with ordered pairs in $[n]^2$, and that routes $\route_{(i,i)}$ are exceptional. In the case $n = 3$, the space $W_{\TheBlackCycle{3}}$ has dimension $3-1 = 2$ and the fan $\DKKred$ is shown in \cref{fig:ex:DKKred_H3} below.
    \begin{figure}[ht]
        \centering
        \hspace{.2\linewidth}
    \begin{tikzpicture}[-stealth,scale=.6]
    \def\r{2cm}
    \foreach \x in {1,...,6}{
        \draw[] (0,0) -- (\x*60:\r);
    }
    \node[orange,circle, draw, inner sep=0pt, minimum size=4pt,fill] at (0,0) {};
    \node[] at ( 0:1.3*\r) {\scriptsize $\route_{(2,1)}$};
    \node[] at ( 60:1.2*\r) {\scriptsize $\route_{(3,1)}$};
    \node[] at (120:1.2*\r) {\scriptsize $\route_{(3,2)}$};
    \node[] at (180:1.3*\r) {\scriptsize $\route_{(1,2)}$};
    \node[] at (240:1.2*\r) {\scriptsize $\route_{(1,3)}$};
    \node[] at (300:1.2*\r) {\scriptsize $\route_{(2,3)}$};
    \end{tikzpicture} \hspace{.1\linewidth}
    \begin{tikzpicture}[-stealth,scale=.4]
    \draw (0,0) --++ (210:1);
    \draw (0,0) --++ (-30:1);
    \draw (0,0) --++ ( 90:1);
    \node[] at (210:1.5) {\scriptsize $100$};
    \node[] at (-30:1.5) {\scriptsize $010$};
    \node[] at ( 90:1.4) {\scriptsize $001$};
    \end{tikzpicture}
        \caption{The fan $\DKKred$ for the blossomed $3$-cycle.
        The ray labelled by a route $\route_{a,b}$ is spanned by the projection to $W_\cyclegraph$ of the point $e_a - e_b$. For reference, we show on the right the projections of the vectors $e_1,e_2,e_3$ to $W_\cyclegraph$.}
        \label{fig:ex:DKKred_H3}
    \end{figure}
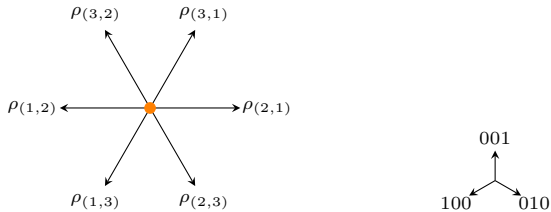
\end{example}

The \textnew{cyclohedron $\cyclohedron_n$} is a polytope that encodes the centrally symmetric triangulations of a regular $2n$-gon.
Explicitly, the $1$-skeleton of $\cyclohedron_n$ is the flip graph of such triangulations.
The facets of $\cyclohedron_n$, equivalently the rays of its normal fan (meaning the normal fan of any realization of $\cyclohedron_n$ as a full-dimensional polytope), correspond to pairs of centrally symmetric diagonals, possibly with both diagonals being equal when they join opposite vertices. For example, for $n=3$, the facets of the cyclohedron correspond to the following diagonals:
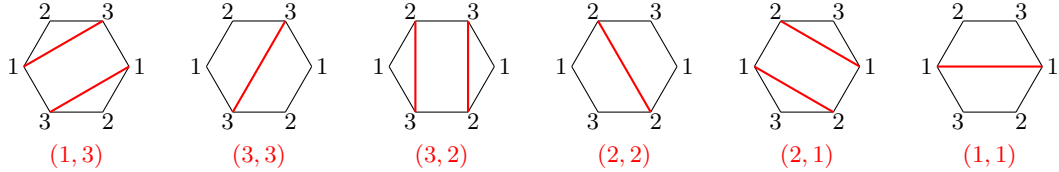
\[
\begin{tikzpicture}[scale=.7]
    \foreach \x in {1,...,6}{\draw[] (\x*60-60:1) -- (\x*60:1);}
    \foreach \x in {1,...,3}{\node at (240-\x*60:1.2) {\small $\x$};\node at (60-\x*60:1.2) {\small $\x$};}
    \draw[red,thick] (60:1) -- (60+120:1);
    \draw[red,thick] (60+180:1) -- (60+180+120:1);
    \node[red] at (-90:1.7) {\small $(1,3)$};
\end{tikzpicture}\quad
\begin{tikzpicture}[scale=.7]
    \foreach \x in {1,...,6}{\draw[] (\x*60-60:1) -- (\x*60:1);}
    \foreach \x in {1,...,3}{\node at (240-\x*60:1.2) {\small $\x$};\node at (60-\x*60:1.2) {\small $\x$};}
    \draw[red,thick] (60:1) -- (60+180:1);
    \node[red] at (-90:1.7) {\small $(3,3)$};
\end{tikzpicture}\quad
\begin{tikzpicture}[scale=.7]
    \foreach \x in {1,...,6}{\draw[] (\x*60-60:1) -- (\x*60:1);}
    \foreach \x in {1,...,3}{\node at (240-\x*60:1.2) {\small $\x$};\node at (60-\x*60:1.2) {\small $\x$};}
    \draw[red,thick] (120:1) -- (120+120:1);
    \draw[red,thick] (120+180:1) -- (120+180+120:1);
    \node[red] at (-90:1.7) {\small $(3,2)$};
\end{tikzpicture}\quad
\begin{tikzpicture}[scale=.7]
    \foreach \x in {1,...,6}{\draw[] (\x*60-60:1) -- (\x*60:1);}
    \foreach \x in {1,...,3}{\node at (240-\x*60:1.2) {\small $\x$};\node at (60-\x*60:1.2) {\small $\x$};}
    \draw[red,thick] (120:1) -- (120+180:1);
    \node[red] at (-90:1.7) {\small $(2,2)$};
\end{tikzpicture}\quad
\begin{tikzpicture}[scale=.7]
    \foreach \x in {1,...,6}{\draw[] (\x*60-60:1) -- (\x*60:1);}
    \foreach \x in {1,...,3}{\node at (240-\x*60:1.2) {\small $\x$};\node at (60-\x*60:1.2) {\small $\x$};}
    \draw[red,thick] (0:1) -- (0+120:1);
    \draw[red,thick] (0+180:1) -- (180+120:1);
    \node[red] at (-90:1.7) {\small $(2,1)$};
\end{tikzpicture}\quad
\begin{tikzpicture}[scale=.7]
    \foreach \x in {1,...,6}{\draw[] (\x*60-60:1) -- (\x*60:1);}
    \foreach \x in {1,...,3}{\node at (240-\x*60:1.2) {\small $\x$};\node at (60-\x*60:1.2) {\small $\x$};}
    \draw[red,thick] (0:1) -- (0+180:1);
    \node[red] at (-90:1.7) {\small $(1,1)$};
\end{tikzpicture}
\]
A set of rays of the normal fan of $\cyclohedron_n$ spans a cone of the fan exactly when the corresponding diagonals do not cross in the interior of the polygon.

\begin{proposition}
    \label{prop:DKKfan-cyclo}
    For the blossomed $n$-cycle $\TheCycle{n}$, the fan $\DKKred$ is combinatorially isomorphic to the normal fan of the cyclohedron.
\end{proposition}

\begin{proof}
    Label the vertices of the $2n$-gon clockwise by $1,2,\dots,n,1,2,\dots,n$, so opposite vertices receive the same label. Pairs of centrally symmetric diagonals correspond to pairs $(a,b)$ with $b \neq a+1$, as follows. The pair $(a,b)$ corresponds to the diagonal joining each of the two vertices labelled $a$ with the first vertex labelled $b$ that you encounter following the boundary of the $2n$-gon in the clockwise direction. For example, when $n=4$, the diagonals $(1,3)$ and $(3,1)$ are shown below
    
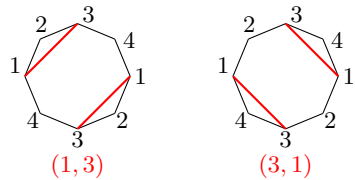
\[
    \begin{tikzpicture}[scale=.7]
        \foreach \x in {1,...,8}{\draw[] (\x*45-45:1) -- (\x*45:1);}
        \foreach \x in {1,...,4}{\node at (215-\x*45:1.2) {\small $\x$};\node at (45-\x*45:1.2) {\small $\x$};}
        \draw[red,thick] (180:1) -- (90:1);
        \draw[red,thick] (0:1) -- (-90:1);
        \node[red] at (-90:1.7) {\small $(1,3)$};
    \end{tikzpicture}\qquad
    \begin{tikzpicture}[scale=.7]
        \foreach \x in {1,...,8}{\draw[] (\x*45-45:1) -- (\x*45:1);}
        \foreach \x in {1,...,4}{\node at (215-\x*45:1.2) {\small $\x$};\node at (45-\x*45:1.2) {\small $\x$};}
        \draw[red,thick] (90:1) -- (0:1);
        \draw[red,thick] (180:1) -- (-90:1);
        \node[red] at (-90:1.7) {\small $(3,1)$};
    \end{tikzpicture}
    \]
    In this manner, pairs $(a,a)$ correspond to long diagonals (diameters) going through the centre of the $2n$-gon.

    Send the ray corresponding to the route $\route_{a,b}$ to the (pair of) diagonal(s) corresponding to the pair $(a,b+1)$.
    Since the route $\route_{a,b}$ is not exceptional precisely when $a \neq b$,
    this gives a bijection between the rays of both fans.
    The result now follows from the compatibility of routes described in \cref{ex:routes-Hn}.
\end{proof}

\subsection{DKK as a normal fan}

We will now realize the fan $\DKKred$ as the normal fan of a polytope. We will use the following result.

\begin{lemma}[{\cite[Lemma 7.11]{zieglerlectures}}]
    \label{lem:normal-fan-subspace}
    Let $\F$ be the normal fan of a polytope $P \subseteq \RR^n$ and $V \subseteq \RR^n$ a linear subspace. The normal fan of $\pi_V(P)$, the orthogonal projection of $P$ to $V$, is the fan
    \[
        \F|_V := \set{ C \cap V }{ C \in \F }.
    \]
\end{lemma}

\begin{theorem}
    \label{thm:polytope-dual-to-cyclic-gfan}
    Let $\cyclegraph$ be a cyclic graph with a fixed cyclic ample framing. Let $\pi_W$ be the orthogonal projection $\RR^{(Q_\cyclegraph)_0} \to W_\cyclegraph$.
    Then $\DKKred$ is (linearly isomorphic to) the normal fan of the polytope $\pi_W( \sum_M \HN(M) )$, where the sum is over the brick modules $M$ of $(Q_\cyclegraph , R_\cyclegraph)$.
\end{theorem}

\begin{proof}
Let $\gamma$ be a cycle (not necessarily oriented) of $Q_H$. If $\gamma$ passes through no relation of $R_H$, then it changes direction exactly when it changes colour. That being the case, the corresponding walk in $H$ is directed, and therefore, by assumption, monochromatic and oriented. It follows that every non-oriented cycle in $Q_H$ passes through a relation.

    Therefore, $(Q_\cyclegraph, R_\cyclegraph)$ satisfies the hypotheses of \cref{prop:locallygentlegvectorfan},
    and we deduce that $\gfan(Q_\cyclegraph, R_\cyclegraph)$ is the normal fan of $\sum_M \HN(M)$.
    Then, \cref{lem:normal-fan-subspace} shows that the normal fan of $\pi_W( \sum_M \HN(M) )$ is $\gfan(Q_\cyclegraph,R_\cyclegraph)|_W$, which by \cref{thm:quotient-subfan} is (linearly isomorphic to) $\DKKred$.
\end{proof}

This result allows us to give another realization of the cyclohedron as a projection of the type D associahedron.
A construction of the cyclohedron as a slice of the type $D$ associahedron is already known \cite[Proposition 6.4]{AHL}. It would be interesting to determine how these realizations are related.

\begin{definition}
    Let $\cyclegraph = \TheCycle{n}$ be the blossomed $n$-cycle from \cref{ex:routes-Hn}.
    For any pair $(a,b) \in [n]^2$, let
    \[
        {\color{defcolor}{\Delta_{(a,b)}}} := \conv \{ e_a + e_{a+1} + \dots + e_b ,\; e_{a+1} + \dots + e_b ,\; e_{b-1} + e_b ,\; e_b ,\; 0 \}
    \]
    be the $\HN$ polytope of the representation of $(Q^\blossom_{\TheCycle{n}}, R^\blossom_{\TheCycle{n}})$ corresponding to the route $\route_{(a,b)}$.
\end{definition}

\begin{example}
    For example, if $n = 4$, then
    \[
        \Delta_{(4,2)} = \conv\{ (1,1,0,1) ,\; (1,1,0,0) ,\; (0,1,0,0) ,\; (0,0,0,0) \}.
    \]
\end{example}

The $n$-cycle is a quiver mutation equivalent to $D_n$. The corresponding cluster-tilted algebra is the path algebra of the $n$-cycle modulo the relations that the composition of any $n-1$ consecutive arrows is zero. Consequently, the HN polytopes of the indecomposable representations of the cluster-tilted algebra are the $\Delta_{(a,b)}$ with $a\ne b+1 \mod n$. By \cite{BM20,PPPP23}, the Minkowski sum of these HN polytopes is therefore a type $D_n$ associahedron.

    By the previous theorem and \cref{prop:DKKfan-cyclo}, we obtain a realization of the cyclohedron as a Minkowski sum of projections of HN polytopes.

   \begin{corollary} \label{cor:cyclo as sum of HN}
    The polytope
    \[
        C_n := \sum_{\substack{(a,b) \in [n]^2 \\ a \neq b+1 \mod n}} \pi_W\big( \Delta_{(a,b)} \big)
    \]
    is a realization of the $(n-1)$-dimensional cyclohedron.
\end{corollary}

\begin{proof}
    Observe that, since $\pi_W((1,1,1,\dots,1)) = \pi_W((0,0,0,\dots,0))$,
    \[
        \pi_W\big( \Delta_{(b+1,b)} \big) = \pi_W\big( \Delta_{(b+2,b)} \big),
    \]
    where the indices are taken modulo $n$.
    Thus, the polytopes $C_n$ and
    $$
        \sum_{(a,b) \in [n]^2} \pi_W\big( \Delta_{(a,b)} \big)
    $$
    have the same normal fan.
    The result now follows by applying \cref{thm:polytope-dual-to-cyclic-gfan} to the graph $\TheCycle{n}$ and using \cref{prop:DKKfan-cyclo}.
\end{proof}

\begin{example}
    We specialize to the case $n = 3$.
    Observe that although the projections of polytopes $\Delta_{(1,3)}$ and $\Delta_{(1,2)}$ are different, they are just translations of one another:
    \[
        \pi_W(\Delta_{(1,3)}) = \pi_W(\Delta_{(2,3)}) = \begin{gathered}\begin{tikzpicture}
    \draw[fill = gray!30!white] (0,0) node[below] {\scriptsize $000 = 111$}
    --++ ( 90:1) node[above] {\scriptsize $001$}
    --++ (-30:1) node[right] {\scriptsize $011$} -- cycle;
    \end{tikzpicture}\end{gathered}
        \qquad
        \pi_W(\Delta_{(3,2)}) = \pi_W(\Delta_{(1,2)}) = \begin{gathered}\begin{tikzpicture}
    \draw[fill = gray!30!white] (0,0) node[above] {\scriptsize $000 = 111$}
    --++ (-30:1) node[right] {\scriptsize $010$}
    --++ (210:1) node[below] {\scriptsize $110$} -- cycle;
    \end{tikzpicture}\end{gathered}
    \]
    In fact, for $n = 3$, all $\Delta_{(a,b)}$ with $a \neq b$, project to the same triangle (modulo translation). When $a = b$, the polytopes we obtain are:
    \[
        \pi_W(\Delta_{(1,1)}) = \begin{gathered}\begin{tikzpicture} \draw[fill = gray!30!white] (0,0) node[above] {\scriptsize $000$} --++ (210:1) node[below] {\scriptsize $100$}; \end{tikzpicture}\end{gathered}
        \qquad
        \pi_W(\Delta_{(2,2)}) = \begin{gathered}\begin{tikzpicture} \draw[fill = gray!30!white] (0,0) node[above] {\scriptsize $000$} --++ (-30:1) node[below] {\scriptsize $010$}; \end{tikzpicture}\end{gathered}
        \qquad
        \pi_W(\Delta_{(3,3)}) = \begin{gathered}\begin{tikzpicture} \draw[fill = gray!30!white] (0,0) node[below] {\scriptsize $000$} --++ ( 90:1) node[above] {\scriptsize $001$}; \end{tikzpicture}\end{gathered}.
    \]
    Their Minkowski sum is:
    \[
        C_3 = \begin{gathered}\begin{tikzpicture}[scale = .3] \draw[fill = gray!30!white] (0,0) --++ ( 90:6) --++ ( 30:1) --++ (-30:6) --++ (-90:1) --++ (210:6) -- cycle; \end{tikzpicture}\end{gathered},
    \]
    whose normal fan appears in \cref{fig:ex:DKKred_H3}.
\end{example}

\section{The family of nested cycles $\cyclegraph_{c,p}$}
\label{s:HCP}

In this section we focus on a two-parameter family of graphs with cyclic ample framings that generalizes the family of blossomed $n$-cycles $\TheCycle{n}$ in \cref{ex:cyclic-framing-Hn}.
We will describe the exceptional routes of graphs in this family, and show that their number of cyclic volume integer flows is given by a multinomial coefficient. For a subfamily of this family of graphs, we give results illustrating a strong link with the permutohedron.

\begin{definition}
    Let $\cyclegraph_{c,p}$ denote the graph consisting of $c$-many \emph{nested} monochromatic $p$-cycles.
    The example for $c = 3$ and $p = 2$ is shown in \cref{fig:H32} below.
    Note that $\cyclegraph_{1,n}$ is the blossomed $n$-cycle $\TheBlackCycle{n}$ from \cref{ex:routes-Hn}.
    Formally, $\cyclegraph_{c,p}$ has
    vertex set $[0,c+1] \times [p]$, edges $(a,b) \to (a+1,b)$ for all $(a,b) \in [0,c] \times [p]$, and $(a,b) \to (a,b+1 \pmod{p})$ for $(a,b) \in  [1,c] \times [p]$.
\end{definition}

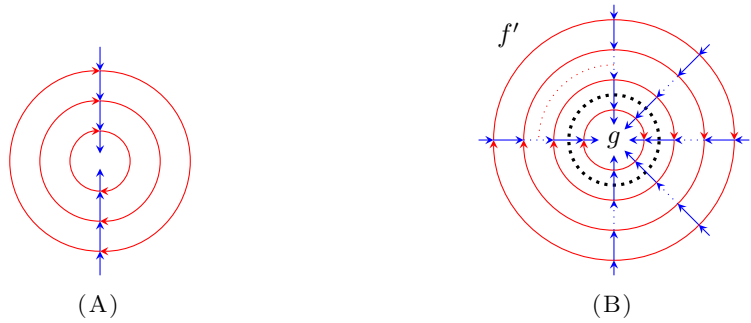
\begin{figure}[ht]
    \centering
    \begin{subfigure}{.4\linewidth}
        \centering
         \begin{tikzpicture}[-stealth,scale=.4,rotate=-90]\input{figures/H32_tikz.tex}\end{tikzpicture}
         \caption{}
         \label{fig:H32}
    \end{subfigure}
    \begin{subfigure}{.5\linewidth}
        \centering
        \begin{tikzpicture}[-stealth,scale=.4]\input{figures/H34_tikz.tex}\end{tikzpicture}
        \caption{}
        \label{fig:breaking-flows-Hcp}
    \end{subfigure}
    \caption{
        (A) The graph $\cyclegraph_{3,2}$ of three nested $2$-cycles
        with its unique (up to swapping colours) cyclic ample framing and (B) the decomposition of cyclic volume integer flow $f$ in $\cyclegraph_{c,p}$ into the flows $(f',g)$.
    }
    \label{fig:Hcp}
\end{figure}

\subsection{Exceptional routes}

Similar to $\TheBlackCycle{n}$, the graph $\cyclegraph_{c,p}$ admits only one cyclic ample framing (up to colour switching): the edges of the cycles are coloured \tcr{red}, and the remaining edges are coloured \tcb{blue} (see \cref{fig:Hcp}).
We fix this framing for the remainder of the section.
We see that $\cyclegraph_{c,p}$ has $c + p$ exceptional routes: $c$ \tcr{red} cycles, and $p$ \tcb{blue} routes that cut through all the cycles without using any edge within a cycle. The following result counts the number of non-exceptional routes of~$\cyclegraph_{c,p}$.

\begin{proposition}
    \label{prop:good_routes_Hcp}
    The graph $\cyclegraph_{c,p}$ has $p^{c+1}-p$ good non-exceptional routes.
\end{proposition}

\begin{proof}
    Every good route of $\cyclegraph_{c,p}$ uses exactly $c + 1$ blue edges: one between each pair of adjacent cycles, plus one entering the first cycle and one leaving the last cycle. Moreover, any arbitrary choice of such a collection of edges can be uniquely extended to a good or exceptional route. The result follows since exactly $p$ out of the $p^{c+1}$ possible choices correspond to an exceptional route.
\end{proof}

\subsection{Cyclic volume integer flows}

We have a closed formula for the normalized volume of the associated flow complex $\bdd(\cyclegraph_{c,p})$ from \cref{def:unbounded_object}.

\begin{theorem}
    \label{thm:hkr-vol}
    For the graph $\cyclegraph_{c,p}$, we have that
    \[
        \vol\bdd(\cyclegraph_{c,p}) = \#\funny(\cyclegraph_{c,p}) = \binom{(c+1)(p-1)}{p-1,\ldots,p-1}.
    \]
In particular, $\vol\bdd(\TheBlackCycle{p})=\#\funny(\TheBlackCycle{p})=\binom{2p-2}{p-1}$ and $\vol\bdd(\cyclegraph_{c,2})=\#\funny(\cyclegraph_{c,2})=(c+1)!$.
\end{theorem}

By \cref{thm:normVolCyclicIntegerFlow}, the volume computation is equivalent to counting cyclic volume integer flows on $\cyclegraph_{c,p}$. In order to do this count recursively, we will need the following enumeration of cyclic integer flows on $\cyclegraph_{1,p}$ with a fixed netflow at the sources.

\begin{lemma}
    \label{lem:cyclicVolIntFlowH1p}
    The number of cyclic integer flows on $\cyclegraph_{1,p}$ with netflow ${\bf a}=(a_1,\dots,a_p)$ in $\ZZ_{\geq 0}^p$ at the sources is $$
    \#\cyclicflows_{\bf a}(\cyclegraph_{1,p}) =
    \binom{a_1+\dots+a_p+2p-2}{p-1}.$$
\end{lemma}

\begin{proof}
	Consider the word $1^{a_1+1}2^{a_2+1}\ldots p^{a_p+1}$ with  $a_1+1$ occurrences of the letter 1, followed by $a_2+1$ occurrences of the letter 2, and so on.
	Among the $a_1+\dots+a_p+p$ positions between successive letters including after the last letter, choose $p-1$ with repetition in which to insert a bar.
	Denote that resulting barred word by $W$. The number of choices of bars in a word is given by $\binom{a_1+\cdots + a_p +2p-2}{p-1}$. Next, we give a bijection between the choices of positions of bars and cyclic integer flows in $\cyclegraph_{1,p}$ with netflow ${\bf a}$ at the sources (see \cref{ex:bijection W to cyclic int flow}).

	Read $W$ and construct a path $P$ which takes an up step each time you cross one of the bars, and takes a down step each time the last occurrence of any one of the letters is reached.
	Each down step can be labelled by $i=1,\ldots,p$.

	Consider the first time $P$ hits its lowest point and let $j$ be the label of the down step heading to that lowest point.

	Let $W'$ be the word obtained by cutting $W$ after the last occurrence of $j$ and inverting the two parts; i.e. if $W=UV$ where $U$ finishes with the last occurrence of $j$, then $W'=VU$.

	We define the cyclic integer flow $f$ on $\cyclegraph_{1,p}$ as follows.
	Let the flow of the blue edge $\textcolor{blue}{(0,i)\to (1,i)}$ be $a_i$ for $i=1,\ldots,p$.
	The red edge of the cycle with zero flow will be the edge  $\textcolor{red}{(1,j) \to (1,j+1)}$.
    Next, label the bars of $W'$ by $j+1,j+2,\dots, j-1$ modulo $p$.
	The bar $j+1$ occurs before the last occurrence of $j+1$ in $W'$, as otherwise, this would contradict hitting the lowest point of $P$ right after the last occurrence of $j$.
	The number of letters before the bar numbered $j+1$ (which are all $j+1$s) determines the flow along the blue edge $\textcolor{blue}{(1,j+1)\to (2,j+1)}$.
	The number of letters between the bar $j+1$ and the last occurrence of $j+1$ corresponds to the flow along the red edge  $\textcolor{red}{(1,j+1) \to (1,j+2)}$.
	The bar numbered $j+2$ occurs before the final occurrence of the letter $j+2$ in $W'$.
	Note that this bar may occur either before or after the final occurrence of $j+1$.
	The number of letters between the bar numbered $j+1$ and the bar numbered $j+2$ corresponds to the flow along the blue edge $\textcolor{blue}{(1,j+2)\to (2,j+2)}$.
	The number of letters between the bar numbered $j+2$ and the last occurrence of $j+2$ corresponds to the flow along the red edge from $\textcolor{red}{(1,j+2) \to (1,j+3)}$.
	Continue in this way around the cycle up until $j-1$.
	The number of letters between the bar numbered $j-1$ and the last occurrence of $j-1$ determines the flow on the red edge  $\textcolor{red}{(1,j-1) \to (1,j)}$.
	Finally, the number of letters between the bar $j-1$ and the final occurrence of the letter $j$, minus 1, determines the flow on the blue edge $\textcolor{blue}{(1,j)\to (2,j)}$.

	We claim that $f$ is a cyclic integer flow on $\cyclegraph_{1,p}$ with netflow ${\bf a}$ at the sources. Indeed, in the red cycle, the edge $\textcolor{red}{(1,j)\to (1,j+1)}$ has zero flow and
    for $i\ne j$, the edge $\textcolor{red}{(1,i)\to (1,i+1)}$ has positive flow since the $i$-th bar occurs before the last occurrence of $i$.
    Also, the netflow at vertex $(1,i)$ for $i\neq j,j+1$ equals
    \begin{multline*}
        f(\textcolor{red}{(1,i)\to (1,i+1)}) + f(\textcolor{blue}{(1,i)\to (2,i)})
        - f(\textcolor{red}{(1,i-1)\to (1,i)}) - f(\textcolor{blue}{(0,i)\to (1,i)}) \\
        = \textcolor{red}{\textrm{\#\{$j+1,j+2\dots,i$ after bar $i$\}}}
        + \textcolor{blue}{\textrm{\#\{letters between bars $i-1$ and $i$\}}} \\
        - \textcolor{red}{\textrm{\#\{$j+1,j+2,\dots,i-1$ after bar $i-1$\}}}
        - \big(\textcolor{blue}{\textrm{\#\{$i$'s in $W'$\}}-1} \big) = 1,
    \end{multline*}
    while the netflow at vertex $(1,j)$ equals
    \begin{multline*}
        f(\textcolor{red}{(1,j)\to (1,j+1)}) + f(\textcolor{blue}{(1,j)\to (2,j)})
        - f(\textcolor{red}{(1,j-1)\to (1,j)}) - f(\textcolor{blue}{(0,j)\to (1,j)}) \\
        = \textcolor{red}{0}
        + \textcolor{blue}{\textrm{\#\{letters between bar $j-1$ and final letter $j$\}}-1} \\
        - \textcolor{red}{\textrm{\#\{letters other than $j$ after bar $j-1$\}}}
        - \big(\textcolor{blue}{\textrm{\#\{$j$'s in $W'$\}}-1} \big)= 0,
        \end{multline*}
    and the netflow at vertex $(1,j+1)$ equals
    \begin{multline*}
        f(\textcolor{red}{(1,j+1)\to (1,j+2)}) + f(\textcolor{blue}{(1,j+1)\to (2,j+1)}) \\
        - f(\textcolor{red}{(1,j)\to (1,j+1)}) - f(\textcolor{blue}{(0,j+1)\to (1,j+1)}) \\
        = \textcolor{red}{\textrm{\#\{$j+1$'s after bar $j+1$\}}}
        + \textcolor{blue}{\textrm{\#\{letters before bar $j+1$\}}} 
        - \textcolor{red}{0}
        - \big(\textcolor{blue}{\textrm{\#\{$j+1$'s in $W'$\}}-1} \big)= 1.
    \end{multline*}
    Lastly, this map $W \mapsto f$ is bijective since $f$ completely determines the word $W'$ with $p-1$ bars. The word $W'$ can be cut as $VU$ after the last occurrence of $p$ and $W=UV$ by inverting the two parts.
\end{proof}

\begin{example} \label{ex:bijection W to cyclic int flow} In \cref{fig:H1pBars}, we demonstrate the results of executing the algorithm in the proof of \cref{lem:cyclicVolIntFlowH1p} for $p=6$, ${\bf a}=(3,2,1,3,0,2)$, and the barred word $$W=11112|22334|44|456||66.$$ From $W$ we obtain the barred word $$W'=4|_4\,44|_5\,456|_6\,|_1\,6611112|_2\,2233$$ by cutting $W$ after the last occurrence of $j=3$ and swapping the resulting two factors.
\end{example}

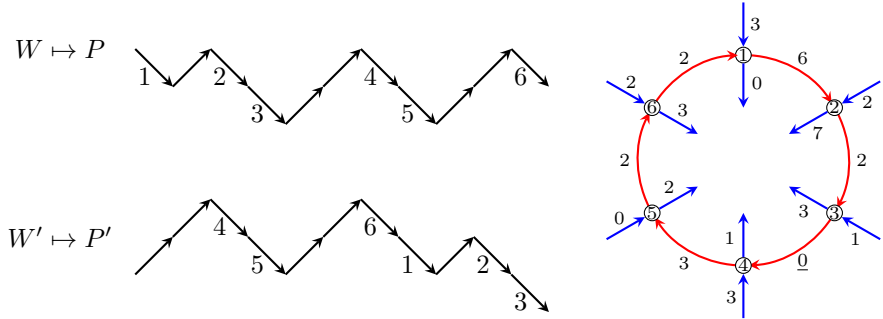
\begin{figure}[ht]
    \centering
    \begin{tikzpicture}[-stealth,scale=.5]
        \node at(-2,6) {$W\mapsto P$};
        \draw[thick] (0,6) -- (1,5) node[near start, below] {1};
        \draw[thick] (1,5) -- (2,6);
        \draw[thick] (2,6) -- (3,5) node[near start, below] {2};
        \draw[thick] (3,5) -- (4,4) node[near start, below] {3};
        \draw[thick] (4,4) -- (5,5);
        \draw[thick] (5,5) -- (6,6);
        \draw[thick] (6,6) -- (7,5) node[near start, below] {4};
        \draw[thick] (7,5) -- (8,4) node[near start, below] {5};
        \draw[thick] (8,4) -- (9,5);
        \draw[thick] (9,5) -- (10,6);
        \draw[thick] (10,6) -- (11,5) node[near start, below] {6};

        \node at(-2,1) {$W'\mapsto P'$};
        \draw[thick] (0,0) -- (1,1);
        \draw[thick] (1,1) -- (2,2);
        \draw[thick] (2,2) -- (3,1) node[near start, below] {4};
        \draw[thick] (3,1) -- (4,0) node[near start, below] {5};
        \draw[thick] (4,0) -- (5,1);
        \draw[thick] (5,1) -- (6,2);
        \draw[thick] (6,2) -- (7,1) node[near start, below] {6};
        \draw[thick] (7,1) -- (8,0) node[near start, below] {1};
        \draw[thick] (8,0) -- (9,1);
        \draw[thick] (9,1) -- (10,0) node[near start, below] {2};
        \draw[thick] (10,0) -- (11,-1) node[near start, below] {3};
    \end{tikzpicture}
    \hspace{.5cm}
    \begin{tikzpicture}[-stealth,scale=.7]
        \def\r{2cm}
        \def\e{3}
        \def\steps{6}

        \foreach \x in {1,...,\steps}{
            \draw[thick,red] (90+360/\steps-360/\steps*\x:\r) arc (90+360/\steps-360/\steps*\x:90-360/\steps*\x+\e:\r);
            \node[circle, draw, inner sep=0pt, minimum size=4pt,fill=white] (N\x) at (90+360/\steps-360/\steps*\x:\r) {\scriptsize \x};
            \draw[thick,blue] (90+360/\steps-360/\steps*\x:1.5*\r) -- (N\x);
            \draw[thick,blue] (N\x) -- (90+360/\steps-360/\steps*\x:0.5*\r);

        }
        \node at (85+360/6-360/6*1:2.6) {\scriptsize $3$};
        \node at (85+360/6-360/6*2:2.6) {\scriptsize $2$};
        \node at (85+360/6-360/6*3:2.6) {\scriptsize $1$};
        \node at (85+360/6-360/6*4:2.6) {\scriptsize $3$};
        \node at (85+360/6-360/6*5:2.6) {\scriptsize $0$};
        \node at (85+360/6-360/6*6:2.6) {\scriptsize $2$};

        \node at (80+360/6-360/6*1:1.5) {\scriptsize $0$};
        \node at (80+360/6-360/6*2:1.5) {\scriptsize $7$};
        \node at (80+360/6-360/6*3:1.5) {\scriptsize $3$};
        \node at (80+360/6-360/6*4:1.5) {\scriptsize $1$};
        \node at (80+360/6-360/6*5:1.5) {\scriptsize $2$};
        \node at (80+360/6-360/6*6:1.5) {\scriptsize $3$};

        \node at (60+360/6-360/6*1:2.25) {\scriptsize $6$};
        \node at (60+360/6-360/6*2:2.25) {\scriptsize $2$};
        \node at (60+360/6-360/6*3:2.25) {\scriptsize $\underline{0}$};
        \node at (60+360/6-360/6*4:2.25) {\scriptsize $3$};
        \node at (60+360/6-360/6*5:2.25) {\scriptsize $2$};
        \node at (60+360/6-360/6*6:2.25) {\scriptsize $2$};
    \end{tikzpicture}
    \caption{Example of the algorithm in the proof of \cref{lem:cyclicVolIntFlowH1p}.
    The two paths built from $W$ and $W'$ are on the left and the cyclic integer flow associated to $W$ is on the right.}
    \label{fig:H1pBars}
\end{figure}

\begin{remark}
\label{rem:surprising_invariance_H_cp}
Note that, surprisingly, the enumeration in \cref{lem:cyclicVolIntFlowH1p} only depends on the total inflow, not how it is divided up.
\end{remark}

\begin{proof}[Proof of \cref{thm:hkr-vol}]
By \cref{thm:normVolCyclicIntegerFlow}, we have that $$
\vol\bdd(\cyclegraph_{c,p}) = \#\funny(\cyclegraph_{c,p}).$$
We proceed by induction on $c$. For $c=1$, by \cref{lem:cyclicVolIntFlowH1p} we have that
\[
    \#\funny(\cyclegraph_{1,p}) = \binom{0+2p-2}{p-1},
\]
as desired.

Note that in a cyclic volume integer flow on $\cyclegraph_{c,p}$ we know that $(p-1)i$ units of flow go from the $i$-th cycle to the $(i+1)$st.
We view each cyclic volume integer flow $f$ in $\cyclegraph_{c,p}$ as a pair $(f',g)$ where $f'$ is a cyclic volume integer flow on the subgraph $\cyclegraph_{c-1,p}$ of the first $c-1$ cycles and $g$ is a cyclic integer flow on the subgraph $\cyclegraph_{1,p}$ of the innermost cycle. In such a pair, denote by ${\bf a}=(a_1,\ldots,a_p)$ the inflow between the $(c-1)$th and $c$th cycle, i.e. $a_i$ is the flow on the edge $(c-1,i)\to (c,i)$. See \cref{fig:breaking-flows-Hcp}. By \cref{lem:cyclicVolIntFlowH1p}, the number of such flows $g$ only depends on the total incoming flow $a_1+\cdots+a_p=(p-1)(c-1)$. From these observations and induction on $c$ we have that
\begin{align*}
   \#\funny(\cyclegraph_{c,p}) &= \#\funny(\cyclegraph_{c-1,p}) \cdot \binom{(c-1)(p-1)+2p-2}{p-1}\\
   &= \binom{c(p-1)}{p-1,\ldots,p-1} \cdot \binom{(c+1)(p-1)}{p-1},
\end{align*}
and the result follows from the definition of the multinomial coefficient.
\end{proof}

\section{The mutoperhedron}
    \label{ss:TheDoppel}

We now restrict our attention to the family of graphs $\cyclegraph_{c,2}$, consisting of $c$-many nested $2$-cycles; see \cref{fig:H32} for a picture of $\cyclegraph_{3,2}$.
In this case, \cref{prop:good_routes_Hcp,thm:hkr-vol} say that $\DKKred(\cyclegraph_{c,2})$ has $2^{c+1} - 2$ rays and $(c+1)!$ maximal cones.
The reader might recognize that these are also the number of rays and maximal cones, respectively, of the (essentialized) \textnew{braid arrangement}, which is the normal fan of the $c$-dimensional \textnew{permutohedron $\Pi_{c+1}$}.
Recall that the $h$-vector (and therefore the $f$-vector) of a simple polytope of dimension $d \leq 5$ is completely determined by its number of vertices and facets.
For example, the $h$-vector of a $5$-dimensional simple polytope $P$ is of the form $(1,a,b,b,a,1)$, where $2+2a+2b$ is the number of vertices of $P$ and $b+d = b+5$ is its number of facets.
Since both the permutohedron $\Pi_{c+1}$ and the polytope whose normal fan is $\DKKred(\cyclegraph_{c,2})$ are simple, they necessarily have the same $f$-vector for $c \leq 5$.

\begin{definition}
    \label{def:doppel}
    The $c$-dimensional \textnew{mutoperhedron $\doppel_c$} is the polytope whose normal fan is $\DKKred(\cyclegraph_{c,2})$.
    One explicit realization of $\doppel_c$ is given by the Minkowski sum in \cref{thm:polytope-dual-to-cyclic-gfan}.
\end{definition}

See \cref{fig:doppel3} for a picture of the $3$-dimensional mutoperhedron.
We will devote the remainder of this section to proving the following result, which explains the choice of the name mutoperhedron.

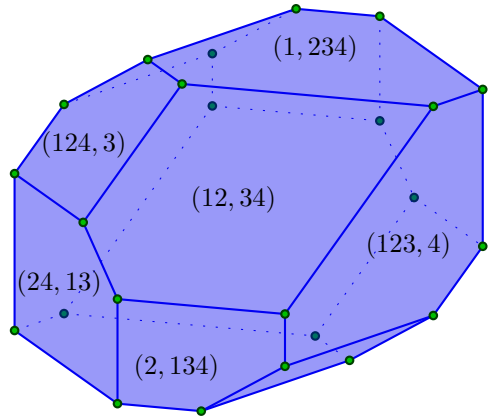
\begin{figure}[ht]
    \centering
    \input{figures/the_non-permutahedron.tex}
    \caption{
        The mutoperhedron $\doppel_3$.
        The picture shows its 24 vertices, 36 edges, and 14 facets.
        The facets corresponding to the route $\route_{(12,34)}$ and its compatible routes are labelled as explained in \cref{ssec:facets}.
        }
    \label{fig:doppel3}
\end{figure}

\begin{theorem}
    \label{thm:f-vect_doppel}
    The mutoperhedron $\doppel_c$ and the permutohedron $\Pi_{c+1}$ have the same $f$-vector.
    That is, for all $c \geq 1$ and $0 \leq i \leq c$,
    \[
        f_i(\doppel_c) = f_i(\Pi_{c+1}).
    \]
    However, for $c \geq 3$, these polytopes are {\em not} combinatorially isomorphic.
\end{theorem}

We can equivalently rewrite the previous result in terms of $h$-vectors as follows.

\begin{corollary}
    For all $c \geq 1$ and $0 \leq i \leq c$,
    \[
        h_i(\doppel_c) = A_{c+1,i},
    \]
    where $A_{n,k} := \# \set{ w \in \mathcal{S}_n }{ {\rm des}(w) = k }$ denotes the Eulerian numbers.
\end{corollary}

We will prove \cref{thm:f-vect_doppel} by providing a bijection between multisets of $k$ intersecting facets of $\doppel_c$ and multisets of $k$ intersecting facets of $\Pi_{c+1}$.

\subsection{The facets of the mutoperhedron}\label{ssec:facets}

Since the normal fan of the mutoperhedron $\doppel_c$ is $\DKKred(\cyclegraph_{c,2})$,
its facets are in bijection with the good non-exceptional routes of $\cyclegraph_{c,2}$,
and a collection of facets intersects if and only if the corresponding routes are compatible.

Recall that $\cyclegraph_{c,2}$ is coloured so that edges along cycles are \tcr{red}, and the remaining edges are \tcb{blue}. We label the \tcb{blue} edges of $\cyclegraph_{c,2}$ in such a way that
\[
    u_1 u_2 \ldots u_{c+1}
    \qquad\text{and}\qquad
    d_1 d_2 \ldots d_{c+1}
\]
are its two exceptional \tcb{blue} routes; \cref{fig:H2c_route} shows this labelling when $c=3$.
In the pictures, the arrows $u_i$ and $d_i$ are on the top and bottom, respectively.

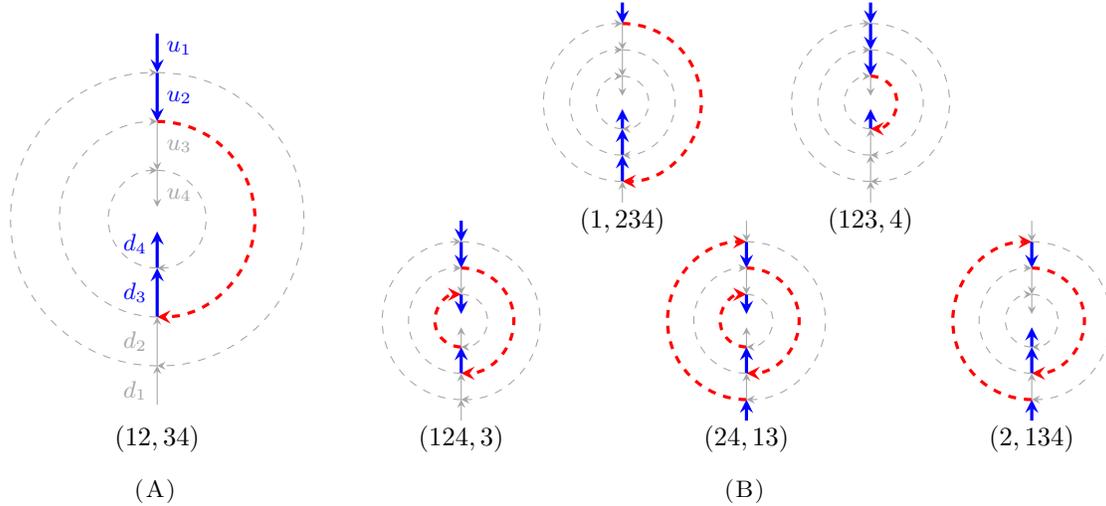
\begin{figure}[ht]
    \begin{subfigure}{.3\linewidth}
        \centering
        \begin{tikzpicture}[-stealth,scale=.65,rotate=-90]
            \input{figures/H32_12-34.tex}
            \node at (4.5, 0) {$(12,34)$};
        \end{tikzpicture}
        \caption{}
        \label{fig:H2c_route}
    \end{subfigure}
    \hfill
    \begin{subfigure}{.65\linewidth}
        \centering
        \begin{tikzpicture}[-stealth,scale=.35,rotate=-90]
            \input{figures/H32_1-234.tex}
            \node at (4.5, 0) {$(1,234)$};
        \end{tikzpicture}
        \qquad\quad
        \begin{tikzpicture}[-stealth,scale=.35,rotate=-90]
            \input{figures/H32_123-4.tex}
            \node at (4.5, 0) {$(123,4)$};
        \end{tikzpicture} \\[-10pt]
        \begin{tikzpicture}[-stealth,scale=.35,rotate=-90]
            \input{figures/H32_124-3.tex}
            \node at (4.5, 0) {$(124,3)$};
        \end{tikzpicture}
        \hfill
        \begin{tikzpicture}[-stealth,scale=.35,rotate=-90]
            \input{figures/H32_24-13.tex}
            \node at (4.5, 0) {$(24,13)$};
        \end{tikzpicture}
        \hfill
        \begin{tikzpicture}[-stealth,scale=.35,rotate=-90]
            \input{figures/H32_2-134.tex}
            \node at (4.5, 0) {$(2,134)$};
        \end{tikzpicture}
        \caption{}
        \label{fig:H2c_pentagon}
    \end{subfigure}
    \caption{
        (A) The labels of the \tcb{blue} edges and the route $\route_{(12,34)} \in \Routes(\cyclegraph_{3,2})$.
        (B) The five routes in $\Routes(\cyclegraph_{3,2})$ compatible with $\route_{(12,34)}$.
    }
    \label{fig:compatibility12_34}
\end{figure}

A good route of $\cyclegraph_{c,2}$ is uniquely determined by the collection of \tcb{blue} edges it uses (\cref{prop:good_routes_Hcp}).
Therefore, the good non-exceptional routes of $\cyclegraph_{c,2}$, and consequently also the facets of $\doppel_c$, are in bijection with ordered partitions of $[c+1]$ into exactly two parts, henceforth referred to as \textnew{bipartitions}.
Explicitly, the route $\route_{(A,B)}$ corresponding to a bipartition $(A,B)$ of $[c+1]$ is the unique good route using \tcb{blue} edges $u_i$ for $i \in A$ and $d_j$ for $j \in B$.
See \cref{fig:compatibility12_34} for an example.
Note that, to simplify the notation, we drop the brackets and the commas from each individual part of an ordered partition. So, for example, we simply write $(12,34)$ instead of $(\{1,2\},\{3,4\})$.

\begin{example}
    \label{ex:Hc2}
    Consider the mutoperhedron $\doppel_3$ depicted in \cref{fig:doppel3}.
    The route $\route_{(12,34)} \in \Routes(\cyclegraph_{3,2})$, shown in \cref{fig:H2c_route},
    is compatible with exactly five other good non-exceptional routes, shown in \cref{fig:H2c_pentagon}.
    Thus, the facet of $\doppel_3$ dual to the ray of $\DKKred(\cyclegraph_{3,2})$ corresponding to $\route_{(12,34)}$ intersects exactly five other facets at a face of codimension two.
    Since $\doppel_3$ is $3$-dimensional, this facet is a pentagon.
\end{example}

The two-dimensional faces of a permutohedron are always either squares or hexagons.
Thus, even though the polytopes $\doppel_3$ and $\Pi_4$ have the same $f$-vector, the previous example shows that they are not combinatorially isomorphic. This obstruction propagates to higher dimensions. This proves the last claim of \cref{thm:f-vect_doppel}.

\begin{lemma} \label{lemma:facet iso to smaller doppel}
The facet of $\doppel_c$ indexed by the bipartition $P_\circ=(\{1\},\{2,3,\dots,c+1\})$ is combinatorially isomorphic to $\doppel_{c-1}$.
\end{lemma}

\begin{proof}
    Let $\route_\circ \in \Routes(\cyclegraph_{c,2})$ be the good route using blue edges $u_1,d_2,d_3,\dots,d_{c+1}$.
    Observe that any non-exceptional good route using edge $d_1$ will necessarily be incompatible with $\route_\circ$, as the common subwalk $d_2,\dots,d_k$ witnesses.
    Thus, good routes $\route \in \Routes(\cyclegraph_{c-1,2})$ are in bijection with routes $\route \in \Routes(\cyclegraph_{c,2})$ compatible with $\route_\circ$.
    Explicitly, using the corresponding bipartition, the bijection is simply $(A,B) \to (\{1\}\cup\set{i+1}{i \in A},\set{i+1}{i \in B})$.
\end{proof}

\subsection{Compatibility}
For the permutohedron $\Pi_{c+1}$, two facets indexed by bipartitions $(A_1,B_1)$ and $(A_2,B_2)$ of $[c+1]$ intersect if and only if the sets $A_1$ and $A_2$ are {\em nested}: either $A_1 \subseteq A_2$ or $A_2 \subseteq A_1$. We will now study some alternative ways to encode the combinatorics of nested bipartitions.

\subsubsection{Arc diagrams}

We consider subgraphs of the graph with vertex set $[0,c+1] = \{0,1,\dots,c+1\}$ and two edges between each pair $i < j$,
one on {\em top} $(i,j)^+$ and one at the {\em bottom} $(i,j)^-$; we call these edges~\textnew{arcs}.
An \textnew{arc diagram} is a set of arcs of the following form:
\[
    (0,i_1)^{\sigma_1} (i_1,i_2)^{\sigma_2} \dots (i_{\ell-2},i_{\ell-1})^{\sigma_{\ell-1}} (i_{\ell-1},c+1)^{\sigma_\ell},
\]
such that the signs $\sigma_1,\sigma_2,\dots,\sigma_\ell$ are alternating.

Bipartitions and arc diagrams are in bijection as follows.
Label the gap between vertices $i-1$ and $i$ by $i$.
Given a bipartition $(A,B)$, there is exactly one arc diagram that goes over the gaps with labels in $A$ and under the gaps with labels in $B$.
Namely, this arc diagram consists of edges $(i-1,j)^+$ for each maximal interval $[i,j]$ contained in $A$, and edges $(i-1,j)^-$ for each maximal interval $[i,j]$ contained in $B$.
See \cref{fig:arcdiag} for an example with $c = 8$.

\begin{figure}[ht]
    \centering
    \begin{tikzpicture}
        \foreach \x in {1,2,...,10}{
            \pgfmathsetmacro\y{int(\x-1)}
            \node [draw,circle] (\x) at (\x,0) {\y};
        }
		 \draw (1) .. controls +(up:1cm) and +(up:1cm) .. (3) {};
		 \draw (3) .. controls +(down:1cm) and +(down:1cm) .. (5) {};
		 \draw (5) .. controls +(up:1cm) and +(up:1cm) .. (9) {};
		 \draw (9) .. controls +(down:1cm) and +(down:1cm) .. (10) {};
	\end{tikzpicture}
    \caption{The arc diagram $\{(0,2)^+,(2,4)^-,(4,8)^+,(8,9)^-\}$ corresponding to the bipartition $(12 \, 5678 \,,\, 34 \, 9)$.}
    \label{fig:arcdiag}
\end{figure}
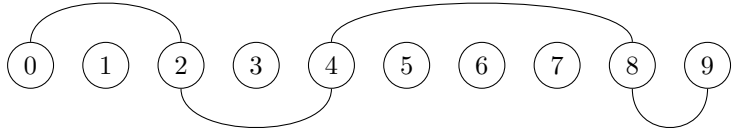

  \begin{definition}\label{def:cross}
        Two arcs of the same sign $(i,j)^\sigma$ and $(k,l)^\sigma$ with $i < k$, are said to
        \textnew{cross} if  $i < k \leq j < l$.
        We say that two diagrams \textnew{cross} if an arc of one crosses an arc of the other.
    \end{definition}

Note that we can and do fix a way of drawing the arcs $(i,j)^\pm$ in the plane so that arcs which do not cross in the above sense, do not cross in the literal sense as curves in the plane (excluding their endpoints).

    Observe that no two arcs of a single arc diagram can cross, since they are either disjoint or they alternate in sign.

We say that a collection of arc diagrams is \textnew{stacked} if the arc diagrams can be ordered $\gamma_1,\dots,\gamma_k$ so that the arcs of $\gamma_1$ lie weakly above the arcs of $\gamma_2$, which lie weakly above the arcs of $\gamma_3$, and so on.
See \cref{fig:stacked} for an example of a stacked collection of arc diagrams for $c=14$.

The following lemma is obvious.

\begin{lemma} A collection of bipartitions is (pairwise) nested if and only if the corresponding collection of arc diagrams is stacked.\end{lemma}

If a collection of arc diagrams is stacked, observe that none of them cross.

\subsubsection{Multiarcs}

A collection of arc diagrams naturally gives rise to a multiset of arcs.

\begin{definition}
    A multiset $\mathcal{A}$ of arcs is \textnew{consistent} if
    no pair of arcs in $\mathcal{A}$ cross, and
    any $x \in [c]$ has the same number of incoming arcs and outgoing arcs.

    Moreover, we say that $\mathcal{A}$ has \textnew{rank $k$} if there are $k$ arcs incident to vertex $0$ (equivalently, to vertex $c+1$).
\end{definition}

A collection of $k$ stacked arc diagrams gives rise to a consistent multiset of arcs of rank $k$. In fact, more is true.

\begin{proposition}
    \label{prop:stacked-bij}
    Multisets of $k$ stacked arc diagrams are in bijection with consistent rank $k$ multisets of arcs.
\end{proposition}

\begin{proof} We have already seen that a collection of $k$ stacked arcs gives rise to a consistent multiset of rank $k$. To go in the reverse direction, starting from $\mathcal A$, a consistent multiset of arcs of rank $k$, form an arc diagram by taking the top arc starting from 0, say $(0,i_1)^\sigma$, then the top arc starting from $i_1$, say $(i_1,i_2)^{-\sigma}$, and continuing in this way to $c+1$.  (Note that, by the condition that there are no crossings, the signs of the arcs will necessarily alternate.) Call the resulting arc diagram $\gamma_1$. Remove its edges from $\mathcal A$, and repeat the same process $k-1$ more times. The result is a collection $\{\gamma_1,\dots, \gamma_k\}$ of $k$ arc diagrams, which is obviously stacked: the arcs of $\gamma_1$ lie weakly above all the others, and by induction the remaining arc diagrams are stacked.

We now check that $\{\gamma_1,\dots, \gamma_k\}$ is the only collection of stacked arc diagrams that is sent to $\mathcal A$. By induction, it is enough to show that any collection of stacked arc diagrams $D$ sent to $\mathcal A$ must contain $\gamma_1$. Suppose that this were false. Then $D$ must contain some arc diagram $\delta$ which uses the topmost arc for the first $r-1$ steps, but on the $r$-th step, does not take the topmost arc $(i,j)^\sigma$. Choose such an arc diagram $\delta$ maximizing $r$.  Now, $D$ contains some other arc diagram $\eta$ which uses $(i,j)^\sigma$. By the maximality of $r$, we know that $\eta$ uses at least one edge prior to $(i,j)^\sigma$ which does not appear in $\delta$, and the rightmost such edge lies below the corresponding edges of $\delta$. This shows that $\delta$ and $\eta$ are not stacked, and thus $D$ is not stacked.
\end{proof}

\subsubsection{Compatible routes}
We now give a similar construction showing that collections of $k$ compatible good routes on $\cyclegraph_{c,2}$ also correspond to consistent multisets of arcs of rank $k$.

Using the bijections between good routes of $\cyclegraph_{c,2}$, bipartitions of $[c+1]$, and arc diagrams on vertex set $[0,c+1]$, we write $\route_D \in \Routes(\cyclegraph_{c,2})$ for the good route associated to an arc diagram $D$.

\begin{definition}
    \label{def:interference}
    Two arc diagrams \textnew{interfere} if they contain arcs
    \begin{eqnarray*}
        (a,i_0)^{\sigma_0} \hspace{2pt} \underbrace{(i_0,i_1)^{\sigma_1} \dots (i_{k-1},i_k)^{\sigma_k}}_{\text{common subpath}} \hspace{2pt} (i_k,b)^{\sigma_{k+1}}\hspace{3pt} & \text{and} \\
        (a',i_0)^{\sigma_0} \overbrace{(i_0,i_1)^{\sigma_1} \dots (i_{k-1},i_k)^{\sigma_k}} (i_k,b')^{\sigma_{k+1}},
    \end{eqnarray*}
    with $a < a'$ and $b < b'$, for some $k \geq 0$.
\end{definition}

\begin{example}
    Below, we show two pairs of interfering diagrams.
    In the first case $k = 0$, so the common {\em subpath} consists of a single node, and in the second $k = 1$.
    
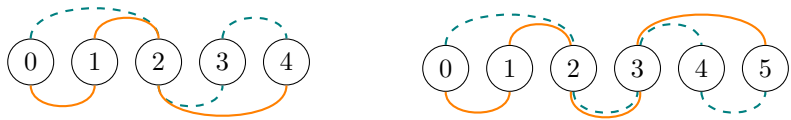
\[
    	\begin{tikzpicture}[scale=0.85]
            \foreach \x in {1,2,...,5}{
                \pgfmathsetmacro\y{int(\x-1)}
                \node [draw,circle] (\x) at (\x,0) {\y};
            }
            \draw[dashed,thick,teal] (1) .. controls +(up:1cm) and +(up:1cm) .. (3) {};
            \draw[dashed,thick,teal] (3) .. controls +(down:0.8cm) and +(down:0.8cm) .. (4) {};
            \draw[dashed,thick,teal] (4) .. controls +(up:0.8cm) and +(up:0.8cm) .. (5) {};
            \draw[thick,orange] (1) .. controls +(down:0.8cm) and +(down:0.8cm) .. (2) {};
            \draw[thick,orange] (2) .. controls +(up:0.8cm) and +(up:0.8cm) .. (3) {};
            \draw[thick,orange] (3) .. controls +(down:1cm) and +(down:1cm) .. (5) {};
	    \end{tikzpicture}
    	\hspace*{.1\linewidth}
    	\begin{tikzpicture}[scale=0.85]
            \foreach \x in {1,2,...,6}{
                \pgfmathsetmacro\y{int(\x-1)}
                \node [draw,circle] (\x) at (\x,0) {\y};
            }
            \draw[dashed,thick,teal] (1) .. controls +( 0.0cm, 1.0cm) and +( 0.0cm, 1.0cm) .. (3) {};
            \draw[dashed,thick,teal] (3) .. controls +( 0.0cm,-0.8cm) and +( 0.0cm,-0.8cm) .. (4) {};
            \draw[dashed,thick,teal] (4) .. controls +( 0.1cm, 0.8cm) and +( 0.0cm, 0.8cm) .. (5) {};
            \draw[dashed,thick,teal] (5) .. controls +( 0.0cm,-0.8cm) and +( 0.0cm,-0.8cm) .. (6) {};
            \draw[thick,orange]      (1) .. controls +( 0.0cm,-0.8cm) and +( 0.0cm,-0.8cm) .. (2) {};
            \draw[thick,orange]      (2) .. controls +( 0.0cm, 0.8cm) and +(-0.1cm, 0.8cm) .. (3) {};
            \draw[thick,orange]      (3) .. controls +(-0.1cm,-0.9cm) and +( 0.1cm,-0.9cm) .. (4) {};
            \draw[thick,orange]      (4) .. controls +( 0.0cm, 1.0cm) and +( 0.0cm, 1.0cm) .. (6) {};
	    \end{tikzpicture}
    \]
\end{example}

\begin{lemma} A collection of good routes in $\cyclegraph_{c,2}$ is compatible if and only if the corresponding arc diagrams neither cross nor interfere. \end{lemma}

\begin{proof} One checks that crosses and interferences between the arc diagrams correspond exactly to the possible configurations which make routes incompatible.
Explicitly, if $D$ and $D'$ are as in \cref{def:interference}, then $\route_D$ and $\route_{D'}$ have a maximal common subwalk $S$ going from the $a'$th cycle to the $b$th cycle; the last edge of $\route_D$ before $S$ is \tcb{blue} and its first edge after $S$ is \tcr{red}.
\end{proof}

From this lemma, it follows that a collection of $k$ compatible good routes corresponds to a consistent rank $k$ multiset. We now prove the analogue of \cref{prop:stacked-bij} for compatible collections of good routes.

\begin{proposition}
    \label{prop:route-bij}
    Multisets of $k$ compatible good routes are in bijection with consistent rank $k$ multisets of arcs.
\end{proposition}

\begin{proof} The proof is similar to that of Proposition \cref{prop:stacked-bij}. Suppose that we have a consistent multiset of arcs $\mathcal A$ of rank $k$. Draw all the arcs of $\mathcal A$, in such a way that arcs which do not cross in the sense of \cref{def:cross}, do not cross as curves in the plane. If $\mathcal A$ contains multiple copies of the same arc, draw the appropriate number of copies of that arc, in a non-crossing fashion.
This drawing of the arcs of $\mathcal A$ puts a total ordering on  the incoming arcs at each vertex $i\in [c]$,
and separately on the outgoing arcs.
Because $\mathcal A$ is consistent, for $j \in [c]$, all arcs enter $j$ on the same side, and all arcs leave $j$ on the same side.
Thus, we can unambiguously number the incoming arcs and the outgoing arcs, counting from the outside in (i.e., starting from the lowest bottom arc or highest top arc).

Now define a collection $D$ of arc diagrams using the multiset of arcs $\mathcal A$, as follows. Choose an arc of $\mathcal A$ leaving 0. If it is the $k$-th arc to arrive at vertex $i_1$, then continue the arc diagram by using the $k$-th arc to leave vertex $i_1$. Proceed similarly when this arc reaches its end vertex, and continue on until we reach vertex $c+1$. Doing this, starting from each arc that leaves 0, leads to a collection $D$ of arc diagrams using the multiset $\mathcal A$.

We now show that the collection $D$ is pairwise non-interfering. Suppose that we have two arc diagrams $\gamma_1,\gamma_2$ in $D$. In order for them to interfere, they must have some common subpath, say from vertex $i$ to vertex $j$, such that $\gamma_1$ and $\gamma_2$ arrive at $i$ using distinct arcs, and leave $j$ using distinct arcs. Say that $\gamma_1$ arrives at $i$ on the arc numbered $n_1$, and $\gamma_2$ arrives on the arc numbered $n_2$, with $n_1<n_2$, so that $\gamma_1$ is on the arc closer to the outside. This means that $\gamma_1$ arrives along a longer arc than $\gamma_2$.
Then, by the definition of $D$, $\gamma_1$ leaves $i$ on the arc numbered $n_1$, and $\gamma_2$ leaves $i$ on the arc numbered $n_2$, so that $\gamma_1$ takes the arc that is closer to the outside. This continues until we reach $j$. At this point, $\gamma_1$ leaves $j$ on the arc that is closer to the outside. As a result, $\gamma_1$ leaves on the longer arc. Thus, $\gamma_1$ and $\gamma_2$ do not interfere.

We now check that the collection $D$ of arc diagrams constructed above is the only pairwise non-interfering collection of arc diagrams using the multiset of arcs $\mathcal A$.
Let $D'$ be some other collection of arc diagrams using $\mathcal A$. Suppose that up to some vertex $j$, $D'$ agrees with $D$, but that once we consider also how the arc diagrams leave vertex $j$, it does not. Let $\gamma_1,\dots,\gamma_r$ be the arc diagrams of $D$ that arrive at node $j$,
ordered from outside to inside as they arrive at $j$,
as in our construction.
Number the arc diagrams of $D'$ incident with node $j$ as $\gamma_1',\dots,\gamma_r'$, so that $\gamma_i'$ agrees with $\gamma_i$ up to node $j$. (There may be more than one way to do this.) Consider the set $I$ of pairs $(m,n)$ with $1\leq m<n\leq r$, such that when leaving $j$, $\gamma_n'$ takes a longer arc than $\gamma_m'$. This set is non-empty because $D'$ disagrees with $D$.
If the only pairs $(m,n)$ in $I$ were such that the arc diagrams $\gamma_m$ and $\gamma_n$ agree when restricted to the part up to vertex $j$, then it would be possible to renumber the arc diagrams of $D'$ so that $\gamma_i'$ agrees with $\gamma_i$ for all $i$ past vertex $j$, contrary to our hypothesis. Thus, there exists a pair $(m,n)\in I$, such that the restrictions up to vertex $j$ of $\gamma_m$ and $\gamma_n$ are distinct. By the construction of $D$, because $\gamma_m$ is closer to the outside than $\gamma_n$ as it enters $j$, the same is true for the entire common segment ending at $j$, and also at the starting point of the common segment. Thus, $\gamma_m$ joins the common segment closer to the outside (i.e., with a longer segment) than $\gamma_n$. By our numbering of the arcs of $D'$, the same is true for $\gamma_m'$ and $\gamma'_n$. By our choice of $(m,n)$, though, $\gamma_m'$ leaves $j$ on a shorter arc than does $\gamma_n'$. This shows that $\gamma_m'$ and $\gamma_n'$ interfere.
\end{proof}

\cref{prop:route-bij} is illustrated in \cref{fig:compatiblediagrams} below.

\begin{figure}[ht]
	\centering
	\begin{tikzpicture}[scale=.9]
        \foreach \x in {1,2,...,16}{
            \pgfmathsetmacro\y{int(\x-1)}
            \node [draw,circle] (\x) at (\x,0) {\y};
        }

		\draw[thick,violet] (1) .. controls +(up:1cm) and +(up:1cm) .. (3) {};
		\draw[thick,orange] (1) .. controls +(up:1.1cm) and +(up:1.1cm) .. (3) {};
		\draw[thick,cyan] (1) .. controls +(up:1.2cm) and +(up:1.2cm) .. (3) {};
		\draw[thick,red] (1) .. controls +(up:1.3cm) and +(up:1.3cm) .. (3) {};

		\draw[thick,green] (1) .. controls +(down:1.4cm) and +(down:1.4cm) .. (6) {};

		\draw[thick,red] (3) .. controls +(down:.7cm) and +(down:.7cm) .. (4) {};
		\draw[thick,cyan] (3) .. controls +(down:.8cm) and +(down:.8cm) .. (4) {};

		\draw[thick,orange] (3) .. controls +(down:1cm) and +(down:1cm) .. (5) {};
		\draw[thick,violet] (3) .. controls +(down:1.1cm) and +(down:1.1cm) .. (5) {};

		\draw[thick,cyan] (4) .. controls +(up:1.3cm) and +(up:1.3cm) .. (7) {};
		\draw[thick,red] (4) .. controls +(up:1.6cm) and +(up:1.6cm) .. (9) {};

		\draw[thick,violet] (5) .. controls +(up:1cm) and +(up:1cm) .. (7) {};
		\draw[thick,orange] (5) .. controls +(up:1.1cm) and +(up:1.1cm) .. (7) {};

		\draw[thick,green] (6) .. controls +(up:.7cm) and +(up:.7cm) .. (7) {};

		\draw[thick,cyan] (7) .. controls +(down:1cm) and +(down:1cm) .. (10) {};
		\draw[thick,orange] (7) .. controls +(down:1.1cm) and +(down:1.1cm) .. (10) {};
		\draw[thick,violet] (7) .. controls +(down:1.2cm) and +(down:1.2cm) .. (10) {};
		\draw[thick,green] (7) .. controls +(down:1.3cm) and +(down:1.3cm) .. (10) {};

		\draw[thick,red] (9) .. controls +(down:.7cm) and +(down:.7cm) .. (10) {};

		\draw[thick,green] (10) .. controls +(up:.7cm) and +(up:.7cm) .. (11) {};
		\draw[thick,violet] (10) .. controls +(up:.8cm) and +(up:.8cm) .. (11) {};
		\draw[thick,orange] (10) .. controls +(up:.9cm) and +(up:.9cm) .. (11) {};
		\draw[thick,cyan] (10) .. controls +(up:1cm) and +(up:1cm) .. (11) {};

		\draw[thick,red] (10) .. controls +(up:1.6cm) and +(up:1.6cm) .. (14) {};

		\draw[thick,cyan] (11) .. controls +(down:.7cm) and +(down:.7cm) .. (12) {};

		\draw[thick,orange] (11) .. controls +(down:1cm) and +(down:1cm) .. (13) {};

		\draw[thick,violet] (11) .. controls +(down:1.4cm) and +(down:1.4cm) .. (15) {};

		\draw[thick,green] (11) .. controls +(down:1.6cm) and +(down:1.6cm) .. (16) {};

		\draw[thick,cyan] (12) .. controls +(up:1cm) and +(up:1cm) .. (14) {};

		\draw[thick,orange] (13) .. controls +(up:.7cm) and +(up:.7cm) .. (14) {};

		\draw[thick,red] (14) .. controls +(down:.7cm) and +(down:.7cm) .. (15) {};
		\draw[thick,cyan] (14) .. controls +(down:.8cm) and +(down:.8cm) .. (15) {};
		\draw[thick,orange] (14) .. controls +(down:.9cm) and +(down:.9cm) .. (15) {};

		\draw[thick,violet] (15) .. controls +(up:.7cm) and +(up:.7cm) .. (16) {};
		\draw[thick,orange] (15) .. controls +(up:.8cm) and +(up:.8cm) .. (16) {};
		\draw[thick,cyan] (15) .. controls +(up:.9cm) and +(up:.9cm) .. (16) {};
		\draw[thick,red] (15) .. controls +(up:1cm) and +(up:1cm) .. (16) {};
	\end{tikzpicture}
    \caption{
        A collection of five stacked arc diagrams.
        Each arc diagram has a different colour.
    }
	\label{fig:stacked}
\end{figure}
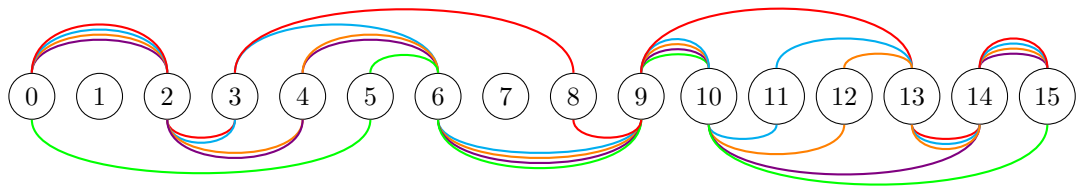

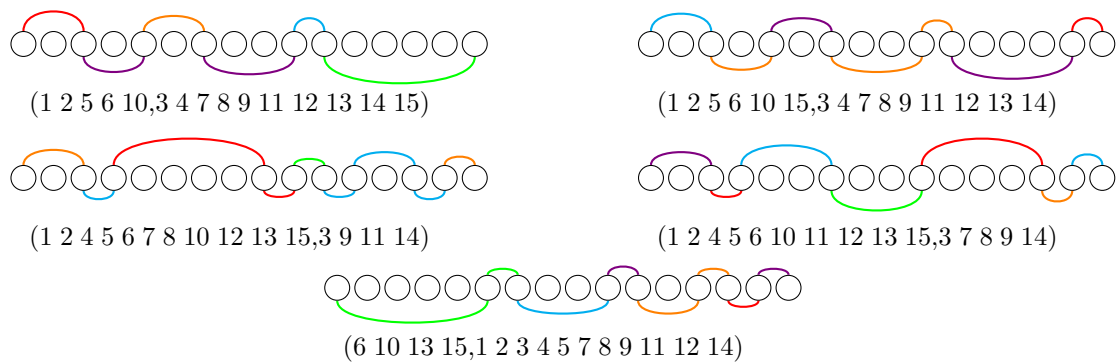
\begin{figure}[ht]
	\centering
	\begin{tikzpicture}[scale=.4]
		\draw (0,0) node foreach \x in {1,2,...,16} [draw,circle] (\x) at (\x,0) {};

		\draw[thick,red] (1) .. controls +(up:1.3cm) and +(up:1.3cm) .. (3) {};
		\draw[thick,violet] (3) .. controls +(down:1.1cm) and +(down:1.1cm) .. (5) {};
		\draw[thick,orange] (5) .. controls +(up:1.1cm) and +(up:1.1cm) .. (7) {};
		\draw[thick,violet] (7) .. controls +(down:1.2cm) and +(down:1.2cm) .. (10) {};
		\draw[thick,cyan] (10) .. controls +(up:1cm) and +(up:1cm) .. (11) {};
		\draw[thick,green] (11) .. controls +(down:1.6cm) and +(down:1.6cm) .. (16) {};
        \node [align=flush center] at (8,-2)
        {
            (1\;2\;5\;6\;10,3\;4\;7\;8\;9\;11\;12\;13\;14\;15)
        };
	\end{tikzpicture}
	\hfill
	\begin{tikzpicture}[scale=.4]
		\draw (0,0) node foreach \x in {1,2,...,16} [draw,circle] (\x) at (\x,0) {};

		\draw[thick,cyan] (1) .. controls +(up:1.2cm) and +(up:1.2cm) .. (3) {};
		\draw[thick,orange] (3) .. controls +(down:1cm) and +(down:1cm) .. (5) {};
		\draw[thick,violet] (5) .. controls +(up:1cm) and +(up:1cm) .. (7) {};
		\draw[thick,orange] (7) .. controls +(down:1.1cm) and +(down:1.1cm) .. (10) {};
		\draw[thick,orange] (10) .. controls +(up:.9cm) and +(up:.9cm) .. (11) {};
		\draw[thick,violet] (11) .. controls +(down:1.4cm) and +(down:1.4cm) .. (15) {};
		\draw[thick,red] (15) .. controls +(up:1cm) and +(up:1cm) .. (16) {};
        \node [align=flush center] at (8,-2)
        {
            (1\;2\;5\;6\;10\;15,3\;4\;7\;8\;9\;11\;12\;13\;14)
        };
	\end{tikzpicture}

	\begin{tikzpicture}[scale=.4]
		\draw (0,0) node foreach \x in {1,2,...,16} [draw,circle] (\x) at (\x,0) {};

		\draw[thick,orange] (1) .. controls +(up:1.1cm) and +(up:1.1cm) .. (3) {};
		\draw[thick,cyan] (3) .. controls +(down:.8cm) and +(down:.8cm) .. (4) {};
		\draw[thick,red] (4) .. controls +(up:1.6cm) and +(up:1.6cm) .. (9) {};
		\draw[thick,red] (9) .. controls +(down:.7cm) and +(down:.7cm) .. (10) {};
		\draw[thick,green] (10) .. controls +(up:.7cm) and +(up:.7cm) .. (11) {};
		\draw[thick,cyan] (11) .. controls +(down:.7cm) and +(down:.7cm) .. (12) {};
		\draw[thick,cyan] (12) .. controls +(up:1cm) and +(up:1cm) .. (14) {};
		\draw[thick,cyan] (14) .. controls +(down:.8cm) and +(down:.8cm) .. (15) {};
		\draw[thick,orange] (15) .. controls +(up:.8cm) and +(up:.8cm) .. (16) {};
        \node [align=flush center] at (8,-2)
        {
            (1\;2\;4\;5\;6\;7\;8\;10\;12\;13\;15,3\;9\;11\;14)
        };
	\end{tikzpicture}
	\hfill
	\begin{tikzpicture}[scale=.4]
		\draw (0,0) node foreach \x in {1,2,...,16} [draw,circle] (\x) at (\x,0) {};

		\draw[thick,violet] (1) .. controls +(up:1cm) and +(up:1cm) .. (3) {};
		\draw[thick,red] (3) .. controls +(down:.7cm) and +(down:.7cm) .. (4) {};
		\draw[thick,cyan] (4) .. controls +(up:1.3cm) and +(up:1.3cm) .. (7) {};
		\draw[thick,green] (7) .. controls +(down:1.3cm) and +(down:1.3cm) .. (10) {};
		\draw[thick,red] (10) .. controls +(up:1.6cm) and +(up:1.6cm) .. (14) {};
		\draw[thick,orange] (14) .. controls +(down:.9cm) and +(down:.9cm) .. (15) {};
		\draw[thick,cyan] (15) .. controls +(up:.9cm) and +(up:.9cm) .. (16) {};
        \node [align=flush center] at (8,-2)
        {
            (1\;2\;4\;5\;6\;10\;11\;12\;13\;15,3\;7\;8\;9\;14)
        };
	\end{tikzpicture}

	\begin{tikzpicture}[scale=.4]
		\draw (0,0) node foreach \x in {1,2,...,16} [draw,circle] (\x) at (\x,0) {};

		\draw[thick,green] (1) .. controls +(down:1.4cm) and +(down:1.4cm) .. (6) {};
		\draw[thick,green] (6) .. controls +(up:.7cm) and +(up:.7cm) .. (7) {};
		\draw[thick,cyan] (7) .. controls +(down:1cm) and +(down:1cm) .. (10) {};
		\draw[thick,violet] (10) .. controls +(up:.8cm) and +(up:.8cm) .. (11) {};
		\draw[thick,orange] (11) .. controls +(down:1cm) and +(down:1cm) .. (13) {};
		\draw[thick,orange] (13) .. controls +(up:.7cm) and +(up:.7cm) .. (14) {};
		\draw[thick,red] (14) .. controls +(down:.7cm) and +(down:.7cm) .. (15) {};
		\draw[thick,violet] (15) .. controls +(up:.7cm) and +(up:.7cm) .. (16) {};
        \node [align=flush center] at (8,-2)
        {
            (6\;10\;13\;15,1\;2\;3\;4\;5\;7\;8\;9\;11\;12\;14)
        };
	\end{tikzpicture}
	\caption{The unique set of compatible routes obtained from the arcs of the set of diagrams in \cref{fig:stacked} with their respective bipartitions.
    }
	\label{fig:compatiblediagrams}
\end{figure}

\begin{proof}[Proof of \cref{thm:f-vect_doppel}]
    We prove that $f_i(\doppel_c) = f_i(\Pi_{c+1})$ by induction on $k = c-i$.
    Recall that faces of codimension $k$ of $\Pi_{c+1}$ (resp. of $\doppel_c$) are in bijection with sets of $k$ stacked (resp. noncrossing and noninterfering) arc diagrams.
    The case $c-i=1$ follows by setting $k=1$ in \cref{prop:stacked-bij,prop:route-bij}, as multisets of cardinality $1$ are simply sets.
    Now, by grouping multisets of cardinality $c-i = k$ by their underlying set, we get that the number of multisets of $k$ stacked arc diagrams and the number of $k$ noncrossing and noninterfering arc diagrams are
    \[
        \sum_{j = 1}^{k} f_{c-j}(\doppel_c) \binom{k-1}{k-j}
        \quad\text{and}\quad
        \sum_{j = 1}^{k} f_{c-j}(\Pi_{c+1}) \binom{k-1}{k-j},
    \]
    respectively.
    Since these quantities are equal by \cref{prop:stacked-bij,prop:route-bij}, and, by induction, $f_{c-j}(\doppel_c) = f_{c-j}(\Pi_{c+1})$ for all $1 \leq j < k$, we deduce that $f_{c-k}(\doppel_c) = f_{c-k}(\Pi_{c+1})$.
\end{proof}

\begin{remark}
    The bijection between 
    multisets of nested arc diagrams and multisets of noninterfering arc diagrams
    is reminiscent of the bijection between noncrossing and nonnesting tableaux via the edge multiset of a Gelfand--Tsetlin pattern from \cite{pps}. It would be interesting to provide a theory that incorporates both. This seems all the more relevant since \cite{pps} was one of the motivations of~\cite{DKK}.
\end{remark}

\begin{example}
The construction above does not directly yield a bijection between the faces of $\doppel_c$ and the faces of $\Pi_{c+1}$.
For instance, starting with the nested set of bipartitions $(2,134),(12,34),(124,3)$ (that corresponds to a vertex of $\Pi_4$),
one obtains the multiset of arc diagrams corresponding to $\{\mskip-5mu\{ (12,34),(12,34),(24,13) \}\mskip-5mu\}$, which does not correspond uniquely to a vertex of $\doppel_3$.
This is illustrated in \cref{fig:non-bijection-of-faces} below.
\begin{figure}[h]
    \centering
    \begin{tikzpicture}[scale=0.85]
        \foreach \x in {1,2,...,5}{
            \pgfmathsetmacro\y{int(\x-1)}
            \node [draw,circle] (\x) at (\x,0) {\y};
        }
        \draw[thick,green]       (1) .. controls +(-0.1cm, 1.1cm) and +( 0.1cm, 1.1cm) .. (3) {};
        \draw[thick,green]       (3) .. controls +( 0.1cm,-0.7cm) and +(-0.1cm,-0.7cm) .. (4) {};
        \draw[thick,green]       (4) .. controls +(-0.1cm, 0.9cm) and +( 0.1cm, 0.9cm) .. (5) {};
        \draw[dashed,thick,teal] (1) .. controls +( 0.0cm, 1.0cm) and +( 0.0cm, 1.0cm) .. (3) {};
        \draw[dashed,thick,teal] (3) .. controls +( 0.0cm,-1.0cm) and +( 0.0cm,-1.0cm) .. (5) {};
        \draw[dotted,thick,orange]      (1) .. controls +( 0.0cm,-1.0cm) and +( 0.0cm,-1.0cm) .. (2) {};
        \draw[dotted,thick,orange]      (2) .. controls +( 0.0cm, 0.8cm) and +(-0.1cm, 0.8cm) .. (3) {};
        \draw[dotted,thick,orange]      (3) .. controls +(-0.1cm,-1.1cm) and +( 0.1cm,-1.1cm) .. (5) {};
	\end{tikzpicture}
    \hspace{.2\linewidth}
    \begin{tikzpicture}[scale=0.85]
        \foreach \x in {1,2,...,5}{
            \pgfmathsetmacro\y{int(\x-1)}
            \node [draw,circle] (\x) at (\x,0) {\y};
        }
        \draw[thick,green]       (1) .. controls +(-0.1cm, 1.1cm) and +( 0.1cm, 1.1cm) .. (3) {};
        \draw[dotted,thick,orange]       (3) .. controls +( 0.1cm,-0.7cm) and +(-0.1cm,-0.7cm) .. (4) {};
        \draw[dotted,thick,orange]       (4) .. controls +(-0.1cm, 0.9cm) and +( 0.1cm, 0.9cm) .. (5) {};
        \draw[dashed,thick,teal] (1) .. controls +( 0.0cm, 1.0cm) and +( 0.0cm, 1.0cm) .. (3) {};
        \draw[dashed,thick,teal] (3) .. controls +( 0.0cm,-1.0cm) and +( 0.0cm,-1.0cm) .. (5) {};
        \draw[dotted,thick,orange]      (1) .. controls +( 0.0cm,-1.0cm) and +( 0.0cm,-1.0cm) .. (2) {};
        \draw[dotted,thick,orange]      (2) .. controls +( 0.0cm, 0.8cm) and +(-0.1cm, 0.8cm) .. (3) {};
        \draw[thick,green]      (3) .. controls +(-0.1cm,-1.1cm) and +( 0.1cm,-1.1cm) .. (5) {};
	\end{tikzpicture}
    \caption{
        Left: A collection of $3$ stacked arc diagrams.
        Right: The unique collection of $3$ noncrossing and noninterfering arc diagrams using the same multiset of arcs.
        }
    \label{fig:non-bijection-of-faces}
\end{figure}
\end{example}

\begin{problem}
    Find an explicit dimension-preserving bijection between the faces of the $c$-dimensional mutoperhedron $\doppel_c$ and the $c$-dimensional permutohedron $\Pi_{c+1}$.
\end{problem}

\begin{example} \label{ex:other realization doppel 3}
    Let $P$ be the polytope inside the hyperplane $x_1+x_2+x_3+x_4=26$ of $\RR^4$ defined by inequalities
    \begin{equation}
        \label{eq:submodular-like}
        \sum_{i \in I} x_i \leq z_I,
    \end{equation}
    for the following pairs $(I,z_I)$:
    \[
        \begin{array}{cccccc}
            & (\{1\}, 11) & (\{2\}, 12) & (\{3\}, 12) & (\{4\}, 11) \\
            (\{1,2\}, 18) & (\{1,3\}, 18) & (\{1,4\}, 19) & (\{2,3\}, 19) & (\{2,4\}, 18) & (\{3,4\}, 18) \\
            & (\{1,2,3\}, 24) & (\{1,2,4\}, 25) & (\{1,3,4\}, 25) & (\{2,3,4\}, 24)
        \end{array}
    \]
    The polytope $P$ is combinatorially isomorphic to the $3$-dimensional mutoperhedron $\doppel_3$ shown in \cref{fig:doppel3}.
    Its vertices are:
    \begin{align*}
        & (6, 12, 6, 2), (11, 1, 7, 7), (11, 7, 1, 7), (11, 7, 6, 2),
        (11, 6, 1, 8), (11, 1, 6, 8),\\
        &(11, 6, 7, 2), (7, 11, 1, 7),
        (2, 12, 6, 6), (2, 6, 12, 6), (6, 12, 2, 6), (6, 6, 12, 2),\\
        &(2, 12, 7, 5), (5, 12, 7, 2), (2, 7, 12, 5), (5, 7, 12, 2),
        (7, 1, 11, 7), (6, 2, 12, 6), \\
        &(7, 1, 7, 11), (2, 6, 7, 11),
        (2, 7, 6, 11), (7, 7, 1, 11), (8, 6, 1, 11), (8, 1, 6, 11).
    \end{align*}
\end{example}

\begin{problem}
    Like $\Pi_{c+1}$ and $\doppel_3$ in \cref{ex:other realization doppel 3}, can $\doppel_c$ be realized inside $\RR^{c+1}$ with inequalities of the form \eqref{eq:submodular-like}?
    That is, are there constants $\set{z_I \in \RR}{\varnothing \subsetneq I \subseteq [n+1]}$ such that the polytope in $\RR^{c+1}$ satisfying the equality
    \[
        \sum_{i \in [c+1]} x_i = z_{[c+1]},
    \]
    and the inequalities in \eqref{eq:submodular-like}
    for $\varnothing \subsetneq I \subsetneq [c+1]$, is combinatorially isomorphic to $\doppel_c$? 
\end{problem}

\bibliographystyle{plain}
\bibliography{references.bib}
\label{s:biblio}

\end{document}

%% file: figures/master_example_acyclic_graph_quiver_smaller.tex
\begin{tikzpicture}[-stealth,scale=.75]
    \node[circle, draw, inner sep=0pt, minimum size=4pt,fill] (u) at (1,1) {};
    \node[circle, draw, inner sep=0pt, minimum size=4pt,fill] (v) at (2,2) {};

    \node[left] at (0.8,1) {$u$};
    \node[left] at (1.8,2) {$v$};

    \draw[red,dashed,very thick] (0,2) -- (u);
    \draw[red,dashed,very thick] (u) -- (2,0);
    \draw[red,dashed,very thick] (1,3) -- (v);
    \draw[red,dashed,very thick] (v) -- (3,1);

    \draw[blue,very thick] (0,0) -- (u);
    \draw[blue,very thick] (u) -- (v);
    \draw[blue,very thick] (v) -- (3,3);
\end{tikzpicture}

%% file: figures/master_example_routes.tex
\begin{tikzpicture}[-stealth,scale=.45]

    \begin{scope}
        \node at (0,3) {\scriptsize $R_1$};
        \node[circle, draw, inner sep=0pt, minimum size=4pt,fill] (u) at (1,1) {};
        \node[circle, draw, inner sep=0pt, minimum size=4pt,fill] (v) at (2,2) {};
        \draw[blue,very thick] (0,0) -- (u);
        \draw[red,dashed,very thick] (u) -- (2,0);
        \draw[gray!70!white,dashed] (0,2) -- (u);
        \draw[gray!70!white] (u) -- (v);
        \draw[gray!70!white,dashed] (v) -- (3,1);
        \draw[gray!70!white,dashed] (1,3) -- (v);
        \draw[gray!70!white] (v) -- (3,3);
    \end{scope}

    \begin{scope}[shift={(04,0)}]
        \node at (0,3) {\scriptsize $R_2$};
        \node[circle, draw, inner sep=0pt, minimum size=4pt,fill] (u) at (1,1) {};
        \node[circle, draw, inner sep=0pt, minimum size=4pt,fill] (v) at (2,2) {};
        \draw[blue,very thick] (0,0) -- (u);
        \draw[gray!70!white,dashed] (u) -- (2,0);
        \draw[gray!70!white,dashed] (0,2) -- (u);
        \draw[blue,very thick] (u) -- (v);
        \draw[red,dashed,very thick] (v) -- (3,1);
        \draw[gray!70!white,dashed] (1,3) -- (v);
        \draw[gray!70!white] (v) -- (3,3);
    \end{scope}

    \begin{scope}[shift={(08,0)}]
        \node at (0,3) {\scriptsize $R_3$};
        \node[circle, draw, inner sep=0pt, minimum size=4pt,fill] (u) at (1,1) {};
        \node[circle, draw, inner sep=0pt, minimum size=4pt,fill] (v) at (2,2) {};
        \draw[gray!70!white] (0,0) -- (u);
        \draw[gray!70!white,dashed] (u) -- (2,0);
        \draw[red,dashed,very thick] (0,2) -- (u);
        \draw[blue,very thick] (u) -- (v);
        \draw[red,dashed,very thick] (v) -- (3,1);
        \draw[gray!70!white,dashed] (1,3) -- (v);
        \draw[gray!70!white] (v) -- (3,3);
    \end{scope}

    \begin{scope}[shift={(12,0)}]
        \node at (0,3) {\scriptsize $R_4$};
        \node[circle, draw, inner sep=0pt, minimum size=4pt,fill] (u) at (1,1) {};
        \node[circle, draw, inner sep=0pt, minimum size=4pt,fill] (v) at (2,2) {};
        \draw[gray!70!white] (0,0) -- (u);
        \draw[gray!70!white,dashed] (u) -- (2,0);
        \draw[red,dashed,very thick] (0,2) -- (u);
        \draw[blue,very thick] (u) -- (v);
        \draw[gray!70!white,dashed] (v) -- (3,1);
        \draw[gray!70!white,dashed] (1,3) -- (v);
        \draw[blue,very thick] (v) -- (3,3);
    \end{scope}

    \begin{scope}[shift={(16,0)}]
        \node at (0,3) {\scriptsize $R_5$};
        \node[circle, draw, inner sep=0pt, minimum size=4pt,fill] (u) at (1,1) {};
        \node[circle, draw, inner sep=0pt, minimum size=4pt,fill] (v) at (2,2) {};
        \draw[gray!70!white] (0,0) -- (u);
        \draw[gray!70!white,dashed] (u) -- (2,0);
        \draw[gray!70!white,dashed] (0,2) -- (u);
        \draw[gray!70!white] (u) -- (v);
        \draw[gray!70!white,dashed] (v) -- (3,1);
        \draw[red,dashed,very thick] (1,3) -- (v);
        \draw[blue,very thick] (v) -- (3,3);
    \end{scope}

    \begin{scope}[shift={(04,-4)}]
        \node at (0,3) {\scriptsize $E_1$};
        \node[circle, draw, inner sep=0pt, minimum size=4pt,fill] (u) at (1,1) {};
        \node[circle, draw, inner sep=0pt, minimum size=4pt,fill] (v) at (2,2) {};
        \draw[blue,very thick] (0,0) -- (u);
        \draw[gray!70!white,dashed] (u) -- (2,0);
        \draw[gray!70!white,dashed] (0,2) -- (u);
        \draw[blue,very thick] (u) -- (v);
        \draw[gray!70!white,dashed] (v) -- (3,1);
        \draw[gray!70!white,dashed] (1,3) -- (v);
        \draw[blue,very thick] (v) -- (3,3);
    \end{scope}

    \begin{scope}[shift={(08,-4)}]
        \node at (0,3) {\scriptsize $E_2$};
        \node[circle, draw, inner sep=0pt, minimum size=4pt,fill] (u) at (1,1) {};
        \node[circle, draw, inner sep=0pt, minimum size=4pt,fill] (v) at (2,2) {};
        \draw[gray!70!white] (0,0) -- (u);
        \draw[red,dashed,very thick] (u) -- (2,0);
        \draw[red,dashed,very thick] (0,2) -- (u);
        \draw[gray!70!white] (u) -- (v);
        \draw[gray!70!white,dashed] (v) -- (3,1);
        \draw[gray!70!white,dashed] (1,3) -- (v);
        \draw[gray!70!white] (v) -- (3,3);
    \end{scope}

    \begin{scope}[shift={(12,-4)}]
        \node at (0,3) {\scriptsize $E_3$};
        \node[circle, draw, inner sep=0pt, minimum size=4pt,fill] (u) at (1,1) {};
        \node[circle, draw, inner sep=0pt, minimum size=4pt,fill] (v) at (2,2) {};
        \draw[gray!70!white] (0,0) -- (u);
        \draw[gray!70!white,dashed] (u) -- (2,0);
        \draw[gray!70!white,dashed] (0,2) -- (u);
        \draw[gray!70!white] (u) -- (v);
        \draw[red,dashed,very thick] (v) -- (3,1);
        \draw[red,dashed,very thick] (1,3) -- (v);
        \draw[gray!70!white] (v) -- (3,3);
    \end{scope}

\end{tikzpicture}

%% file: figures/volume_integer_flows.tex
\begin{tikzpicture}[-stealth,scale=.5]

    \begin{scope}
        \draw[gray!70!white] (0,0) -- (1,1);
        \draw[thick] (1,1) -- (2,0) node [red, midway, fill = white, inner sep = .3pt] {\scriptsize $1$};
        \draw[gray!70!white] (0,2) -- (1,1);
        \draw[gray!70!white] (1,1) -- (2,2);
        \draw[thick] (2,2) -- (3,1) node [red, midway, fill = white, inner sep = .3pt] {\scriptsize $1$};
        \draw[gray!70!white] (1,3) -- (2,2);
        \draw[gray!70!white] (2,2) -- (3,3);
        \node at (1.5,-1) {$R_1,R_2$};
    \end{scope}

    \begin{scope}[shift={(06,0)}]
        \draw[gray!70!white] (0,0) -- (1,1);
        \draw[gray!70!white] (1,1) -- (2,0);
        \draw[gray!70!white] (0,2) -- (1,1);
        \draw[thick] (1,1) -- (2,2) node [red, midway, fill = white, inner sep = .3pt] {\scriptsize $1$};
        \draw[thick] (2,2) -- (3,1) node [red, midway, fill = white, inner sep = .3pt] {\scriptsize $2$};
        \draw[gray!70!white] (1,3) -- (2,2);
        \draw[gray!70!white] (2,2) -- (3,3);
        \node at (1.5,-1) {$R_2,R_3$};
    \end{scope}

    \begin{scope}[shift={(12,0)}]
        \draw[gray!70!white] (0,0) -- (1,1);
        \draw[gray!70!white] (1,1) -- (2,0);
        \draw[gray!70!white] (0,2) -- (1,1);
        \draw[thick] (1,1) -- (2,2) node [red, midway, fill = white, inner sep = .3pt] {\scriptsize $1$};
        \draw[thick] (2,2) -- (3,1) node [red, midway, fill = white, inner sep = .3pt] {\scriptsize $1$};
        \draw[gray!70!white] (1,3) -- (2,2);
        \draw[thick] (2,2) -- (3,3) node [red, midway, fill = white, inner sep = .3pt] {\scriptsize $1$};
        \node at (1.5,-1) {$R_3,R_4$};
    \end{scope}

    \begin{scope}[shift={(18,0)}]
        \draw[gray!70!white] (0,0) -- (1,1);
        \draw[gray!70!white] (1,1) -- (2,0);
        \draw[gray!70!white] (0,2) -- (1,1);
        \draw[thick] (1,1) -- (2,2) node [red, midway, fill = white, inner sep = .3pt] {\scriptsize $1$};
        \draw[gray!70!white] (2,2) -- (3,1);
        \draw[gray!70!white] (1,3) -- (2,2);
        \draw[thick] (2,2) -- (3,3) node [red, midway, fill = white, inner sep = .3pt] {\scriptsize $2$};
        \node at (1.5,-1) {$R_4,R_5$};
    \end{scope}

    \begin{scope}[shift={(24,0)}]
        \draw[gray!70!white] (0,0) -- (1,1);
        \draw[thick] (1,1) -- (2,0) node [red, midway, fill = white, inner sep = .3pt] {\scriptsize $1$};
        \draw[gray!70!white] (0,2) -- (1,1);
        \draw[gray!70!white] (1,1) -- (2,2);
        \draw[gray!70!white] (2,2) -- (3,1);
        \draw[gray!70!white] (1,3) -- (2,2);
        \draw[thick] (2,2) -- (3,3) node [red, midway, fill = white, inner sep = .3pt] {\scriptsize $1$};
        \node at (1.5,-1) {$R_1,R_5$};
    \end{scope}

\end{tikzpicture}

%% file: figures/quiver_grid_new_convention_a.tex
\begin{tikzpicture}[-stealth,scale=.8]
    \node[circle, draw, inner sep=0pt, minimum size=4pt,fill] (d) at (2,1) {};
    \node[circle, draw, inner sep=0pt, minimum size=4pt,fill] (l) at (1,2) {};
    \node[circle, draw, inner sep=0pt, minimum size=4pt,fill] (r) at (3,2) {};
    \node[circle, draw, inner sep=0pt, minimum size=4pt,fill] (u) at (2,3) {};

    \draw[red,dashed,thick] (0,3) -- (l) node[midway,left] {\scriptsize$1$};
    \draw[red,dashed,thick] (l) -- (d) node[midway,left] {\scriptsize$1$};
    \draw[red,dashed,thick] (d) -- (3,0) node[midway,left] {\scriptsize$1$};
    \draw[red,dashed,thick] (1,4) -- (u) node[midway,left] {\scriptsize$1$};
    \draw[red,dashed,thick] (u) -- (r) node[midway,left] {\scriptsize$1$};
    \draw[red,dashed,thick] (r) -- (4,1) node[midway,left] {\scriptsize$1$};

    \draw[blue,thick] (1,0) -- (d) node[midway,left] {\scriptsize$2$};
    \draw[blue,thick] (d) -- (r) node[midway,left] {\scriptsize$2$};
    \draw[blue,thick] (r) -- (4,3) node[midway,left] {\scriptsize$2$};
    \draw[blue,thick] (0,1) -- (l) node[midway,left] {\scriptsize$2$};
    \draw[blue,thick] (l) -- (u) node[midway,left] {\scriptsize$2$};
    \draw[blue,thick] (u) -- (3,4) node[midway,left] {\scriptsize$2$};
\end{tikzpicture}

%% file: figures/quiver_grid_new_convention_b.tex
\begin{tikzpicture}[-stealth,scale=.8]
    \node[circle, draw, inner sep=0pt, minimum size=4pt,fill] (d) at (2,1) {};
    \node[circle, draw, inner sep=0pt, minimum size=4pt,fill] (l) at (1,2) {};
    \node[circle, draw, inner sep=0pt, minimum size=4pt,fill] (r) at (3,2) {};
    \node[circle, draw, inner sep=0pt, minimum size=4pt,fill] (u) at (2,3) {};

    \draw[thick] (0,3) -- (l);
    \draw[thick] (l) -- (d);
    \draw[thick] (d) -- (3,0);
    \draw[thick] (1,4) -- (u);
    \draw[thick] (u) -- (r);
    \draw[thick] (r) -- (4,1);

    \draw[thick] (d) -- (1,0);
    \draw[thick] (r) -- (d);
    \draw[thick] (4,3) -- (r);
    \draw[thick] (l) -- (0,1);
    \draw[thick] (u) -- (l);
    \draw[thick] (3,4) -- (u);

    \begin{scope}[shift={(2,1)}]
        \draw[thick,red,-,dotted] (-45:.4) arc (-45: 45:.4);
        \draw[thick,red,-,dotted] (135:.4) arc (135:225:.4);
    \end{scope}

    \begin{scope}[shift={(1,2)}]
        \draw[thick,red,-,dotted] (-45:.4) arc (-45: 45:.4);
        \draw[thick,red,-,dotted] (135:.4) arc (135:225:.4);
    \end{scope}

    \begin{scope}[shift={(3,2)}]
        \draw[thick,red,-,dotted] (-45:.4) arc (-45: 45:.4);
        \draw[thick,red,-,dotted] (135:.4) arc (135:225:.4);
    \end{scope}

    \begin{scope}[shift={(2,3)}]
        \draw[thick,red,-,dotted] (-45:.4) arc (-45: 45:.4);
        \draw[thick,red,-,dotted] (135:.4) arc (135:225:.4);
    \end{scope}
    
\end{tikzpicture}

%% file: figures/quiver_grid_new_convention_c.tex
\begin{tikzpicture}[-stealth,scale=.8]
    \node at (2,.15) {};
    \node[circle, draw, inner sep=0pt, minimum size=4pt,fill] (d) at (2,1) {};
    \node[circle, draw, inner sep=0pt, minimum size=4pt,fill] (l) at (1,2) {};
    \node[circle, draw, inner sep=0pt, minimum size=4pt,fill] (r) at (3,2) {};
    \node[circle, draw, inner sep=0pt, minimum size=4pt,fill] (u) at (2,3) {};

    \draw[thick] (l) -- (d);
    \draw[thick] (u) -- (r);

    \draw[thick] (r) -- (d);
    \draw[thick] (u) -- (l);

    \begin{scope}[shift={(1,2)}]
        \draw[thick,red,-,dotted] (-45:.4) arc (-45: 45:.4);
    \end{scope}

    \begin{scope}[shift={(3,2)}]
        \draw[thick,red,-,dotted] (135:.4) arc (135:225:.4);
    \end{scope}
   
\end{tikzpicture}

%% file: figures/master_example_acyclic_graph_quiver.tex
G = \;\; \begin{gathered}\begin{tikzpicture}[-stealth,scale=.75]
    \node[circle, draw, inner sep=0pt, minimum size=4pt,fill] (u) at (1,1) {};
    \node[circle, draw, inner sep=0pt, minimum size=4pt,fill] (v) at (2,2) {};

    \node[left] at (0.8,1) {$u$};
    \node[left] at (1.8,2) {$v$};

    \draw[red,dashed,very thick] (0,2) -- (u);
    \draw[red,dashed,very thick] (u) -- (2,0);
    \draw[red,dashed,very thick] (1,3) -- (v);
    \draw[red,dashed,very thick] (v) -- (3,1);

    \draw[blue,very thick] (0,0) -- (u);
    \draw[blue,very thick] (u) -- (v);
    \draw[blue,very thick] (v) -- (3,3);
\end{tikzpicture}\end{gathered}
\hspace*{.1\linewidth}
(Q_G,R_G) = \begin{gathered}\begin{tikzpicture}[-stealth,scale=.75]
    \node[circle, draw, inner sep=0pt, minimum size=4pt,fill] (u) at (1,1) {};
    \node[circle, draw, inner sep=0pt, minimum size=4pt,fill] (v) at (2,2) {};

    \node[left] at (1,1) {$u$};
    \node[left] at (2,2) {$v$};

    \draw[very thick] (v) -- (u);
\end{tikzpicture}\end{gathered}

%% file: figures/master_example_bijection_routes.tex
\begin{tikzpicture}[-stealth,scale=.5]

    \begin{scope}
        \node at (0,3) {$R_1$};
        \node[circle, draw, inner sep=0pt, minimum size=4pt,fill] (u) at (1,1) {};
        \node[circle, draw, inner sep=0pt, minimum size=4pt,fill] (v) at (2,2) {};
        \draw[blue,very thick] (0,0) -- (u);
        \draw[red,dashed,very thick] (u) -- (2,0);
        \draw[gray!70!white,dashed] (0,2) -- (u);
        \draw[gray!70!white] (u) -- (v);
        \draw[gray!70!white,dashed] (v) -- (3,1);
        \draw[gray!70!white,dashed] (1,3) -- (v);
        \draw[gray!70!white] (v) -- (3,3);
        \draw[|->] (1.5,-1) -- (1.5,-2);
        \node[inner sep=1pt] (rep_u) at (0.7,-4.3) {$\KK$};
        \node[inner sep=1pt] (rep_v) at (2.3,-2.7) {$0$};
        \draw[thick] (rep_v) -- (rep_u);
    \end{scope}

    \begin{scope}[shift={(06,0)}]
        \node at (0,3) {$R_2$};
        \node[circle, draw, inner sep=0pt, minimum size=4pt,fill] (u) at (1,1) {};
        \node[circle, draw, inner sep=0pt, minimum size=4pt,fill] (v) at (2,2) {};
        \draw[blue,very thick] (0,0) -- (u);
        \draw[gray!70!white,dashed] (u) -- (2,0);
        \draw[gray!70!white,dashed] (0,2) -- (u);
        \draw[blue,very thick] (u) -- (v);
        \draw[red,dashed,very thick] (v) -- (3,1);
        \draw[gray!70!white,dashed] (1,3) -- (v);
        \draw[gray!70!white] (v) -- (3,3);
        \draw[|->] (1.5,-1) -- (1.5,-2);
        \node[inner sep=1pt] (rep_u) at (0.7,-4.3) {$\KK$};
        \node[inner sep=1pt] (rep_v) at (2.3,-2.7) {$\KK$};
        \draw[thick] (rep_v) -- (rep_u);
    \end{scope}

    \begin{scope}[shift={(12,0)}]
        \node at (0,3) {$R_3$};
        \node[circle, draw, inner sep=0pt, minimum size=4pt,fill] (u) at (1,1) {};
        \node[circle, draw, inner sep=0pt, minimum size=4pt,fill] (v) at (2,2) {};
        \draw[gray!70!white] (0,0) -- (u);
        \draw[gray!70!white,dashed] (u) -- (2,0);
        \draw[red,dashed,very thick] (0,2) -- (u);
        \draw[blue,very thick] (u) -- (v);
        \draw[red,dashed,very thick] (v) -- (3,1);
        \draw[gray!70!white,dashed] (1,3) -- (v);
        \draw[gray!70!white] (v) -- (3,3);
        \draw[|->] (1.5,-1) -- (1.5,-2);
        \node[inner sep=1pt] (rep_u) at (0.7,-4.3) {$0$};
        \node[inner sep=1pt] (rep_v) at (2.3,-2.7) {$\KK$};
        \draw[thick] (rep_v) -- (rep_u);
    \end{scope}

    \begin{scope}[shift={(18,0)}]
        \node at (0,3) {$R_4$};
        \node[circle, draw, inner sep=0pt, minimum size=4pt,fill] (u) at (1,1) {};
        \node[circle, draw, inner sep=0pt, minimum size=4pt,fill] (v) at (2,2) {};
        \draw[gray!70!white] (0,0) -- (u);
        \draw[gray!70!white,dashed] (u) -- (2,0);
        \draw[red,dashed,very thick] (0,2) -- (u);
        \draw[blue,very thick] (u) -- (v);
        \draw[gray!70!white,dashed] (v) -- (3,1);
        \draw[gray!70!white,dashed] (1,3) -- (v);
        \draw[blue,very thick] (v) -- (3,3);
        \draw[|->] (1.5,-1) -- (1.5,-2);
        \node at (1.5,-3.5) {$P_u[1]$};
    \end{scope}

    \begin{scope}[shift={(24,0)}]
        \node at (0,3) {$R_5$};
        \node[circle, draw, inner sep=0pt, minimum size=4pt,fill] (u) at (1,1) {};
        \node[circle, draw, inner sep=0pt, minimum size=4pt,fill] (v) at (2,2) {};
        \draw[gray!70!white] (0,0) -- (u);
        \draw[gray!70!white,dashed] (u) -- (2,0);
        \draw[gray!70!white,dashed] (0,2) -- (u);
        \draw[gray!70!white] (u) -- (v);
        \draw[gray!70!white,dashed] (v) -- (3,1);
        \draw[red,dashed,very thick] (1,3) -- (v);
        \draw[blue,very thick] (v) -- (3,3);
        \draw[|->] (1.5,-1) -- (1.5,-2);
        \node at (1.5,-3.5) {$P_v[1]$};
    \end{scope}

\end{tikzpicture}

%% file: figures/sq_pyramid.tex
\begin{tikzpicture}%
	[x={(-0.768853cm, -0.260038cm)},
	y={(0.639425cm, -0.312731cm)},
	z={(0.000041cm, 0.913553cm)},
	scale=2.5,
	back/.style={loosely dotted, thin},
	edge/.style={color=blue!95!black, thick},
	facet/.style={fill=blue!95!black,fill opacity=.4},
	vertex/.style={inner sep=1pt,circle,draw=green!25!black,fill=green!75!black,thick}]
\coordinate (0, 0, 0) at (0, 0, 0);
\coordinate (0, 0, 1) at (0, 0, 1);
\coordinate (1, 0, 0) at (1, 0, 0);
\coordinate (1, 0, 1) at (1, 0, 1);
\coordinate (1, 1, 1) at (1, 1, 1);
\fill[facet] (1, 1, 1) -- (0, 0, 0) -- (0, 0, 1) -- cycle {};
\fill[facet] (1, 1, 1) -- (0, 0, 0) -- (1, 0, 0) -- cycle {};
\fill[facet] (1, 1, 1) -- (1, 0, 0) -- (1, 0, 1) -- cycle {};
\fill[facet] (1, 1, 1) -- (0, 0, 1) -- (1, 0, 1) -- cycle {};
\draw[edge] (0, 0, 0) -- (0, 0, 1);
\draw[edge] (0, 0, 0) -- (1, 0, 0);
\draw[edge] (0, 0, 0) -- (1, 1, 1);
\draw[edge] (0, 0, 1) -- (1, 0, 1);
\draw[edge] (0, 0, 1) -- (1, 1, 1);
\draw[edge] (1, 0, 0) -- (1, 0, 1);
\draw[edge] (1, 0, 0) -- (1, 1, 1);
\draw[edge] (1, 0, 1) -- (1, 1, 1);
\node[vertex] at (0, 0, 0) {}; \node[right] at (0, 0, 0) {$(0, 0, 0)$};
\node[vertex] at (0, 0, 1) {}; \node[right] at (0, 0, 1) {$(0, 0, 1)$};
\node[vertex] at (1, 0, 0) {}; \node[left]  at (1, 0, 0) {$(1, 0, 0)$};
\node[vertex] at (1, 0, 1) {}; \node[left]  at (1, 0, 1) {$(1, 0, 1)$};
\node[vertex] at (1, 1, 1) {}; \node[right] at (0, 0, .4) {$(1, 1, 1)$};
\end{tikzpicture}

%% file: figures/n-cycle.tex
\begin{tikzpicture}[-stealth,scale=.8]
    \def\r{2cm}
    \def\e{3}
    \def\steps{7}
    \node[circle, draw, inner sep=0pt, minimum size=4pt,fill] (A) at (90-360/\steps*0:\r) {};
    \node[circle, draw, inner sep=0pt, minimum size=4pt,fill] (B) at (90-360/\steps*1:\r) {};
    \node[circle, draw, inner sep=0pt, minimum size=4pt,fill] (C) at (90-360/\steps*2:\r) {};
    \node[circle, draw, inner sep=0pt, minimum size=4pt,fill] (D) at (90-360/\steps*3:\r) {};
    \node[circle, draw, inner sep=0pt, minimum size=4pt,fill] (N) at (90+360/\steps*1:\r) {};

    \node at (90-360/\steps*0-7:1.15*\r) {$1$};
    \node at (90-360/\steps*1-7:1.15*\r) {$2$};
    \node at (90-360/\steps*2-7:1.15*\r) {$3$};
    \node at (90-360/\steps*3-7:1.15*\r) {$4$};
    \node at (90+360/\steps*1-7:1.15*\r) {$n$};

    \draw[red,dashed,thick] (90+360/\steps*1-\e:\r) arc (90+360/\steps*1-\e:90-360/\steps*0+\e:\r);
    \draw[red,dashed,thick] (90-360/\steps*0-\e:\r) arc (90-360/\steps*0-\e:90-360/\steps*1+\e:\r);
    \draw[red,dashed,thick] (90-360/\steps*1-\e:\r) arc (90-360/\steps*1-\e:90-360/\steps*2+\e:\r);
    \draw[red,dashed,thick] (90-360/\steps*2-\e:\r) arc (90-360/\steps*2-\e:90-360/\steps*3+\e:\r);
    \draw[red,-,dotted]     (90-360/\steps*3-\e:\r) arc (90-360/\steps*3-\e:90-360/\steps*(\steps-1)+\e:\r);

    \draw[blue,thick] (90-360/\steps*0:1.5*\r) -- (A); \draw[blue,thick] (A) -- (90-360/\steps*0:0.5*\r);
    \draw[blue,thick] (90-360/\steps*1:1.5*\r) -- (B); \draw[blue,thick] (B) -- (90-360/\steps*1:0.5*\r);
    \draw[blue,thick] (90-360/\steps*2:1.5*\r) -- (C); \draw[blue,thick] (C) -- (90-360/\steps*2:0.5*\r);
    \draw[blue,thick] (90-360/\steps*3:1.5*\r) -- (D); \draw[blue,thick] (D) -- (90-360/\steps*3:0.5*\r);
    \draw[blue,thick] (90+360/\steps*1:1.5*\r) -- (N); \draw[blue,thick] (N) -- (90+360/\steps*1:0.5*\r);
\end{tikzpicture}

%% file: figures/ex_good_route.tex
\begin{tikzpicture}[-stealth,scale=.8]
    \def\r{2cm}
    \def\e{3}
    \def\steps{6} 

    \foreach \x in {1,...,\steps}{
        \draw[gray!70,dashed]
            (90+360/\steps-360/\steps*\x:\r) arc (90+360/\steps-360/\steps*\x:90-360/\steps*\x+\e:\r);
        \node[circle, draw, inner sep=0pt, minimum size=4pt,fill] (N\x) at (90+360/\steps-360/\steps*\x:\r) {};
        \node at (90+360/\steps-360/\steps*\x-7:1.15*\r) {$\x$};
        \draw[gray!70]
            (90+360/\steps-360/\steps*\x:1.5*\r) -- (N\x);
        \draw[gray!70]
            (N\x) -- (90+360/\steps-360/\steps*\x:0.5*\r);
    }

    \draw[blue,very thick] 
        (90+360/\steps-360/\steps*0:1.5*\r) -- (90+360/\steps-360/\steps*0:\r);
        
    \draw[red,dashed,very thick] (90+360/\steps-360/\steps*0-\e:\r) arc (90+360/\steps-360/\steps*0-\e:90-360/\steps*0+\e:\r);

    \draw[red,dashed,very thick] (90+360/\steps-360/\steps*1-\e:\r) arc (90+360/\steps-360/\steps*1-\e:90-360/\steps*1+\e:\r);
    
    \draw[blue,very thick] (90+360/\steps-360/\steps*2:\r) -- (90+360/\steps-360/\steps*2:0.5*\r);
\end{tikzpicture}

%% file: figures/ex_bad_route.tex
\begin{tikzpicture}[-stealth,scale=.8]
    \def\r{2cm}
    \def\e{3}
    \def\steps{6}

    \foreach \x in {1,...,\steps}{
        \node[circle, draw, inner sep=0pt, minimum size=4pt,fill] (N\x) at (90+360/\steps-360/\steps*\x:\r) {};
        \node at (90+360/\steps-360/\steps*\x-7:1.15*\r) {$\x$};
        \draw[gray!70]
            (90+360/\steps-360/\steps*\x:1.5*\r) -- (N\x);
        \draw[gray!70]
            (N\x) -- (90+360/\steps-360/\steps*\x:0.5*\r);
    }

    \draw[blue,very thick] 
        (90-360/\steps*0:1.5*\r) -- (90-360/\steps*0:1.08*\r);
    \draw[blue, very thick]
        (90-360/\steps*2:.92*\r) -- (90-360/\steps*2:0.5*\r);
    \draw[red,dashed,very thick]
            (90+360/\steps-360/\steps*1-\e:1.08*\r) arc (90+360/\steps-360/\steps*1-\e:90-360/\steps*1+\e:1.06*\r);
    \draw[red,dashed,very thick]
            (90+360/\steps-360/\steps*2-\e:1.06*\r) arc (90+360/\steps-360/\steps*2-\e:90-360/\steps*2+\e:1.04*\r);
    \draw[red,dashed,very thick]
            (90+360/\steps-360/\steps*3-\e:1.04*\r) arc (90+360/\steps-360/\steps*3-\e:90-360/\steps*3+\e:1.02*\r);
    \draw[red,dashed,very thick]
            (90+360/\steps-360/\steps*4-\e:1.02*\r) arc (90+360/\steps-360/\steps*4-\e:90-360/\steps*4+\e:1*\r);
    \draw[red,dashed,very thick]
            (90+360/\steps-360/\steps*5-\e:1*\r) arc (90+360/\steps-360/\steps*5:90-\e-360/\steps*5+1.3*\e:.95*\r);
    \draw[red,dashed,very thick]
            (90+360/\steps-360/\steps*0-\e:.98*\r) arc (90+360/\steps-360/\steps*0-\e:90-360/\steps*0+\e:.96*\r);
    \draw[red,dashed,very thick]
            (90+360/\steps-360/\steps*1-\e:.96*\r) arc (90+360/\steps-360/\steps*1-\e:90-360/\steps*1+\e:.94*\r);
    \draw[red,dashed,very thick]
            (90+360/\steps-360/\steps*2-\e:.94*\r) arc (90+360/\steps-360/\steps*2-\e:90-360/\steps*2+\e:.92*\r);
\end{tikzpicture}

%% file: figures/cex_amp_fram_cyc.tex
\begin{tikzpicture}[-stealth,scale=.6]
	\def\r{1.5cm}
	
	\node[circle, draw, inner sep=0pt, minimum size=4pt,fill] (1) at (0,\r) {};
	\node[circle, draw, inner sep=0pt, minimum size=4pt,fill] (2) at (0,-\r) {};
	\node[circle, draw, inner sep=0pt, minimum size=4pt,fill] (3) at (3*\r,\r) {};
	\node[circle, draw, inner sep=0pt, minimum size=4pt,fill] (4) at (3*\r,-\r) {};
	\node[circle, draw, inner sep=0pt, minimum size=4pt,fill] (5) at (1.5*\r,2*\r) {};
	
	\draw
        (0,\r) arc (90:270-3:\r);
	\draw
        (0,-\r) arc (-90:90-3:\r);
	\draw
        (3) arc (90:270-3:\r);
	\draw
        (4) arc (-90:90-3:\r);
	
	\draw
        (0,0.5) -- (0,\r-3);
	\draw
        (3*\r,0.5) -- (3);
	\draw
        (0,-0.5) -- (2);
	\draw
        (4) -- (3*\r,-0.5);
	\draw
        (5) -- (1.4,2.2);
	\draw
        (5) -- (2.9,2.2);
	
	\draw
        (2) edge [bend right=45] (4);
	\draw
        (1) edge [in=135, out=90,looseness=1.3] (5);
	\draw
        (3) edge [in=45, out=90,looseness=1.3] (5);
\end{tikzpicture}
\qquad\qquad
\begin{tikzpicture}[-stealth,scale=.6]
	\def\r{1.5cm}
	
	\node[circle, draw, inner sep=0pt, minimum size=4pt,fill] (1) at (0,\r) {};
	\node[circle, draw, inner sep=0pt, minimum size=4pt,fill] (2) at (0,-\r) {};
	\node[circle, draw, inner sep=0pt, minimum size=4pt,fill] (3) at (3*\r,0) {};
	
	\draw
        (1) arc (90:270-3:\r);
	\draw
        (2) arc (-90:90-3:\r);
	\draw
        (0,0.5) -- (0,\r-3);
	\draw
        (1) edge [in=135, out=90,looseness=1] (3);
	\draw
        (2) edge [in=-135, out=-90,looseness=1] (3);
	\draw
        (0,-0.5) -- (2);
	\draw
        (3) -- (3.5*\r,0.5*\r);
	\draw
        (3) -- (3.5*\r,-0.5*\r);
\end{tikzpicture}

%% file: figures/4-cycle-flow.tex
\begin{gathered}
\begin{tikzpicture}[-stealth,scale= .5]
    \def\r{2cm}
    \def\e{3}
    \def\steps{4}
    \node[circle, draw, inner sep =0pt, minimum size=4pt,fill] (A) at (135-360/\steps*0:\r) {};
    \node[circle, draw, inner sep =0pt, minimum size=4pt,fill] (B) at (135-360/\steps*1:\r) {};
    \node[circle, draw, inner sep =0pt, minimum size=4pt,fill] (C) at (135-360/\steps*2:\r) {};
    \node[circle, draw, inner sep =0pt, minimum size=4pt,fill] (D) at (135-360/\steps*3:\r) {};

    \draw[thick] (135+360/\steps*1-\e:\r) arc (135+360/\steps*1-\e:135-360/\steps*0+\e:\r) node [red, midway, fill = white, inner sep = .8pt] {\scriptsize $2$};
    \draw[thick] (135-360/\steps*0-\e:\r) arc (135-360/\steps*0-\e:135-360/\steps*1+\e:\r) node [red, midway, fill = white, inner sep = .6pt] {\scriptsize $3$};
    \draw[thick] (135-360/\steps*1-\e:\r) arc (135-360/\steps*1-\e:135-360/\steps*2+\e:\r) node [red, midway, fill = white, inner sep = .8pt] {\scriptsize $4$};
    \draw[thick] (135-360/\steps*2-\e:\r) arc (135-360/\steps*2-\e:135-360/\steps*3+\e:\r) node [red, midway, fill = white, inner sep = .6pt] {\scriptsize $2$};

    \draw[thick] (135-360/\steps*0:1.5*\r) -- (A) node [red, midway, fill = white, inner sep = .3pt] {\scriptsize $1$};
    \draw[thick] (A) -- (135-360/\steps*0:0.5*\r) node [red, midway, fill = white, inner sep = .3pt] {\scriptsize $0$};
    \draw[thick] (135-360/\steps*1:1.5*\r) -- (B) node [red, midway, fill = white, inner sep = .3pt] {\scriptsize $1$};
    \draw[thick] (B) -- (135-360/\steps*1:0.5*\r) node [red, midway, fill = white, inner sep = .3pt] {\scriptsize $0$};
    \draw[thick] (135-360/\steps*2:1.5*\r) -- (C) node [red, midway, fill = white, inner sep = .3pt] {\scriptsize $0$};
    \draw[thick] (C) -- (135-360/\steps*2:0.5*\r) node [red, midway, fill = white, inner sep = .3pt] {\scriptsize $2$};
    \draw[thick] (135-360/\steps*3:1.5*\r) -- (D) node [red, midway, fill = white, inner sep = .3pt] {\scriptsize $0$};
    \draw[thick] (D) -- (135-360/\steps*3:0.5*\r) node [red, midway, fill = white, inner sep = .3pt] {\scriptsize $0$};
\end{tikzpicture}
\end{gathered}
\quad = \quad \scalebox{1.2}{2} \;
\begin{gathered}
\begin{tikzpicture}[-stealth,scale= .5]
    \def\r{2cm}
    \def\e{3}
    \def\steps{4}

    \draw[red, dashed, very thick] (135+360/\steps*1-\e:\r) arc (135+360/\steps*1-\e:135-360/\steps*0+\e:\r);
    \draw[red, dashed, very thick] (135-360/\steps*0-\e:\r) arc (135-360/\steps*0-\e:135-360/\steps*1+\e:\r);
    \draw[red, dashed, very thick] (135-360/\steps*1-\e:\r) arc (135-360/\steps*1-\e:135-360/\steps*2+\e:\r);
    \draw[red, dashed, very thick] (135-360/\steps*2-\e:\r) arc (135-360/\steps*2-\e:135-360/\steps*3+\e:\r);

    \node[circle, draw, inner sep =0pt, minimum size=4pt,fill] (A) at (135-360/\steps*0:\r) {};
    \node[circle, draw, inner sep =0pt, minimum size=4pt,fill] (B) at (135-360/\steps*1:\r) {};
    \node[circle, draw, inner sep =0pt, minimum size=4pt,fill] (C) at (135-360/\steps*2:\r) {};
    \node[circle, draw, inner sep =0pt, minimum size=4pt,fill] (D) at (135-360/\steps*3:\r) {};

    \draw[gray!70!white] (135-360/\steps*0:1.5*\r) -- (A);
    \draw[gray!70!white] (A) -- (135-360/\steps*0:0.5*\r);
    \draw[gray!70!white] (135-360/\steps*1:1.5*\r) -- (B);
    \draw[gray!70!white] (B) -- (135-360/\steps*1:0.5*\r);
    \draw[gray!70!white] (135-360/\steps*2:1.5*\r) -- (C);
    \draw[gray!70!white] (C) -- (135-360/\steps*2:0.5*\r);
    \draw[gray!70!white] (135-360/\steps*3:1.5*\r) -- (D);
    \draw[gray!70!white] (D) -- (135-360/\steps*3:0.5*\r);
\end{tikzpicture}
\end{gathered}
\quad + \quad \scalebox{1.2}{1} \;
\begin{gathered}
\begin{tikzpicture}[-stealth,scale= .5]
    \def\r{2cm}
    \def\e{3}
    \def\steps{4}

    \draw[gray!70!white,dashed] (135+360/\steps*1-\e:\r) arc (135+360/\steps*1-\e:135-360/\steps*0+\e:\r);
    \draw[gray!70!white,dashed] (135-360/\steps*2-\e:\r) arc (135-360/\steps*2-\e:135-360/\steps*3+\e:\r);

    \draw[red, dashed, very thick] (135-360/\steps*0-\e:\r) arc (135-360/\steps*0-\e:135-360/\steps*1+\e:\r);
    \draw[red, dashed, very thick] (135-360/\steps*1-\e:\r) arc (135-360/\steps*1-\e:135-360/\steps*2+\e:\r);

    \node[circle, draw, inner sep =0pt, minimum size=4pt,fill] (A) at (135-360/\steps*0:\r) {};
    \node[circle, draw, inner sep =0pt, minimum size=4pt,fill] (B) at (135-360/\steps*1:\r) {};
    \node[circle, draw, inner sep =0pt, minimum size=4pt,fill] (C) at (135-360/\steps*2:\r) {};
    \node[circle, draw, inner sep =0pt, minimum size=4pt,fill] (D) at (135-360/\steps*3:\r) {};

    \draw[blue, very thick] (135-360/\steps*0:1.5*\r) -- (A);
    \draw[gray!70!white] (A) -- (135-360/\steps*0:0.5*\r);
    \draw[gray!70!white] (135-360/\steps*1:1.5*\r) -- (B);
    \draw[gray!70!white] (B) -- (135-360/\steps*1:0.5*\r);
    \draw[gray!70!white] (135-360/\steps*2:1.5*\r) -- (C);
    \draw[blue, very thick] (C) -- (135-360/\steps*2:0.5*\r);
    \draw[gray!70!white] (135-360/\steps*3:1.5*\r) -- (D);
    \draw[gray!70!white] (D) -- (135-360/\steps*3:0.5*\r);
\end{tikzpicture}
\end{gathered}
\quad + \quad \scalebox{1.2}{1} \;
\begin{gathered}
\begin{tikzpicture}[-stealth,scale= .5]
    \def\r{2cm}
    \def\e{3}
    \def\steps{4}

    \draw[gray!70!white,dashed] (135+360/\steps*1-\e:\r) arc (135+360/\steps*1-\e:135-360/\steps*0+\e:\r);
    \draw[gray!70!white,dashed] (135-360/\steps*0-\e:\r) arc (135-360/\steps*0-\e:135-360/\steps*1+\e:\r);
    \draw[gray!70!white,dashed] (135-360/\steps*2-\e:\r) arc (135-360/\steps*2-\e:135-360/\steps*3+\e:\r);

    \draw[red, dashed, very thick] (135-360/\steps*1-\e:\r) arc (135-360/\steps*1-\e:135-360/\steps*2+\e:\r);

    \node[circle, draw, inner sep =0pt, minimum size=4pt,fill] (A) at (135-360/\steps*0:\r) {};
    \node[circle, draw, inner sep =0pt, minimum size=4pt,fill] (B) at (135-360/\steps*1:\r) {};
    \node[circle, draw, inner sep =0pt, minimum size=4pt,fill] (C) at (135-360/\steps*2:\r) {};
    \node[circle, draw, inner sep =0pt, minimum size=4pt,fill] (D) at (135-360/\steps*3:\r) {};

    \draw[gray!70!white] (135-360/\steps*0:1.5*\r) -- (A);
    \draw[gray!70!white] (A) -- (135-360/\steps*0:0.5*\r);
    \draw[blue, very thick] (135-360/\steps*1:1.5*\r) -- (B);
    \draw[gray!70!white] (B) -- (135-360/\steps*1:0.5*\r);
    \draw[gray!70!white] (135-360/\steps*2:1.5*\r) -- (C);
    \draw[blue, very thick] (C) -- (135-360/\steps*2:0.5*\r);
    \draw[gray!70!white] (135-360/\steps*3:1.5*\r) -- (D);
    \draw[gray!70!white] (D) -- (135-360/\steps*3:0.5*\r);
\end{tikzpicture}
\end{gathered}

%% file: figures/ex_clique_flows_H3_original.tex
\begin{gathered}\begin{tikzpicture}[-stealth,scale=.4]
    \def\r{2cm}
    \def\e{3}
    \def\steps{3}
    \foreach \x in {1,...,\steps}{
        \draw[gray!60!white,dashed]
            (90+360/\steps-360/\steps*\x:\r) arc (90+360/\steps-360/\steps*\x:90-360/\steps*\x+\e:\r);
        \node[circle, draw, inner sep=0pt, minimum size=4pt,fill] (N\x) at (90+360/\steps-360/\steps*\x:\r) {};
        \draw[gray!60!white]
            (90+360/\steps-360/\steps*\x:1.5*\r) -- (N\x);
        \draw[gray!60!white]
            (N\x) -- (90+360/\steps-360/\steps*\x:0.3*\r);
    }
        \draw[blue,very thick]
            (90+360/\steps-360/\steps*3:1.5*\r) -- (N3);
        \draw[red,dashed,very thick]
            (90+360/\steps-360/\steps*3:\r) arc (90+360/\steps-360/\steps*3:90-360/\steps*3+\e:\r);
        \draw[red,dashed,very thick]
            (90+360/\steps-360/\steps*1:\r) arc (90+360/\steps-360/\steps*1:90-360/\steps*1+\e:\r);
        \draw[blue,very thick]
            (N2) -- (90+360/\steps-360/\steps*2:0.3*\r);
    
    \foreach \x in {1,...,\steps}{
        \node[circle, draw, inner sep=0pt, minimum size=4pt,fill] (N\x) at (90+360/\steps-360/\steps*\x:\r) {};
    }    
\end{tikzpicture}
\hspace*{.02\linewidth}
\begin{tikzpicture}[-stealth,scale=.4]
    \def\r{2cm}
    \def\e{3}
    \def\steps{3} 

    \foreach \x in {1,...,\steps}{
        \draw[gray!60!white,dashed]
            (90+360/\steps-360/\steps*\x:\r) arc (90+360/\steps-360/\steps*\x:90-360/\steps*\x+\e:\r);
        \node[circle, draw, inner sep=0pt, minimum size=4pt,fill] (N\x) at (90+360/\steps-360/\steps*\x:\r) {};
        \draw[gray!60!white]
            (90+360/\steps-360/\steps*\x:1.5*\r) -- (N\x);
        \draw[gray!60!white]
            (N\x) -- (90+360/\steps-360/\steps*\x:0.3*\r);
    }
        \draw[blue,very thick]
            (90+360/\steps-360/\steps*1:1.5*\r) -- (N1);
        \draw[red,dashed,very thick]
            (90+360/\steps-360/\steps*1:\r) arc (90+360/\steps-360/\steps*1:90-360/\steps*1+\e:\r);
        \draw[blue,very thick]
            (N2) -- (90+360/\steps-360/\steps*2:0.3*\r);

    \foreach \x in {1,...,\steps}{
        \node[circle, draw, inner sep=0pt, minimum size=4pt,fill] (N\x) at (90+360/\steps-360/\steps*\x:\r) {};
    }    
\end{tikzpicture}\end{gathered}
    \qquad\longleftrightarrow\qquad
\begin{gathered}\begin{tikzpicture}[-stealth,scale=.4]
    \def\r{2cm}
    \def\e{3}
    \def\steps{3}

    \draw[thick] (90+360/\steps*1-\e:\r) arc (90+360/\steps*1-\e:90-360/\steps*0+\e:\r) node [red, midway, fill = white, inner sep = .8pt] {\scriptsize $1$};
    \draw[thick] (90-360/\steps*0-\e:\r) arc (90-360/\steps*0-\e:90-360/\steps*1+\e:\r) node [red, midway, fill = white, inner sep = .8pt] {\scriptsize $2$};
    \draw[gray!60!white] (90-360/\steps*1-\e:\r) arc (90-360/\steps*1-\e:90-360/\steps*2+\e:\r);

    \node[circle, draw, inner sep =0pt, minimum size=4pt,fill] (A) at (90-360/\steps*0:\r) {};
    \node[draw, inner sep =0pt, minimum size=4pt,fill] (B) at (90-360/\steps*1:\r) {};
    \node[circle, draw, inner sep =0pt, minimum size=4pt,fill] (C) at (90-360/\steps*2:\r) {};

    \draw[gray!60!white] (90-360/\steps*0:1.5*\r) -- (A);
    \draw[gray!60!white] (A) -- (90-360/\steps*0:0.3*\r);
    \draw[gray!60!white] (90-360/\steps*1:1.5*\r) -- (B);
    \draw[thick] (B) -- (90-360/\steps*1:0.3*\r) node [red, midway, fill = white, inner sep = .8pt] {\scriptsize $2$};
    \draw[gray!60!white] (90-360/\steps*2:1.5*\r) -- (C);
    \draw[gray!60!white] (C) -- (90-360/\steps*2:0.3*\r);
\end{tikzpicture}\end{gathered}

%% file: figures/ex_other_clique_H3.tex
\begin{tikzpicture}[-stealth,scale=.4]
    \def\r{2cm}
    \def\e{3}
    \def\steps{3}

    \draw[thick] (90+360/\steps*1-\e:\r) arc (90+360/\steps*1-\e:90-360/\steps*0+\e:\r) node [red, midway, fill = white, inner sep = .8pt] {\scriptsize $1$};
    \draw[thick] (90-360/\steps*0-\e:\r) arc (90-360/\steps*0-\e:90-360/\steps*1+\e:\r) node [red, midway, fill = white, inner sep = .8pt] {\scriptsize $1$};
    \draw[gray!60!white] (90-360/\steps*1-\e:\r) arc (90-360/\steps*1-\e:90-360/\steps*2+\e:\r);

    \node[circle, draw, inner sep =0pt, minimum size=4pt,fill] (A) at (90-360/\steps*0:\r) {};
    \node[draw, inner sep =0pt, minimum size=4pt,fill] (B) at (90-360/\steps*1:\r) {};
    \node[circle, draw, inner sep =0pt, minimum size=4pt,fill] (C) at (90-360/\steps*2:\r) {};

    \draw[gray!60!white] (90-360/\steps*0:1.5*\r) -- (A);
    \draw[thick] (A) -- (90-360/\steps*0:0.3*\r) node [red, midway, fill = white, inner sep = .8pt] {\scriptsize $1$};
    \draw[gray!60!white] (90-360/\steps*1:1.5*\r) -- (B);
    \draw[thick] (B) -- (90-360/\steps*1:0.3*\r) node [red, midway, fill = white, inner sep = .8pt] {\scriptsize $1$};
    \draw[gray!60!white] (90-360/\steps*2:1.5*\r) -- (C);
    \draw[gray!60!white] (C) -- (90-360/\steps*2:0.3*\r);
\end{tikzpicture}

%% file: figures/ex_nonconvex_flows.tex
\begin{tikzpicture}[-stealth,scale=.7]
    \def\r{2cm}
    \def\e{3}
    \def\steps{6}

    \foreach \x in {1,...,\steps}{
        \draw[gray!70,dashed]
            (90+360/\steps-360/\steps*\x:\r) arc (90+360/\steps-360/\steps*\x:90-360/\steps*\x+\e:\r);
        \node[circle, draw, inner sep=0pt, minimum size=4pt,fill] (N\x) at (90+360/\steps-360/\steps*\x:\r) {};
        \draw[gray!70]
            (90+360/\steps-360/\steps*\x:1.5*\r) -- (N\x);
        \draw[gray!70]
            (N\x) -- (90+360/\steps-360/\steps*\x:0.5*\r);
    }

    \draw[blue,very thick]
            (90+360/\steps-360/\steps*1:1.5*\r) -- (N1);
    \draw[red,very thick,dashed]
            (90+360/\steps-360/\steps*1-\e:\r) arc (90+360/\steps-360/\steps*1-\e:90-360/\steps*1+\e:\r);
    \draw[red,very thick,dashed]
            (90+360/\steps-360/\steps*2-\e:\r) arc (90+360/\steps-360/\steps*2-\e:90-360/\steps*2+\e:\r);
    \draw[red,very thick,dashed]
            (90+360/\steps-360/\steps*3-\e:\r) arc (90+360/\steps-360/\steps*3-\e:90-360/\steps*3+\e:\r);
    \draw[blue,very thick]
            (N4) -- (90+360/\steps-360/\steps*4:0.5*\r);
\end{tikzpicture}
\hspace{.1\linewidth}
\begin{tikzpicture}[-stealth,scale=.7]
    \def\r{2cm}
    \def\e{3}
    \def\steps{6}
    
   \foreach \x in {1,...,\steps}{
        \draw[gray!70,dashed]
            (90+360/\steps-360/\steps*\x:\r) arc (90+360/\steps-360/\steps*\x:90-360/\steps*\x+\e:\r);
        \node[circle, draw, inner sep=0pt, minimum size=4pt,fill] (N\x) at (90+360/\steps-360/\steps*\x:\r) {};
        \draw[gray!70]
            (90+360/\steps-360/\steps*\x:1.5*\r) -- (N\x);
        \draw[gray!70]
            (N\x) -- (90+360/\steps-360/\steps*\x:0.5*\r);
    }

    \draw[blue,very thick]
            (90+360/\steps-360/\steps*4:1.5*\r) -- (N4);
    \draw[red,very thick,dashed]
            (90+360/\steps-360/\steps*4-\e:\r) arc (90+360/\steps-360/\steps*4-\e:90-360/\steps*4+\e:\r);
    \draw[red,very thick,dashed]
            (90+360/\steps-360/\steps*5-\e:\r) arc (90+360/\steps-360/\steps*5-\e:90-360/\steps*5+\e:\r);
    \draw[red,very thick,dashed]
            (90+360/\steps-360/\steps*0-\e:\r) arc (90+360/\steps-360/\steps*0-\e:90-360/\steps*0+\e:\r);
    \draw[blue,very thick]
            (N1) -- (90+360/\steps-360/\steps*1:0.5*\r);
\end{tikzpicture}
\hspace{.1\linewidth}
\begin{tikzpicture}[-stealth,scale=.7]
    \def\r{2cm}
    \def\e{3}
    \def\steps{6}

    \foreach \x in {1,...,\steps}{
        \draw[gray!70,dashed]
            (90+360/\steps-360/\steps*\x:\r) arc (90+360/\steps-360/\steps*\x:90-360/\steps*\x+\e:\r);
        \node[circle, draw, inner sep=0pt, minimum size=4pt,fill] (N\x) at (90+360/\steps-360/\steps*\x:\r) {};
        \draw[gray!70]
            (90+360/\steps-360/\steps*\x:1.5*\r) -- (N\x);
        \draw[gray!70]
            (N\x) -- (90+360/\steps-360/\steps*\x:0.5*\r);
    }

    \foreach \x in {1,...,\steps}{
        \draw[thick] (90+360/\steps-360/\steps*\x:\r) arc (90+360/\steps-360/\steps*\x:90-360/\steps*\x+\e:\r) node [red, midway, fill = white, inner sep = .8pt] {\scriptsize $.5$} ;
    }

    \draw[thick]
            (90+360/\steps-360/\steps*1:1.5*\r) -- (N1) node [red, midway, fill = white, inner sep = .8pt] {\scriptsize $.5$};
    \draw[thick]
            (N1) -- (90+360/\steps-360/\steps*1:0.5*\r) node [red, midway, fill = white, inner sep = .8pt] {\scriptsize $.5$};
    \draw[thick]
            (90+360/\steps-360/\steps*4:1.5*\r) -- (N4) node [red, midway, fill = white, inner sep = .8pt] {\scriptsize $.5$};
    \draw[thick]
            (N4) -- (90+360/\steps-360/\steps*4:0.5*\r) node [red, midway, fill = white, inner sep = .8pt] {\scriptsize $.5$};

\end{tikzpicture}

%% file: figures/ex_bdd_object_flows.tex
\begin{tikzpicture}[-stealth,scale=.4]
    \def\r{2cm}
    \def\e{3}
    \def\steps{2}

    \draw[thick] (90+360/\steps*1-\e:\r) arc (90+360/\steps*1-\e:90-360/\steps*0+\e:\r) node [red, midway, fill = white, inner sep = 1pt] {\scriptsize $1$};
    \draw[gray!70!white] (90-360/\steps*0-\e:\r) arc (90-360/\steps*0-\e:90-360/\steps*1+\e:\r);

    \node[draw, inner sep =0pt, minimum size=4pt,fill] (A) at (90-360/\steps*0:\r) {};
    \node[circle, draw, inner sep =0pt, minimum size=4pt,fill] (B) at (90-360/\steps*1:\r) {};

    \draw[gray!70!white] (90-360/\steps*0:1.5*\r) -- (A);
    \draw[thick] (A) -- (90-360/\steps*0:0.3*\r) node [red, midway, fill = white, inner sep = 1pt] {\scriptsize $1$};
    \draw[gray!70!white] (90-360/\steps*1:1.5*\r) -- (B);
    \draw[gray!70!white] (B) -- (90-360/\steps*1:0.3*\r);
\end{tikzpicture}
\hspace{.2\linewidth}
\begin{tikzpicture}[-stealth,scale=.4]
    \def\r{2cm}
    \def\e{3}
    \def\steps{2}

    \draw[gray!70!white] (90+360/\steps*1-\e:\r) arc (90+360/\steps*1-\e:90-360/\steps*0+\e:\r);
    \draw[thick] (90-360/\steps*0-\e:\r) arc (90-360/\steps*0-\e:90-360/\steps*1+\e:\r) node [red, midway, fill = white, inner sep = 1pt] {\scriptsize $1$};

    \node[circle, draw, inner sep =0pt, minimum size=4pt,fill] (A) at (90-360/\steps*0:\r) {};
    \node[draw, inner sep =0pt, minimum size=4pt,fill] (B) at (90-360/\steps*1:\r) {};

    \draw[gray!70!white] (90-360/\steps*0:1.5*\r) -- (A);
    \draw[gray!70!white] (A) -- (90-360/\steps*0:0.3*\r);
    \draw[gray!70!white] (90-360/\steps*1:1.5*\r) -- (B);
    \draw[thick] (B) -- (90-360/\steps*1:0.3*\r) node [red, midway, fill = white, inner sep = 1pt] {\scriptsize $1$};
\end{tikzpicture}

%% file: figures/ex_bdd_object.tex
\begin{gathered}\begin{tikzpicture}[scale = 1.5]
        \draw[very thick,cyan] (0,0) -- (0,1);
        \draw[very thick,cyan] (0,0) -- (1,0);
        \draw[dashed] (0,1) -- (1,0);
        \node[circle, draw, inner sep=0pt, minimum size=4pt,fill] at (0,1) {};
        \node[circle, draw, inner sep=0pt, minimum size=4pt,fill] at (1,0) {};
        \node[circle, draw, inner sep=0pt, minimum size=3pt,fill] at (.5,.5) {};
        \node[right] at (1,0) {${\bf 1}_{\route_1}$};
        \node[above] at (0,1) {${\bf 1}_{\route_2}$};
        \node[above right] at (.5,.5) {$\tfrac{{\bf 1}_{\route_1} + {\bf 1}_{\route_2}}{2}$};
        \node[orange,circle, draw, inner sep=0pt, minimum size=4pt,fill] at (0,0) {};
\end{tikzpicture}\end{gathered}
\hspace{.15\linewidth}
\route_1 = \begin{gathered}\begin{tikzpicture}[-stealth,scale=.8]
    \def\r{1cm}
    \def\e{3}
    \def\steps{2}

    \foreach \x in {1,...,\steps}{
        \draw[gray!70!white,dashed]
            (90+360/\steps-360/\steps*\x:\r) arc (90+360/\steps-360/\steps*\x:90-360/\steps*\x+\e:\r);
        \node[circle, draw, inner sep=0pt, minimum size=4pt,fill] (N\x) at (90+360/\steps-360/\steps*\x:\r) {};
        \draw[gray!70!white]
            (90+360/\steps-360/\steps*\x:1.5*\r) -- (N\x);
        \draw[gray!70!white]
            (N\x) -- (90+360/\steps-360/\steps*\x:0.5*\r);
    }
        \draw[blue,very thick]
            (90+360/\steps-360/\steps*2:1.5*\r) -- (N2);
        \draw[red,dashed,very thick]
            (90+360/\steps-360/\steps*2:\r) arc (90+360/\steps-360/\steps*2:90-360/\steps*2+\e:\r);
        \draw[blue,very thick]
            (N1) -- (90+360/\steps-360/\steps*1:0.5*\r);            
    \foreach \x in {1,...,\steps}{
        \node[circle, draw, inner sep=0pt, minimum size=4pt,fill] (N\x) at (90+360/\steps-360/\steps*\x:\r) {};
    }    
\end{tikzpicture}\end{gathered}
\hspace{.1\linewidth}
\route_2 = \begin{gathered}\begin{tikzpicture}[-stealth,scale=.8]
    \def\r{1cm}
    \def\e{3}
    \def\steps{2}

    \foreach \x in {1,...,\steps}{
        \draw[gray!70!white,dashed]
            (90+360/\steps-360/\steps*\x:\r) arc (90+360/\steps-360/\steps*\x:90-360/\steps*\x+\e:\r);
        \node[circle, draw, inner sep=0pt, minimum size=4pt,fill] (N\x) at (90+360/\steps-360/\steps*\x:\r) {};
        \draw[gray!70!white]
            (90+360/\steps-360/\steps*\x:1.5*\r) -- (N\x);
        \draw[gray!70!white]
            (N\x) -- (90+360/\steps-360/\steps*\x:0.5*\r);
    }
        \draw[blue,very thick]
            (90+360/\steps-360/\steps*1:1.5*\r) -- (N1);
        \draw[red,dashed,very thick]
            (90+360/\steps-360/\steps*1:\r) arc (90+360/\steps-360/\steps*1:90-360/\steps*1+\e:\r);
        \draw[blue,very thick]
            (N2) -- (90+360/\steps-360/\steps*2:0.5*\r);

    \foreach \x in {1,...,\steps}{
        \node[circle, draw, inner sep=0pt, minimum size=4pt,fill] (N\x) at (90+360/\steps-360/\steps*\x:\r) {};
    }    
\end{tikzpicture}\end{gathered}

%% file: figures/n-cycle-locally-gentle.tex
\begin{tikzpicture}[-stealth,scale=.8]
    \def\r{1.5cm}
    \def\e{3}
    \def\steps{7}
    \node[circle, draw, inner sep=0pt, minimum size=4pt,fill] (A) at (90-360/\steps*0:\r) {};
    \node[circle, draw, inner sep=0pt, minimum size=4pt,fill] (B) at (90-360/\steps*1:\r) {};
    \node[circle, draw, inner sep=0pt, minimum size=4pt,fill] (C) at (90-360/\steps*2:\r) {};
    \node[circle, draw, inner sep=0pt, minimum size=4pt,fill] (D) at (90-360/\steps*3:\r) {};
    \node[circle, draw, inner sep=0pt, minimum size=4pt,fill] (N) at (90+360/\steps*1:\r) {};

    \node at (90-360/\steps*0-8:1.2*\r) {$1$};
    \node at (90-360/\steps*1-8:1.2*\r) {$2$};
    \node at (90-360/\steps*2-8:1.2*\r) {$3$};
    \node at (90-360/\steps*3-8:1.2*\r) {$4$};
    \node at (90+360/\steps*1-8:1.2*\r) {$n$};

    \draw[] (90+360/\steps*1-\e:\r) arc (90+360/\steps*1-\e:90-360/\steps*0+\e:\r);
    \draw[] (90-360/\steps*0-\e:\r) arc (90-360/\steps*0-\e:90-360/\steps*1+\e:\r);
    \draw[] (90-360/\steps*1-\e:\r) arc (90-360/\steps*1-\e:90-360/\steps*2+\e:\r);
    \draw[] (90-360/\steps*2-\e:\r) arc (90-360/\steps*2-\e:90-360/\steps*3+\e:\r);
    \draw[-,dotted]     (90-360/\steps*3-\e:\r) arc (90-360/\steps*3-\e:90-360/\steps*(\steps-1)+\e:\r);

    \draw[stealth-] (90-360/\steps*0:1.5*\r) -- (A); \draw[stealth-] (A) -- (90-360/\steps*0:0.5*\r);
    \draw[stealth-] (90-360/\steps*1:1.5*\r) -- (B); \draw[stealth-] (B) -- (90-360/\steps*1:0.5*\r);
    \draw[stealth-] (90-360/\steps*2:1.5*\r) -- (C); \draw[stealth-] (C) -- (90-360/\steps*2:0.5*\r);
    \draw[stealth-] (90-360/\steps*3:1.5*\r) -- (D); \draw[stealth-] (D) -- (90-360/\steps*3:0.5*\r);
    \draw[stealth-] (90+360/\steps*1:1.5*\r) -- (N); \draw[stealth-] (N) -- (90+360/\steps*1:0.5*\r);

    \draw[red,thick,dotted,-] (90+360/\steps*1:.8*\r) arc (-90+360/\steps*1:-90+360/\steps*1+90:.2*\r);
    \draw[red,thick,dotted,-] (90-360/\steps*0:.8*\r) arc (-90-360/\steps*0:-90-360/\steps*0+90:.2*\r);
    \draw[red,thick,dotted,-] (90-360/\steps*1:.8*\r) arc (-90-360/\steps*1:-90-360/\steps*1+90:.2*\r);
    \draw[red,thick,dotted,-] (90-360/\steps*2:.8*\r) arc (-90-360/\steps*2:-90-360/\steps*2+90:.2*\r);

    \draw[red,thick,dotted,-] (90-360/\steps*0:1.2*\r) arc (+90-360/\steps*0:+90-360/\steps*0+90:.2*\r);
    \draw[red,thick,dotted,-] (90-360/\steps*1:1.2*\r) arc (+90-360/\steps*1:+90-360/\steps*1+90:.2*\r);
    \draw[red,thick,dotted,-] (90-360/\steps*2:1.2*\r) arc (+90-360/\steps*2:+90-360/\steps*2+90:.2*\r);
    \draw[red,thick,dotted,-] (90-360/\steps*3:1.2*\r) arc (+90-360/\steps*3:+90-360/\steps*3+90:.2*\r);
\end{tikzpicture}

%% file: figures/n-cycle-locally-gentle-reduced.tex
\begin{tikzpicture}[-stealth,scale=.8]
    \def\r{1.5cm}
    \def\e{3}
    \def\steps{7}
    \node[circle, draw, inner sep=0pt, minimum size=4pt,fill] (A) at (90-360/\steps*0:\r) {};
    \node[circle, draw, inner sep=0pt, minimum size=4pt,fill] (B) at (90-360/\steps*1:\r) {};
    \node[circle, draw, inner sep=0pt, minimum size=4pt,fill] (C) at (90-360/\steps*2:\r) {};
    \node[circle, draw, inner sep=0pt, minimum size=4pt,fill] (D) at (90-360/\steps*3:\r) {};
    \node[circle, draw, inner sep=0pt, minimum size=4pt,fill] (N) at (90+360/\steps*1:\r) {};

    \node at (90-360/\steps*0-8:1.2*\r) {$1$};
    \node at (90-360/\steps*1-8:1.2*\r) {$2$};
    \node at (90-360/\steps*2-8:1.2*\r) {$3$};
    \node at (90-360/\steps*3-8:1.2*\r) {$4$};
    \node at (90+360/\steps*1-8:1.2*\r) {$n$};

    \draw[] (90+360/\steps*1-\e:\r) arc (90+360/\steps*1-\e:90-360/\steps*0+\e:\r);
    \draw[] (90-360/\steps*0-\e:\r) arc (90-360/\steps*0-\e:90-360/\steps*1+\e:\r);
    \draw[] (90-360/\steps*1-\e:\r) arc (90-360/\steps*1-\e:90-360/\steps*2+\e:\r);
    \draw[] (90-360/\steps*2-\e:\r) arc (90-360/\steps*2-\e:90-360/\steps*3+\e:\r);
    \draw[-,dotted]     (90-360/\steps*3-\e:\r) arc (90-360/\steps*3-\e:90-360/\steps*(\steps-1)+\e:\r);
\end{tikzpicture}

%% file: figures/H32_tikz.tex
\draw[red] (-1,0) arc(180:  0:1);
\draw[red] (-2,0) arc(180:  0:2);
\draw[red] (-3,0) arc(180:  0:3);
\draw[red] ( 1,0) arc(360:180:1);
\draw[red] ( 2,0) arc(360:180:2);
\draw[red] ( 3,0) arc(360:180:3);
\draw[blue] (-3.8,0) -- ( -3,0);
\draw[blue] (  -3,0) -- ( -2,0);
\draw[blue] (  -2,0) -- ( -1,0);
\draw[blue] (  -1,0) -- (-.25,0);
\draw[blue] ( 3.8,0) -- (  3,0);
\draw[blue] (   3,0) -- (  2,0);
\draw[blue] (   2,0) -- (  1,0);
\draw[blue] (   1,0) -- ( .25,0);

%% file: figures/H34_tikz.tex
\draw[dotted, very thick] (0,0) circle (1.5cm);
\draw[red] (-1,0) arc(180:  0:1);
\draw[red] (-2,0) arc(180:  0:2);
\draw[red] (-3,0) arc(180:  0:3);
\draw[red] (-4,0) arc(180:  0:4);
\draw[red] ( 1,0) arc(360:180:1);
\draw[red] ( 2,0) arc(360:180:2);
\draw[red] ( 3,0) arc(360:180:3);
\draw[red] ( 4,0) arc(360:180:4);
\draw[red, dotted, -] (0,2.5) arc[start angle=90, end angle=180, radius=2.5cm];
\draw[blue] (-4.5,0) -- ( -4,0);
\draw[blue] (  -4,0) -- ( -3,0);
\draw[blue,dotted] (  -3,0) -- ( -2,0);
\draw[blue] (  -2,0) -- ( -1,0);
\draw[blue] (  -1,0) -- (-.5,0);
\draw[blue] (4.5,0) -- (4,0);
\draw[blue] ( 4,0) -- (  3,0);
\draw[blue,dotted] (   3,0) -- (  2,0);
\draw[blue] (   2,0) -- (  1,0);
\draw[blue] (   1,0) -- ( .5,0);
\draw[blue] (3.18,3.18) -- (2.82,2.82);
\draw[blue] (2.82,2.82) -- (2.12,2.12);
\draw[blue,dotted]  (2.12,2.12) -- (1.41,1.41 );
\draw[blue] (   1.41,1.41) -- (0.7,0.7);
\draw[blue] (   0.7,0.7) -- (0.35,0.35);
\draw[blue] (3.18,-3.18) -- (2.82,-2.82);
\draw[blue] (2.82,-2.82) -- (2.12,-2.12);
\draw[blue,dotted]  (2.12,-2.12) -- (1.41,-1.41 );
\draw[blue] (   1.41,-1.41) -- (0.7,-0.7);
\draw[blue] (   0.7,-0.7) -- (0.35,-0.35);
\draw[blue] (0,-4.5) -- (0, -4);
\draw[blue] (0,-4) -- (0, -3);
\draw[blue,dotted] (0, -3) -- (0, -2);
\draw[blue] (0,  -2) -- (0, -1);
\draw[blue] (0,  -1) -- (0,-.5);
\draw[blue] (0, 4.5) -- (0,  4);
\draw[blue] (0, 4) -- (0,  3);
\draw[blue,dotted] (0,  3) -- (0,  2);
\draw[blue] (0,   2) -- (0,  1);
\draw[blue] (0,   1) -- (0, .5);
\node at (0,0) {$g$};
\node at (-3.5,3.5) {$f'$};

%% file: figures/the_non-permutahedron.tex
\begin{tikzpicture}%
	[x={(0.570104cm, -0.405562cm)},
	y={(0.821573cm, 0.281466cm)},
	z={(-0.000038cm, 0.869653cm)},
	scale=.8,
	back/.style={loosely dotted, thin},
	edge/.style={color=blue!95!black, thick},
	facet/.style={fill=blue!95!black,fill opacity=.4},
	vertex/.style={inner sep=1pt,circle,draw=green!25!black,fill=green!75!black,thick}]
\coordinate (8, 10, 8) at (8, 10, 8);
\coordinate (8, 10, 11) at (8, 10, 11);
\coordinate (8, 11, 8) at (8, 11, 8);
\coordinate (8, 11, 12) at (8, 11, 12);
\coordinate (8, 14, 11) at (8, 14, 11);
\coordinate (8, 14, 12) at (8, 14, 12);
\coordinate (9, 12, 13) at (9, 12, 13);
\coordinate (9, 15, 13) at (9, 15, 13);
\coordinate (10, 10, 11) at (10, 10, 11);
\coordinate (10, 12, 13) at (10, 12, 13);
\coordinate (10, 16, 11) at (10, 16, 11);
\coordinate (10, 16, 13) at (10, 16, 13);
\coordinate (11, 10, 8) at (11, 10, 8);
\coordinate (11, 10, 10) at (11, 10, 10);
\coordinate (11, 14, 8) at (11, 14, 8);
\coordinate (11, 16, 10) at (11, 16, 10);
\coordinate (12, 11, 8) at (12, 11, 8);
\coordinate (12, 14, 8) at (12, 14, 8);
\coordinate (13, 12, 9) at (13, 12, 9);
\coordinate (13, 12, 10) at (13, 12, 10);
\coordinate (13, 15, 9) at (13, 15, 9);
\coordinate (13, 15, 13) at (13, 15, 13);
\coordinate (13, 16, 10) at (13, 16, 10);
\coordinate (13, 16, 13) at (13, 16, 13);
\draw[edge,back] (8, 10, 8) -- (8, 11, 8);
\draw[edge,back] (8, 11, 8) -- (8, 14, 11);
\draw[edge,back] (8, 11, 8) -- (11, 14, 8);
\draw[edge,back] (8, 11, 12) -- (8, 14, 12);
\draw[edge,back] (8, 14, 11) -- (8, 14, 12);
\draw[edge,back] (8, 14, 11) -- (10, 16, 11);
\draw[edge,back] (8, 14, 12) -- (9, 15, 13);
\draw[edge,back] (10, 16, 11) -- (10, 16, 13);
\draw[edge,back] (10, 16, 11) -- (11, 16, 10);
\draw[edge,back] (11, 14, 8) -- (11, 16, 10);
\draw[edge,back] (11, 14, 8) -- (12, 14, 8);
\draw[edge,back] (11, 16, 10) -- (13, 16, 10);
\node[vertex] at (10, 16, 11)     {};
\node[vertex] at (11, 16, 10)     {};
\node[vertex] at (8, 14, 11)     {};
\node[vertex] at (8, 14, 12)     {};
\node[vertex] at (8, 11, 8)     {};
\node[vertex] at (11, 14, 8)     {};
\fill[facet] (13, 16, 13) -- (13, 15, 13) -- (13, 12, 10) -- (13, 12, 9) -- (13, 15, 9) -- (13, 16, 10) -- cycle {};
\fill[facet] (13, 16, 13) -- (10, 16, 13) -- (9, 15, 13) -- (9, 12, 13) -- (10, 12, 13) -- (13, 15, 13) -- cycle {};
\fill[facet] (10, 12, 13) -- (9, 12, 13) -- (8, 11, 12) -- (8, 10, 11) -- (10, 10, 11) -- cycle {};
\fill[facet] (11, 10, 10) -- (10, 10, 11) -- (8, 10, 11) -- (8, 10, 8) -- (11, 10, 8) -- cycle {};
\fill[facet] (13, 15, 9) -- (12, 14, 8) -- (12, 11, 8) -- (13, 12, 9) -- cycle {};
\fill[facet] (13, 12, 10) -- (11, 10, 10) -- (11, 10, 8) -- (12, 11, 8) -- (13, 12, 9) -- cycle {};
\fill[facet] (13, 15, 13) -- (10, 12, 13) -- (10, 10, 11) -- (11, 10, 10) -- (13, 12, 10) -- cycle {};
\draw[edge] (8, 10, 8) -- (8, 10, 11);
\draw[edge] (8, 10, 8) -- (11, 10, 8);
\draw[edge] (8, 10, 11) -- (8, 11, 12);
\draw[edge] (8, 10, 11) -- (10, 10, 11);
\draw[edge] (8, 11, 12) -- (9, 12, 13);
\draw[edge] (9, 12, 13) -- (9, 15, 13);
\draw[edge] (9, 12, 13) -- (10, 12, 13);
\draw[edge] (9, 15, 13) -- (10, 16, 13);
\draw[edge] (10, 10, 11) -- (10, 12, 13);
\draw[edge] (10, 10, 11) -- (11, 10, 10);
\draw[edge] (10, 12, 13) -- (13, 15, 13);
\draw[edge] (10, 16, 13) -- (13, 16, 13);
\draw[edge] (11, 10, 8) -- (11, 10, 10);
\draw[edge] (11, 10, 8) -- (12, 11, 8);
\draw[edge] (11, 10, 10) -- (13, 12, 10);
\draw[edge] (12, 11, 8) -- (12, 14, 8);
\draw[edge] (12, 11, 8) -- (13, 12, 9);
\draw[edge] (12, 14, 8) -- (13, 15, 9);
\draw[edge] (13, 12, 9) -- (13, 12, 10);
\draw[edge] (13, 12, 9) -- (13, 15, 9);
\draw[edge] (13, 12, 10) -- (13, 15, 13);
\draw[edge] (13, 15, 9) -- (13, 16, 10);
\draw[edge] (13, 15, 13) -- (13, 16, 13);
\draw[edge] (13, 16, 10) -- (13, 16, 13);
\node[vertex] at ( 8, 10,  8) {};
\node[vertex] at ( 8, 10, 11) {};
\node[vertex] at ( 8, 11, 12) {};
\node[vertex] at ( 9, 12, 13) {};
\node[vertex] at ( 9, 15, 13) {};
\node[vertex] at (10, 10, 11) {};
\node[vertex] at (10, 12, 13) {};
\node[vertex] at (10, 16, 13) {};
\node[vertex] at (11, 10,  8) {};
\node[vertex] at (11, 10, 10) {};
\node[vertex] at (12, 11,  8) {};
\node[vertex] at (12, 14,  8) {};
\node[vertex] at (13, 12,  9) {};
\node[vertex] at (13, 12, 10) {};
\node[vertex] at (13, 15,  9) {};
\node[vertex] at (13, 15, 13) {};
\node[vertex] at (13, 16, 10) {};
\node[vertex] at (13, 16, 13) {};
\node at (11.5, 12.0, 11.5) {$(12,34)$};
\node at (11.0, 14.0, 13.5) {$(1,234)$};
\node at (13.0, 14.5, 10.5) {$(123,4)$};
\node at (12.0, 10.5,  9.0) {$(2,134)$};
\node at ( 9.3, 10.0,  9.5) {$(24,13)$};
\node at ( 9.3, 10.5, 12.0) {$(124,3)$};
\end{tikzpicture}

%% file: figures/H32_12-34.tex

\draw[blue,very thick] (-3.8,0) -- ( -3,0);
\draw[blue,very thick] (  -3,0) -- ( -2,0);
\draw[red,dashed,very thick] (-2,0) arc(180:  0:2);
\draw[blue,very thick] (   2,0) -- (  1,0);
\draw[blue,very thick] (   1,0) -- ( .25,0);

\draw[gray!70!white,dashed] (-1,0) arc(180:  0:1);
\draw[gray!70!white,dashed] (-3,0) arc(180:  0:3);
\draw[gray!70!white,dashed] ( 1,0) arc(360:180:1);
\draw[gray!70!white,dashed] ( 2,0) arc(360:180:2);
\draw[gray!70!white,dashed] ( 3,0) arc(360:180:3);

\draw[gray!70!white] ( 3.8,0) -- (  3,0);
\draw[gray!70!white] (   3,0) -- (  2,0);
\draw[gray!70!white] (  -2,0) -- ( -1,0);
\draw[gray!70!white] (  -1,0) -- (-.25,0);

\node[blue] at (-3.5,.45) {\small $u_1$};
\node[blue] at (- 2.5,.45) {\small $u_2$};
\node[gray!70!white] at (- 1.5,.45) {\small $u_3$};
\node[gray!70!white] at (-  .5,.45) {\small $u_4$};
\node[gray!70!white] at ( 3.5,-.45) {\small $d_1$};
\node[gray!70!white] at (  2.5,-.45) {\small $d_2$};
\node[blue] at (  1.5,-.45) {\small $d_3$};
\node[blue] at (   .5,-.45) {\small $d_4$};

%% file: figures/H32_1-234.tex
\draw[blue,very thick] (-3.8,0) -- ( -3,0);
\draw[blue,very thick] (   3,0) -- (  2,0);
\draw[blue,very thick] (   2,0) -- (  1,0);
\draw[blue,very thick] (   1,0) -- ( .25,0);
\draw[red,dashed,very thick] (-3,0) arc(180:  0:3);

\draw[gray!70!white,dashed] (-1,0) arc(180:  0:1);
\draw[gray!70!white,dashed] (-2,0) arc(180:  0:2);

\draw[gray!70!white,dashed] ( 1,0) arc(360:180:1);
\draw[gray!70!white,dashed] ( 2,0) arc(360:180:2);
\draw[gray!70!white,dashed] ( 3,0) arc(360:180:3);

\draw[gray!70!white] (  -3,0) -- ( -2,0);
\draw[gray!70!white] (  -2,0) -- ( -1,0);
\draw[gray!70!white] (  -1,0) -- (-.25,0);
\draw[gray!70!white] ( 3.8,0) -- (  3,0);

%% file: figures/H32_123-4.tex
\draw[blue,very thick] (-3.8,0) -- ( -3,0);
\draw[blue,very thick] (  -3,0) -- ( -2,0);
\draw[blue,very thick] (  -2,0) -- ( -1,0);
\draw[red,dashed,very thick] (-1,0) arc(180:  0:1);
\draw[blue,very thick] (   1,0) -- ( .25,0);

\draw[gray!70!white,dashed] (-2,0) arc(180:  0:2);
\draw[gray!70!white,dashed] (-3,0) arc(180:  0:3);
\draw[gray!70!white,dashed] ( 1,0) arc(360:180:1);
\draw[gray!70!white,dashed] ( 2,0) arc(360:180:2);
\draw[gray!70!white,dashed] ( 3,0) arc(360:180:3);

\draw[gray!70!white] (  -1,0) -- (-.25,0);
\draw[gray!70!white] ( 3.8,0) -- (  3,0);
\draw[gray!70!white] (   3,0) -- (  2,0);
\draw[gray!70!white] (   2,0) -- (  1,0);

%% file: figures/H32_124-3.tex

\draw[blue,very thick] (-3.8,0) -- ( -3,0);
\draw[blue,very thick] (  -3,0) -- ( -2,0);
\draw[red,dashed,very thick] (-2,0) arc(180:  0:2);
\draw[blue,very thick] (   2,0) -- (  1,0);
\draw[red,dashed,very thick] ( 1,0) arc(360:180:1);
\draw[blue,very thick] (  -1,0) -- (-.25,0);

\draw[gray!70!white,dashed] (-1,0) arc(180:  0:1);
\draw[gray!70!white,dashed] (-3,0) arc(180:  0:3);
\draw[gray!70!white,dashed] ( 2,0) arc(360:180:2);
\draw[gray!70!white,dashed] ( 3,0) arc(360:180:3);

\draw[gray!70!white] (  -2,0) -- ( -1,0);
\draw[gray!70!white] ( 3.8,0) -- (  3,0);
\draw[gray!70!white] (   3,0) -- (  2,0);
\draw[gray!70!white] (   1,0) -- ( .25,0);

%% file: figures/H32_24-13.tex

\draw[blue,very thick] ( 3.8,0) -- (  3,0);
\draw[red,dashed,very thick]  ( 3,0) arc(360:180:3);
\draw[blue,very thick] (  -3,0) -- ( -2,0);
\draw[red,dashed,very thick]  (-2,0) arc(180:  0:2);
\draw[red,dashed,very thick] ( 1,0) arc(360:180:1);
\draw[blue,very thick] (   2,0) -- (  1,0);
\draw[blue,very thick] (  -1,0) -- (-.25,0);

\draw[gray!70!white,dashed] (-1,0) arc(180:  0:1);
\draw[gray!70!white,dashed] (-3,0) arc(180:  0:3);
\draw[gray!70!white,dashed]  ( 2,0) arc(360:180:2);
\draw[gray!70!white] (-3.8,0) -- ( -3,0);
\draw[gray!70!white] (  -2,0) -- ( -1,0);
\draw[gray!70!white] (   3,0) -- (  2,0);
\draw[gray!70!white] (   1,0) -- ( .25,0);

%% file: figures/H32_2-134.tex

\draw[blue,very thick] ( 3.8,0) -- (  3,0);
\draw[red,dashed,very thick] ( 3,0) arc(360:180:3);
\draw[blue,very thick] (  -3,0) -- ( -2,0);
\draw[red,dashed,very thick] (-2,0) arc(180:  0:2);
\draw[blue,very thick] (   2,0) -- (  1,0);
\draw[blue,very thick] (   1,0) -- ( .25,0);

\draw[gray!70!white, dashed](-1,0) arc(180:  0:1);
\draw[gray!70!white, dashed](-3,0) arc(180:  0:3);
\draw[gray!70!white, dashed]( 1,0) arc(360:180:1);
\draw[gray!70!white, dashed]( 2,0) arc(360:180:2);

\draw[gray!70!white](-3.8,0) -- ( -3,0);
\draw[gray!70!white] (   3,0) -- (  2,0);
\draw[gray!70!white] (  -2,0) -- ( -1,0);
\draw[gray!70!white] (  -1,0) -- (-.25,0);

%% file: references.bib
@article{Adachi,
    AUTHOR  = {Adachi, Takahide},
    TITLE   = {The classifications of $\tau$-tilting modules over {N}akayama algebras},
    JOURNAL = {J. Algebra},
    NOTE    = {\href{https://doi.org/10.1016/j.jalgebra.2015.12.013}{\texttt{\textsc{Doi}}}},
    PAGE    = {227-262},
    VOLUME  = {450},
    YEAR    = {2016},
}

@article{adachi-iyama-reiten,
    AUTHOR  = {Adachi, Takahide and Iyama, Osamu and Reiten, Idun},
    TITLE   = {{$\tau$}-tilting theory},
    JOURNAL = {Compositio Mathematica},
    ISSUE   = {3},
    NOTE    = {\href{https://doi.org/10.1112/S0010437X13007422}{\texttt{\textsc{Doi}}}},
    PAGES   = {415--452},
    VOLUME  = {150},
    YEAR    = {2014},
}

@misc{aoki2024fanspolytopestiltingtheory,
    AUTHOR = {Aoki, Toshitaka and Higashitani, Akihiro and Iyama, Osamu and Kase, Ryoichi and Mizuno, Yuya},
    TITLE  = {Fans and polytopes in tilting theory {I}: Foundations},
    NOTE   = {Preprint, \href{http://arxiv.org/abs/2203.15213}{\texttt{arXiv:2203.15213}}},
    YEAR   = {2022},
}

@article{aokiyurikusa,
    AUTHOR  = {Aoki, Toshitaka and Yurikusa, Toshiya},
    TITLE   = {Complete gentle and special biserial algebras are $g$-tame},
    JOURNAL = {J. Algebraic. Combin},
    NOTE    = {\href{https://doi.org/10.1007/s10801-023-01216-8}{\texttt{\textsc{Doi}}}},
    NUMBER  = {4},
    PAGES   = {1103--1137},
    VOLUME  = {57},
    YEAR    = {2023},
}

@article{AHL,
    AUTHOR  = {Arkani-Hamed, Nima and He, Song and Lam, Thomas},
    TITLE   = {Cluster configuration spaces of finite type},
    JOURNAL = {SIGMA},
    NOTE    = {\href{https://doi.org/10.3842/SIGMA.2021.092}{\texttt{\textsc{Doi}}}},
    PAGES   = {Paper No. 092, 41},
    VOLUME  = {17},
    YEAR    = {2021},
}

@article{A21,
    AUTHOR  = {Asai, Sota},
    TITLE   = {The wall-chamber structures of the real {G}rothendieck groups},
    JOURNAL = {Advances in Mathematics},
    NOTE    = {\href{https://doi.org/10.1016/j.aim.2021.107615}{\texttt{\textsc{Doi}}}},
    PAGES   = {107615},
    VOLUME  = {381},
    YEAR    = {2021},
}

@book{ASS06,
    AUTHOR     = {Assem, Ibrahim and Skowro{\'{n}}ski, Andrzej and Simson, Daniel},
    TITLE      = {Elements of the Representation Theory of Associative Algebras: Techniques of Representation Theory},
    COLLECTION = {London Mathematical Society Student Texts},
    NOTE       = {\href{https://doi.org/10.1017/CBO9780511614309}{\texttt{\textsc{Doi}}}},
    PLACE      = {Cambridge},
    PUBLISHER  = {Cambridge University Press},
    SERIES     = {London Mathematical Society Student Texts},
    VOLUME     = {1},
    YEAR       = {2006},
}

@article{BV,
    AUTHOR   = {Baldoni, Welleda and Vergne, Mich\`ele},
    TITLE    = {Kostant partitions functions and flow polytopes},
    FJOURNAL = {Transformation Groups},
    JOURNAL  = {Transform. Groups},
    NOTE     = {\href{https://doi.org/10.1007/s00031-008-9019-8}{\texttt{\textsc{Doi}}}},
    NUMBER   = {3-4},
    PAGES    = {447--469},
    VOLUME   = {13},
    YEAR     = {2008},
}

@article{BKT,
    AUTHOR   = {Baumann, Pierre and Kamnitzer, Joel and Tingley, Peter},
    TITLE    = {Affine {M}irkovi\'{c}--{V}ilonen polytopes},
    FJOURNAL = {Publications Math\'{e}matiques. Institut de Hautes \'{E}tudes
                Scientifiques},
    JOURNAL  = {Publ. Math. Inst. Hautes \'{E}tudes Sci.},
    NOTE     = {\href{https://doi.org/10.1007/s10240-013-0057-y}{\texttt{\textsc{Doi}}}},
    PAGES    = {113--205},
    VOLUME   = {120},
    YEAR     = {2014},
}

@phdthesis{BM20,
    AUTHOR = {Bazier-Matte, Véronique},
    TITLE  = {Combinatoire des algèbres amassées},
    NOTE   = {Available at \url{https://archipel.uqam.ca/14297/}},
    SCHOOL = {Université du Québec à Montréal},
    YEAR   = {2020},
}

@article{Kentuckygentle,
    AUTHOR   = {Bell, Matias von and Braun, Benjamin and Bruegge, Kaitlin and
                Hanely, Derek and Peterson, Zachery and Serhiyenko, Khrystyna
                and Yip, Martha},
    TITLE    = {Triangulations of flow polytopes, ample framings, and gentle
                algebras},
    FJOURNAL = {Selecta Mathematica. New Series},
    JOURNAL  = {Selecta Math. (N.S.)},
    NOTE     = {\href{https://doi.org/10.1007/s00029-024-00942-6}{\texttt{\textsc{Doi}}}},
    NUMBER   = {3},
    PAGES    = {Paper No. 55, 34},
    VOLUME   = {30},
    YEAR     = {2024},
}

@misc{vonBellCeballos,
    AUTHOR = {Bell, Matias von and Ceballos, Cesar},
    TITLE  = {Framing lattices and flow polytopes},
    NOTE   = {Preprint, \href{https://arxiv.org/abs/2512.20575}{\texttt{arXiv:2512.20575}}},
    YEAR   = {2025},
}

@article{vonBellGonzalezCetinaYip,
    AUTHOR   = {Bell, Matias von and Gonz{\'a}lez D'Le{\'o}n, Rafael S. and Mayorga Cetina, Francisco A. and Yip, Martha},
    TITLE    = {On framed triangulations of flow polytopes, the {{\(\nu\)}}-{T}amari lattice and {Y}oung's lattice},
    FJOURNAL = {S{\'e}minaire Lotharingien de Combinatoire},
    JOURNAL  = {S{\'e}min. Lothar. Comb.},
    NOTE     = {\href{https://www.mat.univie.ac.at/~slc/wpapers/FPSAC2021/42.html}{\texttt{\textsc{Link}}}},
    NUMBER   = {Art. 42},
    PAGES    = {12},
    VOLUME   = {85B},
    YEAR     = {2021},
}

@misc{BERGGREN,
    AUTHOR = {Berggren, Jonah},
    TITLE  = {Flows on gentle algebras},
    NOTE   = {Preprint, \href{http://arxiv.org/abs/2507.12688}{\texttt{arXiv:2507.12688}}},
    YEAR   = {2025},
}

@misc{berggren2025framingtriangulations,
    AUTHOR = {Berggren, Jonah},
    TITLE  = {Framing Triangulations and Framing Posets of Planar {DAG}s with Nontrivial Netflow Vectors},
    NOTE   = {Preprint, \href{http://arxiv.org/abs/2507.12684}{\texttt{arXiv:2507.12684}}},
    YEAR   = {2025},
}

@article{BerSer_Wilting,
    AUTHOR   = {Berggren, Jonah and Serhiyenko, Khrystyna},
    TITLE    = {Wilting theory of flow polytopes},
    FJOURNAL = {S{\'e}minaire Lotharingien de Combinatoire},
    JOURNAL  = {S{\'e}min. Lothar. Comb.},
    NOTE     = {\href{https://www.mat.univie.ac.at/~slc/wpapers/FPSAC2024/21.html}{\texttt{\textsc{Link}}}},
    NUMBER   = {Art. 21},
    PAGES    = {12},
    VOLUME   = {91B},
    YEAR     = {2024},
}

@article{BDMTY20,
    AUTHOR   = {Br\"{u}stle, Thomas and Douville, Guillaume and Mousavand,
                Kaveh and Thomas, Hugh and Y\i{}ld\i{}r\i{}m, Emine},
    TITLE    = {On the combinatorics of gentle algebras},
    FJOURNAL = {Canadian Journal of Mathematics. Journal Canadien de
                Math\'{e}matiques},
    JOURNAL  = {Canad. J. Math.},
    NOTE     = {\href{https://doi.org/10.4153/s0008414x19000397}{\texttt{\textsc{Doi}}}},
    NUMBER   = {6},
    PAGES    = {1551--1580},
    VOLUME   = {72},
    YEAR     = {2020},
}

@misc{braun2025equatorialflowtriangulationsgorenstein,
    AUTHOR = {Braun, Benjamin and Cornejo, Alvaro},
    TITLE  = {Equatorial Flow Triangulations of {G}orenstein Flow Polytopes},
    NOTE   = {Preprint, \href{http://arxiv.org/abs/2408.05320}{\texttt{arXiv:2408.05320}}},
    YEAR   = {2024},
}

@misc{braun2024volumeinequalitiesflowpolytopes,
    AUTHOR = {Braun, Benjamin and McElroy, James Ford},
    TITLE  = {Volume inequalities for flow polytopes of full directed acyclic graphs},
    NOTE   = {Preprint, \href{http://arxiv.org/abs/2405.02433}{\texttt{arXiv:2405.02433}}},
    YEAR   = {2024},
}

@article{BR87,
    AUTHOR    = {Butler, Michael Charles Richard and Ringel, Claus Michael},
    TITLE     = {{A}uslander--{R}eiten sequences with few middle terms and applications to string algebras},
    JOURNAL   = {Communications in Algebra},
    NOTE      = {\href{https://doi.org/10.1080/00927878708823416}{\texttt{\textsc{Doi}}}},
    NUMBER    = {1-2},
    PAGES     = {145--179},
    PUBLISHER = {Taylor & Francis},
    VOLUME    = {15},
    YEAR      = {1987},
}

@article{CKM,
    AUTHOR    = {Corteel, Sylvie and Kim, Jang Soo and M{\'e}sz{\'a}ros, Karola},
    TITLE     = {Flow polytopes with {C}atalan volumes},
    JOURNAL   = {C. R. Math. Acad. Sci. Paris},
    NOTE      = {\href{https://doi.org/10.1016/j.crma.2017.01.007}{\texttt{\textsc{Doi}}}},
    NUMBER    = {3},
    PAGES     = {248--259},
    PUBLISHER = {Elsevier},
    VOLUME    = {355},
    YEAR      = {2017},
}

@article{CB89,
    AUTHOR  = {Crawley-Boevey, William},
    TITLE   = {Maps between representations of zero-relation algebras},
    JOURNAL = {Journal of Algebra},
    NOTE    = {\href{https://doi.org/10.1016/0021-8693(89)90304-9}{\texttt{\textsc{Doi}}}},
    NUMBER  = {2},
    PAGES   = {259--263},
    VOLUME  = {126},
    YEAR    = {1989},
}

@article{CB18,
    AUTHOR  = {Crawley-Boevey, William},
    TITLE   = {Classification of modules for infinite-dimensional string algebras},
    JOURNAL = {Transactions of the American Mathematical Society},
    ISSUE   = {5},
    NOTE    = {\href{https://doi.org/10.1090/tran/7032}{\texttt{\textsc{Doi}}}},
    PAGES   = {3289--3313},
    VOLUME  = {370},
    YEAR    = {2018},
}

@incollection{DKK,
    AUTHOR    = {Danilov, Vladimir I. and Karzanov, Alexander V. and Koshevoy, Gleb A.},
    TITLE     = {Coherent fans in the space of flows in framed graphs},
    BOOKTITLE = {24th {I}nternational {C}onference on {F}ormal {P}ower {S}eries
                 and {A}lgebraic {C}ombinatorics ({FPSAC} 2012)},
    NOTE      = {\href{https://doi.org/10.46298/dmtcs.3056}{\texttt{\textsc{Doi}}}},
    PAGES     = {481--490},
    PUBLISHER = {Assoc. Discrete Math. Theor. Comput. Sci., Nancy},
    SERIES    = {Discrete Math. Theor. Comput. Sci. Proc., AR},
    YEAR      = {2012},
}

@book{triangulations_book,
    AUTHOR    = {De Loera, Jes\'us A. and Rambau, J\"org and Santos, Francisco},
    TITLE     = {Triangulations},
    NOTE      = {\href{https://doi.org/10.1007/978-3-642-12971-1}{\texttt{\textsc{Doi}}}},
    PAGES     = {xiv+535},
    PUBLISHER = {Springer-Verlag, Berlin},
    SERIES    = {Algorithms and Computation in Mathematics},
    VOLUME    = {25},
    YEAR      = {2010},
}

@article{DIJ,
    AUTHOR   = {Demonet, Laurent and Iyama, Osamu and Jasso, Gustavo},
    TITLE    = {{$\tau$}-tilting finite algebras, bricks, and {$g$}-vectors},
    FJOURNAL = {International Mathematics Research Notices. IMRN},
    JOURNAL  = {Int. Math. Res. Not. IMRN},
    NOTE     = {\href{https://doi.org/10.1093/imrn/rnx135}{\texttt{\textsc{Doi}}}},
    NUMBER   = {3},
    PAGES    = {852--892},
    YEAR     = {2019},
}

@article{derksen-fei,
    AUTHOR   = {Derksen, Harm and Fei, Jiarui},
    TITLE    = {General presentations of algebras},
    FJOURNAL = {Advances in Mathematics},
    JOURNAL  = {Adv. Math.},
    NOTE     = {\href{https://doi.org/10.1016/j.aim.2015.03.012}{\texttt{\textsc{Doi}}}},
    PAGES    = {210--237},
    VOLUME   = {278},
    YEAR     = {2015},
}

@article{edelman84regions,
    AUTHOR   = {Edelman, Paul H.},
    TITLE    = {A partial order on the regions of {${\bf R}^{n}$} dissected by hyperplanes},
    FJOURNAL = {Transactions of the American Mathematical Society},
    JOURNAL  = {Trans. Amer. Math. Soc.},
    NOTE     = {\href{https://doi.org/10.2307/1999150}{\texttt{\textsc{Doi}}}},
    NUMBER   = {2},
    PAGES    = {617--631},
    VOLUME   = {283},
    YEAR     = {1984},
}

@article{JF2,
    AUTHOR   = {Fei, Jiarui},
    TITLE    = {Tropical {$F$}-polynomials and general presentations},
    FJOURNAL = {Journal of the London Mathematical Society. Second Series},
    JOURNAL  = {J. Lond. Math. Soc. (2)},
    NOTE     = {\href{https://doi.org/10.1112/jlms.12734}{\texttt{\textsc{Doi}}}},
    NUMBER   = {6},
    PAGES    = {2079--2120},
    VOLUME   = {107},
    YEAR     = {2023},
}

@misc{FMP,
    AUTHOR = {Ferroni, Luis and Morales, Alejandro H. and Panova, Greta},
    TITLE  = {Skew shapes, {E}hrhart positivity and beyond},
    NOTE   = {Preprint, \href{http://arxiv.org/abs/2503.16403}{\texttt{arXiv:2503.16403}}},
    YEAR   = {2025},
}

@article{s-permutahedra,
    AUTHOR   = {Gonz{\'a}lez D'Le{\'o}n, Rafael S. and Morales, Alejandro H. and Philippe, Eva and Tamayo Jim{\'e}nez, Daniel and Yip, Martha},
    TITLE    = {Realizing the s-permutahedron via flow polytopes},
    FJOURNAL = {Transactions of the American Mathematical Society},
    JOURNAL  = {Trans. Am. Math. Soc.},
    ISSN     = {0002-9947},
    NOTE     = {\href{https://doi.org/10.1090/tran/9461}{\texttt{\textsc{Doi}}}},
    NUMBER   = {11},
    PAGES    = {7625--7666},
    VOLUME   = {378},
    YEAR     = {2025},
}

@misc{GonzalezHanusaYip,
    AUTHOR = {González D'León, Rafael S. and Hanusa, Christopher R. H. and Yip, Martha},
    TITLE  = {Permutation Flows {I}: Triangulations of Flow Polytopes (Research Announcement)},
    NOTE   = {Preprint, \href{http://arxiv.org/abs/2512.04078}{\texttt{arXiv:2512.04078}}},
    YEAR   = {2025},
}

@article{Hille,
    AUTHOR  = {Hille, Lutz},
    TITLE   = {Quivers, cones and polytopes},
    JOURNAL = {Linear Algebra and its Applications},
    NOTE    = {Special Issue on Linear Algebra Methods in Representation Theory. \href{https://doi.org/10.1016/S0024-3795(02)00406-8}{\texttt{\textsc{Doi}}}},
    PAGES   = {215--237},
    VOLUME  = {365},
    YEAR    = {2003},
}

@article{HPS18,
    AUTHOR    = {Hohlweg, Christophe and Pilaud, Vincent and Stella, Salvatore},
    TITLE     = {Polytopal realizations of finite type g-vector fans},
    JOURNAL   = {Advances in Mathematics},
    NOTE      = {\href{https://doi.org/10.1016/j.aim.2018.01.019}{\texttt{\textsc{Doi}}}},
    PAGES     = {713--749},
    PUBLISHER = {Elsevier BV},
    VOLUME    = {328},
    YEAR      = {2018},
}

@article{JangKim,
    AUTHOR   = {Jang, Jihyeug and Kim, Jang Soo},
    TITLE    = {Volumes of flow polytopes related to caracol graphs},
    FJOURNAL = {Electronic Journal of Combinatorics},
    JOURNAL  = {Electron. J. Combin.},
    NOTE     = {\href{https://doi.org/10.37236/9187}{\texttt{\textsc{Doi}}}},
    NUMBER   = {4},
    PAGES    = {Paper No. 4.21, 21},
    VOLUME   = {27},
    YEAR     = {2020},
}

@article{KPY23,
    AUTHOR    = {Keller, Bernhard and Plamondon, Pierre-Guy and Yurikusa, Toshiya},
    TITLE     = {Erratum to: {T}ame algebras have dense g-vector fans},
    JOURNAL   = {Int. Mathem. Res. Not.},
    NOTE      = {\href{https://doi.org/10.1093/imrn/rnab210}{\texttt{\textsc{Doi}}}},
    NUMBER    = {4},
    PAGES     = {3598--3600},
    PUBLISHER = {Oxford University Press (OUP)},
    VOLUME    = {2023},
    YEAR      = {2023},
}

@article{MeszarosSubword,
    AUTHOR   = {M\'{e}sz\'{a}ros, Karola},
    TITLE    = {Pipe dream complexes and triangulations of root polytopes belong together},
    FJOURNAL = {SIAM Journal on Discrete Mathematics},
    JOURNAL  = {SIAM J. Discrete Math.},
    NOTE     = {\href{https://doi.org/10.1137/15M1014802}{\texttt{\textsc{Doi}}}},
    NUMBER   = {1},
    PAGES    = {100--111},
    VOLUME   = {30},
    YEAR     = {2016},
}

@article{MeszarosMorales,
    AUTHOR   = {M\'{e}sz\'{a}ros, Karola and Morales, Alejandro H.},
    TITLE    = {Volumes and {E}hrhart polynomials of flow polytopes},
    FJOURNAL = {Mathematische Zeitschrift},
    JOURNAL  = {Math. Z.},
    NOTE     = {\href{https://doi.org/10.1007/s00209-019-02283-z}{\texttt{\textsc{Doi}}}},
    NUMBER   = {3-4},
    PAGES    = {1369--1401},
    VOLUME   = {293},
    YEAR     = {2019},
}

@article{MMSt,
    AUTHOR   = {M\'{e}sz\'{a}ros, Karola and Morales, Alejandro H. and Striker, Jessica},
    TITLE    = {On flow polytopes, order polytopes, and certain faces of the
                alternating sign matrix polytope},
    FJOURNAL = {Discrete \& Computational Geometry. An International Journal
                of Mathematics and Computer Science},
    JOURNAL  = {Discrete Comput. Geom.},
    NOTE     = {\href{https://doi.org/10.1007/s00454-019-00073-2}{\texttt{\textsc{Doi}}}},
    NUMBER   = {1},
    PAGES    = {128--163},
    VOLUME   = {62},
    YEAR     = {2019},
}

@article{Meszaros-StDizier,
    AUTHOR   = {M\'{e}sz\'{a}ros, Karola and St. Dizier, Avery},
    TITLE    = {From generalized permutahedra to {G}rothendieck polynomials
                via flow polytopes},
    FJOURNAL = {Algebraic Combinatorics},
    JOURNAL  = {Algebr. Comb.},
    NOTE     = {\href{https://doi.org/10.5802/alco.136}{\texttt{\textsc{Doi}}}},
    NUMBER   = {5},
    PAGES    = {1197--1229},
    VOLUME   = {3},
    YEAR     = {2020},
}

@phdthesis{mandel82thesis,
    AUTHOR = {Mandel, Arnaldo},
    TITLE  = {Topology of oriented matroids},
    NOTE   = {Available at \url{https://www.proquest.com/docview/303277038}},
    SCHOOL = {University of Waterloo},
    YEAR   = {1982},
}

@article{Mou,
    AUTHOR   = {Mousavand, Kaveh},
    TITLE    = {{$\tau$}-tilting finiteness of biserial algebras},
    FJOURNAL = {Algebras and Representation Theory},
    JOURNAL  = {Algebr. Represent. Theory},
    NOTE     = {\href{https://doi.org/10.1007/s10468-022-10170-1}{\texttt{\textsc{Doi}}}},
    NUMBER   = {6},
    PAGES    = {2485--2522},
    VOLUME   = {26},
    YEAR     = {2023},
}

@article{PPPP23,
    AUTHOR  = {Padrol, Arnau and Palu, Yann and Pilaud, Vincent and Plamondon, Pierre-Guy},
    TITLE   = {Associahedra for finite-type cluster algebras and minimal relations between g-vectors},
    JOURNAL = {Proceedings of the London Mathematical Society},
    NOTE    = {\href{https://doi.org/10.1112/plms.12543}{\texttt{\textsc{Doi}}}},
    NUMBER  = {3},
    PAGES   = {513--588},
    VOLUME  = {127},
    YEAR    = {2023},
}

@article{ppr20shard,
    AUTHOR   = {Padrol, Arnau and Pilaud, Vincent and Ritter, Julian},
    TITLE    = {Shard polytopes},
    FJOURNAL = {International Mathematics Research Notices. IMRN},
    JOURNAL  = {Int. Math. Res. Not. IMRN},
    NOTE     = {\href{https://doi.org/10.1093/imrn/rnac042}{\texttt{\textsc{Doi}}}},
    NUMBER   = {9},
    PAGES    = {7686--7796},
    YEAR     = {2023},
}

@article{PPPlocg,
    AUTHOR    = {Palu, Yann and Pilaud, Vincent and Plamondon, Pierre-Guy},
    TITLE     = {Non-kissing and non-crossing complexes for locally gentle algebras},
    JOURNAL   = {Journal of Combinatorial Algebra},
    NOTE      = {\href{https://doi.org/10.4171/jca/35}{\texttt{\textsc{Doi}}}},
    NUMBER    = {4},
    PAGES     = {401--438},
    PUBLISHER = {European Mathematical Society - EMS - Publishing House GmbH},
    VOLUME    = {3},
    YEAR      = {2019},
}

@article{PPP,
    AUTHOR   = {Palu, Yann and Pilaud, Vincent and Plamondon, Pierre-Guy},
    TITLE    = {Non-kissing complexes and tau-tilting for gentle algebras},
    FJOURNAL = {Memoirs of the American Mathematical Society},
    JOURNAL  = {Mem. Amer. Math. Soc.},
    NOTE     = {\href{https://doi.org/10.1090/memo/1343}{\texttt{\textsc{Doi}}}},
    NUMBER   = {1343},
    PAGES    = {vii+110},
    VOLUME   = {274},
    YEAR     = {2021},
}

@article{pps,
    AUTHOR  = {Petersen, T. Kyle and Pylyavskyy, Pavlo and Speyer, David E},
    TITLE   = {A noncrossing standard monomial theory},
    JOURNAL = {J. Algebra},
    NOTE    = {\href{https://doi.org/10.1016/j.jalgebra.2010.05.001}{\texttt{\textsc{Doi}}}},
    PAGES   = {951--969},
    VOLUME  = {324},
    YEAR    = {2010},
}

@article{Pla,
    AUTHOR   = {Plamondon, Pierre-Guy},
    TITLE    = {{$\tau$}-tilting finite gentle algebras are representation-finite},
    FJOURNAL = {Pacific Journal of Mathematics},
    JOURNAL  = {Pacific J. Math.},
    NOTE     = {\href{https://doi.org/10.2140/pjm.2019.302.709}{\texttt{\textsc{Doi}}}},
    NUMBER   = {2},
    PAGES    = {709--716},
    VOLUME   = {302},
    YEAR     = {2019},
}

@article{ReadingShards,
    AUTHOR   = {Reading, Nathan},
    TITLE    = {Lattice and order properties of the poset of regions in a hyperplane arrangement},
    FJOURNAL = {Algebra Universalis},
    JOURNAL  = {Algebra Universalis},
    NOTE     = {\href{https://doi.org/10.1007/s00012-003-1834-0}{\texttt{\textsc{Doi}}}},
    NUMBER   = {2},
    PAGES    = {179--205},
    VOLUME   = {50},
    YEAR     = {2003},
}

@article{ReadingArcs,
    AUTHOR   = {Reading, Nathan},
    TITLE    = {Noncrossing arc diagrams and canonical join representations},
    FJOURNAL = {SIAM Journal on Discrete Mathematics},
    JOURNAL  = {SIAM J. Discrete Math.},
    NOTE     = {\href{https://doi.org/10.1137/140972391}{\texttt{\textsc{Doi}}}},
    NUMBER   = {2},
    PAGES    = {736--750},
    VOLUME   = {29},
    YEAR     = {2015},
}

@article{RietschWilliams,
    AUTHOR   = {Rietsch, Konstanze and Williams, Lauren Kiyomi},
    TITLE    = {Root, flow and order polytopes with connections to toric geometry},
    FJOURNAL = {Forum of Mathematics, Sigma},
    JOURNAL  = {Forum Math. Sigma},
    ISSN     = {2050-5094},
    NOTE     = {\href{https://doi.org/10.1017/fms.2025.10056}{\texttt{\textsc{Doi}}}},
    PAGES    = {35},
    VOLUME   = {13},
    YEAR     = {2025},
}

@article{stanley-order-polytopes,
    AUTHOR   = {Stanley, Richard P.},
    TITLE    = {Two poset polytopes},
    FJOURNAL = {Discrete \& Computational Geometry. An International Journal
                of Mathematics and Computer Science},
    JOURNAL  = {Discrete Comput. Geom.},
    NOTE     = {\href{https://doi.org/10.1007/BF02187680}{\texttt{\textsc{Doi}}}},
    NUMBER   = {1},
    PAGES    = {9--23},
    VOLUME   = {1},
    YEAR     = {1986},
}

@misc{PS,
    AUTHOR = {Stanley, Richard P. and Postnikov, Alexander},
    TITLE  = {Acyclic flow polytopes and {K}ostant's partition function},
    NOTE   = {Available at \url{http://www-math.mit.edu/~rstan/transparencies/kostant.ps}},
    YEAR   = {2000},
}

@manual{sagemath,
    AUTHOR = {{The Sage Developers}},
    TITLE  = {{S}ageMath, the {S}age {M}athematics {S}oftware {S}ystem ({V}ersion 10.6)},
    KEY    = {SageMath},
    NOTE   = {\url{https://www.sagemath.org}},
    YEAR   = {2025},
}

@misc{Sage-Combinat,
    AUTHOR = {{The Sage-Combinat Community}},
    TITLE  = {{S}age-{C}ombinat: enhancing {S}age as a toolbox for computer exploration in algebraic combinatorics},
    NOTE   = {\url{http://combinat.sagemath.org}},
    YEAR   = 2008,
}

@misc{Yip2019,
    AUTHOR = {Yip, Martha},
    TITLE  = {A {F}uss--{C}atalan variation of the caracol flow polytope},
    NOTE   = {Preprint, \href{https://arxiv.org/abs/1910.10060}{\texttt{arXiv:1910.10060}}},
    YEAR   = {2019},
}

@article{Z,
    AUTHOR   = {Zeilberger, Doron},
    TITLE    = {Proof of a conjecture of {C}han, {R}obbins, and {Y}uen},
    FJOURNAL = {Electronic Transactions on Numerical Analysis},
    JOURNAL  = {Electron. Trans. Numer. Anal.},
    NOTES    = {Special volume on Orthogonal polynomials: numerical and symbolic algorithms},
    PAGES    = {147--148},
    VOLUME   = {9},
    YEAR     = {1999},
}

@book{zieglerlectures,
    AUTHOR    = {Ziegler, G{\"u}nter M.},
    TITLE     = {Lectures on polytopes},
    NOTE      = {\href{https://doi.org/10.1007/978-1-4613-8431-1}{\texttt{\textsc{Doi}}}},
    PAGES     = {x+370},
    PUBLISHER = {Springer-Verlag, New York},
    SERIES    = {Graduate Texts in Mathematics},
    VOLUME    = {152},
    YEAR      = {1995},
}
